\documentclass[12pt]{amsart}  
 
\usepackage[all,cmtip,2cell]{xy}

\usepackage{graphicx,amsfonts,amssymb,stmaryrd,amscd,amsmath,latexsym,amsbsy,hyperref,eufrak}

\newtheorem{theorem}{Theorem}[section]
\newtheorem{lemma}[theorem]{Lemma}
\newtheorem{proposition}[theorem]{Proposition}
\newtheorem{corollary}[theorem]{Corollary}
\theoremstyle{definition}
\newtheorem{definition}[theorem]{Definition}

\newtheorem{example}[theorem]{Example}
\newtheorem{exercise}[theorem]{Exercise}

\newtheorem{remark}[theorem]{Remark}

\newcommand{\norm}[1]{\left\lVert#1\right\rVert}

\newcommand{\Ext}{\text{Ext}}

\newcommand{\Tr}{\text{Tr}}

\newcommand{\Ker}{\text{Ker\,}}

\newcommand{\End}{\text{End}}
\newcommand{\Hom}{\text{Hom}}

\newcommand{\Aut}{\text{Aut}}

\newcommand{\Rep}{\text{Rep}}

\newcommand{\diag}{\text{diag}}

\renewcommand{\O}{\mathcal{O}}

\newcommand{\g}{\mathfrak{g}}
\newcommand{\kf}{\mathfrak{k}}
\newcommand{\p}{\mathfrak{p}}
\newcommand{\h}{\mathfrak{h}}
\newcommand{\n}{\mathfrak{n}}
\renewcommand{\b}{\mathfrak{b}}

\newcommand{\C}{\mathcal{C}}

\newcommand{\ben}{\begin{enumerate}}
\newcommand{\een}{\end{enumerate}}

\theoremstyle{plain}

\newtheorem*{sol}{Solution}

\theoremstyle{definition}

\theoremstyle{remark}

\newcommand{\solu}[1]{\begin{sol}{\bf (\ref{#1})}}

\def\g{\mathfrak{g}}

\def\sl{\mathfrak{sl}}

\def\C{\mathcal{C}}

\def\Aut{\mathop{\mathrm{Aut}}\nolimits}
\def\h{\mathfrak{h}}

\def\N{\mathcal{N}}
\def\M{\mathcal{M}}

\def\End{\mathrm{End}}
\def\Hom{\mathrm{Hom}}
\def\Ker{\mathrm{Ker}}

\def\n{\mathfrak{n}}

\def\B{\mathcal{B}}

\def\p{\mathfrak{p}}
\def\b{\mathfrak{b}}

\def\O{\mathcal{O}}

\def\Ext{\mathop{\mathrm{Ext}}\nolimits}

\def\A{\mathcal{A}}

\def\Rep{\mathop{\mathrm{Rep}}\nolimits}

\def\gr{\mathop{\mathrm{gr}}\nolimits}

\begin{document}

\title{Representations of Lie groups}

\author{Pavel Etingof}

\address{Department of Mathematics, MIT, Cambridge, MA 02139, USA}

\maketitle 

\centerline{\bf To David Vogan on his 70th birthday with admiration} 

\tableofcontents

\section*{Introduction} 

These notes are based on the course ``Representations of Lie groups" taught by the author at MIT in Fall 2021 and Fall 2023. This is the third semester of Lie theory, 
which follows the standard 2-semester introductory sequence, ``Lie groups and Lie algebras I, II" (for the author's notes of these courses, see \cite{E}). The notes cover the basic theory of representations of non-compact semisimple Lie groups, with a more in-depth study of (non-holomorphic) representations of complex groups. 

Representation theory of (non-compact) semisimple Lie groups is 
an important and deep area of Lie theory with numerous applications, ranging from 
physics (quantum field theory) to analysis (harmonic analysis on homogeneous spaces) and number theory (the theory of automorphic forms, Langlands program). It is a synthetic subject which, besides basic Lie theory and representation theory, uses a plethora of techniques from many other fields, notably analysis, commutative and non-commutative algebra, category theory, homological algebra and algebraic geometry. For basic Lie theory (in particular, structure and finite-dimensional representations of 
compact Lie groups and semisimple complex Lie algebras) we will rely on the notes \cite{E}. In other areas, most of the time we review the necessary material before using it, but at least a superficial previous familiarity with these subjects will be helpful to the reader. 

Representation theory of semisimple Lie groups originated in the work of Bargmann \cite{Ba} and Gelfand-Naimark \cite{GN} in the late 1940s in the case of $SL_2$ and was systematically developed by Harish-Chandra in 1950s and 1960s, and later developed by many prominent mathematicians, notably R. Langlands, who classified the irreducible representations (\cite{La}). A new conceptual approach to the representation theory of complex semisimple Lie groups was proposed by J. Bernstein and S. Gelfand around 1980 (\cite{BG}) based on their previous pioneering work with I. Gelfand on category $\mathcal O$. This ultimately made the representation theory of semisimple Lie groups a part of the new subject of geometric representation theory that emerged in early 1980s from the Kazhdan-Lusztig conjecture on multiplicities in category $\mathcal O$ and its proof using geometric methods ($D$-modules and perverse sheaves) by Beilinson-Bernstein and Brylinski-Kashiwara. Since that time the theory has made giant strides forward (for example, computation of irreducible characters of real reductive groups by G. Lusztig and D. Vogan in 1980s, now implemented in the computer package ATLAS, as well as progress in the classification of unitary representations described in \cite{ALTV}), and the connection with geometry has remained the main driving force of its development all along. 
 
The organization of the notes is as follows. In Sections 1-7 we discuss the analytic aspects of the theory, arising from the fact that interesting representations of non-compact Lie groups are infinite-dimensional and realized in topological vector spaces. 
We introduce the main analytic tools, such as Fr\'echet spaces, the convolution algebra of measures on the group, $K$-finite, smooth and analytic vectors, matrix coefficients, and then discuss theorems of Harish-Chandra which allow one to reduce the study of representations of a semisimple Lie group $G$ to the study of admissible $(\g,K)$-modules, where $\g={\rm Lie}G$ and $K$ is a maximal compact subgroup of $G$, and thereby to the purely algebraic problem of studying Harish-Chandra modules. Then in Section 8 we recall the theory of highest weight modules for $\g$, and in Section 9 use the developed theory to classify irreducible and unitary irreducible representations of $SL_2(\Bbb R)$. 

The rest of the notes is almost completely algebraic. Namely, in Sections 10-14, we prove a number of fundamental properties of the action of the Weyl group $W$
on the symmetric algebra $S\h$ on the Cartan subalgebra $\h$ and of the universal enveloping algebra $U(\g)$ for a semisimple complex Lie algebra $\g$ -- the Chevalley restriction theorem, Chevalley-Shephard-Todd theorems (\cite{Che},\cite{ST}), Kostant theorems (\cite{Ko}). Then in Sections 15, 16 we develop the basic theory of the BGG category $\mathcal O$, and in Section  17 discuss the nilpotent cone of $\g$, proving that it is a reduced irreducible variety. In Sections 18, 19 we prove the Duflo-Joseph theorem and discuss principal series representations of the complex Lie group $G$. Here we also introduce a functor $H_\lambda$ which connects category $\mathcal O$ with the category Harish-Chandra bimodules for $\g$. In Sections 20, 21 we continue to study category $\mathcal O$ 
(BGG theorem, BGG reciprocity, multiplicities, formulation of the Kazhdan-Lusztig conjectures). In Sections 22-25 we give an exposition of the theory of projective functors of J. Bernstein and S. Gelfand (\cite{BG}) and give its applications to representation theory of complex groups (classification of irreducible representations, describing the category of Harish-Chandra bimodules in terms of category $\mathcal O$) as well as to the structure theory of $U(\g)$ (Duflo's theorem on primitive ideals). 
In Section 26, we apply these results to the group
$G=SL_2(\Bbb C)$ and give an explicit classification of its irreducible and unitary irreducible representations. Finally, in Sections 27-31 we outline the geometric approach to representation theory of semisimple Lie groups, starting from Borel-Weil theorem and then proceeding to D-modules, the Beilinson-Bernstein localization theorem, and classification of irreducible Harish-Chandra modules by data attached to $K$-orbits on $G/B$. In these sections the material is more advanced and the exposition is less detailed.

The notes are divided into 31 sections which roughly correspond to 80 minute lectures. So there is a bit more material than in a 1-semester course (which normally consists of 26 lectures). So if these notes are used to teach a course, some material (roughly equivalent to 5 sections) should be skipped. 

A lot of important material is included in exercises, which are often provided with detailed hints and may be assigned as homework. 

{\bf Disclaimer:} These notes contain nothing original except mistakes of the author, and in all sections the exposition is mostly adapted from various existing sources -- 
original articles, textbooks and lecture notes. For example, the exposition in Sections 1-7 is mostly borrowed from \cite{G} and \cite{L}, in Sections 20, 21 parts are adapted from  \cite{Hu}, in Sections 22-25 we closely follow \cite{BG}, and in Sections 28-30 we follow \cite{BCEY}. Since the material is standard, we do not always give a reference.

Also, these notes only scratch the surface in the deep subject of representations of semisimple Lie groups. For a more in-depth study of this theory, we recommend the books \cite{K}, \cite{ABV}, \cite{Wal88}, \cite{Wal92} and the lectures \cite{KT}; for recent progress on unitary representations see \cite{ALTV}. For more about category $\mathcal O$ we refer the reader to \cite{Bez}, \cite{Hu}. For more on the theory of D-modules and their applications to representation theory we recommend the book \cite{HTT}. 

{\bf Acknowledgments.} I'd like to thank David Vogan who prompted me to rework 
the MIT Lie groups graduate sequence, of which this is part 3 (the notes for parts 1 and 2 can be found in \cite{E}). It is my pleasure to dedicate this text to David's 70th birthday. 

I am grateful to the students of the MIT course ``Representations of Lie groups" in the academic years 2021/2022 and 2023/2024 for feedback, and to Jeffrey Lu and David Vogan for reading the manuscript and many corrections. This work was partially supported by the NSF grant DMS-2001318. The latest version
was cleaned up using ChatGPT-5.5.

\section{\bf Continuous representations of topological groups} 

This course will be about representations of Lie groups, with a focus on non-compact groups. While irreducible representations of compact groups are all finite-dimensional, this is not so for non-compact groups, whose most interesting irreducible representations are infinite-dimensional. Thus to have a sensible representation theory of non-compact Lie groups, we need to consider their {\bf continuous} representations on {\bf topological vector spaces}. 

\subsection{Topological vector spaces}

\subsubsection{Basic definitions}
All representations we'll consider will be over the field $\Bbb C$, which is equipped with its usual topology. Recall that a {\bf topological vector space} over $\Bbb C$ is a complex vector space $V$ with a topology in which addition $V\times V\to V$ and scalar multiplication $\Bbb C\times V\to V$ are continuous. The topological vector spaces $V$ we'll consider will always be  assumed to have the following properties: 

$\bullet$ {\bf Hausdorff:} any two distinct points of $V$ have disjoint neighborhoods. 

$\bullet$ {\bf locally convex:} $0\in V$ (hence every point) has a base of convex neighborhoods.\footnote{Recall that a set $X\subset V$ is {\bf convex} if for any $x,y\in X$ and $t\in [0,1]$ we have $tx+(1-t)y\in X$.} Equivalently, the topology on $V$ is defined by a family of {\bf seminorms}\footnote{Recall that a {\bf seminorm} on $V$ is a function $\nu: V\to \Bbb R_{\ge 0}$ such that $\nu(x+y)\le \nu(x)+\nu(y)$ and 
$\nu(\lambda x)=|\lambda|\nu(x)$ for $x,y\in V$, $\lambda\in \Bbb C$. A seminorm is a {\bf norm} iff $\nu(x)=0$ implies $x=0$.} $\lbrace\nu_\alpha, \alpha\in A\rbrace$: a base of neighborhoods 
of $0$ is formed by finite intersections of the sets $U_{\alpha,\varepsilon}:=\lbrace v\in V|\nu_\alpha(v)<\varepsilon\rbrace$, $\alpha\in A$, $\varepsilon>0$. I.e., it is the weakest of the topologies in which all $\nu_\alpha$ are continuous.

$\bullet$ {\bf sequentially complete:} every Cauchy sequence\footnote{Recall that a sequence $a_n\in V$ is {\bf Cauchy} if for any neighborhood $U$ of $0\in V$ there exists $N$ such that for $n,m\ge N$ we have $a_n-a_m\in U$.}  is convergent.

Also, unless specified otherwise, we will assume that $V$ is 

$\bullet$ {\bf first countable}: $0\in V$ (equivalently, every point of $V$) has a countable base of neighborhoods. By the Birkhoff-Kakutani theorem, this is equivalent to $V$ being {\bf metrizable}
(topology defined by a metric), and moreover this metric can be chosen translation invariant: 
$d(x,y)=D(x-y)$ for some function $D: V\to \Bbb R_{\ge 0}$. 

In this case $V$ is called a {\bf Fr\'echet space}. For example, every {\bf Banach spac}e (a complete normed space), in particular, {\bf Hilbert space} is a Fr\'echet space. 

A Hausdorff topological vector space $V$ is said to be {\bf complete} if whenever $V$ is realized as a dense subspace of a Hausdorff topological vector space $\overline V$ with induced topology, we have $V=\overline V$. Every complete space is sequentially complete, and the converse holds for metrizable spaces (albeit not in general). Thus a Fr\'echet space can be defined as a locally convex complete metrizable topological vector space. 

Alternatively, a Fr\'echet space may be defined as a complete (or, equivalently, sequentially complete) topological vector space with topology defined by a {\it countable} system of  seminorms $\nu_n: V\to \Bbb R$, $n\ge 1$ (such a space is always metrizable, see below). Thus, a sequence $x_m\in V$ goes to zero iff $\nu_n(x_m)$ goes to zero for all $n$. Note that the Hausdorff property is then equivalent to the requirement that any vector $x\in V$ with $\nu_n(x)=0$ for all $n$ is zero. 

A translation-invariant metric on a Fr\'echet space may be defined by the formula 
$$
d(x,y)=D(x-y),\ D(x):=\sum_{n=1}^\infty \frac{1}{2^n}\frac{\nu_n(x)}{1+\nu_n(x)}.
$$
Note however that $D$ is not a norm, as it is not homogeneous: for $\lambda\in \Bbb C$, $D(\lambda x)\ne |\lambda| D(x)$. If we had a finite collection of seminorms, we could define a norm simply by $D(x):=\sum_n \nu_n(x)$, but if there are infinitely many, this sum may not converge, and we have to sacrifice the homogeneity property for convergence. In fact, the examples below show that  
there are important Fr\'echet spaces that are not Banach (i.e., do not admit a single norm defining the topology). We also note that the same Fr\'echet space structure on $V$ can be defined by different systems of seminorms $\nu_n$, and there is also nothing canonical about the formula for $D$ (e.g., we can replace $\frac{1}{2^n}$ by any sequence $a_n>0$ with $\sum_n a_n<\infty$), so $\nu_n$ or $D$ are not part of the data of a Fr\'echet space. 
 
Finally, unless specified otherwise, we will assume that $V$ is 

$\bullet$ {\bf second countable:} admits a countable base. For metrizable spaces, this is equivalent to being {\bf separable} (having a dense countable subset).  

\subsubsection{Locally compact spaces} Recall that a Hausdorff topological space $X$ is said to be {\bf locally compact} if every point of $X$ has a neighborhood with compact closure. This implies that every compact subset $K$ of $X$ has a neighborhood with a compact closure, as the open cover of $K$ by such neighborhoods of all its points has a finite subcover.
For example, every manifold is locally compact. 

Let $X$ be a second countable locally compact space. 

\begin{definition} A {\bf compact exhaustion} of $X$ is a sequence of compact subsets $\lbrace K_n,n\ge 1\rbrace$ of $X$ such that $K_n$ is contained in the interior $K_{n+1}^\circ$ of $K_{n+1}$ for all $n$, and $\bigcup_{n\ge 1}K_n=X$.
\end{definition} 

\begin{lemma}\label{cex} (i) Compact exhaustions exist. 

(ii) If $\lbrace K_n,n\ge 1\rbrace$ is a compact exhaustion of $X$ then every compact subset $K\subset X$ is contained in $K_n$ for some $n$. 
\end{lemma} 

\begin{proof} (i) Fix a countable base $\lbrace V_i,i\ge 1\rbrace$ of $X$. We can arrange that the closures $\overline V_i$ are compact. Let $K_n:=\overline{V_1\cup...\cup V_n}$. Then $K_n$ are compact, $K_n\subset K_{n+1}$, and $\cup_n K_n=X$. Also if $K\subset X$ is compact then $K\subset V_1\cup...\cup V_n$ for some $n$, so $K\subset K_n$. Moreover, for any $i_1\in \Bbb Z_{\ge 1}$ the set $K_{i_1}$ has a neighborhood $U_1$ with compact closure $\overline U_1$, hence contained in some $K_{i_2}$, which has a neighborhood $U_2$ with compact closure $\overline U_2$, hence contained in some $K_{i_3}$, and so on. Then $K_{i_1},K_{i_2},...$ is a compact exhaustion of $X$.  

(ii) $\lbrace K_n^\circ\rbrace$ form an open cover of $K$, which must have a finite subcover. Thus $K\subset K_n^\circ\subset K_n$ for some $n$. 
\end{proof} 

\subsubsection{Examples of Fr\'echet spaces} 
\begin{example}\label{frespa} 1. Let $X$ be a second countable locally compact space and $C(X)$ be the space of continuous complex-valued functions on $X$. Let $\lbrace K_i,i\ge 1\rbrace$ 
be a compact exhaustion of $X$ (Lemma \ref{cex}). We can then define seminorms 
$\nu_n$ by 
$$
\nu_n(f)=\max_{x\in K_n}|f(x)|.
$$  
(this is well defined since $K_n$ are compact). This makes $C(X)$ into a Fr\'echet space, and this structure is independent of the choice of the sequence $K_n$.  The convergence in $C(X)$ is uniform convergence on compact sets. 

By the Tietze extension theorem, if $K\subset L$ are compact Hausdorff spaces 
then the restriction map $C(L)\to C(K)$ is surjective. So 
$C(X)=\underleftarrow{\lim_{n\to \infty}}C(K_n)$ as a vector space. 
Alternatively, without making any choices, we may write 
$C(X)=\underleftarrow{\lim_{K\subset X}}C(K)$, where $K$ runs over compact subsets of $X$. 

2. If $X$ is a manifold and $0\le k\le \infty$, we can similarly define 
a Fr\'echet space structure on the space $C^k(X)$ of $k$ times continuously differentiable functions on $X$. Namely, cover $X$ by countably many closed balls $B_n$, each equipped with a local coordinate system, and 
set 
$$
\nu_{n,m}(f)=\max_{x\in B_n}\norm{d^mf(x)},\ 0\le m\le k
$$
where $d^mf(x)$ is the $m$-th differential of $f$ at $x$, comprising 
the $m$-th mixed partial derivatives of $f$ at $x$ with respect to the local coordinates 
(these are labeled by two indices rather than one, but it does not matter since this collection is still countable). The convergence in $C^k(X)$ is uniform convergence with all derivatives up to $k$-th order on compact sets.

These spaces are not Banach unless $X$ is compact. Moreover, 
$C^\infty(X)$ is not Banach even for compact $X$ (of positive dimension). 
For example, for $C^\infty(S^1)$ we may take,
$$
\nu_m(f)=\sum_{i=0}^m \max_{x\in S^1}|f^{(i)}(x)|,
$$
but this is still an infinite collection. Note that these are all norms, not just seminorms, but each of them taken separately does not define the correct topology on $C^\infty(S^1)$ (namely, $\nu_m$ defines the incomplete topology induced by embedding $C^\infty(S^1)$ as a dense subspace into the Banach space $C^m(S^1)$ with norm $\nu_m$).   

3. The Schwartz space $\mathcal S(\Bbb R)\subset C^\infty(\Bbb R)$  is the space of functions $f$ with 
$$
\nu_{m,n}(f):=\sup_{x\in  \Bbb R}|x^n\partial^mf(x)|<\infty,\ m,n\ge 0.
$$
This system of seminorms can then be used to give $\mathcal S(\Bbb R)$ 
the structure of a (non-Banach) Fr\'echet space. The same definition 
can be used for the Schwartz space $\mathcal S(\Bbb R^N)$, 
by taking $n=(n_1,...,n_N)$, $m=(m_1,...,m_N)$, 
$x=(x_1,...,x_N)$, $\partial=(\partial_1,...,\partial_N)$, 
and 
$$
x^n:=\prod_i x_i^{n_i},\ \partial^m:=\prod_i \partial_i^{m_i}.
$$

It is well known that all these spaces are separable (check it!).
\end{example} 

\subsection{Continuous representations} \label{conrep}

Let $G$ be a locally compact topological group, for example, a Lie group.\footnote{Topological groups will always be assumed Hausdorff and second countable.  Important examples of locally compact topological groups include groups $\bold G(F)$, where $F$ is a local field and $\bold G$ is an algebraic group defined over $F$. If $F$
is archimedean ($\Bbb R$ or $\Bbb C$) then $\bold G(F)$ is a real, respectively complex, Lie group.
Another example important in number theory is $\bold G(\Bbb A_k)$, where $k$ is a global field,
$\Bbb A_k$ is its ring of ad\'eles, and $\bold G$ is an algebraic group over $k$.}

\begin{definition} A {\bf continuous representation} of $G$ is a topological vector space $V$ with a {\it continuous} linear action $a: G\times V\to V$.\footnote{It is easy to see that it suffices to check this property at points $(1,v)$ 
for $v\in V$.}  
\end{definition} 

In particular, a continuous representation gives a homomorphism 
$\pi: G\to \Aut(V)$ from $G$ to the group of continuous automorphisms 
of $V$ (i.e., continuous linear maps $V\to V$ with continuous inverse).\footnote{Note that by the open mapping theorem, in a Fr\'echet space any invertible continuous operator has a continuous inverse.} 

\begin{definition} A continuous representation is called {\bf unitary} if $V$ is a Hilbert space and for all $g\in G$, the operator $\pi(g): V\to V$ is unitary; in other words, $\pi$ lands in the unitary group $U(V)\subset \Aut(V)$.  
\end{definition} 

\begin{exercise}\label{transac} Let $1\le p<\infty$ and $L^p(\Bbb R)$ be the Banach space of measurable functions $f: \Bbb R\to \Bbb C$ with 
$$
\norm{f}_p=\left(\int_{\Bbb R} |f(x)|^pdx\right)^{\frac{1}{p}}<\infty
$$
(modulo functions vanishing outside a set of measure zero), with norm $f\mapsto \norm{f}_p$. The Lie group $\Bbb R$ acts 
on $L^p(\Bbb R)$ by translation. 

(i) Show that this is a continuous representation, which is unitary for $p=2$ (use approximation of $L^p$ functions by continuous functions with compact support). 

(ii) Prove the same for the Fr\'echet spaces $C^k(\Bbb R)$ and $\mathcal{S}(\Bbb R)$.
\end{exercise}

Let $G$ be a locally compact group, for example a Lie group. In this case 
$G$ is known to have a unique up to scaling right-invariant {\bf Haar measure} 
$dx$. For Lie groups, this measure is easy to construct by spreading 
a nonzero element of $\wedge^n \g^*$, $\g={\rm Lie}(G)$, $n=\dim \g$, over the group $G$ by right translations. Thus we can define the Banach space $L^p(G)$ similarly to the case $G=\Bbb R$. It is easy to generalize Exercise \ref{transac} to show that the right translation action of $G$ on $L^p(G)$ and $C^k(G)$ is continuous, with $L^2(G)$ unitary. 

\begin{example} Let $X$ be a manifold with a right action of a Lie group $G$. 
We'd like to say that we have a unitary representation of $G$ on $L^2(X)$ via $(gf)(x)=f(xg)$. But for this purpose we need to fix a $G$-invariant measure on $X$, and such a nonzero measure does not always exist (e.g., $G=SL_2(\Bbb R)$, $X=\Bbb R\Bbb P^1=S^1$). 

The way out is to use {\bf half-densities} on $X$ rather than functions. Namely, recall that if $\dim X=m$ then the canonical line bundle $K_X:=\wedge^mT^*X$ has structure group $\Bbb R^\times$. Consider the character $\Bbb R^\times\to \Bbb R^{>0}$ given by $t\mapsto |t|^s$, $s\in \Bbb R$, and denote the associated line bundle $|K|^s$. This is called the bundle 
of $s$-{\bf densities} on $X$ (in particular, {\bf densities} for $s=1$ and {\bf half-densities} for $s=\frac{1}{2}$). Thus in local coordinates $s$-densities are 
ordinary functions, but when we change coordinates by $x\mapsto x'=x'(x)$, 
these functions change as 
$$
f=f'|\det(\tfrac{\partial x'}{\partial x})|^{s}.
$$

The benefit of half-densities is that for any half-density 
$f$, the expression $|f|^2$ is naturally a density 
on $X$, which canonically defines a measure that can be integrated over $X$. 
As a result, the space $L^2(X)$ of half-densities $f$ on $X$ with 
$$
\norm{f}_2=\sqrt{\int_X|f|^2}<\infty
$$
is a Hilbert space attached canonically to $X$ (without choosing any additional structures), and any diffeomorphism $g: X\to X$ defines a unitary operator 
on $L^2(X)$. Thus similarly to Exercise \ref{transac}, $L^2(X)$ is a unitary representation of $G$. Note that if $X$ has a $G$-invariant measure, 
this is the same as a representation of $G$ on $L^2$-functions on $X$. 

In particular, we see that we have a unitary representation of $G\times G$ 
on $L^2(G)$ by left and right translation even though the right-invariant Haar measure is not always left-invariant. 
\end{example} 

If $V$ is finite-dimensional, ${\rm Aut}(V)=GL(V)$ is just the group of invertible matrices, and the continuity condition for representations of $G$ is just that 
the map $\pi: G\to \Aut(V)$ is continuous in the usual topology. Then it is well known 
that this map is smooth and is determined by the corresponding Lie algebra map 
$\g\to {\rm End}(V)=\mathfrak{gl}(V)$, and this correspondence is a bijection if $G$ is simply connected. In this way the theory of finite-dimensional continuous representations of connected Lie groups is immediately reduced to pure algebra. 

On the other hand, for infinite-dimensional representations the situation is more tricky, as there are several natural topologies on ${\rm Aut}(V)$. One of them is the {\bf strong topology} of $\End(V)$ (continuous endomorphisms of $V$), in which $T_n\to T$ iff for all $v\in V$ we have $T_nv\to Tv$. It is clear that if $(V,\pi)$ is a continuous representation of $G$ then the map $\pi: G\to {\rm Aut}(V)$ is continuous in the strong topology, but the converse is not true, in general. However, the converse holds for Banach spaces (in particular, for unitary representations). 

\begin{proposition} If $V$ is a Banach space then 
 a representation $(V,\pi)$ of $G$
is continuous if and only if the map $\pi: G\to \Aut(V)$ is continuous in the strong topology.  
\end{proposition} 

\begin{proof} Recall the {\bf uniform boundedness principle}: If $T_n$ 
is a sequence of bounded operators from a Banach space $V$
to a normed space and for any $v\in V$ the sequence $T_nv$ 
is bounded then the sequence $\norm{T_n}$ is bounded.  

Now assume that $\pi$ is continuous in the strong topology. Let $g_n\in G$, $g_n\to 1$, and $v_n\to v\in V$. Since $G$ is second countable, our job is to show that 
$\pi(g_n)v_n\to v$. We know that $\pi(g_n)v\to v$, as $\pi(g_n)\to 1$ 
in the strong topology. So it suffices to show that $\pi(g_n)(v_n-v)\to 0$. 
As $v_n-v\to 0$, it suffices to show that the sequence $\norm{\pi(g_n)}$ is bounded. But this follows from the uniform boundedness principle.
\end{proof} 

\begin{remark} 1. Another topology on $\End(V)$ for a Banach space $V$ 
is the {\bf norm topology}, defined by the operator norm. It is stronger 
than the strong topology, and a continuous representation $\pi: G\to \Aut(V)$
does {\bf not} have to be continuous in this topology. For example, the action 
of $\Bbb R$ on $L^2(\Bbb R)$ is not. Indeed, denoting by $T_a$ the operator 
$\pi(a)$ given by
$(T_af)(x)=f(x+a),$
we have 
$\norm{T_a-1}=2 $
for all $a\ne 0$ (show it!). 

2. If $\dim V=\infty$ then ${\rm Aut}(V)$ is {\bf not} a topological group with respect to strong topology (multiplication is not continuous). 
\end{remark} 

\subsection{Subrepresentations, irreducible representations}

\begin{definition} A {\bf subrepresentation} of a continuous representation $V$ of $G$ is a {\it closed} $G$-invariant subspace of $V$. We say that $V$ 
is {\bf irreducible} if its only subrepresentations are $0$ and $V$. 
\end{definition} 

\begin{example}\label{transac1} 
The translation representation of $\Bbb R$ on $L^2(\Bbb R)$ 
is not irreducible, although this is not completely obvious. To see this, we apply Fourier transform, which is a unitary automorphism of $L^2(\Bbb R)$. The Fourier transform maps the operator $T_a$ to the operator 
of multiplication by $e^{iax}$. But it is easy to construct closed subspaces 
of $L^2(\Bbb R)$ invariant under multiplication by $e^{iax}$: take any measurable subset $X\subset \Bbb R$ and the subspace $L^2(X)\subset L^2(\Bbb R)$ of functions that essentially vanish outside $X$ (e.g., one can take $X=[0,+\infty)$).
\end{example} 

\begin{example}\label{heis} Here is the most basic example of an irreducible infinite-dimensional representation of a Lie group. Let $G$ be the {\bf Heisenberg} group, i.e., the group of upper triangular unipotent real 3-by-3 matrices. It can be realized as the Euclidean space $\Bbb R^3$ (with coordinates $x,y,z$ being the above-diagonal matrix entries), with multiplication law
$$
(a,b,c)(a',b',c')=(a+a',b+b',c+c'+ab'). 
$$ 
Then we can define a unitary representation of $G$ on $V=L^2(\Bbb R)$ by 
setting $\pi(a,0,0)=e^{iax}$ (multiplication operator) and $\pi(0,b,0)=T_b$ (shift by $b$). 

\begin{exercise} (i) Show that this gives rise to a well defined unitary representation of $G$, and compute $\pi(a,b,c)$ for general $(a,b,c)$. 

(ii) Show that $V$ is irreducible. 

{\it Hint.} Suppose $W\subset V$ is a proper subrepresentation, and denote by $P: V\to V$ 
the orthogonal projector to $W$. We can write $P$ as an integral operator 
with Schwartz kernel\footnote{Recall that every smooth 
function $\phi(x,y)$ on $\Bbb R^2$ with compact support defines a trace class operator 
$T_\phi$ with kernel $\phi(y,x)$, i.e., 
$$
(T_\phi f)(x)=\int_\Bbb R \phi(y,x)f(y)dy.
$$
Then the Schwartz kernel $K$ of a continuous endomorphism $A$ of $L^2(\Bbb R)$ is defined by the formula $(K,\phi)=\Tr(AT_\phi)$ (which is well defined since the operator $AT_\phi$ is trace class).} $K(x,y)$, a distribution on $\Bbb R^2$. Show that $K$ is translation invariant, i.e., 
$K(x+a,y+a)=K(x,y)$, and deduce $K(x,y)=k(x-y)$ for some distribution $k(x)$ 
on $\Bbb R$.\footnote{This means that $(K,\phi)=(k,\widetilde \phi)$, where 
$\widetilde \phi(x):=\int_{\Bbb R}\phi(x+y,y)dy$.}  
 Show that $(e^{iax}-1)k(x)=0$ for all $a\in \Bbb R$. Deduce that $P$ is a scalar operator. Conclude that $P=0$, so $W=0$. 
\end{exercise}

\end{example} 

\section{\bf $K$-finite vectors and matrix coefficients} 

\subsection{$K$-finite vectors} 

Let $K$ be a {\it compact} topological group. In this case $K$ has a unique right-invariant Haar measure of volume $1$, which is therefore also left-invariant; we will denote this measure by $dg$. Thus if $V$ is a finite-dimensional (continuous) representation of $K$ 
and $B$ a positive definite Hermitian form on $V$ then the form 
$$
\overline B(v,w):=\int_K B(gv,gw)dg
$$
is positive definite and $K$-invariant, which implies that $V$ is unitary. 
If $V$ is irreducible then by Schur's lemma this unitary structure is unique up to scaling.  

This implies that finite-dimensional representations of $K$ are completely reducible: if $W\subset V$ is a subrepresentation then 
$V=W\oplus W^\perp$, where $W^\perp$ is the orthogonal complement of $W$ under the Hermitian form.  

Now let $V$ be any continuous representation of $K$ (not necessarily finite-dimensional). 

\begin{definition} A vector $v\in V$ is $K$-{\bf finite} if it is contained 
in a finite-dimensional subrepresentation of $V$. The space of $K$-finite vectors of $V$ is denoted by $V^{\rm fin}$. 
\end{definition} 

Let ${\rm Irr}K$ be the set of isomorphism classes of irreducible finite-dimensional representations of $K$. We have a natural $K$-invariant linear map 
$$
\xi: \oplus_{\rho\in {\rm Irr}K}\Hom(\rho,V)\otimes \rho\to V^{\rm fin}
$$
(where $K$ acts trivially on $\Hom(\rho,V)$) defined by 
$$
\xi(h\otimes u)=h(u). 
$$

\begin{lemma}\label{isolem} $\xi$ is an isomorphism. 
\end{lemma} 

\begin{proof} To show $\xi$ is injective, assume the contrary, and let $\widetilde\rho$ be an irreducible 
subrepresentation of ${\rm Ker}\xi$. Then $\widetilde \rho=h\otimes \rho$ for a suitable $h\in \Hom(\rho,V)$, so for any $u\in \rho$ we have 
$h(u)=\xi(h\otimes u)=0$. Thus $h=0$, a contradiction. 

It remains to show that $\xi$ is surjective. For $v\in V^{\rm fin}$, let 
$W\subset V^{\rm fin}$ be a finite-dimensional subrepresentation 
of $V$ containing $v$. By complete reducibility, 
$W$ is a direct sum of irreducible representations. 
Thus it suffices to assume that $W$ is irreducible. 
 Let $h: W\hookrightarrow V$ be the corresponding inclusion. Then $v=h(v)=\xi(h\otimes v)$. 
 \end{proof} 
 
 \begin{example} Let $K=S^1=\Bbb R/2\pi \Bbb Z$. The irreducible finite-dimensional representations of $K$ are the characters $\rho_n(x)=e^{inx}$ 
 for integer $n$. Let $V=L^2(S^1)$. Then $\Hom(\rho_n,V)$ 
 is the space of functions on $S^1$ such that $f(x+a)=e^{ina}f(x)$, 
 which is a 1-dimensional space spanned by the function 
 $e^{inx}$. It follows that $V^{\rm fin}$ is the space of 
 trigonometric polynomials $\sum_n a_ne^{inx}$, where only finitely many coefficients $a_n\in \Bbb C$ are nonzero.  
 \end{example} 

\subsection{Matrix coefficients} Let us now consider the special case $V=L^2(K)$, and view it as a representation of $K\times K$ via 
$$
(\pi(a,b)f)(x)=f(a^{-1}xb).
$$
For every irreducible representation $\rho\in {\rm Irr}K$ we have a 
homomorphism of representations of $K\times K$: 
$$
\xi_\rho: \End_{\Bbb C}\rho=\rho^*\otimes \rho\to L^2(K)
$$
defined by 
$$
\xi_\rho(h\otimes v)(g):=h(gv). 
$$
This map is nonzero, hence injective (as $\rho^*\otimes \rho$ is an irreducible $K\times K$-representation), and is called {\bf the matrix coefficient map}, as the right hand side is a matrix coefficient of the representation $\rho$. 
The {\bf theorem on orthogonality of matrix coefficients} tells us that 
the images of $\xi_\rho$ for different $\rho$ are orthogonal, 
and for $A,B\in \End_{\Bbb C}\rho$ we have 
$$
(\xi_\rho(A),\xi_\rho(B))=\frac{{\rm Tr}(AB^\dagger)}{\dim \rho},
$$
where $B^\dagger$ is the Hermitian adjoint of $B$ with respect to the unitary structure on $\rho$. Thus, choosing orthonormal bases 
$\lbrace v_{\rho i}\rbrace$ in each $\rho$, we find that the functions 
$$
\psi_{\rho ij}:=(\dim \rho)^{\frac{1}{2}}\xi_\rho(E_{ij}),
$$
where $E_{ij}:=v_{\rho j}^*\otimes v_{\rho i}$ are elementary matrices, 
form an orthonormal system in $L^2(K)$. 

Let us view $L^2(K)$ as a representation of $K$ via left translations. 
Let $\rho\in {\rm Irr}K$. Then every $h\in \rho$ defines a homomorphism of representations $f_h: \rho^*\to L^2(K)$
which, when viewed as an element of $L^2(K,\rho)$, is given by the formula 
$f_h(y):=yh$. Conversely, suppose $f: \rho^*\to L^2(K)$ is a homomorphism. 
Then $f$ can be represented by an $L^2$-function 
$\widetilde f: K\to \rho$ such that for any $b\in K$, the function $x\mapsto \widetilde f(bx)-b\widetilde f(x)$ vanishes outside a set $S_b\subset K$ of measure $0$. Let $S\subset K\times K$ 
be the set of pairs $(b,x)$ such that $x\in S_b$. Then $S$ has measure $0$, hence 
the set $T_x$ of $b\in K$ such that $(b,x)\in S$ (i.e., $x\in S_b$) has measure zero 
almost everywhere with respect to $x$. So pick $x\in K$ such that $T_x$ has measure zero. 
For $b\notin T_x$ and $y=bx$, we have $\widetilde f(y)=yx^{-1}\widetilde f(x)$. Thus 
$f=f_h$ where $h=x^{-1}\widetilde f(x)$. It follows that the assignment $h\mapsto f_h$ 
is an isomorphism $\rho\cong \Hom(\rho^*,L^2(K))$. This shows 
that the map 
$$
\bigoplus_{\rho\in {\rm Irr}K}\xi_\rho:  \bigoplus_{\rho\in {\rm Irr}K}\rho^*\otimes \rho\to L^2(K)^{\rm fin}
$$
is an isomorphism, where $L^2(K)^{\rm fin}$ is the space of $K$-finite vectors in $L^2(K)$ under left translations. 
Thus any $K$-finite function under left (or right) translations is actually $K\times K$-finite, and we have a natural orthogonal decomposition 
$$
L^2(K)^{\rm fin}\cong \bigoplus_{\rho\in {\rm Irr}K}\rho^*\otimes \rho.
$$
Moreover, since $L^2(K)$ is separable, it follows that ${\rm Irr}K$ is a countable set. 

\subsection{The Peter-Weyl theorem} 
The following non-trivial theorem is proved in the basic Lie groups course. 

\begin{theorem} (Peter-Weyl) $L^2(K)^{\rm fin}$ is a dense subspace of $L^2(K)$. Hence $\lbrace \psi_{\rho ij}\rbrace $ form an orthonormal basis of $L^2(K)$, and we have 
$$
L^2(K)=\widehat\oplus_{\rho\in {\rm Irr}K}\rho^*\otimes \rho.
$$
(completed orthogonal direct sum under the Hilbert space norm). 
\end{theorem} 

\begin{example} For $K=S^1=\Bbb R/2\pi \Bbb Z$ the Peter-Weyl theorem says that the Fourier system $\lbrace e^{inx}\rbrace$ is complete, i.e., a basis of $L^2(S^1)$.
\end{example} 

\subsection{Partitions of unity} 

Let $X$ be a metric space, and $\lbrace U_i,i\ge 1\rbrace$ be a countable open cover of $X$. 

\begin{definition} A {\bf partition of unity subordinate to the cover} $\lbrace U_i,i\ge 1\rbrace$
is a collection of non-negative functions $\phi_i\in C(X)$ such that $\phi_i|_{U_i^c}=0$ for all $i\ge 1$
and 
$$
\sum_{i\ge 1}\phi_i=1
$$
uniformly on compact subsets of $X$. 
\end{definition} 

\begin{lemma}\label{parun} Every countable open cover of $X$ admits a partition of unity subordinate to this cover. 
\end{lemma} 

\begin{proof} Let $d$ be the distance function on $X$ and $C\subset X$ a closed subset.
For $x\in X$ define $d(x,C)\in [0,+\infty]$ by the formula
$$
d(x,C):={\rm inf}_{y\in C}d(x,y)
$$ 
if $C\ne \emptyset$; for $C=\emptyset$, we set $d(x,C)=+\infty$. 
For non-empty $C$, we have
$d(x,C)\le d(x,y)+d(y,C)$, hence 
$|d(x,C)-d(y,C)|\le d(x,y)$, so this function is continuous.
Thus the function $f_C(x):=\frac{d(x,C)}{1+d(x,C)}$ (defined to be $1$ if $C=\emptyset$) 
is continuous on $X$, takes values in $[0,1]$, and $f_C(x)=0$ iff $x\in C$. So if $\lbrace U_i,i\ge 1\rbrace$ is a countable open cover of $X$ then the function $\sum_{i\ge 1}2^{-i}f_{U_i^c}$ is continuous and strictly positive. Thus we may define the continuous functions on $X$
$$
\phi_i:=\frac{2^{-i}f_{U_i^c}}{\sum_{j\ge 1}2^{-j}f_{U_j^c}}, i\ge 1.
$$
These functions are non-negative, vanish outside $U_i$, and 
$$
\sum_{i\ge 1}\phi_i=1
$$
pointwise. 

Moreover, this series converges uniformly on every compact subset $K\subset X$. 
Indeed, there exists $n$ such that $K\subset U_1\cup...\cup U_n$, so we have 
$K\cap (U_1^c\cap...\cap U_n^c)=\emptyset$. Thus for any $x\in K$ we have $d(x,U_1^c\cap...\cap U_n^c)
>0$. So there exists $j=j_x\in [1,n]$ such that $d(x,U_j^c)>0$, hence $f_{U_j^c}(x)>0$. It follows that for all $x\in K$ we have $\sum_{i=1}^n2^{-i}f_{U_i^c}(x)>0$. This is a continuous function on $K$, hence attains a minimal value $D>0$. 
Thus $\sup|\phi_j|\le \frac{2^{-j}}{D}$ for all $j$, implying the uniform convergence. 

It follows that $\lbrace \phi_i,i\ge 1\rbrace$ is a desired partition of unity. 
\end{proof} 

\section{\bf Algebras of measures on locally compact groups} 

\subsection{The space of measures} 

Let $X$ be a locally compact second countable Hausdorff topological space. 
It is well known that such a space is metrizable (hence separable), so let us fix a metric $d$ 
defining the topology on $X$. 

As we have seen in Example \ref{frespa}, the space $C(X)$ of continuous functions on $X$ is a separable Fr\'echet space. So let us consider the topological dual space, $C(X)^*$, of continuous linear functionals on $C(X)$. This space is denoted by ${\rm Meas}_c(X)$; its elements are called (complex-valued) {\bf compactly supported (Radon) measures on $X$}. If $X$ is compact, we will also denote ${\rm Meas}_c(X)$ by ${\rm Meas}(X)$. We will often use the standard notation from measure theory: for $f\in C(X)$ and $\mu\in {\rm Meas}_c(X)$, 
$$
\mu(f)=\int_X f(x)d\mu(x). 
$$

Pick a compact exhaustion $\lbrace K_i,i\ge 1\rbrace$ of $X$ (Lemma \ref{cex}). We claim that for any $\mu\in {\rm Meas}_c(X)$ there exists $i$ such that if $f\in C(X)$ satisfies $f|_{K_i}=0$ then $\mu(f)=0$. 
Indeed, if not then for each $i$ there is $f_i\in C(X)$ with $f_i|_{K_i}=0$ but $\mu(f_i)=1$. Then the series $\sum_i f_i$ converges in $C(X)$ (as it terminates on each $K_i$, and every compact subset of $X$ is contained in some $K_i$) while the series $\mu(\sum_i f_i)=\sum_i \mu(f_i)=\sum_i 1$ diverges, a contradiction. 
Thus we see that ${\rm Meas}_c(X)=\bigcup_{i\ge 1} {\rm Meas}(K_i)$ as a vector space, 
or, without making any choices, 
\begin{equation}\label{indli}
{\rm Meas}_c(X)=\underrightarrow{\lim_{K\subset X}}{\rm Meas}(K),
\end{equation}
where $K$ runs over compact subsets of $X$. 

Let us now endow ${\rm Meas}_c(X)$ with a topology. First assume that $X$ is compact. We then  equip ${\rm Meas}(X)$ with a {\bf weak topology},  defined by the condition that $\mu_n\to \mu$ iff $\mu_n(f)\to \mu(f)$ for any $f\in C(X)$ (i.e., by the family of seminorms $\mu\mapsto |\mu(f)|$, $f\in C(X)$).\footnote{Note that $C(X)$ is a Banach space and thus so is ${\rm Meas}(X)$, in the corresponding norm topology. However, this norm topology is stronger than the weak topology.} Thus ${\rm Meas}(X)$ is Hausdorff and locally convex. 
We will also show that ${\rm Meas}(X)$ is sequentially complete and separable. However, if $X$ is infinite, ${\rm Meas}(X)$ is {\bf not} a Fr\'echet space, as it is not metrizable (or first countable). This follows from the following exercise. 

\begin{exercise} Show that if $X$ is an infinite compact space then 
the weak topology of ${\rm Meas}(X)$ cannot be defined by a countable 
system of seminorms. 

{\bf Hint:} Assume the contrary and show that then there exist $f_i\in C(X),i\in \Bbb N$ such that the topology is defined by the seminorms $|\mu(f_i)|$. Deduce that $\lbrace f_i, i\in \Bbb N\rbrace$ must be an algebraic spanning set for $C(X)$ and derive a contradiction. 
\end{exercise} 

In general, if $X$ is not necessarily compact, we endow ${\rm Meas}_c(X)$ with the strict locally convex inductive limit topology defined by \eqref{indli}, and call it the {\bf weak topology} on ${\rm Meas}_c(X)$. 
Thus, a sequence $\lbrace\mu_n\in {\rm Meas}_c(X)\rbrace$ converges in the weak topology iff it is contained in ${\rm Meas}(K)$ for some compact $K\subset X$ and converges there. 

\begin{lemma}\label{seqco} (i) If a sequence $\lbrace\mu_n,n\ge 1\rbrace\in {\rm Meas}_c(X)$ is Cauchy then there is a compact subset $K\subset X$ 
such that $\mu_n\in {\rm Meas}(K)\subset {\rm Meas}_c(X)$ for all $n$. 

(ii) ${\rm Meas}_c(X)$ is sequentially complete. 
\end{lemma} 

\begin{proof} (i) We have $\mu_n-\mu_{n+1}\to 0$, $n\to \infty$. 
Thus there exists a compact subset $K\subset X$ and 
$N\in \Bbb N$ such that for $n\ge N$, $\mu_n-\mu_{n+1}\in {\rm Meas}(K)$. 
We can enlarge $K$ so that also $\mu_N\in {\rm Meas}(K)$. 
Then $\mu_n\in {\rm Meas}(K)$ for all $n\ge N$. We can now further enlarge $K$ 
to make sure that $\mu_i\in {\rm Meas}(K)$ for $i<N$, which completes the proof. 

(ii) Let $\lbrace \mu_n,n\ge 1\rbrace$ be a Cauchy sequence in ${\rm Meas}_c(X)$. By (i), $\mu_n\in {\rm Meas}(K)$ for some compact $K\subset X$, so we may assume that $X$ is compact. Since $\mu_n$ is Cauchy, so is $\mu_n(f)$ for any $f\in C(X)$. Thus $\mu_n$ weakly converges to some linear functional $\mu: C(X)\to \Bbb C$ given by $\mu(f):=\lim_{n\to \infty}\mu_n(f)$, and our job is to show, that $\mu$ is continuous. Since $\mu_n(f)$ is convergent, it is bounded, so
by the uniform boundedness principle, the sequence $\norm{\mu_n}$ is bounded above by some 
constant $C$, i.e., $|\mu_n(f)|\le C\norm{f}$. But then $|\mu(f)|\le C\norm{f}$, so $\norm{\mu}\le C$, as desired. 
\end{proof} 

\begin{remark} In fact, it can be shown that ${\rm Meas}_c(X)$ is complete, not just sequentially complete. 
But we will not need this. 
\end{remark} 

\subsection{Support of a measure} 

Define the {\bf support} of $\mu\in {\rm Meas}_c(X)$, denoted ${\rm supp}\mu$, to be the set of 
all $x\in X$ such that for any neighborhood $U$ of $x$ in $X$ there exists 
$f\in C(X)$ vanishing outside $U$ with $\mu(f)\ne 0$. Thus the complement 
$({\rm supp}\mu)^c$ is the set of $x\in X$ which admit a neighborhood $U$ 
such that for every $f\in C(X)$ vanishing outside $U$ we have $\mu(f)=0$. 
In this case, $U\subset ({\rm supp}\mu)^c$, so $({\rm supp}\mu)^c$ is open, 
hence ${\rm supp}\mu$ is closed. Moreover, since ${\rm Meas}_c(X)=\underrightarrow{\lim_{K\subset X}}{\rm Meas}(K)$, 
${\rm supp}\mu$ is contained in some compact subset $K\subset X$, so it is itself compact. 

\begin{proposition} If $f\in C(X)$ and $f|_{{\rm supp}\mu}=0$ then $\mu(f)=0$. 
\end{proposition} 

\begin{proof} For every $z\in ({\rm supp}\mu)^c$ there is a neighborhood $U_z$ of $z$ such that $\overline{U_z}\subset ({\rm supp}\mu)^c$ and for any $\phi\in C(X)$ vanishing outside $U_z$, $\mu(\phi)=0$. These neighborhoods form an open cover of $({\rm supp}\mu)^c$. Since $({\rm supp}\mu)^c$ is second countable, this cover has a countable subcover
$\lbrace U_i,i\in \Bbb N\rbrace$. Let $\lbrace \phi_i,i\in \Bbb N\rbrace$ 
be a partition of unity 
subordinate to this cover (which exists by Lemma \ref{parun}). We extend $\phi_i$ by zero 
to ${\rm supp}\mu$ (this extension is continuous since $\overline U_z\cap {\rm supp}\mu=\emptyset$). Then $\mu(\phi_if)=0$ for all $i$. 

Moreover, we claim that $\sum_{i\ge 1}\phi_i f=f$ uniformly on any compact subset $K\subset X$.  
Indeed, let $\varepsilon>0$, and let $U\subset X$ be the set of points $x$ where $|f(x)|<\varepsilon$ (so $U$ is a neighborhood of ${\rm supp}\mu$). Since $K\cap U^c\subset ({\rm supp}\mu)^c$ is compact and the series $\sum_i \phi_i=1$ converges uniformly on compact sets in $({\rm supp}\mu)^c$, there is $n$ such that 
$$
\sup_{K\cap U^c}|(1-\sum_{i=1}^n \phi_i)f|<\varepsilon. 
$$
But this inequality also holds on $U$, hence on the whole $K$, as desired. 

It follows that $\mu(f)=\mu(\sum_i \phi_if)=\sum_i \mu(\phi_if)=0$, as claimed. 
\end{proof} 

\subsection{Finitely supported measures} 
A basic example of an element of ${\rm Meas}_c(X)$ is a {\bf Dirac measure} $\delta_a$, 
$a\in X$, such that $\delta_a(f)=f(a)$. Thus if $a_n\to a$ in $X$ as $n\to \infty$ 
then $\delta_{a_n}\to \delta_a$ in the weak topology. A finite linear combination of Dirac measures is called a {\bf finitely supported} measure, since such measures are exactly the measures with finite support. The subspace of finitely supported measures is denoted ${\rm Meas}_c^0(X)$, or ${\rm Meas}^0(X)$ for compact $X$. 

\begin{lemma}\label{le1} ${\rm Meas}_c^0(X)$ is a sequentially dense (in particular, dense) 
subspace in ${\rm Meas}_c(X)$, i.e., every element $\mu \in {\rm Meas}_c(X)$ is the limit of a sequence $\mu_n\in {\rm Meas}_c^0(X)$ in the weak topology. 
\end{lemma} 

\begin{proof} By replacing $X$ with ${\rm supp}\mu$, we may assume that $X$ is compact. For every $n\ge 1$, let $X_n$ be a finite subset of $X$ such that the open balls $B(x,\frac{1}{n})$ around $x\in X_n$ cover $X$.
Let $\lbrace \phi_{nx},x\in X_n  \rbrace$ be a partition of unity subordinate to this cover (Lemma \ref{parun}), and let 
$$
\mu_n:=\sum_{x\in X_n}\mu(\phi_{nx})\delta_x\in {\rm Meas}_c^0(X).
$$ 
We claim that $\mu_n\to \mu$ in the weak topology. 

Indeed, let $f\in C(X)$. Then 
$$
|\mu_n(f)-\mu(f)|= |\mu(\sum_{x\in X_n}\phi_{nx}(f-f(x))|\le \norm{\mu}\sup_{y\in X}\sum_{x\in X_n}\phi_{nx}(y)|f(y)-f(x)|.
$$
But $f$ is uniformly continuous, so for every $\varepsilon>0$ 
there is $N$ such that if $d(x,y)<\frac{1}{N}$ then $|f(x)-f(y)|<\varepsilon$. 
So for $n\ge N$, whenever $\phi_{nx}(y)\ne 0$, we have $|f(y)-f(x)|<\varepsilon$. Thus we get 
$$
|\mu_n(f)-\mu(f)|\le \varepsilon \norm{\mu}\sup_{y\in X}\sum_{x\in X_n}\phi_{nx}(y)=\varepsilon \norm{\mu},
$$
which implies the desired statement. 
\end{proof} 

Note that since $X$ is separable, so is ${\rm Meas}_c^0(X)$ (given a countable dense subset $T\subset X$, 
finitely supported measures with support in $T$ and Gaussian rational coefficients form a countable, sequentially dense subset $E_T\subset {\rm Meas}_c^0(X)$). Thus we get that ${\rm Meas}_c(X)$ 
is separable; moreover, since $E_T$ is {\it sequentially} dense in ${\rm Meas}_c(X)$, the latter is {\it sequentially separable}. 

\begin{corollary}\label{boxt} If $X,Y$ are locally compact second countable Hausdorff spaces then 
the natural bilinear map 
$$
\boxtimes: {\rm Meas}_c^0(X)\times {\rm Meas}_c^0(Y)\to {\rm Meas}_c(X\times Y)
$$
uniquely extends to a bilinear map 
$$
\boxtimes: {\rm Meas}_c(X)\times {\rm Meas}_c(Y)\to {\rm Meas}_c(X\times Y)
$$ 
which is continuous in each variable. 
\end{corollary} 

\begin{proof} We may assume that $X,Y$ are compact. 
Given $\mu\in {\rm Meas}(X),\nu\in {\rm Meas}(Y)$, define a linear functional $\mu\boxtimes \nu$ on $C(X)\otimes C(Y)\subset C(X\times Y)$ by 
$$
(\mu\boxtimes \nu)(f\otimes g):=\mu(f)\nu(g).
$$ 
We claim that 
$\norm{\mu\boxtimes \nu}\le \norm{\mu}\norm{\nu}$ (in fact, the opposite inequality is obvious, so 
we have an equality). Thus our job is to show that for $f_i\in C(X),g_i\in C(Y)$, $1\le i\le n$, we have 
$$
|\sum_i \mu(f_i)\nu(g_i)|\le \norm{\mu}\norm{\nu}\max_{x\in X,y\in Y}|\sum_i f_i(x)g_i(y)|
$$
i.e., that 
$$
|\nu(\sum_i \mu(f_i)g_i)|\le \norm{\mu}\norm{\nu}\max_{x\in X,y\in Y}|\sum_i f_i(x)g_i(y)|.
$$
To this end, it suffices to show that 
$$
\max_{y\in Y}|\sum_i \mu(f_i)g_i(y)|\le \norm{\mu}\max_{x\in X,y\in Y}|\sum_i f_i(x)g_i(y)|,
$$
which would follow from the inequality
$$
|\sum_i \mu(f_i)g_i(y)|\le \norm{\mu}\max_{x\in X}|\sum_i f_i(x)g_i(y)|.
$$
for all $y\in Y$. But this is just the inequality $|\mu(F_y)|\le \norm{\mu}\max_{x\in X}|F_y(x)|$
applied to $F_y(x):=\sum_i g_i(y)f_i(x)$. 

Now note that by the Stone-Weierstrass theorem, $C(X)\otimes C(Y)$ is dense in $C(X\times Y)$, 
so $\mu\boxtimes \nu$ extends continuously to $C(X\times Y)$.
\end{proof} 

\subsection{The algebra of measures on a locally compact group} 

Now let $G$ be a locally compact group. In this case ${\rm Meas}_c^0(G)=\Bbb CG$ is the group algebra of $G$ as an abstract group. Namely, the algebra structure is given by $\delta_x\delta_y=\delta_{xy}$. 
This multiplication uniquely extends to ${\rm Meas}_c(G)$ since the latter is sequentially complete 
and ${\rm Meas}_c^0(G)$ is sequentially dense in ${\rm Meas}_c(G)$ (cf. Corollary \ref{boxt}). Thus ${\rm Meas}_c(G)$ is a unital associative algebra with unit $\delta_1$. 
The multiplication in this algebra may be written as 
$$
(\mu_1\ast \mu_2)(f)=(\mu_1\boxtimes \mu_2,\Delta(f))=\int_{G\times G}f(xy)d\mu_1(x)d\mu_2(y), 
$$
where $\Delta: C(G)\to C(G\times G)$ is given by $\Delta(f)(x,y):=f(xy)$. 
This multiplication is called the {\bf convolution product}. 

Moreover, if $dg$ is a right-invariant Haar measure on $G$ then 
any compactly supported continuous function (or, more generally, $L^1$-function) 
$\phi$ on $G$ gives rise to a measure $\mu=\phi dg\in {\rm Meas}_c(G)$. 
For such measures $\mu_1=\phi_1dg,\mu_2=\phi_2dg$ we have 
$$
(\mu_1\ast \mu_2)(f)=\int_{G\times G}f(xy)\phi_1(x)\phi_2(y)dxdy=
\int_{G\times G}f(z)\phi_1(zy^{-1})\phi_2(y)dzdy.
$$
Thus $\mu_1\ast \mu_2=\phi dg$ where 
$$
\phi(z)=\int_G \phi_1(zy^{-1})\phi_2(y)dy.
$$
This operation is called the {\bf convolution of functions}.  

Now let $V$ be a continuous representation of $G$ with the associated homomorphism 
$\pi: G\to {\rm Aut}(V)$. This map $\pi$ extends by linearity to a homomorphism
$\pi: \Bbb CG={\rm Meas}_c^0(G)\to \End(V)$. Given $v\in V$, define the linear map 
$\pi_v: {\rm Meas}_c^0(G)\to V$ by the formula $\pi_v(\mu):=\pi(\mu)v$. 

\begin{lemma}\label{le2} (i) Let \(K\) be compact metric space and \(E\) be a Banach space.
Then the algebraic tensor product $C(K)\otimes E$ is dense in $C(K,E)$. 
Thus for every $\mu\in {\rm Meas}(K)$ the operator $\mu\otimes 1_E: C(K)\otimes E\to E$ 
extends to an operator $\widehat\mu: C(K,E)\to E$ with norm $\|\mu\|$. 

(ii) Let $\mu_n\to 0$ in ${\rm Meas}(K)$ in the weak topology. 
Then for every $f\in C(K,E)$ one has $\widehat\mu_n(f)\to 0$. 

(iii) The map $\pi_v: {\rm Meas}_c^0(G)\to V$ is sequentially continuous.
In other words, if $\mu_n\to 0$ in ${\rm Meas}^0_c(G)$ then $\pi_v(\mu_n)\to 0$. 

(iv) Let $\mu_n\in {\rm Meas}^0_c(G)$ converge to $\mu\in {\rm Meas}_c(G)$ 
in the weak topology. Then the sequence $\pi_v(\mu_n)$ converges, and its limit depends only on $\mu$. Moreover, the formula $\pi_v(\mu):=\lim_{n\to \infty}\pi_v(\mu_n)$ defines an extension of $\pi_v$ to a linear map ${\rm Meas}_c(G)\to V$.

(v) This gives rise to an extension of $\pi$ to a unital algebra homomorphism ${\rm Meas}_c(G)\to \End(V)$. 

(vi) The map $\pi_v: {\rm Meas}_c(G)\to V$ is sequentially continuous. 
In other words, if $\mu_n\to 0$ in ${\rm Meas}_c(G)$ then $\pi_v(\mu_n)\to 0$. 
Thus $\pi$ is sequentially continuous in the strong topology of $\End(V)$. 
\end{lemma} 

\begin{proof} (i) Let $f\in C(K,E)$. Then $f(K)$ is compact, so for any $\varepsilon>0$ 
it is covered by a finite collection of balls $B(e_i,\varepsilon)$ of radius $\varepsilon$ and centers 
$e_i\in E,i=1,...,n$. Let $U_i:=f^{-1}(B(e_i,\varepsilon))$; these sets form a finite open cover of $K$. 
Let $\lbrace\phi_i\rbrace$ be a partition of unity subordinate to this cover. 
Then for every $x\in K$, 
$$
\| f(x)-\sum_i \phi_i(x)e_i\|=\|\sum_i \phi_i(x)(f(x)-e_i)\|\le \sum_i \phi_i(x)\|f(x)-e_i\|.
$$
If $x\in U_i$, the $i$-th term in this sum is $<\varepsilon \phi_i(x)$, while 
if $x\notin U_i$, it is zero as $\phi_i(x)=0$. Thus $\| f(x)-\sum_i \phi_i(x)e_i\|<\varepsilon$. 
Thus $C(K)\otimes E$ is dense in $C(K,E)$. The second statement now follows 
from the continuous linear extension theorem.  

(ii) Let $f_m\to f$ be a sequence of elements of $C(K)\otimes E$ such that $\|f_m-f\|<\frac{1}{m}$ (it exists by (i)). Then we have 
$$
\| \widehat\mu_n(f)\|\le \|\widehat\mu_n(f-f_m)\|+\|\widehat\mu_n(f_m)\|.
$$
By the uniform boundedness principle, there exists $L>0$ such that 
$\|\mu_n\|\le L$. Also $\|\widehat\mu_n(f_m)\|\to 0$, $n\to \infty$. Thus for fixed $m$ 
$$
\limsup_{\n\to \infty}\| \widehat\mu_n(f)\|\le \frac{L}{m}.
$$
Since this holds for all $m$, we get (ii).

(iii) There exists a compact subset $K\subset G$ such that 
$\mu_n\in {\rm Meas}^0(K)$ and converges to $0$ in this space. 
It suffices to show that for every continuous seminorm $p$ on $V$ one has 
$p(\pi(\mu_n)v)\to 0$. Let $V_p$ be the Banach completion of 
$V/{\rm Ker}p$ and $P_p: V\to V_p$ be the natural map; so it is enough to show that 
$\| P_p \pi(\mu_n)v\|_{V_p}\to 0$. However, we have 
a map $f: g\mapsto P_p \pi(g)v$ which belongs to $C(K,V_p)$,  
and $\| P_p \pi(\mu_n)v\|_{V_p}=\|\widehat\mu_n(f)\|_{V_p}$. 
Thus the statement follows by (ii). 

(iv) There is a compact subset $K\subset G$ such that $\mu_n,\mu\in {\rm Meas}(K)$ and $\mu_n\to \mu$ in ${\rm Meas}(K)$. Thus $\mu_n$ is Cauchy, i.e., 
if $m_i,n_i\to \infty$ then $\mu_{m_i}-\mu_{n_i}\to 0$. Hence by (iii), 
$\pi_v(\mu_{m_i})-\pi_v(\mu_{n_i})\to 0$. It follows that the sequence 
$\pi_v(\mu_n)\in V$ is Cauchy. But $V$ is sequentially complete, 
so this sequence converges to some $w\in V$. If $\mu_n'\in {\rm Meas}^0(K)$ 
is another sequence converging to $\mu$ then $\mu_n-\mu_n'$ converges to $0$, 
so $\pi_v(\mu_n-\mu_n')=\pi_v(\mu_n)-\pi_v(\mu_n')$ also converges to $0$ by (iii).
Thus $w$ depends only on $\mu$. The last statement follows from Lemma \ref{le1}. 

(v) Let \(\mu\in \operatorname{Meas}_c(G)\)
be supported on a compact set \(K\subset G\), and let \(p\) be a
continuous seminorm on \(V\).  Since the action map \(K\times V\to V\)
is continuous and \(K\) is compact, there is a continuous seminorm
\(q\) on \(V\) such that
\[
        p(\pi(g)w)\le q(w),\qquad g\in K,\ w\in V .
\]
Hence
\begin{equation}\label{semiest}
        p(\pi(\mu)w)
        \le \|\mu\|
             \sup_{g\in K}p(\pi(g)w)
        \le \|\mu\|q(w).
\end{equation}
Thus \(\pi(\mu):V\to V\) is continuous.

Now let \(\mu,\nu\in\operatorname{Meas}_c(G)\).  Choose compact sets
\(K,L\subset G\) supporting \(\mu,\nu\), and choose sequences
$\mu_n\in\operatorname{Meas}^0(K),\
        \nu_n\in\operatorname{Meas}^0(L),$
such that \(\mu_n\to\mu\) in \(\operatorname{Meas}(K)\) and
\(\nu_n\to\nu\) in \(\operatorname{Meas}(L)\).  Then
$$
        \mu_n*\nu_n\longrightarrow \mu*\nu
        $$
in \(\operatorname{Meas}(KL)\).  Indeed, for \(f\in C(KL)\) this is
the convergence of
\[
        \int_{K\times L} f(xy)\,d\mu_n(x)d\nu_n(y)
\]
to the corresponding integral against \(\mu\boxtimes\nu\), which
follows from the same argument as in part (ii).

Now, for finitely supported measures we obviously have
$$
\pi(\mu_n*\nu_n)=\pi(\mu_n)\pi(\nu_n).
$$
Let us apply both sides to \(v\in V\) and pass to the limit.
The left hand side converges by the sequential continuity of
\(\eta\mapsto \pi(\eta)v\).  On the right hand side,
\(\pi(\nu_n)v\to\pi(\nu)v\), and \eqref{semiest}, applied
to the uniformly bounded sequence \(\{\mu_n\}\subset\operatorname{Meas}(K)\),
gives
\[
        \pi(\mu_n)\pi(\nu_n)v-\pi(\mu_n)\pi(\nu)v\to 0,
\]
while \(\pi(\mu_n)\pi(\nu)v\to\pi(\mu)\pi(\nu)v\).  Hence we obtain
\[
        \pi(\mu*\nu)v=\pi(\mu)\pi(\nu)v,
\]
i.e., $\pi$ is an algebra homomorphism. Also \(\pi(\delta_1)=\operatorname{Id}_V\), so the extended map
\(\operatorname{Meas}_c(G)\to \operatorname{End}(V)\) is a unital
algebra homomorphism.

(vi) The proof is the same as in (iii). 
\end{proof} 

\begin{remark}
Note that if $V$ is infinite dimensional then the map 
$\pi_v: {\rm Meas}_c^0(G)\to V$ need not be continuous, even if $G$ is compact.  
Namely, let $G=S^1$ and recall that every continuous linear map from ${\rm Meas}^0(S^1)=\mathbb C S^1$ with weak topology to a normed space has finite rank. Now let $V=C(S^1)$ with the left-regular action, and choose \(v\in C(S^1)\) whose translates span an infinite-dimensional subspace (a continuous function with infinitely many nonzero Fourier
coefficients). Then $\pi_v$ is not continuous.
\end{remark}

\section{\bf Plancherel formulas, Dirac sequences, smooth vectors} 

\subsection{Plancherel formulas} For a compactly supported $L^1$-function $f$ 
on $G$, for brevity let us denote $\pi(fdg)$ just by $\pi(f)$. 

\begin{proposition} (Plancherel's theorem for compact groups) Let $K$ be a compact group and $f_1,f_2\in L^2(K)$. 
Then 
$$
(f_1,f_2)=\sum_{\rho\in {\rm Irr}K}\dim \rho\cdot {\rm Tr}(\pi_\rho(f_1)\pi_\rho(f_2)^\dagger) 
$$
and this series is absolutely convergent.
\end{proposition} 

\begin{proof} Recall that if $e_i$ is an orthonormal basis of 
a Hilbert space $H$ and $f_1,f_2\in H$ then 
$$
(f_1,f_2)=\sum_i (f_1,e_i)(e_i,f_2)
$$
and this series is absolutely convergent. 
The result now follows by applying this 
formula to the orthonormal basis provided by the Peter-Weyl theorem: 
$$
\psi_{\rho ij}=\sqrt{\dim \rho}(\pi_\rho(g)v_{\rho i},v_{\rho j}),
$$
where $\lbrace v_{\rho i}\rbrace $ is an orthonormal basis of $\rho$. 
\end{proof} 

\begin{example} If $K=S^1$, Plancherel's theorem reduces to the usual Parseval equality
in Fourier analysis: 
$$
(f_1,f_2)=\sum_{n\in \Bbb Z}c_n(f_1)\overline{c_n(f_2)},
$$
where $c_n(f)$ are the Fourier coefficients of $f$.  
\end{example} 

\begin{proposition} (Plancherel's formula) If $K$ is a compact Lie group and 
$f\in C^\infty(K)$ then 
$$
f(1)=\sum_{\rho \in {\rm Irr}K}\dim \rho\cdot {\rm Tr}(\pi_\rho(f)) 
$$
and this series is absolutely convergent. 
\end{proposition}  

\begin{example} If $K=S^1$ then this formula says that for $f\in C^\infty(S^1)$
$$
f(1)=\sum_{n\in \Bbb Z}c_n(f),
$$
i.e., the Fourier series of $f$ absolutely converges at $1$. Note that for $f\in C(S^1)$ 
this is false in general!\footnote{One can show that for an $N$-dimensional group, the differentiability needed for the Plancherel formula is $C^k$ for $k>N/2$.}
\end{example} 

\begin{proof} Consider the integral operator $A$ of convolution with the function $f$: 
$$
(A\psi)(x)=(f\ast \psi)(x)=\int_K f(xy^{-1})\psi(y)dy.
$$
This operator is trace class, since it has smooth integral kernel $F(x,y)=f(xy^{-1})$, and 
$$
{\rm Tr}(A)=\int_K F(x,x)dx=\int_K f(1)dx=f(1). 
$$
On the other hand, $A$ is right-invariant, so it preserves the decomposition 
of $L^2(K)$ into the direct sum of $\rho\otimes \rho^*$ and acts on each such summand as 
$\pi_\rho(f)\otimes 1$. Thus we also have 
$$
{\rm Tr}(A)=\sum_{\rho\in {\rm Irr}K} \dim(\rho)\cdot {\rm Tr}(\pi_\rho(f)),
$$
as desired. 
\end{proof} 

\subsection{Dirac sequences} 

If $G$ is a locally compact group then multiplication by $dg$ defines an inclusion 
$C_c(G)\hookrightarrow {\rm Meas}_c(G)$ of compactly supported continuous functions into compactly supported measures as a (non-unital) subalgebra. Moreover, if $G$ is a Lie group then we have a nested sequence of subalgebras $C_c^k(G)$, $0\le k\le \infty$ (compactly supported $C^k$-functions). The following lemma shows that while these subalgebras are non-unital, they are ``almost unital". 

\begin{lemma}\label{le3} There exists a sequence $\phi_n\in C_c(G)$ such that $\phi_n\to \delta_1$ in the weak topology as $n\to \infty$. Moreover, if $G$ is a Lie group, we can choose $\phi_n\in C^\infty_c(G)$. 
\end{lemma} 

\begin{proof} (sketch) $\phi_n$ can be constructed as a sequence of ``hat" functions
(normalized to have integral $1$) supported on a decreasing sequence of balls $B_1\supset B_2\supset...$ whose intersection is $1\in G$. Such hat functions can be chosen smooth 
if $G$ is a Lie group.  
\end{proof} 

Such sequences $\phi_n$ are called {\bf Dirac sequences}. 

\begin{corollary}\label{co4} $C_c(G)$ is sequentially dense in ${\rm Meas}_c(G)$. For Lie groups, $C_c^\infty(G)$ is sequentially dense in ${\rm Meas}_c(G)$. 
\end{corollary} 

\begin{proof} By translating a Dirac sequence, for any $g\in G$ we can construct a sequence $\psi_n\to \delta_g$. This implies that ${\rm Meas}_c^0(G)$ is contained in the sequential closure of $C_c(G)$ (and of $C_c^\infty(G)$ in the Lie case). So the result follows from 
Lemma \ref{le1}. 
\end{proof} 

\subsection{Density of $K$-finite vectors} 

\begin{corollary}\label{co5} Let $V$ be a continuous representation of a compact 
group $K$. Then $V^{\rm fin}$ is dense in $V$. 
\end{corollary} 

\begin{proof} Let $v\in V$, and $\phi_n\to \delta_1$ a continuous Dirac sequence, which exists by Lemma \ref{le3}. Then 
$\pi(\phi_n)v\to v$ as $n\to \infty$. But $\phi_n\in L^2(K)$, so by 
the Peter-Weyl theorem, there exists $\psi_n\in L^2(K)^{\rm fin}=\oplus_\rho \rho^*\otimes \rho$ such that 
$$
\norm{\psi_n-\phi_n}_2<\frac{1}{n}.
$$ 
Then 
$\psi_n-\phi_n\to 0$ in $L^2(K)$, hence in ${\rm Meas}_c(K)$. 
So by Lemma \ref{le2}, $\pi(\psi_n-\phi_n)v\to 0$ as $n\to \infty$. It follows that 
$\pi(\psi_n)v\to v$ as $n\to \infty$. But $\pi(\psi_n)v\in V^{\rm fin}$. 
\end{proof} 

\begin{corollary} $L^2(K)^{\rm fin}\subset C(K)$ is a dense subspace. 
Moreover, if $K$ is a Lie group then $L^2(K)^{\rm fin}\subset C^k(K)$
is a dense subspace for $0\le k\le \infty$. 
\end{corollary} 

\begin{proof} The claimed inclusions follow since matrix coefficients of finite-dimensional 
representations of $K$ are continuous, and moreover $C^\infty$ in the case of Lie groups. 
The density then follows from Corollary \ref{co5}.  
\end{proof} 

\begin{corollary} If $V$ is an irreducible continuous representation of $K$ then $V$ is finite-dimensional. 
\end{corollary}

\begin{proof} By Corollary \ref{co5}, $V^{\rm fin}$ is dense in $V$. Hence $V^{\rm fin}\ne 0$. Let $\rho$ be a finite-dimensional subrepresentation of $V^{\rm fin}$. Then $\rho$ is a closed invariant subspace of $V$. Hence $\rho=V$.  
\end{proof} 

\subsection{Smooth vectors} 

Let $G$ be a Lie group. As we have noted in Subsection \ref{conrep}, any continuous {\it finite-dimensional} representation $\pi: G\to {\rm Aut}(V)$ is automatically smooth and thereby defines a 
representation $\pi_*: \g\to \End(V)$ of the corresponding Lie algebra, which determines $\pi$ if $G$ is connected. Moreover, if $G$ is simply connected, this correspondence is an equivalence of categories. This immediately reduces the problem to pure algebra and is the main tool of studying finite-dimensional representations of Lie groups.

We would like to have a similar theory for infinite-dimensional representations. But 
in the infinite-dimensional setting the above statements don't hold in the literal sense. 

\begin{example} Consider the action of $S^1$ on $L^2(S^1)$. Then the Lie algebra 
should act by $\frac{d}{d \theta}$. But this operator does not act on $L^2(S^1)$. 
The largest subspace of $L^2(S^1)$ preserved by this operator (acting on distributions on $S^1$) is $C^\infty(S^1)$. 
\end{example} 

This motivates the notion of a {\bf smooth vector} in a continuous representation of a Lie group. To define this notion, for a manifold $X$ and a topological vector space $V$, denote by $C^\infty(X,V)$ the space of smooth maps $X\to V$ (where smooth maps are defined in the same way as in the case of finite-dimensional $V$). 

\begin{definition} Let $(V,\pi)$ be a continuous representation of a Lie group $G$. 
A vector $v\in V$ is called {\bf smooth} if the map $G\to V$ given by $g\mapsto \pi(g)v$ 
is smooth, i.e., belongs to $C^\infty(G,V)$. The space of smooth vectors is denoted by $V^\infty$. 
\end{definition} 

It is clear that $V^\infty\subset V$ is a $G$-invariant subspace (although not a closed one). 

\begin{example} For the representation of a compact Lie group $K$ on $V=L^2(K)$, 
we have $V^\infty=C^\infty(K)$.  
\end{example} 

\begin{proposition}\label{act} Let $(V,\pi)$ be a continuous representation of a Lie group $G$ with $\g={\rm Lie}(G)$. Let $v\in V^\infty$. Then we have a linear map 
$\pi_{*,v}: \g\to V^\infty$ given by 
$$
\pi_{*,v}(b)=\frac{d}{dt}|_{t=0}\pi(e^{tb})v.
$$ 
This defines a Lie algebra homomorphism $\pi_*: \g\to {\rm End}_{\Bbb C}(V^\infty)$ (algebra of all linear endomorphisms of $V^\infty$) given by $\pi_*(b)v:=\pi_{*,v}(b)$. 
\end{proposition} 

\begin{exercise} Prove Proposition \ref{act}.
\end{exercise} 

Let $V^{\rm fin}\subset V$ be the subspace of vectors generating a finite dimensional representation of $G$. 

\begin{proposition} (i) $V^\infty$ is dense in $V$.

(ii) $V^{\rm fin}\subset V^\infty$. 
\end{proposition} 

\begin{proof} (i) Let $\phi_n\to \delta_1$ be a smooth Dirac sequence. 
Then $\pi(\phi_n)v\to v$ as $n\to \infty$. But it is easy to see that 
$\pi(\phi_n)v\in V^\infty$.

(ii) This follows since matrix coefficients of finite-dimensional representations 
are smooth. 
\end{proof} 

\section{\bf Admissible representations and $(\g,K)$-modules} 

\subsection{Admissible representations} 

Now let $G$ be a Lie group and $K\subset G$ a compact subgroup. For a continuous representation $V$ of $G$, denote by $V^{K\text{-fin}}$ the space $(V|_K)^{\rm fin}$.
In general $V^{K\text{-fin}}$ is not contained in $V^\infty$; for example, if $K=1$ then 
 $V^{K\text{-fin}}=V$. However, this inclusion holds if $K$ is sufficiently large and $V$ is sufficiently small. 
 
\begin{definition} $V$ is said to be {\bf $K$-admissible} (or of {\bf finite $K$-type}) if for every finite-dimensional irreducible representation $\rho$ of $K$, the space $\Hom_K(\rho,V)$ is finite-dimensional. 
\end{definition}  

\begin{example} Let $G$ be a connected Lie group and $V=L^2(G/B)$ where $B$ is a closed subgroup of $G$ (half-densities on $G/B$). Then $V$ is $K$-admissible iff 
$K$ acts transitively on $G/B$, i.e., $KB=G$. In this case 
setting $T=K\cap B$, we have $G/B=K/T$, so $V=L^2(K/T)$ and 
$\Hom_K(\rho,V)\cong (\rho^T)^*$.\footnote{Note that here we don't have to distinguish between half-densities and functions on $K/T$ since $K/T$ always has a $K$-invariant volume form as $K$ is compact.} 
\end{example} 

\begin{example}\label{prinser} For $G=SL_2(\Bbb C)$ and $K=SU(2)$, the unitary representation 
of $G$ on the space $V=L^2(\Bbb C \Bbb P^1)$ of square-integrable half-densities on $\Bbb C\Bbb P^1$ is $K$-admissible. Indeed, taking $\rho_n$ 
to be the representation of $SU(2)$ with highest weight $n$, we have 
$\dim\Hom(\rho_n,V)=0$ for odd $n$ and $1$ for even $n$. 

More generally, 
for a real number $s$ we may consider the representation $V_s$ 
of square integrable $\frac{1}{2}+is$-densities on $\Bbb C\Bbb P^1$;
this space is canonically defined since for a $\frac{1}{2}+is$-density 
$f$, the complex conjugate $\overline f$ is a $\frac{1}{2}-is$-density, 
so $|f|^2=f\overline f$ is a density and can be integrated canonically over $\Bbb C\Bbb P^1$. This representation has the same $K$-multiplicities as $V=V_0$. 

Similarly, for $G=SL_2(\Bbb R)$, $K=SO(2)$, we have a unitary $K$-admissible 
representation $V=L^2(\Bbb R\Bbb P^1)$ (half-densities) and more generally 
$V_s$ ($\frac{1}{2}+is$-densities). For the $K$-multiplicities we have equalities
$\dim\Hom(\chi_n,V_s)=1$ for even $n$ and $0$ for odd $n$, where 
$\chi_n(\theta)=e^{in\theta}$. 

We will see that the representations $V_s$ in both cases are irreducible and $V_s,V_t$ are isomorphic iff $s=\pm t$. The family of representations $V_s$ is called the {\bf unitary spherical principal series}. 

Note that this family makes sense also when $s$ is a complex number which is not necessarily real. 
In this case $V_s$ is not necessarily unitary but still a continuous representation on square integrable $\frac{1}{2}+is$-densities.
The space of such densities is canonically defined as a topological vector space, although its Hilbert norm 
is not canonically defined unless $s$ is real (however, we will see that for some non-real $s$, corresponding to so-called {\bf complementary series}, this representation is still unitary, even though the inner product is not given by the standard formula).  The family $V_s$ with arbitrary complex $s$ 
is called the {\bf spherical principal series}. 

Explicitly, the action of $G$ on $V_s$ looks as follows (realizing elements of $V_s$ as functions 
on $\Bbb R$ or $\Bbb C$, removing the point at infinity): 
$$
\left(\begin{pmatrix} a& b\\ c&d
\end{pmatrix}^{-1} f\right)(z)=f\left(\frac{az+b}{cz+d}\right)|cz+d|^{-m(1+2is)},
$$
where $m=1$ in the real case and $m=2$ in the complex case. 
\end{example} 
  
\begin{proposition} 
If $V$ is $K$-admissible then $V^{K\text{-fin}}\subset V^\infty$, and it is a $\g$-submodule 
(although not in general a $G$-submodule). 
\end{proposition} 

\begin{proof} 
For a finite-dimensional irreducible representation $\rho$ of $K$, 
let $V^\rho:=\Hom(\rho,V)\otimes \rho$ be the isotypic component of $\rho$. 

We claim that for any continuous representation $V$ the space $V^\infty\cap V^\rho$ is dense in $V^\rho$. Indeed, let $\psi_\rho\in L^2(K)^{\rm fin}$ be the character of $\rho$ given by 
$$
\psi_\rho=(\dim \rho)^{-\frac{1}{2}}\sum_i \psi_{\rho ii}.
$$
Let $\xi_\rho$ be the pushforward of $\dim(\rho)\overline\psi_\rho dx$ from $K$ to $G$ (a measure on $G$ supported on $K$). 
Then $\pi(\xi_\rho)$ is the projector to $V^\rho$ annihilating the closure $\overline{\oplus_{\eta\ne \rho}V^\eta}$ 
of $\oplus_{\eta\ne \rho}V^\eta$. Let $\phi_n\to \delta_1$ be a smooth Dirac sequence on $G$. Then for $v\in V^\rho$, 
$$
\pi(\xi_\rho\ast \phi_n)v=\pi(\xi_\rho)\pi(\phi_n)v\to \pi(\xi_\rho)v=v
$$ 
as $n\to \infty$. 
However, $\xi_\rho\ast \phi_n$ is smooth, so $\pi(\xi_\rho\ast \phi_n)v\in V^\infty\cap V^\rho$. 

Thus if $V^\rho$ is finite-dimensional (which happens for $K$-admissible $V$) then 
$V^\infty\cap V^\rho=V^\rho$, so $V^\rho\subset V^\infty$. Hence $V^{K\text{-fin}}\subset V^\infty$. 

Finally, it is clear that 
for $b\in \g$ and $v\in V^\rho$, the vector $bv$ generates a $K$-submodule of a multiple of $\g\otimes \rho$, so $bv\in V^{K\text{-fin}}$. It follows that 
$V^{K\text{-fin}}$ is a $\g$-submodule.  
\end{proof} 

\begin{example} If $G=SL_2(\Bbb R)$, $K=SO(2)$, $V=V_s=L^2(S^1)$ is a spherical principal series representation, then $V^{K\text{-fin}}$ is the space of trigonometric polynomials. Note that this space is {\it not} invariant under the action of $G$. However, 
the Lie algebra $\g=\mathfrak{sl}_2(\Bbb R)$ does act on this space. 
\end{example}  

\begin{exercise} Compute this Lie algebra action in the basis $v_n=e^{in\theta}$ and write it as first order differential operators in the angle $\theta$. (Pick generators $e,h,f$ in $\g_{\Bbb C}$ 
so that $h$ acts diagonally in the basis $v_i$). 
\end{exercise} 

\subsection{$(\g,K)$-modules} This motivates the following definition. 
Let $K$ be a compact connected Lie group and $\kf={\rm Lie}K$. 
Let $\g$ be a finite-dimensional real Lie algebra containing $\kf$, and suppose 
the adjoint action of $\kf$ on $\g$ integrates 
to an action of $K$. In this case we say that $(\g,K)$ 
is a {\bf Harish-Chandra pair}. 

\begin{definition} Let $(\g,K)$ be a Harish-Chandra pair. 

(i) A $(\g,K)$-{\bf module} is a vector space $M$ with actions of $K$ and $\g$ 
such that 

$\bullet$ $M$ is a direct sum of finite-dimensional continuous $K$-modules; 

$\bullet$ the two actions of $\kf$ on $M$ (coming from the actions of $\g$ and $K$) coincide. 

(ii) Such a module is said to be {\bf admissible} if for every $\rho\in {\rm Irr}K$ we have $\dim \Hom_K(\rho,M)<\infty$. 

(iii) An admissible $(\g,K)$-module which is finitely generated over $U(\g)$ is called a {\bf Harish-Chandra module}. 
\end{definition} 

\begin{exercise} Show that if $M$ is a $(\g,K)$-module then for every
$g\in K,a\in \g, v\in M$ we have
$$
gav={\rm Ad}(g)(a)gv, 
$$ 
where ${\rm Ad}$ denotes the $K$-action on $\g$.
\end{exercise} 

In fact, a $(\g,K)$-module is a purely algebraic object, since finite-dimensional $K$-modules can be described as algebraic representations of the complex reductive group $K_{\Bbb C}$. 
Moreover, we can represent them even more algebraically in terms of the action of $\kf$. 
Namely, let us say that a finite-dimensional representation of $\kf$ is {\bf integrable} to $K$ if it corresponds to a representation of  $K$ (note that this is automatic if $K$ is simply connected). Then $(\g,K)$-modules are simply $\g$-modules which are locally integrable to $K$ when restricted to $\kf$ (i.e., sum of integrable modules). So if $K$ is simply connected (in which case $\kf$ is semisimple) then a $(\g,K)$-module is the same thing as a  $\g$-module which is locally finite when restricted to $\kf$ (i.e., sum of finite-dimensional modules).

 Thus $(\g,K)$-modules form an abelian category closed under extensions (and this category can be defined over any algebraically closed field of characteristic zero). The same applies to admissible $(\g,K)$-modules and to Harish-Chandra modules (the latter is closed under taking kernels of morphisms because 
the algebra $U(\g)$ is Noetherian, as so is its associated graded $S\g$ by the Hilbert basis theorem). 

\begin{example} Let $G$ be a simply connected complex semisimple Lie group. Then its maximal compact subgroup
is the compact form $K=G_c$. Thus a $(\g,K)$-module is a $\g$-module $M$ which is 
locally finite for $\g_c\subset \g$, where $\g_c={\rm Lie}G_c$. Note that the action of $\g$ here is only 
{\bf real linear}. Thus we may pass to complexifications: 
$(\g_c)_{\Bbb C}=\g$, $\g_{\Bbb C}=\g\oplus \g$, and $\g$ sits inside 
$\g\oplus \g$ as the diagonal. Thus a $(\g,K)$-module is simply a $\g\oplus \g$-module which is locally finite for the diagonal copy of $\g$. This is the same as a $\g$-bimodule\footnote{Indeed, every $\g\oplus \g$-module $M$ with action 
$(a,b,v)\mapsto (a,b)\circ v$, $a,b\in \g$, $v\in M$ is a 
$\g$-bimodule with $av=(a,0)\circ v$ and $vb=(0,-b)\circ v$, and vice versa.} with locally finite adjoint action 
$$
{\rm ad}(b)m:=[b,m]=bm-mb.
$$ 
For example, if $I$ is any two-sided ideal in $U(\g)$ then $U(\g)/I$ 
is a $(\g,K)$-module. 
\end{example} 

Thus we obtain the following proposition. 

\begin{proposition} If $V$ is a $K$-admissible continuous representation of $G$ then $V^{K\text{-fin}}$ is an admissible $(\g,K)$-module.
\end{proposition}  

\begin{exercise} Show that for any continuous representation $V$ of $G$, 
the intersection $V^\infty\cap V^{K\text{-fin}}$ is a $(\g,K)$-module (not necessarily admissible). 
\end{exercise}

\begin{exercise}\label{tenspro} Show that if $V$ is an admissible representation 
of $G$ and $L$ a finite-dimensional (continuous) representation of $G$ then $V\otimes L$ is also admissible. Prove the same statement for $(\g,K)$-modules.
\end{exercise} 

\subsection{Harish-Chandra's admissibility theorem} 

We will now restrict our attention to {\bf semisimple} Lie groups $G$. By this we will mean 
a connected linear real Lie group $G$ with semisimple Lie algebra $\g$. ``Linear" means 
that it has a faithful finite-dimensional representation, i.e., is isomorphic to a closed subgroup
of $GL_n(\Bbb C)$. In other words, $G$ is the connected component of the identity in $\bold G(\Bbb R)$, where $\bold G$ is a semisimple algebraic group defined over $\Bbb R$.  Typical examples of such groups include $SL_n(\Bbb R)$ and $SL_n(\Bbb C)$ (in the latter case $\bold G=SL_n\times SL_n$ 
and the real structure defined by the involution permuting the two factors).  

A fundamental result about the structure of semisimple Lie groups is 

\begin{theorem} (E. Cartan)
Every semisimple Lie group $G$ has a maximal compact subgroup $K\subset G$ 
which is unique up to conjugation. 
\end{theorem} 

\begin{example} For $G=SL_n(\Bbb R)$ we have $K=SO(n)$ and 
for $G=SL_n(\Bbb C)$ we have $K=SU(n)$. 
\end{example} 

We will say that a continuous representation $V$ of $G$ is {\bf admissible} if it is $K$-admissible with respect to a maximal compact subgroup $K\subset G$ (this does not depend on the choice of $K$ since they are all conjugate).  

\begin{theorem}\label{admi} (Harish-Chandra's admissibility theorem, \cite{HC2}) Every irreducible unitary representation of a semisimple Lie group is admissible. 
\end{theorem} 

We will not give a proof (see \cite{HC2},\cite{Ga}). 

\begin{remark} 1. This theorem extends straightforwardly to the more general case of real reductive Lie groups.  

2. Let $G=\widetilde{SL_2(\Bbb R)}$ be the universal covering of $SL_2(\Bbb R)$. Then $G$ is not linear (why?) and so it is {\bf not} viewed as a semisimple Lie group according to our definition. In fact, Harish-Chandra's theorem does not hold as stated for this group, since it has no nontrivial compact subgroups. This happens because when we take the universal cover, the maximal compact subgroup $SO(2)=S^1$ gets replaced by the noncompact group $\Bbb R$. However, if we take for $K$ the universal cover of $SO(2)$ (even though it is not compact) then Harish-Chandra's theorem extends straightforwardly to this case. 
\end{remark}

\begin{exercise} \label{dualmod} 
Let $M$ be an admissible $(\g,K)$-module and 
$$
M^\vee:=\oplus_{V\in {\rm Irr}K}(\Hom(V,M)\otimes V)^*\subset M^*
$$ 
be the restricted dual to $M$. 
Show that $M^\vee$ has a natural structure of an admissible $(\g,K)$-module, and $(M^\vee)^\vee\cong M$. 
\end{exercise} 

\section{\bf Weakly analytic vectors} 

\subsection{Weakly analytic vectors and Harish-Chandra's analyticity} 

Let $G$ be a Lie group, $V$ a continuous representation of $G$, and $V^*$ its continuous dual. 

\begin{definition} A vector $v\in V$ is called {\bf weakly analytic} if for each 
$h\in V^*$ the matrix coefficient $h(gv)$ is a real analytic complex-valued function of $g$. 
\end{definition}

\begin{example} Let $V=L^2(S^1)$ and $G=S^1$ act by rotations.
So if $v(x)=\sum_{n\in \Bbb Z}v_ne^{inx}$ and $h(x)=\sum_{n\in \Bbb Z}h_ne^{-inx}$ 
then for $g=e^{i\theta}$ we have 
$$
h(g(\theta)v)=\sum_{n\in \Bbb Z}h_nv_ne^{in\theta}.
$$
Thus $v$ is a weakly analytic vector iff
the sequence $h_nv_n$ decays exponentially for any $\ell_2$-sequence $\lbrace h_n\rbrace$, which is equivalent to saying that $v_n$ decays exponentially, i.e., $v(\theta)$ is analytic.  
\end{example} 

\begin{theorem}\label{analy} (Harish-Chandra's analyticity theorem) If $V$ is an admissible representation of a semisimple Lie group $G$ with maximal compact subgroup $K$ then every $v\in V^{K\text{-fin}}$ is a weakly analytic vector. 
\end{theorem} 

Theorem \ref{analy} is proved in the next two subsections. 

\subsection{Elliptic regularity} 
The proof of Theorem \ref{analy} is based on the {\bf analytic elliptic regularity theorem}, which is a fundamental result in analysis (see \cite{Ca}). To state it, let $X$ be a smooth manifold, and $D(X)$ the algebra of (real) differential operators on $X$. This algebra has a filtration by order: $D_0(X)=C^\infty(X)\subset D_1(X)\subset...$, such that 
$$
D_n(X)=\lbrace D\in \End_{\Bbb C} C^\infty(X): [D,f]\in D_{n-1}(X)\forall f\in C^\infty(X)\rbrace,\ n\ge 1,
$$ 
and 
${\rm gr}D(X)=\oplus_{n\ge 0}\Gamma(X,S^nTX)$, where $\Gamma$ takes sections of the vector bundle. Thus for every differential operator $D$ on $X$ of order $n$ we have its {\bf symbol} 
$\sigma(D)\in {\rm gr}_n D(X)=\Gamma(X,S^nTX)$. For every $x\in X$, $\sigma(D)_x$ 
is thus a homogeneous polynomial of degree $n$ on $T_x^*X$. 

\begin{definition} We say that $D$ is {\bf elliptic} at $x$ if $\sigma(D)_x(p)\ne 0$ for nonzero 
$p\in T_x^*X$. We say that $D$ is {\bf elliptic} (on $X$) if it is elliptic at all points $x\in X$. 
\end{definition}

\begin{example} 1. If $\dim X=1$ then any differential operator with nonvanishing symbol is elliptic. 

2. Fix a Riemannian metric on $X$ and let $\Delta$ be the corresponding Laplace operator, $\Delta f={\rm div}({\rm grad}f)$. 
Then $\Delta$ is elliptic. 

3. If $D$ is elliptic then for any nonzero polynomial $P\in \Bbb R[t]$ the operator $P(D)$ is elliptic. 
\end{example} 

Note that ellipticity is an open condition, since it is equivalent to non-vanishing of $\sigma(D)_x$ 
on the unit sphere in $T_x^*X$ (under some inner product). Thus the set of $x\in X$ 
on which a given operator $D$ is elliptic is open in $X$. 

\begin{theorem} (Elliptic regularity) Suppose $D$ is an elliptic operator with real analytic coefficients 
on an open set $U\subset \Bbb R^N$, and $f(x)$ is a smooth solution of the PDE 
$$
Df=0
$$ 
on $U$. 
Then $f$ is real analytic on $U$.  
\end{theorem} 

\begin{corollary}\label{ellreg} Let $X$ be a real analytic manifold and $D$ an elliptic operator on $X$ with analytic coefficients. 
Then every smooth solution of the equation $Df=0$ on $X$ is actually real analytic. 
\end{corollary} 

\begin{remark} 1. This is, in fact, true much more generally, when $f$ is a weak (i.e., distributional) solution of the equation $Df=0$. Also the equation 
$Df=0$ can be replaced by a more general inhomogeneous equation $Df=g$, where $g$ 
is analytic. 

2. If $D$ is not elliptic, there is an obvious counterexample: the equation 
$\frac{\partial^2f}{\partial x\partial y}=0$ on $\Bbb R^2$ has smooth non-analytic solutions of the form $f(x)+g(y)$, $f,g\in C^\infty(\Bbb R)$. 
\end{remark} 

\begin{example} 1. For $N=1$ this theorem just says that a solution of an ODE 
$$
f^{(n)}(x)+a_1(x)f^{(n-1)}(x)+....+a_n(x)f(x)=0
$$
with real analytic coefficients is itself real analytic, a classical fact about ODE.

2. Let $N=2$ and $D=\Delta$ be the Laplace operator on $U\subset \Bbb R^2$ with respect to some Riemannian metric with real analytic coefficients. This metric defines a conformal structure with real analytic local complex coordinate $z$. Then every harmonic function $f$ (i.e., one satisfying $\Delta f=0$) is a real part of a holomorphic function of $z$, hence is real analytic, which proves elliptic regularity in this special case. 

3. Suppose $f,g$ are Schwartz functions on $\Bbb R^n$ and $D=Q(\partial)$ is an elliptic operator with constant coefficients, where $Q$ is a real polynomial (so the leading term of $Q$ is nonvanishing for nonzero vectors). Suppose $Df=g$. Then elliptic regularity says that if $g$ is analytic, so is $f$. This can be easily proved using Fourier transform. Indeed, for Fourier transforms we get $Q(p)\widehat f(p)=\widehat g(p)$. Thus $\widehat f(p)=\frac{{\widehat g}(p)}{Q(p)}$, so this must be a smooth function. Note that $|Q(p)|\to \infty$ as $p\to \infty$ because $Q$ has non-vanishing leading term. So, since $g$ is analytic, $\widehat g$ decays exponentially at infinity, hence so does $\widehat f$. Thus $f$ is analytic.
\end{example}

\subsection{Proof of Harish-Chandra's analyticity Theorem}
We are now ready to prove Theorem \ref{analy}. Let $\g={\rm Lie}G$ and $b\in U(\g)$. 
Then we have a linear operator $\pi_*(b): V^\infty\to V^\infty$, which we will write just as $b$ for short. Moreover, if $b\in U(\g)^K$ then it preserves the subspace $V^\rho\subset V^\infty$ 
for each irreducible representation $\rho$ of $K$. Therefore, since all $V^\rho$ are finite-dimensional, for any $v\in V^{K\text{-fin}}$ there exists a 
nonzero polynomial $P\in \Bbb R[t]$ such that $P(b)v=0$. E.g., denoting by $\widetilde{Kv}$ the sum of isotypic 
components in $V$ of all $K$-modules occurring in $Kv$ (which is finite dimensional by admissibility of $V$), we can take $P$ to be the product of the characteristic polynomial of $b$ on $\widetilde{Kv}$ by its complex conjugate.  

Now recall that $U(\g)$ can be thought of as the algebra of left-invariant real differential operators 
on $G$. Let $\psi_{h,v}(g):=h(gv)$ be the matrix coefficient function. 
We know that this function is smooth, and we have 
$$
(P(b)\psi_{h,v})(g)=h(gP(b)v)=0. 
$$
Thus if $b$ is an elliptic differential operator on $G$, it 
will follow from Corollary \ref{ellreg} that $\psi_{h,v}$ is real analytic, as desired. 

It remains to find $b\in U(\g)^K$ which defines an elliptic operator on $G$. 
For this purpose fix a left-invariant Riemannian metric on $G$, and make it $K$-invariant 
(under right, or, equivalently, adjoint action) by averaging over $K$. Then the Laplace operator $\Delta$ 
corresponding to this metric is elliptic and given by some element $\Delta\in U(\g)^K$, so we may take $b=\Delta$. This proves Theorem \ref{analy}. 

\begin{remark} If $G$ is simple, there exists a unique up to scaling two-sided invariant 
metric on $G$. This metric, however, is pseudo-Riemannian rather than Riemannian 
if $G$ is not compact. Thus the corresponding Laplace operator is hyperbolic rather than elliptic, so not suitable for our purposes. 
\end{remark} 

\subsection{Applications of weakly analytic vectors} 

\begin{corollary} The action of $G$ on $V$ is completely determined by the 
corresponding $(\g,K)$-module 
$V^{K\text{-fin}}$. 
\end{corollary} 

\begin{proof} Since $V^{K\text{-fin}}$ is dense in $V$, it suffices to specify $gv$ for $v\in V^{K\text{-fin}}$. For this it suffices to specify $h(gv)$ for all $h\in V^*$. 
By Theorem \ref{analy} and the analytic continuation principle, this is determined by the derivatives of all orders of 
$h(gv)$ at $g=1$. But these have the form $h(bv)$ where $b\in U(\g)$, 
so are determined by $bv$. 
\end{proof} 

\begin{corollary} Let $W\subset V^{K\text{-fin}}$ be a sub-$(\g,K)$-module. 
Then the closure $\overline{W}\subset V$ is $G$-invariant. 
\end{corollary} 

\begin{proof} Let $w\in W$, $g\in G$. It suffices to show that $gw\in \overline W$. 
If not, then the space $W':=\overline W\oplus \Bbb Cgw$ is a closed subspace
of $V$ containing $\overline W$ as a subspace of codimension $1$. 
So there exists a unique continuous linear functional $h: W'\to\Bbb C$ such that 
$h(gw)=1$ and $h|_{\overline W}=0$. By the Hahn-Banach theorem, $h$ can be extended 
to an element of $V^*$. Thus to get a contradiction, it is enough to show that for every $h\in V^*$ that vanishes on $\overline W$, we have $h(gw)=0$.  
But by Theorem \ref{analy}, this function is analytic in $g$. So it suffices to check 
that its derivatives at $g=1$ vanish. But these derivatives are of the form $h(bw)$ for $b\in U(\g)$, so vanish since $bw\in W$. 
\end{proof} 

\begin{corollary}\label{bije} Let $V$ be an admissible representation of $G$. There is a bijection between subrepresentations of $V$ and $(\g,K)$-submodules of $V^{K\text{-fin}}$, given by 
$\alpha: U\subset V\mapsto U^{K\text{-fin}}$. The inverse is given by $\beta: W\mapsto \overline W$.  
\end{corollary} 

\begin{proof} Since $U^{K\text{-fin}}$ is dense in $U$, we have $\beta\circ \alpha={\rm Id}$. 
To show that $\alpha\circ\beta={\rm Id}$, we need to show that $\overline W^{K\text{-fin}}=W$. 
Clearly $\overline W^{K\text{-fin}}$ contains $W$, so we just need to prove the opposite inclusion. Let $w\in \overline W^\rho$, then we have a sequence $w_n\to w$, $w_n\in W$. Now apply the projector $\xi_\rho$: 
$$
w_n':=\pi(\xi_\rho)w_n\to \pi(\xi_\rho)w=w, \ n\to \infty, 
$$ 
and $w_n'\in W^\rho$. Thus $w\in \overline{W^\rho}=W^\rho$, since $W^\rho$ is finite-dimensional. Hence $\overline W^{K\text{-fin}}=W$.
\end{proof} 

\begin{corollary} If $V$ is irreducible then $V^{K\text{-fin}}$ is an irreducible $(\g,K)$-module, and vice versa. 
\end{corollary} 

\begin{corollary} If $V$ is of finite length then $V^{K\text{-fin}}$ is a Harish-Chandra module. 
\end{corollary} 

\begin{proof} By Corollary \ref{bije}, $V^{K\text{-fin}}$ is a finite length $(\g,K)$-module. But any finite length $(\g,K)$-module is finitely generated over $U(\g)$, hence a Harish-Chandra module.   
\end{proof} 

Let $\Rep G$ denote the category of admissible representations of $G$ of finite length, 
and $\mathcal{HC}_G$ the category of Harish-Chandra modules for $G$. 
Thus we obtain 

\begin{theorem}\label{exfai} The assignment $V\mapsto V^{K\text{-fin}}$ defines an exact, faithful functor 
$\Rep G\to \mathcal{HC}_G$, which maps irreducibles to irreducibles.  
\end{theorem}  

\section{\bf Infinitesimal equivalence and globalization} 

\subsection{Infinitesimal equivalence} 

The functor of Theorem \ref{exfai} is not full, however, since there exist pairs of non-isomorphic $V,W\in \Rep G$ such
that $V^{K-\rm fin}\cong W^{K-\rm fin}$ as Harish-Chandra modules. Representations $V,W\in \Rep G$ such that $V^{K-\rm fin}\cong W^{K-\rm fin}$ as Harish-Chandra modules
are called {\bf infinitesimally equivalent}. In other words, infinitesimally equivalent representations with the same underlying Harish-Chandra module $M$ differ by what topology we put on $M$ (namely, the corresponding representation $\widehat M$
is the completion of $M$ in this topology). An example of infinitesimally equivalent but non-isomorphic representations are $L^2(\Bbb R\Bbb P^1)$ and $C^\infty(\Bbb R\Bbb P^1)$ as representations of $G=SL_2(\Bbb R)$ (with $G$-action on half-densities).
 
However, we have the following proposition. 

\begin{proposition} Let $V,W$ be two unitary representations in $\Rep G$. 
If $V^{K-\rm fin}\cong W^{K-\rm fin}$ as Harish-Chandra modules, then $V\cong W$ as unitary representations. 
In other words, infinitesimally equivalent unitary representations in $\Rep G$ are isomorphic.  
\end{proposition} 

\begin{proof} Clearly, it suffices to assume that $V,W$ are irreducible. 
If $V$ is unitary irreducible then $V^{K-\rm fin}$ has an invariant positive Hermitian inner product $B=B_V$ restricted from $V$. Moreover, $B$ is the unique invariant Hermitian inner product on $V^{K-\rm fin}$ up to scaling.\footnote{An invariant inner product on a $(
\g,K)$-module is one that is invariant under both $\g$ and $K$, i.e., $K$-invariant 
and satisfying the equality $B(av,w)+B(v,aw)=0$ 
for all $a\in \g$.} 
Indeed, if $B'$ is another, choose a nonzero finite-dimensional $K$-isotypic component $E\subset V^{K\text{-fin}}$ and define $T\in\End_{\Bbb C}(E)$ by $B'(v,w)=B(Tv,w)$ for $v,w\in E$.
Pick an eigenvector $v\ne 0$ of $T$ with eigenvalue $\lambda$. Since distinct $K$-isotypic components are orthogonal for both forms, $v$ lies in the kernel of $B'-\lambda B$, which is a $(\g,K)$-submodule of $V^{K-\rm fin}$.
This kernel therefore must be the whole  $V^{K-\rm fin}$, so $B'=\lambda B$.  

Thus if $A: V^{K-\rm fin}\to W^{K-\rm fin}$ is an isomorphism then it is an isometry with respect to $B_V$, $B_W$ under suitable normalization of these forms. Then $A$ extends by continuity to a unitary 
isomorphism $V\to W$ which commutes with $K$. 

It remains to show that $A$ commutes with $G$. For $v\in V$, $w\in W$, consider the function 
$$
f_{w,v}(g):=B_W((gA-Ag)v,w)=B_W(gAv,w)-B_V(gv,A^{-1}w),\ g\in G. 
$$
Our job is to show that $f_{w,v}(g)=0$. It suffices to check this when $v\in V^{K-\rm fin}$, as it is dense in $V$. In this case by Harish-Chandra's analyticity theorem, the 
function $f_{w,v}(g)$ is analytic on $G$. Also all its derivatives at $1$ vanish 
since $bA-Ab=0$ for any $b\in U(\g)$. This implies that $f_{w,v}$ is indeed zero, as desired. 
\end{proof} 

\subsection{Dixmier's lemma and infinitesimal character} \label{dixle} 

The following is an infinite-dimensional analog of Schur's lemma. 

\begin{lemma}\label{Dixlemm} (Dixmier) Let $A$ be a countable-dimensional $\Bbb C$-algebra and $M$ a simple $A$-module. Then $\End_A(M)=\Bbb C$. 
In particular, the center $Z$ of $A$ acts on $M$ by a character $\chi: Z\to \Bbb C$. 
\end{lemma} 

Note that the condition of countable dimension cannot be dropped. Without it, a counterexample is $A=M=\Bbb C(x)$ (the field of rational functions in one variable), then $\End_A(M)=\Bbb C(x)$. 

\begin{proof} Let $D:=\End_A(M)$. By the usual Schur lemma, 
$D$ is a division algebra. Assume the contrary, that $D\ne \Bbb C$. Then for any $x\in D\setminus \Bbb C$, $D$ contains the field $\Bbb C(x)$ of rational functions of $x$ (as $\Bbb C$ has no finite field extensions). 
But $\Bbb C(x)$ has uncountable dimension (contains linearly independent elements $\frac{1}{x-a}$, $a\in \Bbb C$), hence so does $D$. On the other hand, let $v\in M$ be a nonzero vector, then $M=Av$ and the map $D\to M$ given by $T\mapsto Tv$ is injective. 
Thus $M$ is countable-dimensional, hence so is $D$, contradiction. 
\end{proof} 

Now let $\g$ be a countable-dimensional complex Lie algebra and $M$ a simple $\g$-module. By Lemma \ref{Dixlemm}, the center $Z(\g)$ of $U(\g)$ acts 
on $M$ by a character, $\chi: Z(\g)\to \Bbb C$. This character is called the {\bf infinitesimal character} of $M$. 

In particular, for semisimple groups we obtain 

\begin{corollary}\label{Schur} (Schur's lemma for $(\g,K)$-modules) 
Any endomorphism of an irreducible $(\g,K)$-module $M$ is a scalar. 
Thus the center $Z(\g)$ of $U(\g)$ acts on $M$ by an infinitesimal character $\chi: Z(\g)\to \Bbb C$. 
\end{corollary} 

The character $\chi$ is often also called the {\bf infinitesimal character} of $M$. 

\begin{exercise} Show that the action of $Z(\g)$ on every admissible $(\g,K)$-module $M$ is locally finite. 
\end{exercise} 

\subsection{Harish-Chandra's globalization theorem} 

\begin{theorem}\label{inte} (Harish-Chandra's globalization theorem) Every unitary irreducible Harish-Chandra module $M$ for $G$ uniquely integrates (=globalizes) to an irreducible admissible unitary representation of $G$. 
\end{theorem} 

\begin{proof} Let $\mathfrak{k}={\rm Lie}K$ and $\g=\mathfrak{k}\oplus \mathfrak{p}$ 
be the decomposition of $\g$ into $\pm 1$-eigenspaces of the involution $\theta: \g\to \g$ attached to $G$. 
Let $B$ be the Killing form of $\g$. Define a positive definite $K$-invariant inner product on $\g$ by setting \scriptsize
$$
(x,y)=-B(x,y),\ x,y\in \mathfrak{k};\ (x,y)=B(x,y),\ x,y\in \mathfrak \p;\ (x,y)=0,\ x\in \mathfrak k,\ y\in \mathfrak p 
$$ \normalsize
and consider the element $C_\g^+:=\sum_{j=1}^{\dim\g}a_j^2\in U(\g)$, where $a_j$ is an orthonormal basis of $\g$ under this inner product. If $C_\g$ 
is the quadratic Casimir of $\g$ defined by the Killing form, then 
$C_\g^+=C_\g+2C_{\mathfrak k}$, where $C_{\mathfrak k}$ is the Casimir of $\mathfrak{k}$ 
corresponding to the restriction of the inner product to $\mathfrak{k}$. If $L_\nu$ 
is the highest weight irreducible representation with highest weight $\nu$ then 
$-C_{\mathfrak k}|_{L_\nu}=|\nu+\rho_K|^2-|\rho_K|^2$, where $\rho_K$ 
is the half-sum of positive roots of $K$. Also $C_\g|_M=C_M$ is a scalar.
Thus if $M^\nu:=M^{L_\nu}$ then 
$$
-C_\g^+|_{M^\nu}=2|\nu+\rho_K|^2-2|\rho_K|^2-C_M=:q(\nu).
$$ 
Note that for $v\in M^\nu$ we have 
$$
\sum_{j=1}^{\dim \g} \norm{a_jv}^2=-(\sum_{j=1}^{\dim \g} a_j^2v,v)=-(C_\g^+v,v)=q(\nu)\norm{v}^2; 
$$
in particular, $q(\nu)\ge 0$ and $q(\nu)\sim 2|\nu|^2$ for large $\nu$. 
It follows that for any $a\in \g, v\in M^\nu$, 
$$
\norm{av}^2\le q(\nu)\norm{a}^2\norm{v}^2.
$$
Now, for $v\in M^{\nu_0}$ all components of $a^nv$ belong to $M^\nu$, where 
$\nu=\nu_0+\beta_1+...+\beta_n$ and $\beta_j$ are weights of $\g$ as a $K$-module.  
So there exist $R,c=c(\nu_0)>0$ such that 
$$
 \norm{a^nv}\le (Rn+c)\norm{a}\norm{a^{n-1}v}, n\ge 1.
 $$
Thus 
$$
\norm{a^nv}\le (R+c)...(Rn+c)\norm{a}^n\norm{v}. 
$$
So the series 
$$
e^a v:=\sum_{n\ge 0} \frac{a^nv}{n!}
$$
converges absolutely in the Hilbert space $\widehat M$ in the region 
$\norm{a}<R^{-1}$, and convergence is uniform on compact sets with all derivatives, and defines an analytic function of $a$. Moreover, it is easy to check that $\norm{e^av}=\norm{v}$ (since $a$ is skew-symmetric under the inner product of $M$). Thus the operator $e^a: M\to \widehat M$ 
extends to a unitary operator on $\widehat M$. The formal Campbell-Hausdorff formula then implies 
that this defines a continuous unitary action $\pi$ of a neighborhood $U$ of $0$ in $G$ on $\widehat M$ such that $\pi(xy)=\pi(x)\pi(y)$ if $x,y,xy\in U$. It is well known that this implies 
that $\pi$ extends to a unitary representation of the universal cover 
$\widetilde G$ of $G$ on $\widehat M$. Now let $\widetilde K$ be the preimage of $K$ in $\widetilde G$ 
(by the polar decomposition, it is also the universal cover of $K$). Since by definition $\pi|_{\widetilde K}$ factors through $K$, and the kernel of $\widetilde G\to G$ coincides with the kernel $\widetilde K\to K$, it follows that $\pi$ actually factors through $G$. 
\end{proof} 

Thus, using Harish-Chandra's admissibility theorem, we obtain  

\begin{corollary} For a semisimple Lie group $G$, the assignment $V\mapsto V^{K-\rm fin}$ is an equivalence of categories between unitary representations of $G$ of finite length and unitary Harish-Chandra modules of finite length (i.e., Harish-Chandra modules which admit an invariant positive Hermitian inner product). 
\end{corollary} 

However, while irreducible Harish-Chandra modules for any $G$ have been classified, determining which of them are unitary is a very difficult problem which is not yet fully solved. 

\subsection{Casselman--Wallach globalization}

The unitarity assumption in Theorem~7.5 is not needed if one works in
an appropriate category of smooth Fr\'echet representations.  Assume, as
above, that $G$ is a semisimple Lie group with maximal compact subgroup
$K$; the following statements hold more generally for real reductive
Lie groups.  A smooth Fr\'echet representation $V$ of $G$ (i.e., such that $V=V^\infty$) 
is said to be of {\bf moderate growth} if, for every continuous seminorm $p$ on $V$, there
are a continuous seminorm $q$ and an integer $N$ such that
\[
        p(gv) \leq \|g\|^N q(v), \qquad g\in G,\ v\in V,
\]
where $\|g\|$ is the exponential of the distance from $1$ to $g$ in a left-invariant Riemannian metric (clearly, this condition is independent of the choice of this metric). 
Let $\mathcal{CW}(G)$ be the category of smooth Fr\'echet representations of moderate growth
whose $K$-finite vectors form Harish-Chandra modules, with continuous
$G$-maps as morphisms.

\begin{theorem}[Casselman--Wallach globalization]
The functor
\[
        V \longmapsto V^{K\text{-}\mathrm{fin}}
\]
from $\mathcal{CW}(G)$ to the category $\mathrm{HC}_G$ of
Harish-Chandra modules is an equivalence of categories.  Thus every
Harish-Chandra module $M$ has a smooth Fr\'echet moderate-growth
globalization $M^\infty$, unique up to unique isomorphism over $M$, and
$(M^\infty)^{K\text{-}\mathrm{fin}}\cong M$.  Equivalently, every
morphism of Harish-Chandra modules extends uniquely to a continuous
$G$-map between their Casselman--Wallach globalizations.
\end{theorem}

There is also a weaker, but often useful, Hilbert-space form of
existence.

\begin{theorem}[Hilbert globalizations]
Every Harish-Chandra module $M$ admits a $G$-continuous Hilbert norm:
the Hilbert completion $H$ is a continuous admissible representation of
$G$ and
\[
        H^{K\text{-}\mathrm{fin}} \cong M
\]
as a $(\mathfrak g,K)$-module.  
\end{theorem}

This Hilbert globalization is not
canonical, and there is no uniqueness statement for the Hilbert space
representation.  However, the smooth vectors $H^\infty$ form a
Casselman--Wallach globalization of $M$; hence $H^\infty$ is independent
of the chosen Hilbert globalization up to unique isomorphism.

For proofs and further discussion see \cite{Cas89},
\cite[Ch.~11]{Wal92}, and \cite{BK14}; the
existence of Hilbert globalizations follows, for instance, from
Casselman's subrepresentation theorem and the existence of
$G$-continuous Hilbert norms, see~\cite[\S5.1]{BK14}.

\section{\bf Highest weight modules and Verma modules} 

\subsection{$\g$-modules with a weight decomposition} 
Let us recall basic results on highest weight modules and Verma modules for a complex semisimple Lie algebra $\g$. 

Let $\g=\n_-\oplus \h\oplus \n_+$ be a triangular decomposition and $\lambda\in \h^*$ be a weight. 
We have $\n_\pm=\oplus_{\alpha\in R_\pm}\g_\alpha$, where 
$R_\pm$ are the sets of positive and negative roots.  Let $Q\subset \h^*$ be the root lattice of $\g$ spanned by its roots. Let $e_i,f_i,h_i, i=1,...,r$ 
be the Chevalley generators of $\g$.
Let $P\subset \h^*$ be the weight lattice, consisting of $\lambda \in \h^*$ with $\lambda(h_i)\in \Bbb Z$ for all $i$ and $P_+\subset P$ be the set of dominant integral weights, defined by the condition $\lambda(h_i)\in \Bbb Z_{\ge 0}$ for all $i$. Finally, let $Q_+\subset Q$ be the set of sums of positive roots. 

\begin{definition} Let $V$ be a representation of $\g$ (possibly infinite-dimensional). 
Then a vector $v\in V$ is said to have {\bf weight} $\lambda$ if $hv=\lambda(h)v$ for all $h\in \h$. 
The subspace of such vectors is denoted by $V[\lambda]$. If $V[\lambda]\ne 0$, we say that $\lambda$ is a weight of $V$, and the set of weights of $V$ is denoted by $P(V)$. 
\end{definition} 

It is easy to see that $\g_\alpha V[\lambda]\subset V[\lambda+\alpha]$. 

Let $V'\subset V$ be the span of all weight vectors in $V$. Then it is clear that 
$V'=\oplus_{\lambda\in \h^*}V[\lambda]$. 

\begin{definition} We say that $V$ {\bf has a weight decomposition} (with respect to a Cartan subalgebra $\h\subset \g$), or is $\h$-{\bf semisimple} if $V'=V$, i.e., if $V=\oplus_{\lambda\in \h^*}V[\lambda]$. 
\end{definition} 

Note that not every representation of $\g$ has a weight decomposition (e.g., for $V=U(\g)$ with $\g$ acting by left multiplication all weight subspaces are zero).  

\begin{definition} A vector $v$ in $V[\lambda]$ is called a {\bf singular (or highest weight) vector of weight $\lambda$} if $e_iv=0$ for all $i$, i.e., if $\n_+v=0$. A representation $V$ of $\g$ is a {\bf highest weight representation with highest weight $\lambda$} if it is generated by such a nonzero vector. 
\end{definition} 

\subsection{Verma modules}
The {\bf Verma module} $M_\lambda$ is defined as ``the largest highest weight module with highest weight $\lambda$". Namely, it is generated by a single highest weight vector $v_\lambda$ 
with {\bf defining relations} $hv=\lambda(h)v$ for $h\in \h$ and $e_iv=0$. More formally, 
we make the following definition. 

\begin{definition} Let $I_\lambda\in U(\g)$ be the left ideal generated by the elements $h-\lambda(h),h\in \h$ and 
$e_i$, $i=1,...,r$. Then the {\bf Verma module} $M_\lambda$ is the quotient 
$U(\g)/I_\lambda$. 
\end{definition} 

In this realization, the highest weight vector $v_\lambda$ is just the class of the unit $1$ of $U(\g)$. 
 
\begin{proposition} The map $\phi: U(\n_-)\to M_\lambda$ given by $\phi(x)=xv_\lambda$ 
is an isomorphism of left $U(\n_-)$-modules.  
\end{proposition}  

\begin{proof} 
By the PBW theorem, the multiplication map 
$$
\xi: U(\n_-)\otimes U(\h\oplus \n_+)\to U(\g)
$$ 
is a linear isomorphism. It is easy to see that $\xi^{-1}(I_\lambda)=U(\n_-)\otimes K_\lambda$, where 
$$
K_\lambda:=\sum_i U(\h\oplus \n_+)(h_i-\lambda(h_i))+\sum_i U(\h\oplus \n_+)e_i
$$
is the kernel of the homomorphism $\chi_\lambda: U(\h\oplus \n_+)\to \Bbb C$ 
given by $\chi_\lambda(h)=\lambda(h)$, $h\in \h$, $\chi_\lambda(e_i)=0$. 
Thus, we have a natural isomorphism of left $U(\n_-)$-modules
$$
U(\n_-)=U(\n_-)\otimes U(\h\oplus \n_+)/K_\lambda\to M_\lambda,
$$
as claimed.
\end{proof} 

\begin{remark} The definition of $M_\lambda$ means that 
it is the {\bf induced module} $U(\g)\otimes_{U(\h\oplus \n_+)}\Bbb C_\lambda$, where 
$\Bbb C_\lambda$ is the one-dimensional representation of $\h\oplus \n_+$ on which it acts via $\chi_\lambda$. 
\end{remark} 

\begin{corollary} $M_\lambda$ has a weight decomposition with 
$P(M_\lambda)=\lambda-Q_+$, $\dim M_\lambda[\lambda]=1$, and weight subspaces of $M_\lambda$ are finite-dimensional. 
\end{corollary} 

\begin{proposition} (i) If $V$ is a representation of $\g$ and $v\in V$ is a vector such that 
$hv=\lambda(h)v$ for $h\in h$ and $e_iv=0$ then there is a unique homomorphism 
$\eta: M_\lambda\to V$ such that $\eta(v_\lambda)=v$. In particular, if $V$ is generated by such $v\ne 0$ (i.e., $V$ is a highest weight representation with highest weight vector $v$) then $V$ is a quotient 
of $M_\lambda$.  

(ii) Every highest weight representation  has a weight decomposition into finite-dimensional weight subspaces. 

(iii) Every highest weight representation $V$ has a unique highest weight generator, up to scaling. 
\end{proposition} 

\begin{proof} (i) Uniqueness follows from the fact that $v_\lambda$ generates $M_\lambda$. 
To construct $\eta$, note that we have a natural map of $\g$-modules 
$\widetilde \eta: U(\g)\to V$ given by $\widetilde \eta(x)=xv$. Moreover, $\widetilde\eta|_{I_\lambda}=0$
thanks to the relations satisfied by $v$, so $\widetilde \eta$ descends to a map 
$\eta: U(\g)/I_\lambda=M_\lambda\to V$. Moreover, if $V$ is generated by $v$ then this map is surjective, as desired. 

(ii) This follows from (i) since a quotient of any representation with a weight decomposition 
must itself have a weight decomposition. 

(iii) Suppose $v,w$ are two highest weight generators of $V$ of weights $\lambda,\mu$. 
If $\lambda=\mu$ then they are proportional since $\dim V[\lambda]\le \dim M_\lambda[\lambda]=1$, as $V$ is a quotient of $M_\lambda$. On the other hand, if $\lambda\ne \mu$, then we can assume without loss of generality that $\lambda-\mu\notin Q_+$ (otherwise switch $\lambda,\mu$). Then 
$\mu\notin \lambda-Q_+$, hence $\mu\notin P(V)$, a contradiction. 
\end{proof} 

\subsection{Irreducible highest weight $\g$-modules} 
\begin{proposition} For every $\lambda\in \h^*$, the Verma module $M_\lambda$ has a unique 
irreducible quotient $L_\lambda$. Moreover, $L_\lambda$ is a quotient of every highest weight $\g$-module $V$ with highest weight $\lambda$. 
\end{proposition} 

\begin{proof} Let $Y\subset M_\lambda$ be a proper submodule. Then $Y$ has a weight decomposition, and cannot contain a nonzero multiple of $v_\lambda$ (as otherwise $Y=M_\lambda$), so $P(Y)\subset 
(\lambda-Q_+)\setminus \lbrace \lambda \rbrace$. Now let $J_\lambda$  be the sum of all 
proper submodules $Y\subset M_\lambda$. Then $P(J_\lambda)\subset (\lambda-Q_+)\setminus \lbrace \lambda \rbrace$, so $J_\lambda$ is also a proper submodule of $M_\lambda$  (the maximal one). 
Thus, $L_\lambda:=M_\lambda/J_\lambda$ is an irreducible highest weight module with highest weight $\lambda$. Moreover, if $V$ is any nonzero quotient of $M_\lambda$ then the kernel 
$K$ of the map $M_\lambda\to V$ is a proper submodule, hence contained in $J_\lambda$. 
Thus the surjective map $M_\lambda\to L_\lambda$ descends to a surjective map $V\to L_\lambda$. 
The kernel of this map is a proper submodule of $V$, hence zero if $V$ is irreducible. 
Thus in the latter case $V\cong L_\lambda$. 
\end{proof} 

\begin{corollary} Irreducible highest weight $\g$-modules are classified by their highest weight $\lambda\in \h^*$, via the bijection $\lambda\mapsto L_\lambda$.
\end {corollary}   

\begin{exercise}\label{sl2}  Let $\g=\mathfrak{sl}_2$ with standard generators $e,f,h$ and identify $\h^*\cong \Bbb C$ via $\lambda\mapsto \lambda(h)$. Show that $M_\lambda$ 
is irreducible if $\lambda\notin \Bbb Z_{\ge 0}$, while for $\lambda$ a nonnegative integer
we have $J_\lambda=M_{-\lambda-2}$, so $L_\lambda$ is the $\lambda+1$-dimensional 
irreducible representation of $\mathfrak{sl}_2$. 
\end{exercise} 

It is known from the theory of finite-dimensional representations of $\g$ that 
its irreducible finite-dimensional representations are $L_\lambda$ with $\lambda\in P_+$. 
Thus we have 

\begin{proposition} $L_\lambda$ is finite-dimensional if and only if 
$\lambda\in P_+$. 
\end{proposition} 

Note that the ``only if" direction of this proposition follows immediately from Exercise \ref{sl2}. 

\subsection{Exercises}

\begin{exercise}\label{intertw} Let $\g$ be a finite-dimensional simple complex Lie algebra, and $V$ a finite-dimensional representation of $\g$. Let $\lambda,\mu\in \h^*$ be weights for $\g$, and 
$X,Y$ be representations of $\g$ with $P(X)\subset \lambda-Q_+$, $P(Y)\subset \mu-Q_+$, and $X[\lambda]=\Bbb Cv_\lambda$, $Y[\mu]=\Bbb C v_\mu$ for nonzero vectors $v_\lambda,v_\mu$. Given a linear map
$\Phi: X\to V\otimes Y$, let the {\bf expectation value} of $\Phi$ be defined by 
$$
\langle \Phi\rangle:=({\rm Id}\otimes v_\mu^*,\Phi v_\lambda)\in V
$$
 where $v_\mu^*\in Y[\mu]^*$ is such that $(v_\mu^*,v_\mu)=1$. In other words, we have 
$$
\Phi v_\lambda=\langle \Phi\rangle \otimes v_\mu+\text{lower terms}
$$ 
where the lower terms have lower weight than $\mu$ in the second component. 

(i) Show that if $\Phi$ is a homomorphism then $\langle \Phi\rangle$ has weight $\lambda-\mu$. 

(ii) Let $M_\lambda$ be the Verma module with highest weight $\lambda\in \h^*$, and $\overline M_{-\mu}$ be the {\bf lowest weight} Verma module 
with lowest weight $-\mu$, i.e., generated by a vector $v_{-\mu}$ with defining relations 
$hv_{-\mu}=-\mu(h)v_{-\mu}$ for $h\in \h$ and $f_iv_{-\mu}=0$.  
Show that the map $\Phi\mapsto \langle \Phi\rangle$ defines an isomorphism 
$$
\Hom_\g(M_\lambda, V\otimes \overline M_{-\mu}^*)\cong V[\lambda-\mu]
$$
where $*$ denotes the restricted dual (the direct sum of duals of all weight subspaces).

(iii) Let $\lambda\in P_+$ and $V[\nu]_\lambda$ be the subspace of vectors $v\in V[\nu]$ of weight $\nu$ which satisfy 
the equalities $f_i^{(\lambda,\alpha_i^\vee)+1}v=0$ for all $i$. Show that a map $\Phi\in \Hom_\g(M_\lambda, V\otimes \overline M_{-\mu}^*)$ factors through $L_\lambda$ 
iff $\langle \Phi\rangle\in V[\lambda-\mu]_\lambda$, i.e., $f_i^{(\lambda,\alpha_i^\vee)+1}\langle \Phi\rangle=0$ (for this, use that $e_jf_i^{(\lambda,\alpha_i^\vee)+1}v_\lambda=0$, and that the kernel of $M_\lambda\to L_\lambda$ is generated by the vectors $f_i^{(\lambda,\alpha_i^\vee)+1}v_\lambda$). Deduce that the map $\Phi\mapsto \langle \Phi\rangle$ defines an isomorphism 
$\Hom_\g(L_\lambda, V\otimes \overline M_{-\mu}^*)\cong V[\lambda-\mu]_\lambda$.

(iv) Now let both $\lambda,\mu$ be in $P_+$. Show that every homomorphism $L_\lambda\to V\otimes \overline M_{-\mu}^*$
in fact lands in $V\otimes L_\mu\subset V\otimes \overline M_{-\mu}^*$.  
Deduce that the map 
$\Phi\mapsto \langle \Phi\rangle$ defines an isomorphism 
$$
\Hom_\g(L_\lambda, V\otimes L_\mu)\cong V[\lambda-\mu]_\lambda.
$$

(v) Let $V=\Bbb C^n$ be the vector representation of $SL_n(\Bbb C)$. Determine the weight subspaces of $S^mV$, and compute the decomposition of $S^mV\otimes L_\mu$ into irreducibles for all $\mu\in P_+$ (use (iv)). 

(vi) For any $\g$, compute the decomposition of $\g\otimes L_\mu$, $\mu\in P_+$, where $\g$ is the adjoint representation of $\g$ (again use (iv)). 

In both (v) and (vi) you should express the answer in terms of the numbers $k_i$ such that $\mu=\sum_i k_i\omega_i$ and the Cartan matrix entries of $\g$.  
\end{exercise} 

\begin{exercise}\label{dimhom}(D. N. Verma) (i) Let $\g=\n_-\oplus \h\oplus \n_+$ be a finite-dimensional simple complex Lie algebra, and $\lambda,\mu\in \h^*$. 
Show that every nonzero homomorphism $M_\mu\to M_\lambda$ is injective. (Use that $U(\n_-)$ has no zero divisors). Deduce that if $M_\lambda$ is reducible then there exists $\lambda'\in \lambda-Q_+$, 
$\lambda'\ne \lambda$ with $M_{\lambda'}\subset M_\lambda$. 

(ii) Show that for every $\lambda\in \h^*$ there is $\lambda'\in \lambda-Q_+$ with $M_{\lambda'}\subset M_\lambda$ and $M_{\lambda'}$ irreducible. (Assume the contrary and construct an infinite sequence of proper inclusions 
$$
...M_{\lambda_2}\subset M_{\lambda_1}\subset M_\lambda.
$$ 
Then derive a contradiction by looking at the eigenvalues 
of the quadratic Casimir $C\in U(\g)$).

(iii) Show that if $M_\mu$ is irreducible then $\dim \Hom_\g(M_
\mu,M_\lambda)\le 1$. (Look at the growth of the dimensions of weight subspaces). 

(iv) Show that  $\dim \Hom_\g(M_
\mu,M_\lambda)\le 1$ for any $\lambda,\mu\in \h^*$. (Look at the restriction 
of a homomorphism $M_\mu\to M_\lambda$ to $M_{\mu'}\subset M_\mu$ 
which is irreducible).  
\end{exercise}

\begin{exercise}\label{Shapova} (i) Keep the notation of Exercise \ref{dimhom}. Let $\lambda\in \h^*$ be such that $(\lambda,\alpha_i^\vee)=n-1$ for a positive integer $n$ and simple root $\alpha_i$. Show that there is an inclusion $M_{\lambda-n\alpha_i}\hookrightarrow M_\lambda$. 

(ii) Let $\rho$ be the sum of fundamental weights of $\g$ and 
$W$ be the Weyl group of $\g$. For $w\in W$, $\lambda\in \h^*$ 
let $w\bullet \lambda:=w(\lambda+\rho)-\rho$ (the {\bf shifted action} of $W$). Deduce from (i) that if $\lambda\in P_+$ then for every $w\in W$, there is an inclusion $\iota_w: M_{w\bullet \lambda}\hookrightarrow M_\lambda$, and that if $w=w_1w_2$ 
with $\ell(w)=\ell(w_1)+\ell(w_2)$ (where $\ell(w)$ is the length of $w$) 
then $\iota_w$  factors through $\iota_{w_2}$. In particular, we have an inclusion $M_{w\bullet \lambda}\hookrightarrow M_{w_2\bullet \lambda}$. 

(iii) Show that $M_\lambda$ is irreducible unless $(\lambda+\rho,\alpha^\vee)=1$ for some $\alpha\in Q_+\setminus 0$, where $\alpha^\vee:=\frac{2\alpha}{(\alpha,\alpha)}$ (look at the eigenvalues of the quadratic Casimir). 

(iv) For $\beta\in Q_+$ define the {\bf Kostant partition function} 
$K(\beta)$ to be the number of unordered representations of 
$\beta$ as a sum of positive roots of $\g$ (thus $K(\beta)=\dim U(\n_+)[\beta]$). Also define  the
{\bf Shapovalov pairing} 
$$
B_\beta(\lambda):U(\n_+)[\beta]\times U(\n_-)[-\beta]\to \Bbb C
$$ 
by the formula 
$$
xyv_\lambda=B_\beta(\lambda)(x,y)v_\lambda, 
$$ 
where $x\in U(\n_+)[\beta], y\in U(\n_-)[-\beta]$, and $v_\lambda$ 
is the highest weight vector of $M_\lambda$. Let 
$$
D_\beta(\lambda):=\det B_\beta(\lambda),
$$ 
the determinant 
of the matrix of $B_\beta(\lambda)$ in some bases of $U(\n_+)[\beta],  U(\n_-)[-\beta]$. This is a (non-homogeneous) polynomial in $\lambda$ well defined up to scaling. Show that the leading term of $D_\beta$ 
is 
$$
D_\beta^0(\lambda)={\rm const}\cdot \prod_{\alpha\in R_+}(\lambda,\alpha^\vee)^{\sum_{n\ge 1}K(\beta-n\alpha)}.
$$
(Hint: show that the leading term comes from the product of the diagonal entries 
of the matrix of the Shapovalov pairing in the PBW bases).  

(v) Show that 
$$
D_\beta(\lambda)={\rm const}\cdot \prod_{\alpha\in Q_+\setminus 0}((\lambda+\rho,\alpha^\vee)-1)^{m_\alpha}
$$
for some nonnegative integers $m_\alpha=m_\alpha(\beta)$. Then use (iv) to show that moreover $m_\alpha=0$ unless $\alpha$ is a multiple of a positive root. 

(vi) Let $V,U$ be finite-dimensional vector spaces over a field $k$ of dimension $n$ 
and $B(t): V\times U\to k[[t]]$ be a bilinear form. Denote by $V_0\subset V,U_0\subset U$ the left and right kernels of $B(0)$. Suppose that $B'(0)$ is a perfect pairing $V_0\times U_0\to k$. Show that the vanishing order of $\det B(t)$ at $t=0$ (computed with respect to any bases of $V,U$) equals $\dim V_0=\dim U_0$. ({\it Hint:} Pick a basis 
$e_1,...,e_m$ of $V_0$, complete it to a basis $e_1,...,e_n$ of $V$. Choose 
vectors $f_{m+1},...,f_n\in U$ such that $B(0)(e_i,f_j)=\delta_{ij}$ for $m<i,j\le n$. 
Let $f_1,...,f_m$ be the basis $U_0$ dual to $e_1,...,e_m$ with respect to $B'(0)$. 
Show that $\lbrace f_i\rbrace$ is a basis of $U$ 
and the determinant of $B(t)$ in the bases $\lbrace e_i\rbrace$, $\lbrace f_i\rbrace$
equals $t^m+O(t^{m+1})$.)

(vii) Show that if $\lambda$ is generic on the hyperplane 
$(\lambda+\rho,\alpha^\vee)=n$ for $n\in \Bbb Z_{>0}$ and $\alpha\in R_+$ and 
$m_{n\alpha}(\beta)>0$ then $M_\lambda$ contains an irreducible submodule $M_{\lambda-n\alpha}$ and the quotient $M_\lambda/M_{\lambda-n\alpha}$ is irreducible. (Use Casimir eigenvalues to show that the only irreducible modules which could occur in the composition series of $M_\lambda$ are $L_\lambda$ and $L_{\lambda-n\alpha}$ and apply Exercise \ref{dimhom}). 

(viii) Let $\lambda$ be as in (vii) and let $B(\beta,t):=B_\beta(\lambda+t\alpha)$. 
Show that $B(\beta,t)$ satisfies the assumption of (vi) for all $\beta$. 

{\bf Hint:} Use that $\oplus_\beta {\rm Ker}B(\beta,0)$ is naturally identified 
with $M_{\lambda-n\alpha}$ and $B'(\beta,0)$ restricts 
on it to a multiple of its Shapovalov form, and show that one has
$B_{n\alpha}'(0)(v_{\lambda-n\alpha},v_{\lambda-n\alpha})\ne 0$.
For the latter, assume the contrary and show that there exists a homogeneous lift $u$ 
of $v_{\lambda-n\alpha}$ modulo $t^2$ such that 
$B_{n\alpha}(t)(u,w)=0$ modulo $t^2$ for all $w$ of weight 
$\lambda+(t-n)\alpha$. Deduce that $e_iu$ vanishes modulo $t^2$ for all $i$. Conclude that 
$$
Cu=((\lambda+(t-n)\alpha+\rho)^2-\rho^2)u+O(t^2)
$$
and derive a contradiction with 
$$
Cu=((\lambda+t\alpha+\rho)^2-\rho^2)u.
$$
 
(ix) Deduce that 
$m_{n\alpha}(\beta)=K(\beta-n\alpha)$; in particular, 
in general $m_{n\alpha}(\beta)\le K(\beta-n\alpha)$. 

(x) Prove the {\bf Shapovalov determinant formula}: 
$$
D_\beta(\lambda)=\prod_{\alpha\in R_+}\prod_{n\ge 1} ((\lambda+\rho,\alpha^\vee)-n)^{K(\beta-n\alpha)}
$$
up to scaling.

(xi) Determine all $\lambda\in \h^*$ for which $M_\lambda$ is irreducible. 
\end{exercise} 

\section{\bf Representations of $SL_2(\Bbb R)$}\label{sl2R}

\subsection{Irreducible $(\g,K)$-modules for $SL_2(\Bbb R)$}
Let us now apply the general theory to the simplest example -- representations of the group $G=SL_2(\Bbb R)$
of real 2 by 2 matrices with determinant $1$. Note that $SL_2(\Bbb R)\cong SU(1,1)$, and in this realization the maximal compact subgroup $SO(2)$ becomes $U(1)$. So we have ${\rm Lie}(G)=\g=\mathfrak{su}(1,1)$, hence $\g_{\Bbb C}=\mathfrak{sl}_2(\Bbb C)$ with standard basis $e,f,h$, so that a maximal compact subgroup $K$ of $G$ consists of elements 
$e^{ith}$, $t\in [0,2\pi)$. Thus a $(\g,K)$-module is the same thing as a $\g_{\Bbb C}$-module with a weight decomposition and integer weights. 

Let us classify irreducible $(\g,K)$-modules $M$. To this end, recall that we have the central Casimir element 
$C\in U(\g_{\Bbb C})$ given by   
$$
C=fe+\frac{(h+1)^2}{4},
$$
and note that by the PBW theorem, $U(\g_{\Bbb C})$ is free as a right module over the commutative subalgebra 
$\Bbb C[h,fe]=\Bbb C[h,C]$ with basis $1,f^n,e^n$, $n\ge 1$. Thus if $v$ is a nonzero weight vector of $M$ 
then $M$ is spanned by $v,f^nv,e^nv$. It follows that weight subspaces of $M$ are 1-dimensional, and 
$P(M)$ is an arithmetic progression with step $2$. Thus we have four cases: 

{\bf 1.} $P(M)$ is finite. Then $M=L_m$, the $m+1$-dimensional irreducible representation. 

{\bf 2.} $P(M)$ is infinite, bounded above. In this case let $v$ have the maximal weight $m$. Then $f^nv$, 
$n\ge 0$ is a basis of $M$, and we have $hv=mv,ev=0$. Thus $M=M_m$ is the Verma module with highest weight $m\in \Bbb Z$. This module is irreducible iff $m<0$ (Exercise \ref{sl2}). Thus in this case we get modules $M_{-m}=M_{-m}^+$, $m\ge 1$. 

{\bf 3.} $P(M)$ is infinite, 
bounded below. The situation is completely parallel (with $f$ replaced by $e$) and we 
obtain lowest weight Verma modules $M_m^-$ for $m\ge 1$. The $(\g,K)$-modules $M_m^-,M_{-m}^+$ 
are called the {\bf discrete series modules} for $m\ge 2$, and {\bf limit of discrete series} for $m=1$.

{\bf 4.} $P(M)$ is unbounded on both sides. Let $c$ be the scalar by which $C$ acts on $M$. 
We have two cases -- the even case $P(M)=2\Bbb Z$ and the odd case 
$P(M)=2\Bbb Z+1$. In both cases we have a basis $v_n$, $n\in P(M)$ such that 
\begin{equation}\label{modprinser}
hv_n=nv_n,\ fv_n=v_{n-2},\  ev_n=\Lambda_nv_{n+2},
\end{equation}
where $\Lambda_n\ne 0$. 
To compute $\Lambda_n$, we write 
$$
\Lambda_nv_n=fev_n=(C-\tfrac{(h+1)^2}{4})v_n=(c-\tfrac{(n+1)^2}{4})v_n.
$$
Thus 
$$
\Lambda_n=c-\tfrac{(n+1)^2}{4}.
$$
Let $c=\frac{s^2}{4}$. Then 
\begin{equation}\label{lan}
\Lambda_n=\tfrac{1}{4}(s-1-n)(s+1+n).
\end{equation} 
 Thus we can 
replace $v_n$ by its multiple $w_n$ so that 
$$
hw_n=nw_n,\ fw_n=\tfrac{1}{2}(s-1+n)w_{n-2},\  ew_n=\tfrac{1}{2}(s-1-n)w_{n+2}.
$$ 
These formulas define $\g_{\Bbb C}$-modules for any $s\in \Bbb C$. 
We will denote these modules by $P_\pm(s)$ (plus for the even case, minus for the odd case). The $(\g,K)$-modules $P_\pm(s)$ are called the {\bf principal series modules}. 
We see that $P_+(s)$ is irreducible if $s\notin 2\Bbb Z+1$ and $P_-(s)$ is irreducible iff $s\notin 2\Bbb Z$, and $P_\pm(s)=P_\pm(-s)$ in this case. 
 
Moreover, when these conditions fail, we have short exact sequences 
$$
0\to L_{2m}\to P_+(2m+1)\to M^+_{-2m-2}\oplus M_{2m+2}^-\to 0,\ m\in \Bbb Z_{\ge 0},  
$$
$$
0\to M^+_{-2m-2}\oplus M_{2m+2}^-\to P_+(-2m-1)\to L_{2m}\to 0,\ m\in \Bbb Z_{\ge 0},  
$$
$$
0\to L_{2m+1}\to P_-(2m+2)\to M^+_{-2m-3}\oplus M_{2m+3}^-\to 0,\ m\in \Bbb Z_{\ge 0}, 
$$
$$
0\to M^+_{-2m-3}\oplus M_{2m+3}^- \to P_-(-2m-2)\to L_{2m+1}\to 0,\ m\in \Bbb Z_{\ge 0}, 
$$
and for $s=0$ we have an isomorphism 
$$
P_-(0)\cong M_{-1}^+\oplus M_1^-.
$$
All these modules except $P_-(0)$ are indecomposable.
Thus we see that $P_\pm(s)\ncong P_\pm(-s)$ when it is reducible and $s\ne 0$.  

As a result, we get 

\begin{proposition} The simple $(\g,K)$-modules (or equivalently, Harish-Chandra modules) 
are $L_m,m\in \Bbb Z_{\ge 0}$, 
$M_m^-,M_{-m}^+$, $m\in \Bbb Z_{\ge 1}$, and $P_+(s)$, $s\notin 2\Bbb Z+1$, $P_-(s)$, $s\notin 2\Bbb Z$, 
with the only isomorphisms $P_\pm(s)\cong P_{\pm}(-s)$. 
\end{proposition} 

\begin{exercise} Let $\widetilde P_+(s),\widetilde P_-(s)$ be the modules defined by 
\eqref{modprinser},\eqref{lan}; so they are isomorphic to $P_+(s),P_-(s)$ when $s$ 
is not an odd integer, respectively not a nonzero even integer. But we will consider 
$\widetilde P_+(s)$ when $s=2k+1$ and $\widetilde P_-(s)$ when 
$s=2k$, $k\ne 0$ (where $k$ is an integer). 

(i) Compute the Jordan-H\"older series of  $\widetilde P_+(s),\widetilde P_-(s)$ and show that they are uniserial, i.e., have a unique filtration with irreducible successive quotients. 

(ii) Do there exist isomorphisms $\widetilde P_+(s)\cong P_+(s)$, $\widetilde P_-(s)\cong P_-(s)$? 
\end{exercise} 

\subsection{Realizations}
Let us discuss realizations of these representations by admissible representations of $G$. For $L_m$ there is nothing to discuss, so we'll focus on principal series and discrete series. 

The realization of principal series has already been discussed in Example \ref{prinser}. 
Namely, let $B\subset G$ be the subgroup of upper triangular matrices $b$ with diagonal entries $(t(b),t(b)^{-1})$.  
As before we consider the spaces
$$
\Bbb V_+(s)=\lbrace F\in C^\infty(G): F(gb)=F(g)|t(b)|^{s-
1}\rbrace, 
$$
$$
\Bbb V_-(s)=\lbrace F\in C^\infty(G): F(gb)=F(g)|t(b)|^{s-1}{\rm sign}(t(b))\rbrace.
$$

These are admissible representations of $G$ acting by left multiplication. 
Let us compute $\Bbb V_\pm(s)^{\rm fin}$. 
To this end, note that the group $K=U(1)=S^1$ acts transitively on $G/B$ with stabilizer $\Bbb Z/2=\lbrace \pm 1\rbrace$. Thus, pulling the function $F$ back to $K$, we can realize $\Bbb V_\pm(s)$ as the space $\Bbb V_\pm$ of functions $F\in C^\infty(S^1)$ such $F(-z)=\pm F(z)$. 

A more geometric way of thinking about this is the following. Given a Lie group $G$ and a closed subgroup $B$ with Lie algebras $\g,\b$, every finite-dimensional representation $V$ of $B$ gives rise to a vector bundle 
$E_V:=(G\times V)/B$ over $G/B$, where the action of $B$ on $G\times V$ is given by $(g,v)b=(gb,b^{-1}v)$. 
For example, the tangent bundle $T(G/B)$ is obtained from the representation $V=\g/\b$. 
In our example, $\g/\b$ is the 1-dimensional representation of $B$ given by $b\mapsto t(b)^{-2}$. Thus 
sections of the tangent bundle on $G/B$ (i.e., vector fields) can be interpreted as 
functions $F$ on $G$ such that 
$$
F(gb)=F(g)t(b)^2. 
$$
It follows that elements of $\Bbb V_+(s)$ can be interpreted as 
sections of the bundle ${\rm K}^{\frac{1-s}{2}}$ where ${\rm K}=T^*(G/B)$
is the canonical bundle, which coincides with the cotangent bundle since $\dim(G/B)=1$ (this bundle is trivial topologically but the action of diffeomorphisms of $G/B=S^1$, in particular, of elements of $SL_2(\Bbb R)$ on its sections depends on $s$). 
In other words,  elements of $\Bbb V_+(s)$ can be  interpreted as ``tensor fields 
of non-integer rank": $\phi(u)(d\ {\rm arg}u)^{\frac{1-s}{2}}$, where $u=e^{i\theta}$, 
$\theta$ is the angle coordinate on $G/B=\Bbb R\Bbb P^1$ and $\phi$ is a smooth function.
Similarly, elements of $\Bbb V_-(s)$ can be interpreted as expressions 
 $u^{\frac{1}{2}}\phi(u)(d\ {\rm arg}u)^{\frac{1-s}{2}}$, i.e., two-valued smooth sections of the same bundle 
 which change sign when one goes around the circle. Thus the Lie algebra action on these modules  
 is by the vector fields 
$$
h=2u\partial_u,\ f=\partial_u,\ e=-u^2\partial_u,
$$ 
but they act on elements of $\Bbb V_\pm(s)$ not as on functions but as on tensor fields. 
Thus $\Bbb V_\pm(s)^{\rm fin}\subset \Bbb V_\pm(s)$ is the subspace of vectors such that $\phi\in \Bbb C[u,u^{-1}]$. Taking the basis $w_{2k}=u^{k}(d\ {\rm arg}u)^{\frac{1-s}{2}}$ in the even case 
and $w_{2k+1}=u^{k+\frac{1}{2}}(d\ {\rm arg}u)^{\frac{1-s}{2}}$ in the odd case, we have 
$$
hw_n=nw_n,\ fw_n=\tfrac{1}{2}(s-1+n)w_{n-2},\ ew_{n}=\tfrac{1}{2}(s-1-n)w_{n+2}. 
$$
Thus we get that $\Bbb V_\pm(s)^{\rm fin}\cong P_\pm(s)$ for all $s\in \Bbb C$.  

In particular, at points where $P_\pm(s)$ are reducible, this gives realizations of the discrete series. 
Namely, consider the modules $\Bbb V_+(-r)$ for odd $r\ge 1$ and $\Bbb V_-(-r)$ for even $r\ge 1$. 
The space $\Bbb V_+(-r)$ consists of elements $\phi(u)(\frac{du}{i u})^{\frac{1+r}{2}}$ where $\phi$ is smooth (note that $d\ {\rm arg}u=\frac{du}{iu}$). 
So it has the subrepresentation $\Bbb V_+^0(-r)$ of forms that extend holomorphically to the disk $|u|\le 1$. 
This means that $\phi(u)=\sum_{N\ge 0}a_Nu^{N+\frac{1+r}{2}}$, where $a_N$ is a rapidly decaying sequence (faster than any power of $N$). In other words, $\Bbb V_+^0(-r)$
consists of elements $\psi(u)(du)^{\frac{1+r}{2}}$,
where $\psi$ is smooth on the disk $|u|\le 1$ and holomorphic 
for $|u|<1$. Thus the eigenvalues of $h$ 
on $\Bbb V_+^0(-r)$ are $1+r+2N$, hence $\Bbb V_+^0(-r)^{\rm fin}=M_{r+1}^-$. 

Also, 
$\Bbb V_+(-r)$ has a subrepresentation $\Bbb V_+^\infty(-r)$ of forms that extend holomorphically to $|u|\ge 1$ (including infinity), which means that $\phi(u)=\sum_{N\ge 0}a_Nu^{-N-\frac{1+r}{2}}$.
In other words, $\Bbb V_+^\infty(-r)$
consists of elements $\psi(u^{-1})(du^{-1})^{\frac{1+r}{2}}$,
where $\psi$ is smooth on the disk $|u|\le 1$ and holomorphic 
for $|u|<1$. Thus we get $\Bbb V_+^\infty(-r)^{\rm fin}=M_{-r-1}^+$.  

Similarly, for even $r$ we get $\Bbb V_-^0(-r)^{\rm fin}=M_{r+1}^-$, $\Bbb V_-^\infty(-r)^{\rm fin}=M_{-r-1}^+$. 

\subsection{Unitary representations} 
These Fr\'echet space realizations can easily be made Hilbert space realizations, by completing with respect to the usual $L^2$-norm given by 
$$
\norm{\phi}^2=\frac{1}{2\pi}\int_0^{2\pi}|\phi(e^{i\theta})|^2d\theta. 
$$
However, this norm is only preserved by $G$ when $s$ is imaginary. In this case we obtain that the completed representations $\widehat{\Bbb V}_\pm(s)$, in particular $\widehat{\Bbb V}_-^0(0),\widehat{\Bbb V}_-^\infty(0)$, are unitary. It follows that the Harish-Chandra modules $P_\pm(s)$ for $s\in i\Bbb R$ and $M_1^-,M_{-1}^+$ are unitary. 

It turns out, however, that there are other irreducible unitary representations. Let us classify them. It suffices to classify irreducible unitary Harish-Chandra modules. Note that the relevant anti-involution on $\g$ is given by $e^\dagger=-f$, $f^\dagger=-e$, $h^\dagger=h$. Let $M$ be irreducible and $v\in M$ a vector of weight $n$. Then if $(,)$ is an invariant Hermitian form on $M$ then  
$$
(ev,ev)=-(fev,v)=((\tfrac{n+1}{2})^2-c)(v,v),
$$
where $c$ is a Casimir eigenvalue on $M$. We see that a nonzero invariant Hermitian form exists 
iff $c=\frac{s^2}{4}\in \Bbb R$, and such a form can be chosen positive definite iff $c\le (\frac{n+1}{2})^2$ for every $n\in P(M)$. This shows that all discrete series representations are unitary 
and also determines the unitarity range of $s$ for the principal series representations. 
Thus we obtain the following theorem.

\begin{theorem} (Gelfand-Naimark \cite{GN}, Bargmann \cite{Ba}). 
The irreducible unitary representations of $SL_2(\Bbb R)$ are Hilbert space completions 
of the following unitary Harish-Chandra modules: 

$\bullet$ Discrete series and limit of discrete series $M_m^-,M_{-m}^+$, $m\in \Bbb Z_{\ge 1}$; 

$\bullet$ Unitary principal series $P_+(s)$, $s\in i\Bbb R$, and $P_-(s)$, $s\in i\Bbb R^\times$; 

$\bullet$ The {\bf complementary series} $P_+(s)$, $s\in \Bbb R$, $0<|s|<1$; 

$\bullet$ The trivial representation $\Bbb C$. 

Here $P_\pm(s)\cong P_{\pm}(-s)$ and there are no other isomorphisms. 
\end{theorem} 

Let us discuss explicit Hilbert space realizations of the unitary representations. We have already described such unitary realizations of principal series in $L^2(S^1)$, except the complementary series. 
For discrete series we only gave realizations for $m=1$, as $M_1^-,M_{-1}^+$ are direct summands 
in $P_-(0)$. However, one can give a realization for any $m$. To this end, note that 
$G=SL_2(\Bbb R)$ acts by fractional linear transformations on the disk $|u|\le 1$. Moreover, 
we have the Poincar\'e (hyperbolic) metric on the disk which is $G$-invariant. 
The volume element for this metric looks like 
$$
\mu=\frac{dud\overline u}{(1-|u|^2)^2}.
$$
Thus for expressions $\omega=\psi(u)(du)^{m\over 2}$ where $m\ge 2$ is an integer and $\psi(u)$ 
is holomorphic for $|u|<1$ we may define the $G$-invariant norm
$$
\norm{\omega}^2=\int_{|u|<1}\frac{\omega\overline\omega}{\mu^{\frac{m}{2}-1}}=\int_{|u|< 1}|\psi(u)|^2(1-|u|^2)^{m-2}dud\overline u. 
$$
Hence the Hilbert space completion $\widehat M_{m}^-$ may be realized as the space 
$H_m$ of holomorphic $m\over 2$-forms $\omega=\psi(u)(du)^{m\over 2}$ for $|u|<1$ for which 
$\norm{\omega}^2<\infty$ (note that this space is nonzero only if $m\ge 2$). 

Likewise, $\widehat M_{-m}^+$ can be similarly realized via antiholomorphic forms. Indeed, conjugation by the matrix $\begin{pmatrix} 0& 1\\ 1& 0\end{pmatrix}$ (of determinant $-1$) defines an outer automorphism of $SL_2(\Bbb R)$ 
which is induced by complex conjugation on the unit disk, and this automorphism exchanges $M_m^-$ with $M_{-m}^+$. 

\begin{exercise} Let $G_\ell$ be the $\ell$-fold cover of $PSL_2(\Bbb R)$ (for example, $G_1=PSL_2(\Bbb R)$, $G_2=SL_2(\Bbb R)$). Classify irreducible admissible representations (up to infinitesimal equivalence) and irreducible unitary representations 
of $G_\ell$ for all $\ell$. 

{\bf Hint.} The maximal compact subgroup of $G_\ell$ is $K_\ell$, the $\ell$-fold cover of $PSO(2)$. 
Thus irreducible Harish-Chandra modules for $G_\ell$ are irreducible $\mathfrak{sl}_2(\Bbb C)$-modules on which the element $h$ acts diagonalizably with eigenvalues in $\frac{2}{\ell}\Bbb Z$. 
\end{exercise} 

\begin{exercise} Compute the matrix coefficients of the principal series modules, $\psi_{m,n}(g)=(w_m,gw_n)$, $g\in SL_2(\Bbb R)$. 

{\bf Hint.} Write $g$ as $g=U_1DU_2$ where 
$$
U_k=\exp(i\theta_kh)\in SO(2),\ \theta_k\in \Bbb R/2\pi \Bbb Z
$$ 
for $k=1,2$ and $D=\diag(a,a^{-1})$ 
is diagonal, and express $\psi_{m,n}(g)$ as $e^{i(n\theta_2-m\theta_1)}\psi(m,n,a,s)$. Write the function $\psi(m,n,a,s)$ in terms of the Gauss hypergeometric function ${}_2F_{1}$. 
\end{exercise}

\begin{exercise}\label{comserR} (i) Show that for $-1<s<0$ the formula 
$$
(f,g)_s:=\int_{\Bbb R^2}f(y)\overline{g(z)}|y-z|^{-s-1}dydz
$$
defines a positive definite inner product on the space $C_0(\Bbb R)$ 
of continuous functions $f: \Bbb R\to \Bbb C$ 
with compact support ({\it Hint}: pass to Fourier transforms). 

(ii) Let $\mathcal H_s$ be the Hilbert space completion of $C_0(\Bbb R)$ under $(,)_s$.
Show that it identifies with the space of tempered distributions $f$ whose Fourier transform is a function satisfying
$$
\norm{f}_s^2=c_s\int_{\Bbb R}|\widehat f(\xi)|^2|\xi|^s d\xi<\infty,
$$
for a constant $c_s>0$.  

(iii) Let us view $\mathcal H_s$ as the space of 
distributional tensor fields\linebreak[4] $f(y)(dy)^{\frac{1-s}{2}}$, where $f$ is as in (ii).  
Show that the complementary series unitary representation $\widehat P_+(s)$ of 
$SL_2(\Bbb R)$ may be realized in $\mathcal H_s$ 
with the natural action of each $g\in G$ extended by continuity from continuous compactly supported tensor fields whose images under $g$ are compactly supported.
({\it Hint:} show that the differential form $\frac{dydz}{(y-z)^2}$ 
is invariant under simultaneous M\"obius transformations of $y,z$ by the same matrix). 
\end{exercise} 

\section{\bf Chevalley restriction theorem and Chevalley-Shephard-Todd theorem} 

\subsection{Chevalley restriction theorem} 
Let $\g$ be a semisimple complex Lie algebra with Cartan subalgebra $\h$, and let $W$ be the corresponding Weyl group. 
Given $F\in \Bbb C[\g]^\g$, let ${\rm Res}(F)$ be its restriction to $\h$. 

\begin{theorem} (Chevalley restriction theorem) (i) ${\rm Res}(F)\in \Bbb C[\h]^W$.

(ii) The map ${\rm Res}: \Bbb C[\g]^\g\to \Bbb C[\h]^W$ is a graded algebra isomorphism. 
\end{theorem} 

\begin{proof} (i) Let $G$ be the adjoint complex Lie group corresponding to $\g$. 
Then $\Bbb C[\g]^\g=\Bbb C[\g]^G$, so $F$ is $G$-invariant. 
Thus, denoting by $H$ the maximal torus in $G$ with ${\rm Lie}H=\h$, 
we see that the normalizer $N(H)$ preserves ${\rm Res}(F)$. 
Since $H$ acts trivially on $\h$, we get that $W=N(H)/H$ preserves ${\rm Res}(F)$, as desired. 

(ii) It is clear that ${\rm Res}$ is a graded algebra homomorphism, so we just need to show that 
it is bijective. The injectivity of this map follows immediately from the fact that ${\rm Res}(F)$ determines the values of $F$ on the subset of semisimple elements $\g_{s}\subset \g$, and this subset is dense in $\g$. 

It remains to prove the surjectivity of Res. Consider the functions 
$$
F_{\lambda,n}(x):=\Tr_{L_\lambda}(x^n)=\chi_\lambda(x^n), \ x\in \g
$$
in $\Bbb C[\g]^\g$, where $\chi_\lambda$ is the character of $L_\lambda$.
We'll show that the functions ${\rm Res}(F_{\lambda,n})$ for various $\lambda$ span $\Bbb C[\h]^W[n]=(S^n\h^*)^W$ for each $n$, which implies that Res is surjective. 

To this end, for every dominant integral weight $\lambda\in P_+$ let $m_\lambda$ 
be the orbit sum 
$$
m_\lambda :=\sum_{\mu\in W\lambda}e^\mu\in \Bbb C[P]^W.
$$
We have 
$$
\chi_\lambda=\sum_{\mu\le \lambda}N_{\lambda \mu}m_\mu,
$$
where $\mu\le \lambda$ means that $\lambda-\mu$ is a (possibly empty) sum of positive roots, and $N_{\lambda\mu}$ is the matrix of weight multiplicities (in particular, $N_{\lambda\lambda}=1$). 
This matrix is triangular with ones on the diagonal, so we can invert it and get 
\begin{equation}\label{inve}
m_\lambda=\sum_{\mu\le\lambda} \widetilde N_{\lambda\mu} \chi_\mu 
\end{equation} 
for some integers $\widetilde N_{\lambda\mu}$. 
Now, for $h\in \h$, let 
$$
M_{\lambda,n}(h):=\sum_{\mu\in W\lambda}\mu(h)^n=\frac{|W\lambda|}{|W|}\sum_{w\in W}\lambda(wh)^n. 
$$
(note that $\mu(x)^n=\mu(x^n)$). By \eqref{inve} we have 
$$
M_{\lambda,n}(h)=\sum_{\mu\le \lambda} \widetilde N_{\lambda\mu} F_{\mu,n}(h).
$$
Thus it suffices to show that $M_{\lambda,n}(h)$ for various $\lambda$ span 
$(S^n\h^*)^W[n]$ for each $n$. Since averaging over $W$ is a surjection $S^n\h^*\to (S^n\h^*)^W$, it suffices to show that the functions $\lambda^n$ for $\lambda\in P_+$ span $S^n\h^*$. 

Denote the span of these functions by $Y$. Since $P_+$ is Zariski dense in $\h^*$, we find that 
$\lambda^n\in Y$ for all $\lambda\in \h^*$. Thus $Y\subset S^n\h^*$ is a subrepresentation of $GL(\h^*)$. But $S^n\h^*$ is an irreducible representation of $GL(\h^*)$, hence $Y=S^n\h^*$. 
This completes the proof of (ii). 
\end{proof} 

\begin{remark}\label{dualfo} 1. Since the Killing form allows us to identify $\g\cong \g^*$ and $\h\cong \h^*$, the Chevalley restriction theorem is equivalent to the statement 
that the restriction map ${\rm Res}: \Bbb C[\g^*]^\g=(S\g)^\g\to \Bbb C[\h^*]^W=(S\h)^W$
is a graded algebra isomorphism. 

2. The Chevalley restriction theorem trivially generalizes to reductive Lie algebras. 
\end{remark} 

\begin{example}\label{glex} Let $\g=\mathfrak{gl}_n(\Bbb C)$. 
Then by the fundamental theorem on symmetric functions,  
$\Bbb C[\h]^W=\Bbb C[x_1,...,x_n]^{S_n}=\Bbb C[e_1,...,e_n]$ 
where 
$$
e_i(x_1,...,x_n)=\sum_{k_1<...<k_i}x_{k_1}...x_{k_i}
$$
 are elementary symmetric functions. 
The Chevalley restriction theorem thus says 
that restriction defines an isomorphism between 
the algebra $\Bbb C[\g]^\g$ of conjugation-invariant polynomials 
of a single matrix $A$ and  $\Bbb C[e_1,...,e_n]$. Namely, 
let $a_i:={\rm Tr}(\wedge^iA)$ be the coefficients of the characteristic 
polynomial of $A$ (up to sign). Then $\Bbb C[\g]^\g=\Bbb  C[a_1,...,a_n]$ 
and $a_i|_\h=e_i(x_1,...,x_n)$. Another set of generators are $b_i:={\rm Tr}(A^i)$, 
$1\le i\le n$; we have $b_i|_\h=p_i(x_1,...,x_n)$, where 
$$
p_i(x_1,...,x_n):=\sum_{k=1}^n x_k^i
$$
are the power sums, another set of generators of the algebra of symmetric functions. Yet another generating set is $c_i:={\rm Tr}(S^iA)$
which restrict to complete symmetric functions 
$$
h_i(x_1,...,x_n)=\sum_{k_1\le...\le k_i}x_{k_1}...x_{k_i}.
$$ 
Thus 
$$
a_i(A)=e_i(x_1,....,x_n),\ b_i(A)=p_i(x_1,...,x_n),\ c_i(A)=h_i(x_1,...,x_n),
$$ 
where $x_1,...,x_n$ are the eigenvalues of $A$. Note that $a_1(A)=b_1(A)=c_1(A)={\rm Tr}(A)$ and $a_n(A)=\det A$. 

For $\g=\mathfrak{sl}_n$ (type $A_{n-1}$), the story is the same, except that 
$e_1=p_1=h_1=0$ and $a_1=b_1=c_1=0$, so they should be removed. 
\end{example}

\begin{example}\label{ospex}
Similarly, for $\g=\mathfrak{so}_{2n+1}(\Bbb C)$ and $\g=\mathfrak{sp}_{2n}(\Bbb C)$
(types $B_n$ and $C_n$) we have 
$$
\Bbb C[\h]^W=\Bbb C[x_1,...,x_n]^{S_n\ltimes (\Bbb Z/2)^n}=
$$
$$
\Bbb C[x_1^2,...,x_n^2]^{S_n}=\Bbb C[e_2,e_4,...,e_{2n}]=\Bbb C[p_2,p_4,...,p_{2n}]=C[h_2,h_4,...,h_{2n}],
$$ 
where $e_k,p_k,h_k$ are symmetric functions of $2n$ variables evaluated at the point 
$(x_1,...,x_n,-x_n,...,-x_1)$, and $e_{2i}=a_{2i}|_\h$, $p_{2i}=b_{2i}|_\h$, $h_{2i}=c_{2i}|_\h$ (note that the odd-indexed symmetric functions evaluate to $0$). This is so because the eigenvalues of $A$ are $x_1,...,x_n,-x_n,...,-x_1$, and also $0$ in the orthogonal case. 

The case $\g=\mathfrak{so}_{2n}(\Bbb C)$ (type $D_n$) is a bit trickier. 
In this case the Weyl group is $W=S_{n}\ltimes (\Bbb Z/2)^n_+$, where 
$(\Bbb Z/2)^n_+$ is the group of binary $n$-dimensional vectors with zero sum of coordinates. Thus it is easy to check that
$$
\Bbb C[\h]^W=\Bbb C[e_2,...,e_{2n-2},\sqrt{e_{2n}}]. 
$$
where $e_j=e_j(x_1,...,x_n,-x_n,...,-x_1)$. 
The polynomial $\sqrt{e_{2n}}={\rm i}^nx_1...x_n$ 
is the restriction of the {\bf Pfaffian} ${\rm Pf}(A)=\sqrt{\det A}$. 
Thus 
$$
\Bbb C[\g]^\g=\Bbb C[a_2(A),...,a_{2n-2}(A),{\rm Pf}(A)]. 
$$ 

The generators 
of $\Bbb C[\g]^\g$ for exceptional $\g$ are less explicit, however. 
\end{example} 

\subsection{Chevalley-Shephard-Todd theorem, part I}

In Examples \ref{glex}, \ref{ospex} we observe that the algebras $\Bbb C[\h]^W$ of Weyl group invariant polynomials for classical groups are free (polynomial) algebras. This is not true for a general finite group: e.g. if $G=\Bbb Z/2$ acting on $\Bbb C^2$ by $(x,y)\mapsto (-x,-y)$ then the ring of invariants $\Bbb C[x,y]^{\Bbb Z/2}$ is $\Bbb C[a,b,c]$ where $a=x^2,b=xy,c=y^2$, and it is not free -- it has a relation 
$ac=b^2$ (and the set of generators is minimal). It turns out, however, that this is true for all Weyl groups and more generally complex reflection groups. 

\begin{definition} A diagonalizable automorphism $g: V\to V$ of a finite-dimensional complex vector space $V$ is called a {\bf complex reflection} if ${\rm rank}(g-1)=1$; in other words, in some basis $g=\diag(\lambda,1,...,1)$ where $
\lambda\ne 0,1$. A {\bf complex reflection group} is a finite subgroup $G\subset GL(V)$ generated by complex reflections. 
\end{definition} 

For example, the Weyl group $W\subset GL(\h)$ of a semisimple Lie algebra $\g$ and, more generally, a finite Coxeter group is a complex reflection group, but there are others, e.g. $S_n\ltimes (\Bbb Z/m)^n$ acting on $\Bbb C^n$ for $m>2$, or, more generally, the subgroup 
$G(m,d,n)$ in this group consisting of elements for which the sum of $\Bbb Z/m$-coordinates lies in $d\cdot \Bbb Z/m$ for some divisor $d$ of $m$. 

It is easy to see that any complex reflection group is uniquely a product of irreducible ones, and irreducible complex reflection groups were classified by Shephard and Todd in 1954. Besides symmetric groups $S_n$ acting on $\Bbb C^{n-1}$ and $G(m,d,n)$ acting on $\Bbb C^n$ (which includes dihedral groups), there are $34$ exceptional groups, which include 19 subgroups of $GL_2$, 6 exceptional Coxeter groups of rank $\ge 3$ ($H_3,H_4,F_4,E_6,E_7,E_8$), and $9$ other groups. 

\begin{theorem}\label{CST1} (Chevalley-Shephard-Todd theorem, part I, \cite{Che}, \cite{ST}) Let $V$ be a finite-dimensional complex vector space and  $G\subset GL(V)$ 
be a finite subgroup. Then $\Bbb C[V]^G$ is a polynomial algebra
if and only if $G$ is a complex reflection group.  
\end{theorem} 

\section{\bf Proof of the CST theorem, part I} 

\subsection{Proof of the CST theorem, part I, the ``if" direction} 

We first need a lemma from invariant theory. 
Let $G\subset GL(V)$ be a finite subgroup, 
and $I\subset \Bbb C[V]$ be the ideal generated by 
positive degree elements of $\Bbb C[V]^G$. Let 
$f_1,...,f_r\in \Bbb C[V]^G$ be homogeneous 
generators of $I$ (which exist by the Hilbert basis theorem).

\begin{lemma}\label{lee1} The algebra $\Bbb C[V]^G$ is generated by $f_1,...,f_r$; in particular, it is finitely generated.
\end{lemma} 

\begin{proof} We need to show that every homogeneous 
$f\in \Bbb C[V]^G$ is a polynomial of $f_1,...,f_r$. 
The proof is by induction in $d=\deg f$. The base $d=0$ is obvious. 
If $d>0$, we have $f\in I$, so 
$$
f=s_1f_1+...+s_rf_r
$$
where $s_i\in \Bbb C[V]$ are homogeneous of degrees $<d$. 

For $h\in \Bbb C[V]$ let $h^*:=\frac{1}{|G|}\sum_{g\in G}gh\in \Bbb C[V]^G$ be the $G$-average of $h$. Then we have  
$$
f=s_1^*f_1+...+s_r^*f_r.
$$
But by the induction assumption, $s_i^*$ are polynomials of $f_1,..,f_r$, which proves the lemma. 
\end{proof} 

\begin{remark} Let $A$ be a finitely generated commutative $\Bbb C$-algebra with an action of a finite group $G$. Lemma \ref{lee1} implies that 
the algebra $A^G$ is also finitely generated (the {\bf Hilbert-Noether lemma}). Indeed, pick generators $a_1,...,a_m$ of $A$ and 
let $V\subset A$ be the (finite-dimensional) $G$-submodule generated by them. Then $A^G$ is a quotient of $(SV)^G=\Bbb C[V^*]^G$, which is 
finitely generated by Lemma \ref{lee1}.  
\end{remark} 

The next lemma establishes a special property of algebras of invariants of complex reflection groups which will allow us to prove that they are polynomial algebras. 

\begin{lemma} \label{lee2} Assume that $G$ is a complex reflection group. Let $I$ be as above, $F_1,...,F_m\in \Bbb C[V]^G$ be homogeneous, and suppose that $F_1$ does not belong to the ideal in $\Bbb C[V]^G$ generated by $F_2,...,F_m$. Suppose $g_i\in \Bbb C[V]$ for $1\le i\le m$ 
are homogeneous and $\sum_{i=1}^m g_iF_i=0$. Then $g_1\in I$. 
\end{lemma} 

\begin{proof} Let $J=(F_2,...,F_m)\subset \Bbb C[V]$. We claim that $F_1\notin J$. Indeed, if $F_1=s_2F_2+...+s_m F_m$ then $F_1=s_2^*F_2+...+s_m^*F_m$, contradicting our assumption. 

We prove the lemma by induction in $D:=\deg g_1$. If $D=0$ then $g_1=0$, as $F_1\notin J$. This establishes the base of induction. 

Now assume $D>0$. Let $\sigma\in G$ be a complex reflection and $\alpha$ be the linear function on $V$ defining the reflection hyperplane $V^\sigma$ (i.e., the eigenvector of $\sigma$ in $V^*$ with eigenvalue $\ne 1$). Then $\sigma g_i-g_i$ vanishes on $V^\sigma$, so is divisible by $\alpha$. Thus 
$$
\sigma g_i-g_i=h_i\alpha
$$
for some polynomials $h_i$ with $\deg h_i=\deg g_i-1$, in particular $\deg h_1=D-1$. Applying the operator 
$\sigma-1$ to the relation $\sum_{i=1}^m g_iF_i=0$ and dividing by $\alpha$, we obtain 
$$
\sum_{i=1}^m h_iF_i=0.
$$
By the induction assumption $h_1\in I$, so $\sigma g_1-g_1\in I$. {\bf Since $G$ is generated by complex reflections,} this implies that $wg_1-g_1\in I$ for any $w\in G$. Thus $g_1^*-g_1\in I$. But $g_1^*$ is a positive degree invariant, so 
$g_1^*\in I$. Hence $g_1\in I$, which justifies the induction step. 
\end{proof} 

Now we are ready to prove the ``if" direction of the Chevalley-Shephard-Todd theorem. Suppose that $f_1,...,f_r\in \Bbb C[V]^G$ are homogeneous of positive degree and form a {\it minimal} set of homogeneous generators of $I$. 

\begin{lemma}\label{lee3} $f_1,...,f_r$ are algebraically independent. 
\end{lemma}

\begin{proof} Assume the contrary, i.e., 
\begin{equation}\label{eqqq1}
h(f_1,...,f_r)=0,
\end{equation} 
where 
$h(y_1,...,y_r)$ is a nonconstant polynomial. Let $d_i:=\deg f_i$. We may assume that $h$ is quasi-homogeneous (with $\deg y_i=d_i$), of the lowest possible degree. Let $x_k$ be linear coordinates on $V$, $\partial_k:=\frac{\partial}{\partial x_k}$. Differentiating \eqref{eqqq1} with respect to $x_k$ and using the chain rule, we get 
\begin{equation}\label{linrel}
\sum_{j=1}^r h_j(\bold f)\partial_kf_j=0,
\end{equation} 
where $\bold f:=(f_1,...,f_r)$ and 
$h_j:=\frac{\partial h}{\partial y_j}$. By renumbering $f_j$ if needed, we may assume that $h_1(\bold f),...,h_m(\bold f)$ is a minimal generating set of the ideal $(h_1(\bold f),...,h_r(\bold f))\subset \Bbb C[V]$. Moreover, since $h$ is nonconstant, 
$h_j\ne 0$ for some $j\in [1,r]$, and since $h$ is of lowest degree, this implies that $h_j(\bold f)\ne 0$. So $m\ge 1$. 
Then for $i>m$ we have 
$$
h_i(\bold f)=\sum_{j=1}^m g_{ij}h_j(\bold f)
$$
for some homogeneous polynomials $g_{ij}\in \Bbb C[V]$ of degree 
$$
\deg h_i-\deg h_j=d_j-d_i.
$$ 
Substituting this into \eqref{linrel}, we get
$$
\sum_{j=1}^m p_jh_j(\bold f)=0,
$$
where 
$$
p_j:=\partial_k f_j+\sum_{i=m+1}^r g_{ij}\partial_k f_i. 
$$
Since $h_1(\bold f)\notin (h_2(\bold f),...,h_m(\bold f))$, by Lemma \ref{lee2} applied to $F_i=h_i(\bold f)$, $1\le i\le m$, 
we have $p_1\in I$. Thus 
$$
\partial_k f_1+\sum_{i=m+1}^r g_{i1}\partial_k f_i=\sum_{i=1}^r q_{ik}f_i,
$$
where $q_{ik}\in \Bbb C[V]$ are homogeneous of degree $d_1-d_i-1$. Let us multiply this equation by $x_k$ and add over all $k$. Then we get 
\begin{equation}\label{eqqq2}
d_1f_1+\sum_{i=m+1}^r g_{i1}d_if_i=\sum_{i=1}^r q_{i}f_i,
\end{equation}
where $q_i:=\sum_k x_kq_{ik}$. In particular, $q_i$ are homogeneous 
of strictly positive degree. All terms in this equation are homogeneous of the same degree $d_1$, so we must have $q_1=0$. Thus \eqref{eqqq2} implies that $f_1\in (f_2,...,f_r)$, a contradiction with our minimality assumption. 
\end{proof} 
  
Now, by Lemmas \ref{lee3} and \ref{lee1}, we have $\Bbb C[V]^G=\Bbb C[f_1,...,f_r]$. This proves the ``if" direction of the Chevalley-Shephard-Todd theorem. 

\begin{remark} Note that $r={\rm trdeg}(\Bbb C(V)^G)={\rm trdeg}(\Bbb C(V))=n$, where $n=\dim V$ and trdeg denotes the transcendence degree of a field, since transcendence degree does not change under finite extensions.
\end{remark}  

\subsection{A lemma on group actions} 

\begin{lemma} \label{grac}
Let $U$ be an affine space over $\Bbb C$ and $G$ a finite group acting on $U$ by polynomial automorphisms. 

(i) Let $u\in U$ be a point with trivial stabilizer in 
$G$. Then there exists a local coordinate system on $U$ near $u$ consisting of elements of $\Bbb C[U]^G$.  

(ii) Maximal ideals in $\Bbb C[U]^G$ (i.e., characters $\chi: \Bbb C[U]^G\to \Bbb C$) are in bijection 
with $G$-orbits on $U$, which assigns to an orbit $Gu$ the character $\chi_u(f):=f(u)$. 
Thus the set of maximal ideals in $\Bbb C[U]^G$ is $U/G$. 
\end{lemma} 

\begin{proof} (i) Pick a basis $\lbrace e_i\rbrace$ of $T_u^*U$. Since $gu\ne u$ for any $g\in G$, $g\ne 1$,
there exist $y_i\in \Bbb C[U]$, $1\le i\le \dim U$ such that the linear approximation 
of $y_i$ at $gu$ is zero for all $g\ne 1$, $y_i(u)=0$, and $dy_i(u)=e_i$. 
Let $y_i^*$ be the average of $y_i$ over $G$. Then $\lbrace y_i^*\rbrace$ form a required coordinate system.  

(ii) Suppose $v,u\in U,v\notin Gu$, then $Gu\cap Gv=\emptyset$, so there exists $f\in \Bbb C[U]$ such that $f|_{Gv}=0$, 
$f|_{Gu}=1$. Moreover, by replacing $f$ by $f^*$, we may choose such $f\in \Bbb C[U]^G$. 
Then $\chi_v(f)=0$ while $\chi_u(f)=1$, so $\chi_u\ne \chi_v$, hence $u\mapsto \chi_u$ 
is injective. To show that it's also surjective, take a maximal ideal $\mathfrak m\subset \Bbb C[U]^G$. It generates an ideal 
$I\subset \Bbb C[U]$ whose projection to $\Bbb C[U]^G$ is $\mathfrak m$. Thus $I$ is a proper ideal, so by the Nullstellensatz, 
its zero set $Z\subset U$ is non-empty. Let $u\in Z$, then for any $f\in \mathfrak m$, $\chi_u(f)=f(u)=0$. Hence $\mathfrak m={\rm Ker}\chi_u$, as desired.  
\end{proof} 

\subsection{Proof of the CST theorem, part I, the ``only if" direction}\footnote{This proof uses some very basic algebraic geometry.} 

Let $G\subset GL(V)$ be a finite subgroup. Let $H$ be the normal subgroup of $G$ generated by the complex reflections of $G$. Then by the ``if" part of the theorem, $\Bbb C[V]^H$ is a polynomial algebra with an action of $G/H$. In other words, using Lemma \ref{grac}(ii), $U:=V/H$ is an affine space with a (possibly non-linear) action of $G/H$. 

Moreover, we claim that $G/H$ acts freely on $U$ outside of a set of codimension $\ge 2$. Indeed, if $1\ne s\in G/H$ and $a\in s$ then $a$ is not a reflection, so $Y_s:=\cup_{a\in s}V^{a}$ has codimension $\ge 2$. Now, for any $v$ in the preimage of $U^s$ in $V$ and $a\in s$ we have $av=h^{-1}v$ for some $h\in H$, thus $hav=v$ and $v\in Y_s$. Thus $U^s$ is contained in the image of $Y_s$ in $U$, hence ${\rm codim}(U^s)\ge 2$, as claimed. 

Now assume that $\Bbb C[V]^G$ is a polynomial algebra, and
let $V/G=W$ be the corresponding affine space. Consider the natural regular map $\eta: V/H=U\to V/G=W$ between $n$-dimensional affine spaces, and let $J\in \Bbb C[U]$ be the Jacobian of this map (well defined up to scaling). If $u\in U$ and the stabilizer of $u$ in $G/H$ is trivial then by Lemma \ref{grac}, $\eta$ is \'etale at $u$, hence $J(u)\ne 0$. But as shown above, the complement of such points has codimension $\ge 2$. This implies that $J={\rm const}$, as a nonconstant polynomial would vanish on a subset of codimension $1$. Thus the quotient map $\eta$ by $G/H$ is \'etale at $0$. Hence the stabilizer $G/H$ of $0$ is trivial, i.e. $H=G$. 

\begin{remark} Let $X$ be a smooth affine algebraic variety over $\Bbb C$ and $G$ be a finite group of automorphisms of $X$. Then by the Hilbert-Noether lemma, $\Bbb C[X]^G$ is finitely generated, so $X/G:= {\rm Spec}\Bbb C[X]^G$ is an affine algebraic variety. The Chevalley-Shephard-Todd theorem implies that $X/G$ is smooth at the image $x^*\in X/G$ of 
$x\in X$  if and only if the stabilizer $G_x$ of $x$ is a complex reflection 
group in $GL(T_xX)$. In particular, $X/G$ is smooth iff all stabilizers are complex reflection groups. This follows from the {\bf formal Cartan lemma}: any action of a finite group $G$ on a formal polydisk $D$ over a field of characteristic zero is equivalent to its linearization (i.e., to the action of $G$ on the formal neighborhood of $0$ in the tangent space to $D$ at its unique geometric point). 
\end{remark} 

\section{\bf Chevalley-Shephard-Todd theorem, part II}

\subsection{Degrees of a complex reflection group}

The degrees $d_i$ of the generators $f_i$ of $\Bbb C[V]^G$ 
for a complex reflection group $G$ are uniquely determined up to relabelings (even though $f_i$ themselves are not). 
Indeed, recall that for a $\Bbb Z$-graded vector space $M$ with finite-dimensional homogeneous components its {\bf Hilbert series} 
is 
$$
H(M,q)=\sum_{i\in \Bbb Z}\dim M[i]q^i 
$$
(also called Hilbert polynomial if $\dim M<\infty$). Then the Hilbert series of $\Bbb C[V]^G$ is 
$$
H(\Bbb C[V]^G,q)=\frac{1}{\prod_{i=1}^r (1-q^{d_i})},
$$
which uniquely determines $d_i$. These numbers are usually arranged in non-decreasing order 
and are called the {\bf degrees} of $G$. For instance, 
for Weyl groups of classical simple Lie algebras we saw in Examples \ref{glex},\ref{ospex} that in type $A_{n-1}$ the degrees
are $2,3,...,n$, for $B_n$ and $C_n$ 
they are $2,4,...,2n$, and for $D_n$ they are $2,4,...,2n-2$ and $n$. In particular, in the last case, if $n$ is even, the degree $n$ occurs twice. 

\subsection{$\Bbb C[V]$ as a $\Bbb C[V]^G$-module} 

Let $R$ be a commutative ring. Let $A$ be a commutative $R$-algebra
with an $R$-linear action of a finite group $G$. 
 
\begin{proposition}\label{hn} (Hilbert-Noether theorem) (i) $A$ is integral over $A^G$. In particular, if $A$ finitely generated then it is 
module-finite over $A^G$. 

(ii) If $R$ is Noetherian and $A$ is finitely generated then so is $A^G$. 
\end{proposition} 

\begin{proof} (i) We will prove only the first statement, as the second one
then follows immediately. For $a\in A$, consider the monic polynomial 
$$
P_a(x):=\prod_{g\in G} (x-ga).
$$
It is easy to see that $P_a\in A^G[x]$ and $P_a(a)=0$, which implies the statement. 

(ii) This follows from (i) and the Artin-Tate lemma: 
If $B\subset A$ is an $R$-subalgebra of a finitely generated $R$-algebra 
$A$ over a Noetherian ring $R$ and $A$ is module-finite over $B$ then $B$ is finitely generated.\footnote{Recall the proof of the Artin-Tate lemma. Let $x_1,...,x_m$ generate $A$ as an $R$-algebra and let $y_1,..., y_n$ generate $A$ as a $B$-module. Then we can write
$$
x_i = \sum_j b_{ij} y_j,\quad y_i y_j = \sum_k b_{i j k} y_k
$$
with $b_{ij} , b_{i j k} \in B$. Then $A$ is module-finite over the $R$-algebra $B_0\subset B$  generated by $b_{i j}, b_{i j k}$ (namely, it is generated as a module over $B_0$ by the $y_i$).  Using that $R$ and hence $B_0$ is Noetherian, we obtain that $B$ is also module-finite over $B_0$. Since $B_0$ is a finitely generated $R$-algebra, so is $B$.}
 \end{proof} 

This shows for any finite $G\subset GL(V)$, the algebra $\Bbb C[V]$ is module-finite over $\Bbb C[V]^G$. Note that in (ii) we again proved that 
$\Bbb C[V]^G$ is finitely generated. 

\begin{theorem}\label{CST2} (Chevalley-Shephard-Todd theorem, part II, \cite{Che}, \cite{ST}) If $G$ is a complex reflection group then for any irreducible representation $\rho$ of $G$, the $\Bbb C[V]^G$-module 
$\Hom_G(\rho,\Bbb C[V])$ is free of rank $\dim \rho$. Thus the 
$G$-module $R_0=\Bbb C[x_1,...,x_n]/(f_1,...,f_n)$ 
is the regular representation and $\prod_{i=1}^n d_i=|G|$. 
Moreover, the Hilbert polynomial $H(R_0,q):=\sum_{N\ge 0} \dim R_0[N]q^N$ is 
$$
H(R_0,q)=\prod_{i=1}^n [d_i]_q,
$$
where $[d]_q:=\frac{1-q^d}{1-q}=1+q+...+q^{d-1}$. 
\end{theorem} 

Thus we see that the Hilbert polynomial of $\Hom_G(\rho,R_0)$ 
is some polynomial $K_\rho(q)$ with nonnegative integer coefficients and $K_\rho(1)=\dim \rho$. It is called the {\bf fake degree polynomial} (or {\bf Kostka polynomial} for $G=S_n$). We have 
$$
\sum_\rho K_\rho(q)\dim \rho=H(R_0,q)=\prod_{i=1}^n[d_i]_q.
$$
For example, for $G=S_3$ and $V$ the reflection representation 
we have three irreducible representations: $\Bbb C_+$ (trivial), 
$\Bbb C_-$ (sign) and $V$. We have $K_{\Bbb C_+}(q)=1$ 
and 
$$
1+2K_V(q)+K_{\Bbb C_-}(q)=(1+q)(1+q^2)=1+2q+2q^2+q^3.
$$
It follows that 
$$
K_V(q)=q+q^2,\ K_{\Bbb C_-}(q)=q^3.
$$

\subsection{Graded modules} For the proof of Theorem \ref{CST2} we need to recall some basics from commutative algebra, which we discuss in the next few subsections. 

Let $k$ be a field,  $S$ a $\Bbb Z_+$-graded (not necessarily commutative) $k$-algebra with generators $f_i$ of positive integer degrees $\deg f_i=d_i$, 
$M$ a $\Bbb Z_+$-graded left $S$-module, and $M_0:=M/S_+M$, where $S_+\subset S$ is the augmentation ideal.

\begin{lemma}\label{homog} (i) Any homogeneous lift $\lbrace v_i^*\rbrace$ of a homogeneous basis $\lbrace v_i\rbrace$ of $M_0$ to $M$ is a system of generators for $M$; in particular, if $\dim M_0<\infty$ then $M$ is finitely generated. 

(ii) If in addition $M$ is projective, then $\lbrace v_i^*\rbrace$ is actually a basis 
of $M$ (in particular, $M$ is free). Thus if $\dim M_0[i]<\infty$ for all $i$ 
then 
$$
H(M,q)=H(M_0,q)H(S,q).
$$ 
In particular, if $S=k[f_1,...,f_n]$ then 
$$
H(M,q)=
\frac{H(M_0,q)}{\prod_{i=1}^n(1-q^{d_i})}.
$$
\end{lemma} 

\begin{proof} 
(i) We prove that any homogeneous element $u\in M$ is a linear combination of 
$v_i^*$ with coefficients in $S$ by induction in $\deg u$ (with obvious base). 
Namely, if $u_0$ is the image of $u$ 
in $M_0$ then $u_0=\sum_i c_iv_i$ for some $c_i\in k$ ($c_i=0$ unless $\deg v_i=\deg u$), and 
so  
$$
u-\sum_i c_iv_i^*=\sum_j f_ju_j,
$$
with $u_j\in M$, $\deg u_j=\deg u-d_j$. So by the induction assumption 
$$
u_j=\sum_i p_{ij}v_i^*
$$ 
for some homogeneous $p_{ij}\in S$ of degree $\deg u-d_j-\deg v_i^*$, and we get 
$$
u=\sum_i p_iv_i^*,
$$
where $p_i:=c_i+\sum_j f_jp_{ij}$.\footnote{Note that for each $i$, one of these two summands is necessarily $0$.} 

(ii) Let $M'$ be the free graded $S$-module with basis $w_i$ of degrees $\deg w_i=\deg v_i$, and 
$f: M'\to M$ be the surjection sending $w_i$ to $v_i^*$. Since $M$ is projective, the map 
$$
f\circ : \Hom(M,M')\to \Hom(M,M)
$$ 
is surjective, so we can pick a homogeneous $g: M\to M'$ of degree $0$ 
such that $f\circ g={\rm id}_M$. Then $g\circ f: M'\to M'$ is a projection which identifies 
$M'$ with $M\oplus {\rm Ker}f$ as a graded $S$-module. But the map 
$f_0: M_0'\to M_0$ induced by $f$ sends the basis $w_i$ of $M_0'$ to the basis $v_i$ of $M_0$, so is an isomorphism. It follows that $({\rm Ker}f)_0=0$, so 
${\rm Ker}f=0$ and $f$ is an isomorphism, as claimed. 
\end{proof}  

\subsection{Koszul complexes} 
Let $R$ be a commutative ring and $f\in R$. Then we can define a 2-step {\bf Koszul complex} $K_R(f)=[R\to R]$ with the differential given by multiplication by $f$ 
(the two copies of $R$ sit in degrees $-1$ and $0$). We have $H^0(K_R(f))=R/(f)$, and $K_R(f)$ is exact in degree $-1$ if and only if $f$ is not a zero divisor in $R$. This allows us to define the Koszul complex of several elements of $R$: 
$$
K_R(f_1,...,f_m)=K_R(f_1)\otimes_R...\otimes_R K_R(f_m)
$$
with $H^0(K_R(f_1,...,f_m))=R/(f_1,...,f_m)$. Thus 
$$
K_R(f_1,...,f_m)=K_R(f_1,...,f_{m-1})\otimes_R K_R(f_m).
$$ 

For example, let $R:=k[x_1,...,x_n]$ for a field $k$. 
Then the complex $K_n:=K_R(x_1,...,x_n)=K_1^{\otimes n}$ is acyclic in negative degrees 
and has $H^0=k$. Thus for any commutative $k$-algebra $S$, the complex $K_{R\otimes S}(x_1,...,x_n):=
K_R(x_1,...,x_n)\otimes S$ is acyclic in negative degrees and has $H^0=S$.
By taking $S=R$ and making a linear change of variable, this yields a 
free resolution of $R$ as an $R$-bimodule called the {\bf Koszul resolution}, which 
we'll denote by $K_n$:
$$
0\to R\otimes\wedge^n k^n\otimes R\to ...\to R\otimes \wedge^2k^n\otimes R\to R\otimes k^n \otimes R\to R\otimes R\to R.
$$
Moreover, this exact sequence is split as a sequence of $R$-modules (under right multiplication by $R$), since all participating $R$-modules are free. Hence if $M$ is any $R$-module then 
$K_n\otimes_R M$ is a free resolution of $M$ of length $n$. Thus we obtain 

\begin{proposition} \label{vani} If $i>n$ then 
 for any $k[x_1,...,x_n]$-modules $M,N$, one has ${\rm Ext}^i(M,N)=0$.
\end{proposition}

\subsection{Syzygies} 
Now assume that $M$ is a finitely generated graded module over $R=k[x_1,...,x_n]$. Then $M=:M_0$ is a quotient of $R\otimes V_0$, where $V_0$ is a finite-dimensional graded vector space. By the Hilbert basis theorem, the kernel $M_1$ of the map $\phi_0: R\otimes V_0\to M$ is finitely generated, so 
is a quotient of $R\otimes V_1$ for some finite-dimensional graded space $V_1$, and the kernel $M_2$ of $\phi_1: R\otimes V_1\to M_1$ is finitely generated, and so on. 
The long exact sequences of Ext groups associated to the short exact sequences
$$
0\to M_{j+1}\to R\otimes V_j\to M_j\to 0
$$
and Proposition \ref{vani} then imply by induction in $j$ that $\Ext^i(M_j,N)=0$ for any $R$-module $N$ if $i>n-j$. In particular, the module $M_n$ is projective, hence free by Lemma \ref{homog}, i.e., we may take $V_n$ such that $M_n=R\otimes V_n$. This gives a free resolution of $M$ by finitely generated graded $R$-modules: 
$$
0\to R\otimes V_n\to...\to R\otimes V_0\to M.
$$ 
Thus, taking graded Euler characteristic we obtain

\begin{theorem}\label{syz} ({\bf Hilbert syzygy theorem}) We have 
$$
H(M,q)=\frac{p(q)}{(1-q)^n},
$$
where $p$ is a polynomial with integer coefficients. 
\end{theorem} 

\begin{proof} Indeed, $p$ is just the alternating sum of the Hilbert polynomials of $V_j$. 
\end{proof} 

\subsection{The Hilbert-Samuel polynomial} 

Let $R$ be a commutative Noetherian ring and $\mathfrak{m}\subset R$ a maximal ideal. Then $R/\mathfrak{m}=k$ is a field and $\mathfrak{m}^N/\mathfrak{m}^{N+1}$ is a finite-dimensional $k$-vector space. Thus ${\rm gr}(R):=\oplus_{N\ge 0}\mathfrak{m}^N/\mathfrak{m}^{N+1}$ (where $\mathfrak{m}^0:=R$) is a graded algebra generated in degree $1$. 
So by the Theorem \ref{syz}, the Hilbert series 
$$
H({\rm gr}(R),q)=\sum_{N\ge 0} \dim_k (\mathfrak{m}^N/\mathfrak{m}^{N+1})q^N
$$
is a rational function of the form $\frac{p(q)}{(1-q)^m}$, where $p$ is a polynomial
and $m=\dim_k (\mathfrak{m}/\mathfrak{m}^{2})$. Hence
$$
P_{R,\mathfrak{m}}(N):=\sum_{j=0}^{N-1} \dim_k (\mathfrak{m}^j/\mathfrak{m}^{j+1})=
{\rm length}(R/\mathfrak{m}^N)
$$
is a polynomial in $N$ for large enough $N$ called the {\bf Hilbert-Samuel polynomial} of $R$ at $\mathfrak{m}$. 
The degree of this polynomial equals the order of the pole of 
$H({\rm gr}(R),q)$ at $q=1$. We call this degree the {\bf dimension} of $R$ at $\mathfrak{m}$, denoted $\dim_{\mathfrak{m}}R$. 
For example, if $R=k[x_1,...,x_n]$ and $\mathfrak{m}$ 
is any maximal ideal then $P_{R,\mathfrak{m}}(N)=\binom{N+n-1}{n}$, so $\dim_{\mathfrak{m}}R=n$. 

\begin{lemma}\label{bou} Let $f\in \mathfrak m$. Then $\dim_{\mathfrak{m}}(R/f)\ge \dim_{\mathfrak{m}}R-1$.
\end{lemma} 

\begin{proof} The ideal $(f)$ in $R/\mathfrak{m}^N$ is the image of $fR/\mathfrak{m}^{N-1}$. 
So we have 
$$
P_{R/f,\mathfrak{m}}(N)={\rm length}((R/\mathfrak{m}^N)/f)\ge {\rm length}(R/\mathfrak{m}^N)-{\rm length}(R/\mathfrak{m}^{N-1})
$$
$$
=P_{R,\mathfrak{m}}(N)-P_{R,\mathfrak{m}}(N-1),
$$
which implies the statement. 
\end{proof} 

Let $k$ be an algebraically closed field and $\mathfrak{m}_p\subset k[x_1,...,x_n]$ be the maximal ideal corresponding to $p\in k^n$. 

\begin{corollary}\label{bou1} Let $f_1,...,f_m\in k[x_1,...,x_n]$
be homogeneous polynomials of positive degrees. Let $Z$ be an irreducible component of the zero set $Z(f_1,...,f_m)\subset k^n$. Then $\dim_{\mathfrak{m}_0}k[Z]\ge n-m$.
\end{corollary} 

\begin{proof} Let $p\in Z$ be not contained in other components 
of $Z(f_1,...,f_m)$. Applying Lemma \ref{bou} repeatedly, we get 
$\dim_{\mathfrak{m}_p}k[Z]\ge n-m$. But ${\rm gr}(k[Z]/\mathfrak{m}_p^N)$ (the associated graded under the filtration induced by the grading on $k[Z]$) is a quotient of $k[Z]/\mathfrak{m}_0^N$.
Thus\footnote{In fact these dimensions are equal (to $\dim Z$), but we don't use it here.}
$$
 \dim_{\mathfrak{m}_0}k[Z]\ge \dim_{\mathfrak{m}_p}k[Z],
$$ 
so 
$\dim_{\mathfrak{m}_0}k[Z]\ge n-m$. 
\end{proof}

\subsection{Regular sequences} 
Let $R$ be a commutative ring. A sequence 
$f_1,...,f_n\in R$ is called a {\bf regular sequence} 
if for each $j\in [1,n]$, $f_j$ is not a zero 
divisor in $R/(f_1,...,f_{j-1})$, and $R/(f_1,...,f_n)\ne 0$. 

\begin{lemma}\label{exaa} If $f_1,...,f_n\in R$ is a regular sequence then the complex $K_R(f_1,...,f_n)$ is exact in negative degrees.
\end{lemma} 

\begin{proof} The proof is by induction in $n$ with obvious base. For the induction step, note that by the inductive assumption $K_R(f_1,...,f_{n-1})$ is exact in negative degrees with $H^0=R/(f_1,...,f_{n-1})$, so the cohomology of
$K_R(f_1,...,f_n)$ coincides with the cohomology 
$K_{R/(f_1,...,f_{n-1})}(f_n)$, which vanishes in negative degrees 
since $f_n$ is not a zero divisor in $R/(f_1,...,f_{n-1})$. 
\end{proof} 

Now let $k$ be an algebraically closed field. 

\begin{proposition}\label{coo2} Suppose $f_1,...,f_n\in R:=k[x_1,...,x_n]$ are homogeneous polynomials of positive degree such that the zero set $Z(f_1,...,f_n)$ consists of the origin. Then 
$f_1,...,f_n$ is a regular sequence. 
\end{proposition} 

\begin{proof} We argue by induction on $m$ to show that for each $m\le n-1$, $f_{m+1}$ is not a zero divisor 
in $R_m:=k[x_1,...,x_n]/(f_1,...,f_m)$. Let $Z_m=Z(f_1,...,f_m)$. By the induction hypothesis and the unmixedness theorem for quotients of polynomial rings by regular sequences (see \cite[Corollary~18.14]{Eis}), $R_m$ has no embedded associated primes. Thus it suffices to show that 
$f_{m+1}$ does not vanish on any irreducible component of $Z_m$. Assume the contrary, i.e., that it vanishes on such a component $Z_m^0$. By Corollary \ref{bou1}, we have $\dim_{\mathfrak{m}_0} k[Z_m^0]\ge n-m$. Since 
$f_{m+1}=0$ on $Z_m^0$, using Lemma \ref{bou} repeatedly, we get 
$$
\dim_{\mathfrak{m}_0} k[Z_m^0]/(f_{m+1},...,f_n)\ge 1,
$$ 
which is a contradiction, as the zero set of $f_{m+1},...,f_n$ on $Z_m^0$ 
consists just of the origin, so this dimension must be zero.  
\end{proof} 

\begin{proposition}\label{kosz} Suppose $f_1,..,f_n\in R:=k[x_1,...,x_n]$ are homogeneous polynomials of degrees $d_1,...,d_n>0$ such that $R$ is a finitely generated module over $S:=k[f_1,...,f_n]$. Then this module is free of rank $\prod_{i=1}^n d_i$. 
Moreover, the Hilbert polynomial of 
$R_0:=k[x_1,...,x_n]/(f_1,...,f_n)$ (or, equivalently, of a space of free homogeneous generators of this module) is 
\begin{equation}\label{hils} 
H(R_0,q)=\prod_{i=1}^n [d_i]_q.
\end{equation} 
\end{proposition}  

\begin{proof} By Lemma \ref{homog}, it suffices to show that 
$R$ is a free $S$-module. 
By assumption $R_0$ is finite-dimensional, i.e., the equations 
$$
f_1=...=f_n=0
$$ 
have only the zero solution. By Proposition \ref{coo2}, this implies that 
$f_1,...,f_n$ is a regular sequence, so by Lemma \ref{exaa} the Koszul complex $K_R(f_1,...,f_n)$ associated to this sequence is exact in negative degrees. The finite-generation assumption implies that $f_1,...,f_n$ are algebraically independent, so let us identify $S$ with $k[a_1,...,a_n]$ via $a_i\mapsto f_i$ and put $\deg a_i = 0$.
Now consider the complex $K_{R\otimes S}(f_1-a_1,...,f_n-a_n)$. This complex is filtered by degree with associated graded being 
$$
K_{R\otimes S}(f_1,...,f_n)=K_R(f_1,...,f_n)\otimes S.
$$ 
Thus $K_{R\otimes S}(f_1-a_1,...,f_n-a_n)$ is also exact in nonzero degrees with 
$$
H^0=k[x_1,...,x_n,a_1,...,a_n]/(f_1-a_1,...,f_n-a_n)\cong R
$$ 
by eliminating $a_i$, 
and the associated graded under the above filtration is ${\rm gr}(R)=R_0\otimes S$ as an $S$-module.
This module is free over $S$, hence so is $R$.
\end{proof} 

\begin{remark}\label{remci} Let $f_1,...,f_r$ be a regular sequence of homogeneous polynomials in $k[x_1,...,x_n]$ of positive degree and $Z_m\subset k^n$ be the zero set of $f_1,...,f_m$. Then $f_{m+1}$ is not a zero divisor in $k[x_1,...,x_n]/(f_1,...,f_m)$, hence does not vanish identically on any irreducible component of $Z_m$. So by induction in $m$ we get that the dimension of every irreducible component of $Z_m$ is $\le n-m$. By Corollary \ref{bou1}, this implies that this dimension is precisely $n-m$; in particular, $r\le n$, and every irreducible component of the affine scheme $\mathcal Z:={\rm Spec} k[x_1,...,x_n]/(f_1,...,f_r)$ has dimension $n-r$. Such a scheme is called a {\bf complete intersection}. In fact, it follows by induction in  $r$ that $\mathcal Z$ is a complete intersection precisely when all its irreducible components have dimension $\le n-r$ (in which case they have dimension exactly $n-r$). In particular, if $r=n$, this means that the only $k$-point of $\mathcal Z$ is the origin, as indicated in Proposition \ref{coo2}. Thus the converse of this proposition also holds. 
\end{remark} 

\subsection{Proof of the CST Theorem, Part II} We are now ready to prove Theorem \ref{CST2}. It follows from Proposition \ref{kosz}, Proposition \ref{hn} and Theorem \ref{CST1} that  $\Bbb C[V]$ is a free $\Bbb C[V]^G$-module. Since $\Bbb C[V]=\oplus_\rho \Hom_G(\rho,\Bbb C[V])\otimes \rho$, it follows by Lemma \ref{homog}(ii) that $\Hom_G(\rho,\Bbb C[V])$ is also a free $\Bbb C[V]^G$-module (as it is graded and projective). Finally, the rank of this module equals 
$$
\dim_{\Bbb C(V)^G}(\Bbb C(V)^G\otimes_{\Bbb C[V]^G}\Hom_G(\rho,\Bbb C[V]))=
\dim_{\Bbb C(V)^G}\Hom_G(\rho,\Bbb C(V)), 
$$
which equals $\dim \rho$ by basic Galois theory 
($\Bbb C(V)$ is a regular representation of $G$ over $\Bbb C(V)^G$). 

\section{\bf Kostant's theorem}

\subsection{Kostant's theorem for $S\g$}\label{kothe}

Let $\g$ be a semisimple complex Lie algebra. 

\begin{theorem}\label{Kos} (Kostant) $S\g$ is a free $(S\g)^\g$-module. 
Moreover, for every finite-dimensional irreducible representation 
$V$ of $\g$, the space $\Hom_\g(V,S\g)$ is a free $(S\g)^\g$ module
of rank $\dim V[0]$, the dimension of the zero weight space of $V$. 
\end{theorem} 

The rest of the subsection is dedicated to the proof of this theorem. 
Introduce a filtration on $S\g$ by setting $\deg(\g_\alpha)=1$ for all roots $\alpha$ and $
\deg \h=2$. Then ${\rm gr}(S\g)=S\n_-\otimes S\h\otimes S\n_+$  and by the Chevalley restriction theorem, ${\rm gr}((S\g)^\g)$ is identified with the subalgebra $(S\h)^W$ of the middle factor. 
Thus by the Chevalley-Shephard-Todd theorem, ${\rm gr}(S\g)$ is a free ${\rm gr}((S\g)^\g)$-module. It follows that 
$S\g$ is a free $(S\g)^\g$-module (namely, any lift of a homogeneous basis 
of the graded module is a basis of the filtered module). 

Now recall that 
\begin{equation}\label{deco}
S\g=\oplus_{V\in {\rm Irr}(\g)} V\otimes \Hom_\g(V,S\g).
\end{equation}
Thus $\Hom_\g(V,S\g)$ is a graded direct summand in $S\g$. 
It follows that $\Hom_\g(V,S\g)$ is a projective, hence free $(S\g)^\g$-module (using Lemma \ref{homog}(ii)). 

It remains to prove the formula for the rank of $\Hom_\g(V,S\g)$. 
To this end, consider the $Q$-graded Hilbert series of $S\g$, i.e., the generating function of the characters of symmetric powers of $\g$:
$$
H_Q(S\g,q):=\sum_{m\ge 0}(\sum_{\mu\in Q}\dim S^m\g[\mu] e^\mu)q^m \in \Bbb C[Q][[q]]. 
$$
Since $S\g=S\h\otimes \bigotimes_{\alpha\in R} S\g_\alpha$, we have 
$$
H_Q(S\g,q)=\frac{1}{(1-q)^r}\prod_{\alpha\in R}\frac{1}{1-qe^\alpha},
$$
where $r={\rm rank}(\g)$. 
On the other hand, by \eqref{deco},  
$$
H_Q(S\g,q)=\sum_{V\in {\rm Irr}(\g)}H(\Hom_\g(V,S\g),q)\chi_V,
$$
where $\chi_V$ is the character of $V$.

Now, by the Chevalley restriction theorem $(S\g)^\g\cong (S\h)^W$, so 
$$
H(\Hom_\g(V,S\g),q)=H(\Hom_\g(V,(S\g)_0),q)H((S\h)^W,q),
$$
where $(S\g)_0$ is the quotient of $S\g$ by the ideal generated by the invariants of positive degree. Thus by the Chevalley-Shephard-Todd theorem,  
$$
H(\Hom_\g(V,S\g),q)=H(\Hom_\g(V,(S\g)_0),q)\prod_{j=1}^r \frac{1}{1-q^{d_j}}.
$$
So we get 
$$
\sum_{V\in {\rm Irr}(\g)}H(\Hom_\g(V,(S\g)_0),q)\chi_V=\frac{\prod_{j=1}^r[d_j]_q}{\prod_{\alpha\in R}(1-qe^\alpha)}.
$$
By character orthogonality, $H(\Hom_\g(V,(S\g)_0),q)$ is the inner product of the right hand side of this equality with $\chi_V$:
$$
H(\Hom_\g(V,(S\g)_0),q)=\left(\frac{\prod_{j=1}^r[d_j]_q}{\prod_{\alpha\in R}(1-qe^\alpha)},\chi_V\right).
$$
Recall that the inner product on $\Bbb C[P]$ making the characters orthonormal is given by the formula
$$
(\phi,\psi)=\frac{1}{|W|}{\rm CT}(\phi\psi^*\prod_{\alpha\in R}(1-e^\alpha)),
$$
where CT denotes the constant term and $*$ is the automorphism of $\Bbb C[P]$ given by $(e^\mu)^*=e^{-\mu}$. Thus, using that $\chi_V^*=\chi_{V^*}$, we get 
\begin{equation}\label{Kostq}
H(\Hom_\g(V,(S\g)_0),q)=\frac{\prod_{j=1}^r[d_j]_q}{|W|}{\rm CT}\left(\chi_{V^*}\prod_{\alpha\in R}\frac{1-e^\alpha}{1-qe^\alpha}\right).
\end{equation}
In this formula $q$ is a formal parameter, but the right hand side
converges to an analytic function in the disk $|q|<1$, since 
it can be written as an integral: 
$$
H(\Hom_\g(V,(S\g)_0),q)=\frac{\prod_{j=1}^r[d_j]_q}{|W|}\int_{\h_{\Bbb R}/Q^\vee} \chi_{V^*}(e^{2\pi ix})\prod_{\alpha\in R}\frac{1-e^{2\pi i\alpha(x)}}{1-qe^{2\pi i\alpha(x)}}dx,
$$
where $Q^\vee$ is the coroot lattice. If $0\le q<1$, this can also be written as 
\begin{equation}\label{L2norm} 
H(\Hom_\g(V,(S\g)_0),q)=\frac{\prod_{j=1}^r[d_j]_q}{|W|}\int_{\h_{\Bbb R}/Q^\vee} \chi_{V^*}(e^{2\pi ix})\left|\prod_{\alpha\in R_+}\frac{1-e^{2\pi i\alpha(x)}}{1-qe^{2\pi i\alpha(x)}}\right|^2dx.
\end{equation} 

\begin{lemma}\label{L2} As $q\to 1$ in $(0,1)$, 
the function $F_q(x):=\prod_{\alpha\in R_+}\frac{1-e^{2\pi i\alpha(x)}}{1-qe^{2\pi i\alpha(x)}}$ goes to $1$ in $L^2(\h/Q^\vee)$.\footnote{Note however that $F_q(x)$ does not go to $1$ pointwise (hence not in $C(\h/Q^\vee)$) since $F_q(0)=0$.} 
\end{lemma} 

\begin{proof} If $x\in \Bbb R$, $|x|\le 1$ then 
${\rm min}_{q\in [0,1]}(1-2qx+q^2)$ is $1$ if $x\le 0$ and 
$1-x^2$ if $x>0$. So if $z=x+iy$ is on the unit circle and $0\le q<1$ then
$$
\left|\frac{1-z}{1-qz}\right|^2=\frac{2(1-x)}{1-2qx+q^2}\le \begin{cases} 2(1-x),\ x\le 0\\ \frac{2}{1+x},\ x>0\end{cases} \le 4. 
$$
Note also that by the residue formula
$$
\int_0^1 \frac{dt}{|1-qe^{2\pi it}|^2}=\frac{1}{2\pi i}\int_{|z|=1}\frac{z^{-1}dz}{(1-qz)(1-qz^{-1})}=\frac{1}{1-q^2}.
$$
Thus 
$$
\int_0^1 \left|\frac{1-e^{2\pi it}}{1-qe^{2\pi it}}-1\right|^2dt=
\int_0^1 \left|\frac{(q-1)e^{2\pi it}}{1-qe^{2\pi it}}\right|^2dt=\frac{1-q}{1+q}.
$$
So $\frac{1-z}{1-qz}\to 1$ as $q\to 1$ in $L^2(S^1)$.  
But if $X$ is a finite measure space and 
for $j=1,...,N$, $f_n^{(j)}\to f^{(j)}$ in $L^2(X)$ as $n\to \infty$ 
and $|f_n^{(j)}(z)|\le C$ for all $z\in X$ and all $n,j$ then $\prod_j f_n^{(j)}\to \prod_j f_j$ in $L^2(X)$. This implies the statement. 
\end{proof} 

By Lemma \ref{L2} we may take the limit $q\to 1$ under the integral in \eqref{L2norm}. Then, using that $\prod_{j=1}^r d_j=|W|$, we get 
$$
\dim \Hom_\g(V,(S\g)_0)=\int_{\h/Q^\vee}\chi_{V^*}(e^{2\pi ix})dx=
$$
$$
{\rm CT}(\chi_{V^*})=\dim V^*[0]=\dim V[0],
$$
which concludes the proof of Kostant's theorem. 

\subsection{The structure of $S\g$ as a $(S\g)^\g$-module}
As a by-product, we obtain 

\begin{theorem} (Kostant) 
For $\lambda\in P_+$ we have 
$$
H(\Hom_\g(L_\lambda^*,(S\g)_0),q)=
\frac{\prod_{j=1}^r[d_j]_q}{|W|}{\rm CT}\left(\prod_{\alpha\in R}\frac{1-e^\alpha}{1-qe^\alpha}\chi_{L_\lambda}\right)=
$$
$$
\prod_{j=1}^r[d_j]_q\cdot {\rm CT}\left(\frac{e^{\lambda}\prod_{\alpha\in R_+}(1-e^{\alpha})}
{\prod_{\alpha\in R}(1-qe^\alpha)}\right).
$$
\end{theorem} 

Indeed, the first expression is \eqref{Kostq} and second expression is obtained from \eqref{Kostq} using the Weyl character formula for $\chi_{L_\lambda}$ and observing that all terms in the resulting sum over $W$ are the same. 

Substituting $\lambda=0$, we get

\begin{corollary} \label{diform}
$$
\frac{1}{|W|}{\rm CT}\left(\prod_{\alpha\in R}\frac{1-e^\alpha}{1-qe^\alpha}\right)={\rm CT}\left(\frac{\prod_{\alpha\in R_+}(1-e^\alpha)}{\prod_{\alpha\in R}(1-qe^\alpha)}\right)=\frac{1}{\prod_{j=1}^r[d_j]_q}.
$$
\end{corollary} 

For example, if $\g=\mathfrak{sl}_2$, this formula looks like
\begin{equation}\label{sl2for}
\frac{1}{2}{\rm CT}\left(\frac{(1-z)(1-z^{-1})}{(1-qz)(1-qz^{-1})}\right)=
{\rm CT}\left(\frac{1-z}{(1-qz)(1-qz^{-1})}\right)=\frac{1}{1+q},
\end{equation} 
which is easy to check using the residue formula. 

For $\g=\mathfrak{sl}_n$ we obtain the identity
$$
\frac{1}{n!}{\rm CT}\left(\prod_{1\le i<j\le n}\frac{(1-\frac{X_i}{X_j})(1-\frac{X_j}{X_i})}{(1-q\frac{X_i}{X_j})(1-q\frac{X_j}{X_i})}\right)={\rm CT}\left(\prod_{1\le i<j\le n}\frac{1-\frac{X_i}{X_j}}{(1-q\frac{X_i}{X_j})(1-q\frac{X_j}{X_i})}\right)
$$
$$
=\frac{1}{(1+q)...(1+q+...+q^{n-1})}.
$$

\subsection{The structure of $U(\g)$ as a $Z(\g)$-module} 

Recall that the universal enveloping algebra $U(\g)$ of any Lie algebra $\g$ has the standard filtration defined on generators by $\deg(\g)=1$, which is called the {\bf Poincar\'e-Birkhoff-Witt filtration}. 

Let $\g$ be a semisimple complex Lie algebra of rank $r$, and $W$ be the Weyl group of $\g$ with degrees $d_i,i=1,...,r$. 

\begin{theorem}\label{Kos1} (Kostant) (i) The center $Z(\g)=U(\g)^\g$ of $U(\g)$ is a polynomial algebra in $r$ generators $C_i$ of Poincar\'e-Birkhoff-Witt filtration degrees $d_i$.

(ii) $U(\g)$ is a free module over 
$Z(\g)$, and for every irreducible finite-dimensional representation 
$V$ of $\g$, the space $\Hom_\g(V,U(\g))$ is a free $Z(\g)$-module of rank $\dim V[0]$. 
\end{theorem} 

\begin{proof} By the Poincar\'e-Birkhoff-Witt theorem, for any Lie algebra $\g$ we have ${\rm gr}(U(\g))=S\g$. Moreover, we have the symmetrization map $S\g\to U(\g)$ given by 
$$ a_1\otimes...\otimes a_n\mapsto \frac{1}{n!}\sum_{s\in S_n}a_{s(1)}...a_{s(n)},
$$ 
$a_i\in \g$, which is an isomorphism of $\g$-modules.
Using this map, any homogeneous element of $(S\g)^\g$ can be lifted into $U(\g)^\g$. It follows that ${\rm gr}(U(\g)^\g)=(S\g)^\g$.
Thus Theorem \ref{Kos} implies all the statements of the theorem. 
\end{proof} 

\begin{example} Suppose $\g$ is simple. Then $d_1=2$ and $C_1$ is the quadratic 
Casimir of $\g$. 
\end{example} 

\begin{exercise} Consider the Lie algebra $\g=\mathfrak{sl}_n(\Bbb C)$ 
spanned by elementary matrices $E_{ij}$ with $\sum_{i=1}^nE_{ii}=0$. 

(i) Show that the center $Z(\g)$ is freely generated by the elements 
$$
C_{k-1}:=\sum_{i_1,...,i_k=1}^n \prod_{j=1}^k E_{i_j,i_{j+1}},\ k=2,...,n.
$$
where $j$ is viewed as an element of $\Bbb Z/k$. 

{\bf Hint:} It is slightly more convenient (and equivalent) to consider $\g=\mathfrak{gl}_n(\Bbb C)$, in which case one also has the generator $C_0$. 
Identify $\g$ with $\g^*$ using the trace pairing on $\g$. Let 
$T_k: \g^{\otimes k}\to \Bbb C$ be the $\g$-module map defined by 
$T_k(a_1\otimes...\otimes a_k):={\rm Tr}(a_k...a_1)$. Let $T_k^*: \Bbb C\to \g^{\otimes k}$ be the dual map. Show that 
$$
T_k^*(1)=\sum_{i_1,...,i_k=1}^n E_{i_1i_2}\otimes E_{i_2i_3}\otimes...\otimes E_{i_ki_1}.
$$
Use that this element is $\g$-invariant to show that the element $C_{k-1}$ is central. 

(ii) Generalize these statements to $\mathfrak{so}_{2n+1}(\Bbb C)$ and $\mathfrak{sp}_{2n}(\Bbb C)$. What happens for $\mathfrak{so}_{2n}$?
\end{exercise} 

\section{\bf Harish-Chandra isomorphism, maximal quotients} 

\subsection{The Harish-Chandra isomorphism} Let $\g$ be a complex semisimple Lie algebra. Fix a triangular decomposition $\g=\n_-\oplus \h\oplus \n_+$. By the PBW theorem, we then have 
a linear isomorphism 
$$
\mu: U(\n_-)\otimes U(\h)\otimes U(\n_+)\to U(\g)
$$ 
given by multiplication. We also have the linear map 
$$
\beta: U(\n_-)\otimes U(\h)\otimes U(\n_+)\to U(\h)
$$ 
given by 
$$
a_-\otimes h\otimes a_+\mapsto \varepsilon(a_-)\varepsilon(a_+)h,\ a_\pm \in U(\n_\pm), h\in U(\h),
$$
where $\varepsilon: U(\n_\pm)\to \Bbb C$ is the augmentation homomorphism (the counit).
Thus we get a linear map 
$$
HC:=\beta\circ \mu^{-1}: U(\g)\to U(\h)=S\h=\Bbb C[\h^*] 
$$
called the {\bf Harish-Chandra map}.

\begin{theorem} (Harish-Chandra) 
(i) If $b\in U(\g)$ and $c\in Z(\g)$ then $HC(bc)=HC(b)HC(c)$. 
In particular, the restriction of $HC$ to $Z(\g)$ is an algebra homomorphism. 

(ii) Define the shifted action of $W$ on $\h^*$ by $w\bullet x:=w(x+\rho)-\rho$ 
where $\rho$ is the half sum of positive roots (or, equivalently, sum of fundamental weights). 
Then $HC$ maps $Z(\g)$ into the space of invariants $\Bbb C[\h^*]^{W\bullet }$. That is,  
for any $b\in Z(\g)$ we have $HC(b)(\lambda)=f_b(\lambda+\rho)$ for some $f_b\in \Bbb C[\h^*]^W$.

(iii) If $V$ is a highest weight representation of $\g$ with highest weight $\lambda$ 
then 
$$
f_b(\lambda+\rho)=(v_\lambda^*,bv_\lambda)
$$
where $v_\lambda$ is a highest weight vector 
of $V$ and $v_\lambda^*$ the lowest weight vector of $V^*$ 
such that $(v_\lambda^*,v_\lambda)=1$. Thus if $b\in  Z(\g)$ 
then $HC(b)(\lambda)$ is the scalar by which $b$ acts on a highest weight module with highest weight $\lambda$.  

(iv) The map $HC: Z(\g)\to \Bbb C[\h^*]^{W\bullet }$ is a filtered algebra homomorphism and ${\rm gr}(HC)={\rm Res}$, the Chevalley restriction homomorphism $(S\g)^\g\to (S\h)^W$. 

(v) $HC$ is an algebra isomorphism. 
\end{theorem} 

The isomorphism $HC: Z(\g)\to \Bbb C[\h^*]^{W\bullet }$ is called the {\bf Harish-Chandra isomorphism}. 

\begin{proof} Let $b=a_-ha_+\in U(\g)$. We have 
$$
(v_\lambda^*,bv_\lambda)=(v_\lambda^*,a_-ha_+v_\lambda)=\varepsilon(a_-)\varepsilon(a_+)\lambda(h)=
HC(b)(\lambda).
$$
Thus 
$$
HC(bc)(\lambda)=(v_\lambda^*,bcv_\lambda)=(v_\lambda^*,bv_\lambda)(v_\lambda^*,cv_\lambda)=HC(b)(\lambda)HC(c)(\lambda)
$$
since $c$ is central; namely, the last factor is just the eigenvalue of $c$ on $V$. This proves (i). 

To establish (ii),(iii), it remains to show that for $b\in Z(\g)$, 
$HC(b)$ is invariant under the shifted action of all $w\in W$. To this end, it suffices to show this 
for $w=s_i$, a simple reflection. For this purpose, consider the Verma module 
$M_{\lambda}$ with $(\lambda+\rho,\alpha_i^\vee)=n\in \Bbb Z_{>0}$. Then 
$f_i^nv_\lambda$ generates a copy of $M_{\lambda-n\alpha_i}=M_{s_i\bullet \lambda}$ inside $M_\lambda$. 
Thus we get $HC(b)(\lambda)=HC(b)(s_i\bullet \lambda)$. Since this holds on a Zariski dense set, 
it holds identically, which yields (ii),(iii). 

(iv) follows immediately from (iii). 

Finally, (v) follows from (iv) and the Chevalley restriction theorem, since any filtered map whose associated graded is an isomorphism is itself an isomorphism. 
\end{proof} 

\begin{remark} Kostant theorems and the Harish-Chandra isomorphism extend trivially to reductive Lie algebras.  
\end{remark}

\subsection{Maximal quotients}
Let $\g$ be a semisimple Lie algebra and $M$ a $\g$-module on which 
the center $Z(\g)$ acts by a character 
$$
\chi: Z(\g)\to \Bbb C
$$
(for example, $M$ is irreducible). In view of the Harish-Chandra isomorphism theorem, 
we have $\chi=\chi_\lambda$, where 
$$
\chi_\lambda(z)=HC(z)(\lambda-\rho)
$$ 
for a unique $\lambda\in \h^*$ modulo the action of $W$. As mentioned in Subsection \ref{dixle}, the element $\chi_\lambda$ is called the {\bf infinitesimal character} or {\bf central character} of $M$. 

If $M$ is a $\g$-bimodule then it carries two actions of $Z(\g)$, by left 
and by right multiplication. If these actions are by characters, then they are called the {\bf left and right infinitesimal characters} of $M$. The infinitesimal character of $M$ is then the pair $(\theta,\chi)$ where $\theta$ is the left infinitesimal character and $\chi$ the right infinitesimal character of $M$. 
 
For a character $\chi: Z(\g)\to \Bbb C$ let 
$$
U_\chi=U_\chi(\g):=U(\g)/(z-\chi(z),z\in Z(\g)).
$$ 
This algebra is called 
the {\bf maximal quotient} of $U(\g)$ with infinitesimal character $\chi$, as every $U(\g)$-module with such infinitesimal character factors through $U_\chi$.  Note that $U_\chi$ is a $\g$-bimodule with infinitesimal character $(\chi,\chi)$ (as it is a $U_\chi$-bimodule). 

Theorem \ref{Kos1} immediately implies 

\begin{corollary}\label{Ul} For any finite-dimensional irreducible $\g$-module $V$ we have $\dim\Hom_\g(V,U_\chi)=\dim V[0]$, where $\g$ acts on $U_\chi$ by the adjoint action. Thus $U_\chi$ is a Harish-Chandra $\g$-bimodule.  
\end{corollary} 

\begin{corollary}\label{tenspro1} If $V$ 
is a finite-dimensional $\g$-bimodule then 
$V\otimes U_\chi$ is a Harish-Chandra $\g$-bimodule. 
\end{corollary} 

\begin{proof} This follows from Corollary \ref{Ul} and Exercise \ref{tenspro}. 
\end{proof} 

\begin{corollary}\label{quot} (i) Every irreducible 
$\g$-bimodule $M$ locally finite under the adjoint $\g$-action is a quotient 
of $V\otimes U_\chi$ for some finite-dimensional irreducible $\g$-module $V$ with trivial right action of $\g$, where $\chi$ is the right infinitesimal character of $M$.

(ii) Every irreducible $\g$-bimodule locally finite under the adjoint $\g$-action is a Harish-Chandra bimodule. 
\end{corollary} 

\begin{proof}  (ii) follows from (i) and Corollary \ref{tenspro1}, so it suffices to prove (i). 
By Dixmier's lemma (Lemma \ref{Dixlemm}), $M$ has some infinitesimal character $(\theta,\chi)$. Let $V\subset M$ be an irreducible finite-dimensional subrepresentation under $\g_{\rm ad}$. 
Let us view $V^*$ as a $\g$-bimodule 
with action 
$$
(af)(x)=-f(ax),\ fb=0
$$ 
for $a,b\in \g, x\in V, f\in V^*$, and consider the tensor product 
$V^*\otimes M$, which is a $\g$-bimodule
with action 
$$
a\circ (f\otimes m):=af\otimes m+f\otimes am,\ 
(f\otimes m)\circ b:=f\otimes mb.
$$ 
The canonical element 
$u\in V^*\otimes V\subset V^*\otimes M$ 
is $\g_{\rm ad}$-invariant (i.e., commutes with $\g$). 
Thus we have a bimodule homomorphism 
$\psi: U(\g)\to V^*\otimes M$ given by $\psi(c):=uc=\sum v_i^*\otimes v_ic$, where $v_i$ is a basis of $V$ and $v_i^*$ the dual basis of $V^*$. Moreover, since the right infinitesimal character of $M$ is $\chi$, this homomorphism descends to $\overline{\psi}: U_\chi\to V^*\otimes M$. This gives rise to a nonzero homomorphism of bimodules $\xi: 
V\otimes U_\chi\to M$, where the right $\g$-module structure of $V$ is trivial. 
Since $M$ is irreducible, $\xi$ is surjective. Thus the result follows from Corollary \ref{tenspro1}.  
\end{proof} 

\section{\bf Category $\mathcal O$ of $\g$-modules - I} 

\subsection{Category $\mathcal O$} 

Let $\g$ be a semisimple complex Lie algebra. 

\begin{definition} The category $\mathcal O=\mathcal O_\g$ is the full subcategory of 
$\g$-mod, which consists of finitely generated $\g$-modules $M$ with weight decomposition and $P(M)\subset \cup_{i=1}^m (\lambda_i-Q_+)$, where 
$\lambda_1,...,\lambda_m\in \h^*$. 
\end{definition} 

It is clear that $\mathcal O$ is closed under taking subquotients and direct sums, so it 
is an abelian category (recall that a submodule of a finitely generated $\g$-module 
is finitely generated since $U(\g)$ is Noetherian). 

Also it is easy to see that any nonzero object $M\in \mathcal O$ has a singular vector (namely, take any nonzero vector of a maximal weight in $P(M)$). Thus the simple objects (=modules) of $\mathcal O$ are $L_\lambda$, $\lambda\in \h^*$. 

\begin{example} All highest weight $\g$-modules, in particular
a Verma module $M_\lambda$ and its simple quotient $L_\lambda$ belong to $\mathcal O$. Another example is $\overline M_{-\lambda}^*$, the restricted dual to the lowest weight Verma module $\overline M_{-\lambda}$, introduced in Exercise \ref{intertw}(ii). 
This module is called the {\bf contragredient Verma module} and denoted $M_\lambda^\vee$. 
\end{example}

\begin{lemma} \label{findim} If $M\in \mathcal O$ then the weight subspaces of $M$ are finite-dimensional.  
\end{lemma} 

\begin{proof} Let $v_1,..,v_m$ be generators of $M$ which are eigenvectors 
of $\h$ (they exist since $M$ is finitely generated and has weight decomposition). 
Let $E:=\sum_{i=1}^m U(\h\oplus \n_+)v_i=\sum_{i=1}^m U(\n_+)v_i$. Then $E$ is finite-dimensional by the condition on the weights of $M$. On the other hand, the natural map 
$U(\n_-)\otimes E\to M$ is surjective. The lemma follows, as weight subspaces of $U(\n_-)\otimes E$ are finite-dimensional. 
\end{proof} 

Let $\mathcal R$ be the ring of series $F:=\sum_{\mu\in \h^*}c_\mu e^\mu$, where $c_\mu\in \Bbb Z$ and the set $P(F)$ of $\mu$ with $c_\mu\ne 0$ is contained in a finite union of sets of the form $\lambda-Q_+$, $\lambda\in \h^*$. If $M$ is an $\h$-semisimple $\g$-module with finite-dimensional 
weight spaces and weights in a finite union of sets $\lambda-Q_+$ then we can define 
the {\bf character} of $M$, 
$$
{\rm ch}(M)=\sum_{\lambda\in \h^*}\dim M[\lambda]e^\lambda\in \mathcal R.
$$
For example, 
$$
{\rm ch}(M_\lambda)=\frac{e^\lambda}{\prod_{\alpha\in R_+}(1-e^{-\alpha})}.
$$
We have ${\rm ch}(M\otimes N)={\rm ch}(M){\rm ch}(N)$ and 
$$
{\rm ch}(M)={\rm ch}(L)+{\rm ch}(N)
$$
when $0\to L\to M\to N\to 0$ is a short exact sequence. Lemma \ref{findim} implies that we can define such characters ${\rm ch}(M)$ for $M\in \O$.

\begin{corollary}\label{finquo} The action of $Z(\g)$ on every $M\in \mathcal{O}$ factors through a finite-dimensional quotient. 
\end{corollary} 

\begin{proof} Since $Z(\g)$ is finitely generated, it suffices to show that every 
$z\in Z(\g)$ satisfies a polynomial equation $F(z)=0$ in $M$. Let $\mu_1,...,\mu_k$
be weights such that $M$ is generated by $E:=M[\mu_1]\oplus...\oplus M[\mu_k]$.  
By Lemma \ref{findim}, this space is finite-dimensional, and it is preserved
by $z$. Let $F$ be the minimal polynomial of $z$ on $E$. 
Then $F(z)=0$ on $E$, hence on the whole $M$ (as $z$ is central and $E$ generates $M$). 
\end{proof} 

\begin{exercise}\label{finquo1} Show that the action of $Z(\g)$ on any Harish-Chandra $(\g,K)$-module factors through a finite-dimensional quotient. (Mimic the proof of Corollary \ref{finquo}).
\end{exercise}

\begin{exercise}\label{ext1van} (i) Show that for any $\mu\in \h^*$, ${\rm Ext}^1_{\mathcal O}(M_\mu,M_\mu)=0$. 

(ii) Show that $\Ext^1(M_\mu,M_\mu)$ (Ext in the category of all $\g$-modules) 
is nonzero. 
\end{exercise} 

\begin{corollary}\label{decom} (i) Any $M\in \mathcal O$ 
has a canonical decomposition 
$$
M=\oplus_{\chi\in \h^*/W}M(\chi), 
$$
where $M(\chi)$ is the generalized eigenspace of $Z(\g)$ in $M$ with eigenvalue $\chi$, and this direct sum is finite. In other words, 
$$
\mathcal O=\oplus_{\chi\in \h^*/W}\mathcal O_\chi,
$$
where $\mathcal O_\chi$ is the subcategory of $\mathcal O$ 
of modules where every $z\in Z(\g)$ acts with generalized eigenvalue  
$\chi(z)$. 

(ii) Each $M\in \mathcal \O_\chi$ has a finite filtration 
with successive quotients having infinitesimal character $\chi$.  
\end{corollary} 

\begin{proof} (i) Let $R:=Z(\g)/{\rm Ann}(M)$ be the quotient of $Z(\g)$ by its annihilator  
in $M$. This algebra is finite-dimensional, so has the form $R=\prod_{i=1}^m R_i$, where 
$R_i$ are local with units $\bold e_i$, corresponding to the generalized eigenvalues $\chi_1,...,\chi_m\in \h^*/W$ of $Z(\g)$ on $M$. 
So $M=\oplus_{i=1}^m M(\chi_i)$, where $M(\chi_i):=\bold e_iM$.  

(ii) If $M\in \mathcal O_\chi$ then the algebra $R$ is local. 
Let $\mathfrak m$ be its unique maximal ideal. Then the required finite filtration on $M$ is 
$$
M\supset \mathfrak m M\supset \mathfrak m^2M...
$$
\end{proof} 

Thus the simple objects of $\O_\chi$ are $L_{\mu-\rho}$, where $\chi=\chi_\mu$, 
i.e., $\mu\in \chi$. 

We can partition the $W$-orbit $\chi$ into equivalence classes according to the relation $\mu\sim \nu$ if $\mu-\nu\in Q$. It is clear that this partition defines a decomposition 
$\O_\chi=\oplus_{S}\O_\chi(S)$, where 
$S$ runs over the equivalence classes in $\chi$ under the relation $\sim$. 
Namely, $\O_\chi(S)$ is the subcategory of modules with all weights in $\mu-\rho+Q$, where $\mu\in S$. 

\begin{example} Suppose that $\lambda\in \h^*$ is such that 
$w\lambda-\lambda\notin Q$ for any $1\ne w\in W$. 
In this case the equivalence relation on $W\lambda$ is trivial, so for any $\mu\in W\lambda$ the category 
$\O_{\chi_\lambda}(\mu)$ has a unique simple object $M_{\mu-\rho}$. It thus follows from Exercise \ref{ext1van} that for any $\mu\in W\lambda$, the category $\mathcal O_{\chi_\lambda}(\mu)$ is equivalent to the category of finite-dimensional vector spaces (as $M_{\mu-\rho}$ has no nontrivial self-extensions), and the category $\mathcal O_{\chi_\lambda}$ is semisimple with $|W|$ simple objects. 
\end{example}  

\begin{lemma} Every object $M$ of $\mathcal O$ has finite length. 
\end{lemma} 

\begin{proof} By Corollary \ref{decom} we may assume that $M$ has infinitesimal character $\chi_\lambda$. We may also assume that $P(M)\subset \mu+Q$ for some $\mu\in \Bbb \h^*$. Recall that the quadratic Casimir $C$ of $\g$ acts on $M$ in the same way as in $M_{\lambda-\rho}$, i.e., by the scalar $\lambda^2-\rho^2$. Suppose that $v$ is a singular vector in a nonzero subquotient $M'$ of $M$ of some weight $\gamma\in \mu+Q$ (it must exist since weights of $M'$ belong to a finite union of $\lambda_i-Q_+$). Then $Cv=((\gamma+\rho)^2-\rho^2)v$, so we must have 
$$
(\gamma+\rho)^2=\lambda^2.
$$
Since the inner product on $Q$ is positive definite, this equation has a finite set $S$ of solutions $\gamma\in \mu+Q$. 

For a semisimple $\h$-module $Y$ set 
$Y[S]:=\oplus_{\gamma\in S}Y[\gamma]$. It follows that $M'[S]\ne 0$. 
Also by Lemma \ref{findim} we have $\dim M[S]<\infty$. Thus 
length$(M)\le \dim M[S]$ is finite, as claimed. 
\end{proof} 

\subsection{Partial orders of $\h^*$} 

Introduce a partial order on $\h^*$: we say that $\mu\le\lambda$ if $\lambda-\mu \in Q_+$ and $\mu<\lambda$ if $\mu\le \lambda$ but $\mu\ne \lambda$. We write 
$\lambda\ge \mu$ if $\mu\le \lambda$ and $\lambda>\mu$ if $\mu<\lambda$. 

If $\mu=s_\alpha \lambda$ 
for some $\alpha\in R_+$ and $\mu<\lambda$ (i.e., 
$(\lambda,\alpha^\vee)\in \Bbb Z_{\ge 1}$ and $\mu=\lambda-(\lambda,\alpha^\vee)\alpha$), then we write $\mu<_{\alpha}\lambda$. We write $\mu\preceq \lambda$ 
if there exist sequences $\alpha^1,...,\alpha^m\in R_+$ and 
$\mu=\mu_0,\mu_1,...,\mu_m=\lambda$ such that 
for all $i$, $\mu_{i-1}<_{\alpha^i}\mu_i$, and write $\mu\prec\lambda$ if 
$\mu\preceq\lambda$ but $\mu\ne\lambda$ (i.e., $m\ne 0$). 
We write $\lambda\succeq \mu$ if $\mu\preceq \lambda$ and $\lambda\succ\mu$ if $\mu\prec\lambda$. 

\begin{remark}\label{falseingen} 
It is easy to see that if $\mu\prec\lambda$ then $\mu<\lambda$ and $\mu\in W\lambda$, but the converse is false, in general. For example, consider the root system of type $A_3$, 
and let us realize $\h^*$ as $\Bbb C^4/\Bbb C_{\rm diagonal}$. Let 
$\mu=(0,3,1,2)$, $\lambda=(1,2,3,0)$. Then $\mu\in W\lambda$ and $\mu<\lambda$, since $\lambda-\mu=(1,-1,2,-2)=\alpha_1+2\alpha_3$. However, 
$\mu\nprec\lambda$. Indeed, otherwise there would exist 
$\alpha\in R_+$ such that $\mu\le s_\alpha\lambda<\lambda$, and it is easy 
to check that there is no such $\alpha$. 
\end{remark} 

\subsection{Verma's theorem} 

\begin{theorem}\label{vermath} (D. N. Verma) Let $\lambda,\mu\in \h^*$ and $\mu\preceq\lambda$. 
Then $\dim \Hom(M_{\mu-\rho},M_{\lambda-\rho})=1$ and $M_{\mu-\rho}$ can be uniquely realized as a submodule of $M_{\lambda-\rho}$. In particular, $L_{\mu-\rho}$ occurs in the composition series of $M_{\lambda-\rho}$. 
\end{theorem} 

\begin{proof} By Exercise \ref{dimhom}, $\dim \Hom(M_{\mu-\rho},M_{\lambda-\rho})\le 1$ and any nonzero homomorphism $M_{\mu-\rho}\to M_{\lambda-\rho}$ is injective, so it suffices to show that 
$\dim \Hom(M_{\mu-\rho},M_{\lambda-\rho})\ge 1$. By definition of the partial order $\preceq$, it suffices to do so when $\mu<_\alpha \lambda$ for some $\alpha\in R_+$, i.e., when $\mu=s_\alpha\lambda=\lambda-n\alpha$ where $n:=(\lambda,\alpha^\vee)\in  \Bbb Z_+$. For generic $\lambda$ with $(\lambda,\alpha^\vee)=n\in \Bbb Z_+$, this follows from the Shapovalov determinant formula (Exercise \ref{Shapova}), and the general case follows by taking the limit. 
\end{proof} 

We will see below that the converse to Verma's theorem also holds: if $L_{\mu-\rho}$ occurs in the composition series of $M_{\lambda-\rho}$
 then $\mu\preceq\lambda$. This was proved by J. Bernstein, I. Gelfand and S. Gelfand, see Theorem \ref{BGGth} below. 

\subsection{The stabilizer in $W$ of a point in $\h^*/Q$}

Let $x\in \h^*/Q$ and $W_x\subset W$ be the stabilizer of $x$. 

\begin{proposition}\label{stab} $W_x$ is generated by the reflections $s_\alpha\in W_x$. Moreover, the roots $\alpha$ such that $s_\alpha\in W_x$ form a root system $R_x\subset R$, and $W_x$ is the Weyl group of $R_x$. The corresponding dual root system $R_x^\vee$ 
is a root subsystem of $R^\vee$, i.e., $R_x^\vee={\rm span}_{\Bbb Z}(R_x^\vee)\cap R^\vee$. 
\end{proposition} 

\begin{proof} Let $T:=\h^*/Q$. 
The ring $\Bbb C[T/W]:=\Bbb C[T]^W$ is freely generated by 
the orbit sums $m_i=\sum_{\beta\in W\omega_i^\vee}e^\beta$, where $\omega_i^\vee$ are the fundamental coweights. Hence $T/W$ is smooth (in fact, an affine space). It follows by the Chevalley-Shephard-Todd theorem that for each $x\in T$ the stabilizer $W_x$ is generated by a subset of reflections of $W$. Moreover, if $s_\alpha,s_\beta\in W_x$ then $s_\alpha s_\beta s_\alpha=s_{s_\alpha(\beta)}\in W_x$, which implies that the set $R_x$ of $\alpha$ such that $s_\alpha\in W_x$
is a root system in $R$, and $W_x$ is its Weyl group. 
Moreover, picking a preimage $\widetilde x$ of $x$ in $\h^*$, 
we see that $\alpha\in R_x$ if and only if $(\alpha^\vee,\widetilde x)\in \Bbb Z$. Thus $R_x^\vee$ is a root subsystem of $R^\vee$. 
\end{proof} 

\begin{remark} 1. Note that unlike the case $x\in \h^*$, for $x\in \h^*/Q$ the group $W_x$ is not necessarily a {\bf parabolic} subgroup of $W$, i.e., it is not necessarily conjugate to a subgroup generated by simple reflections. In fact, the Dynkin diagram of $R_x$ or $R_x^\vee$ may not be a subdiagram of the Dynkin diagram of $W$. Such subgroups are called {\bf quasiparabolic subgroups}. 

For example, if $R$ is of type $B_2$ with simple roots $\alpha_1=(1,0)$ and $\alpha_2=(-1,1)$
then for $x=(\frac{1}{2},0)$, $R_x$ is the root system of type $A_1\times A_1$ consisting of $\pm \alpha_1$ and $\pm (\alpha_1+\alpha_2)$. 
The same example shows that $R_x$ is not necessarily a root subsystem of $R$, as $\alpha_1+(\alpha_1+\alpha_2)\notin R_x$. 

2. If $G^\vee$ is the simply connected complex semisimple Lie group corresponding to $R^\vee$ then $T$ is the maximal torus of $G^\vee$, and it is easy to see that 
$R_x^\vee$ is the root system of the centralizer $\mathfrak{z}_x$ of $x$ in $\g^\vee:={\rm Lie}(G^\vee)$.
\end{remark} 

\newpage

\section{\bf Category $\mathcal O$ of $\g$-modules - II} 

\subsection{Dominant weights} \label{domwei}

Let us say that a weight $\lambda\in \h^*$ is {\bf dominant} for the partial order 
$\le$ (respectively, $\preceq$) if it is maximal with respect to this order 
in its equivalence class (or, equivalently, in its $W$-orbit). 

\begin{corollary}\label{domchar} The following conditions on a weight $\lambda\in \h^*$ are equivalent: 

(i) $\lambda$ is dominant for $\le$;

(ii) $\lambda$ is dominant for $\preceq$; 

(iii) For every root $\alpha\in R_+$, $(\lambda,\alpha^\vee)\notin \Bbb Z_{<0}$. 

(iv) For every $w\in W_{\lambda+Q}$, $w\lambda\preceq\lambda$.

(v) For every $w\in W_{\lambda+Q}$, $w\lambda\le \lambda$.
\end{corollary} 

\begin{proof} It is clear that (iv) implies (v) implies (i) implies (ii). 
It is also easy to see that (ii) implies (iii), 
since if $(\lambda,\alpha^\vee)\in \Bbb Z_{<0}$ then $s_\alpha\lambda\sim \lambda$ and $s_\alpha\lambda>\lambda$ so $\lambda$ is not maximal under $\preceq$ in its equivalence class. It remains to show that (iii) implies (iv). By Proposition \ref{stab}, $W_{\lambda+Q}$ is the Weyl group of some root system $R'\subset R$, and the equivalence class $S$ of $\lambda$ is simply the orbit $W_{\lambda+Q}\lambda$. 
By our assumption, for $\alpha\in R_+'$ we have $(\lambda,\alpha^\vee)\in \Bbb Z\setminus \Bbb Z_{<0}=\Bbb Z_{\ge 0}$. Thus, $\lambda=\lambda'+\nu$ 
where $\lambda'$ is a dominant integral weight for $R'$ (meaning that $(\lambda,\alpha^\vee)\in \Bbb Z_{\ge 0}$ for $\alpha\in R_+'$) and $(\nu,\alpha^\vee)=0$ for all $\alpha\in R_+'$. Now for any $w\in W_{\lambda+Q}$, fix a reduced 
decomposition $w=s_{i_m}...s_{i_1}$, where $s_i=s_{\beta_i}$ and 
$\beta_i$ are the simple roots of $R'$. Let $\lambda_k:=s_{i_k}...s_{i_1}\lambda$, so 
$\lambda_0=\lambda$ and $\lambda_m=w\lambda$. Setting $\lambda_k':=s_{i_k}...s_{i_1}\lambda'=\lambda_k-\nu$, we then have 
$$
\lambda_{k-1}-\lambda_k=\lambda_{k-1}'-\lambda_k'=(\lambda_{k-1}',\beta_{i_k}^\vee)\beta_{i_k}=(\lambda',s_{i_1}...s_{i_{k-1}}\beta_{i_k}^\vee)\beta_{i_k}.
$$
The coroot $s_{i_1}...s_{i_{k-1}}\beta_{i_k}^\vee$ is positive, 
so we get that $\lambda_{k}\preceq\lambda_{k-1}$, which yields (iv). 
\end{proof} 

Corollary \ref{domchar} shows that every equivalence class of weights contains a unique maximal element with respect to each of the orders $\preceq$ and $\le$, namely the unique dominant weight in this class. The same is true for minimal elements by changing signs.  

\subsection{Projective objects} Let $\C$ be an abelian category over a field $k$. 
Recall that $\C$ is said to be {\bf Noetherian} if any ascending chain of subobjects of any object $X\in \C$ stabilizes. This holds, for instance, when objects of $\C$ have finite length. 

Recall also that an object $P\in \C$ is {\bf projective} if the functor 
$\Hom(P,-)$ is (right) exact, and that $\C$ is said to have {\bf enough projectives} if every object $L\in \C$ is a quotient of a projective object $P$. 
Note that if objects of $\C$ have finite length then it is sufficient for this property to hold for every simple $L$: it can then be extended to all $L$ by induction in length. Indeed, suppose we have a short exact sequence 
$$
0\to L_1\to L\to L_2\to 0
$$
with $L_1,L_2\ne 0$ and projectives $P_1,P_2$ with epimorphisms $p_j: P_j\twoheadrightarrow L_j$. Then the map $p_2$ lifts to $\widetilde p_2: P_2\to L$, which yields an epimorphism 
$p_1+\widetilde p_2: P_1\oplus P_2\twoheadrightarrow L$. 

Suppose that Hom spaces in $\mathcal C$ are finite-dimensional. Then by the { \bf Krull-Schmidt theorem}, every object of $\mathcal C$ has a unique representation as a finite direct sum of indecomposable ones (up to isomorphism and permutation of summands). 

\begin{proposition}\label{genera} Let $\C$ be a Noetherian abelian category with enough projectives and finite-dimensional Hom spaces over an algebraically closed field $k$. Then 

(i) Let $I$ be the set labeling the isomorphism classes of indecomposable projectives $P_i$ of $\C$. Then 
the isomorphism classes of simple objects $L_i$ of $\C$ are labeled by the same set $I$, and 
$\dim\Hom(P_i,L_j)=\delta_{ij}$, $i,j\in I$. 

(ii) For $M\in \C$ of finite length, the multiplicities $[M:L_i]$ 
equal $\dim \Hom(P_i,M)$. 
\end{proposition} 
 
\begin{proof} Let $P\in \C$ be an indecomposable projective. Then $\End(P)$ has no idempotents other than $0,1$, so 
$\End(P)=k\oplus N$ where $N$ is the nilradical, i.e., it is a local algebra. 

Suppose $Q\subset P$ is a maximal proper subobject (it exists by Zorn's lemma since $\C$ is Noetherian). Let $Q'\subset P$ be a subobject not contained in $Q$. Then $Q+Q'=P$. So we have an epimorphism $Q\oplus Q'\to P$, which, by the projectivity of $P$, gives a surjection 
$$
\Hom(P,Q)\oplus \Hom(P,Q')\to \End(P).
$$ 
So we have $1_P=a+a'$, where $a,a': P\to P$ factor through $Q,Q'$. Thus $a$ is not an isomorphism (since $Q$ is proper). As $\End(P)$ is local, it follows that $a'$ is an isomorphism, so $Q'=P$. 

It follows that $P$ has a unique maximal proper subobject $J(P)$, and 
$L_P:=P/J(P)$ is simple. Moreover, if $L:=P/Q$ is simple then $Q=J(P)$, so $L=L_P$. 
So if $I'$ labels the isomorphism classes of simples in $\C$, then we get 
a map $\ell: I\to I'$ such that $\ell(P)=L_P$, and we have $\dim \Hom(P_i,L_{i'})=\delta_{\ell(i),i'}$. Moreover, $\ell$ is surjective since 
every simple $L$ is a quotient of some projective $P$ which may be chosen indecomposable (if $P\twoheadrightarrow L$ and $P=\oplus_{i=1}^N P_i$ where $P_i$ are indecomposable then there exists $i$ such that the map $P_i\to L$ is nonzero, hence an epimorphism as $L$ is simple). 

It remains to show that $\ell$ is injective, i.e., if $L_m\cong L_n$ then $P_m\cong P_n$. To this end, note that the epimorphisms $a_0: P_m\twoheadrightarrow L_n$, $b_0: P_n\twoheadrightarrow L_m$ lift to 
morphisms $a: P_m\to P_n$, $b: P_n\to P_m$, such that $ab\in \End(P_n)$ and $ba\in \End(P_m)$ are not nilpotent (as they define isomorphisms on the corresponding simple quotients). Since the algebras $\End(P_n),\End(P_m)$ are local, it follows that $ab$ and $ba$ are isomorphisms, as claimed. 

This proves (i). Part (ii) now follows 
from the exactness of the functor $\Hom(P_i,?)$. 
\end{proof} 

The object $P_i$ is called the {\bf projective cover} of $L_i$, and $L_i$ is called the 
{\bf head} of $P_i$; by Proposition \ref{genera}, it is the unique simple quotient of $P_i$. 

\begin{remark} In general, objects of a category satisfying the assumptions of 
Proposition \ref{genera} need not have finite length. An example 
when they can have infinite length is the category of finitely generated 
$\Bbb Z$-graded $\Bbb C[x]$-modules, where $\deg(x)=1$. 
The simple objects in this category are 1-dimensional modules 
$L_n$, $n\in \Bbb Z$, which sit in degree $n$ (with $x$ acting by zero). 
The projective cover of $L_n$ is $P_n=\Bbb C[x]_n$, the free rank $1$ module 
sitting in degrees $n,n+1,...$, which has infinite length.  
\end{remark} 

\subsection{Projective objects in $\mathcal O$} 

\begin{proposition}\label{proje} If $\lambda$ is dominant then $M_{\lambda-\rho}$ is a projective object in $\mathcal O$.\footnote{Note that this does not mean that $M_\lambda$ is a projective $U(\g)$-module; in fact, it is not.} 
\end{proposition} 

\begin{proof} Our job is to show that the functor $\Hom(M_{\lambda-\rho},\bullet)$ is exact on $\O$. It suffices to show this on $\O_{\chi_{\lambda}}(S)$, where $S$ is the equivalence class of $\lambda$.   
To this end, note that all weights of any $X\in \O_{\chi_{\lambda}}(S)$ are not $> \lambda-\rho$. Thus every $v\in X[\lambda-\rho]$ is singular, so there is a unique homomorphism $M_{\lambda-\rho}\to X$ sending $v_{\lambda-\rho}$ to $v$. It follows that $\Hom(M_{\lambda-\rho},X)\cong X[\lambda-\rho]$, which implies the statement. 
\end{proof} 

Now let $V$ be a finite-dimensional $\g$-module. Then we have an exact functor $V\otimes: \mathcal O\to \mathcal O$.

\begin{corollary} \label{proje1} (i) If $P\in \mathcal O$ is projective then so is 
$V\otimes P$. 

(ii) If $\lambda\in \h^*$ is dominant then 
the object $V\otimes M_{\lambda-\rho}\in \mathcal{O}$ is projective. 
\end{corollary} 

\begin{proof} (i) For $X\in \O$
$$
\Hom_\g(V\otimes P,X)=\Hom_\g(P,V^*\otimes X),
$$
which is exact since $P$ is projective. 

(ii) follows from (i) and Proposition \ref{proje}. 
\end{proof} 

\begin{corollary}\label{enoughpro} (i) For every $\mu\in \h^*$, there exists dominant $\lambda\in \h^*$  and a finite-dimensional $\g$-module $V$ such that 
$\Hom(V\otimes M_{\lambda-\rho},L_{\mu})\ne 0$. Thus $\mathcal O$ has enough projectives.  

(ii) Every projective object $P$ of $\mathcal O$ is a free $U(\n_-)$-module.
\end{corollary} 

\begin{proof} (i) We have 
$$
\Hom(V\otimes M_{\lambda-\rho},L_\mu)=\Hom_\g(M_{\lambda-\rho},V^*\otimes L_\mu).
$$
Now take $V=V^*=L_{N\rho}$ for large $N$ and $\lambda=\mu+(N+1)\rho$. It is clear that $\lambda$ is dominant, and $\Hom_\g(M_{\lambda-\rho},V^*\otimes L_\mu)=\Bbb C$, as claimed. 

(ii) This follows by Lemma \ref{homog} since every indecomposable projective object $P\in \mathcal O$ is an $\h^*$-graded direct summand in $V\otimes M_{\lambda-\rho}$, which is a free graded $U(\n_-)$-module. 
\end{proof} 

It follows that every simple object $L_\lambda$ of $\mathcal O$ has a projective cover $P_\lambda$, with $\dim \Hom(P_\lambda,L_\mu)=\delta_{\lambda\mu}$.  

\section{\bf The nilpotent cone of $\g$} 

\subsection{The nilpotent cone}

Let $(S\g)_0$ be the quotient of $S\g$ by the ideal generated by the positive degree part of $(S\g)^\g$, i.e. by the free homogeneous generators $p_1,...,p_r$ of $(S\g)^\g$ (which exist by Kostant's theorem). The scheme 
$$
\mathcal N:={\rm Spec}(S\g)_0\subset \g^*\cong \g
$$ 
is called the {\bf nilpotent cone} of $\g$. It follows from the Kostant theorem that $p_1,...,p_r$ is a regular sequence, i.e., this scheme is a complete intersection of codimension $r$ in $\g$ (see Remark \ref{remci}), i.e., of dimension 
$$
\dim \mathcal N=\dim \g-r=|R|=2|R_\pm|=2\dim \n_\pm,
$$ 
the number of roots of $\g$. 

Let $x\in \g$ be a nilpotent element. Recall that then $x$ is conjugate to an element $y\in \n_+$ and ${\rm Ad}(t^{2\rho^\vee})y\to 0$ as $t\to 0$, where $\rho^\vee$ is the half-sum of positive coroots of $\g$. Thus $p_i(x)=p_i(y)=0$ and 
hence $x\in \mathcal N(\Bbb C)$. On the other hand, if $x$ is not nilpotent then ${\rm ad}(x)$ is not a nilpotent operator, so $ {\rm Tr}({\rm ad}(x)^N)\ne 0$ for some $N$, hence $x\notin \mathcal{N}(\Bbb C)$. It follows that $\mathcal N(\Bbb C)$ is exactly the set of nilpotent elements of $\g$, hence the term ``nilpotent cone". 

For example, for $\g=\mathfrak{sl}_2$  
we have $r=1$ and 
$$
p_1(A)=-\det A=x^2+yz
$$ 
for $A:=\begin{pmatrix} x & y\\ z & -x\end{pmatrix}\in \g$, so $\mathcal N$ is the usual quadratic cone in $\Bbb C^3$ defined by the equation $x^2+yz=0$. 

\subsection{The principal $\mathfrak{sl}_2$ subalgebra}
 The {\bf principal $\mathfrak{sl}_2$ subalgebra} of $\g$ is the subalgebra spanned by $e:=\sum_{i=1}^r e_i$, $f:=\sum_i c_if_i$ and $h:=[e,f]=\sum_i c_ih_i=2\rho^\vee$. Thus $c_i$ are found from the equations $\sum_i c_ia_{ij}=2$ for all $j$, where $A=(a_{ij})$ is the Cartan matrix of $\g$.
 
\begin{lemma}\label{prisl2} 
The restriction of the adjoint representation of $\g$ to its principal $\mathfrak{sl}_2$-subalgebra is isomorphic to $L_{2m_1}\oplus...\oplus L_{2m_r}$ for appropriate $m_i\in \Bbb Z_{>0}$.
\end{lemma} 

\begin{proof} Consider the corresponding action of the group $SL_2(\Bbb C)$. The element $-1\in SL_2(\Bbb C)$ acts on $\g$ by $\exp(2\pi i\rho^\vee)=1$ since $\rho^\vee$ is an integral coweight. Thus only even highest weight $\mathfrak{sl}_2$-modules may occur in the decomposition of $\g$. 
Since $\rho^\vee$ is regular, the $0$-weight space of this module (the centralizer $Z_\g(\rho^\vee)$) is $\h$, i.e., has dimension $r$. Thus $\g$ has $r$ indecomposable direct summands over the principal $\mathfrak{sl}_2$, as claimed.
\end{proof} 

The numbers $m_i$ (arranged in non-decreasing order) are called the {\bf exponents} of $\g$. We will soon see that $m_i=d_i-1$, where $d_i$ are the degrees of $\g$. 

\subsection{Regular elements} 
Recall that $x\in \mathcal \g$ is {\bf regular} if the dimension of its centralizer is $r={\rm rank}\g$ (the smallest it can be). Thus regular elements form an open set $\g_{\rm reg}\subset \g$. 

\begin{lemma}\label{reguu} The element $e=\sum_{i=1}^r e_i$ is regular.
\end{lemma} 

\begin{proof} By Lemma \ref{prisl2}, the centralizer $Z_\g(e)$ is spanned by the highest vectors of the representations $L_{2m_1},...,L_{2m_r}$, hence has dimension $r$. 
\end{proof} 

\begin{corollary}\label{den} Let $B_+$ be the Borel subgroup of $G$ with Lie algebra 
$\mathfrak{b}_+:=\h\oplus \n_+$. Then ${\rm Ad}(B_+)e$ is the set of elements 
$\sum_{\alpha\in R_+} c_\alpha e_\alpha$ with $c_\alpha\in \Bbb C$ and $c_{\alpha_i}\ne 0$ for all $i$.  
\end{corollary}

\begin{proof} Since by Lemma \ref{reguu} $\dim Z_\g(e)=r$, we have 
$$
\dim [e,\n_+]\ge |R_+|-r=\dim [\n_+,\n_+].
$$ 
Since $[e,\n_+]\subset [\n_+,\n_+]$, we get that $[e,\n_+]=[\n_+,\n_+]$. It follows that if $N_+=\exp(\n_+)$ then 
${\rm Ad}(N_+)e=e+[\n_+,\n_+]$ is the set of expressions $\sum_{\alpha\in R_+}c_\alpha e_\alpha$ with $c_{\alpha_i}=1$ for all $i$. The statement follows by adding the action of the maximal torus 
$H=\exp(\h)$, which allows to set $c_{\alpha_i}$ to arbitrary nonzero values. 
\end{proof} 

\subsection{Properties of the nilpotent cone} 
\begin{proposition}\label{redu} The nilpotent cone is reduced.
\end{proposition} 

Proposition \ref{redu} is proved in the following exercise. 

\begin{exercise} Let $\g$ be a finite-dimensional simple Lie algebra. 

(i) Let $R_0$ be the graded algebra in Theorem \ref{CST2}. Show that the top degree of this algebra is $D:=\sum_{i=1}^r (d_i-1)$ and $R_0[D]=\Bbb C\Delta$, where 
$\Delta:=\prod_{\alpha\in R_+}\alpha$. Deduce that $\sum_{i=1}^r (d_i-1)=|R_+|$, the number of positive roots. 

(ii) Let $\g=\oplus_{i=1}^r L_{2m_i}$ be the decomposition of $\g$ as a module over the principal $\sl_2$-subalgebra $(e,f,h)$ given by Lemma \ref{prisl2}, i.e., $m_i$ are the exponents of $\g$. Show that $m_1=1$ and $\sum_{i=1}^r m_i=|R_+|$. Moreover, show that if $\mu_\g$ is the partition $(m_r,...,m_1)$ then the conjugate partition $\mu_\g^\dagger$ is $(n_1,...,n_{\rm h-1})$, where $n_i$ is the number of positive roots $\alpha$ of height $i$ (i.e., $(\rho^\vee,\alpha)=i$) and ${\rm h}:=m_r+1$. Conclude that ${\rm h}=(\rho^\vee,\theta)+1$ where $\theta$ is the maximal root, i.e., the {\bf Coxeter number} of $\g$. 

(iii)(a) Let $b_i$ be the lowest weight vectors of $L_{2m_i}$, and 
$$
\mathfrak{z}_f:=\oplus_{i=1}^r\Bbb Cb_i\subset \g
$$
be the centralizer of $f$. 
Show that $\g=\mathfrak{z}_f\oplus T_eO_e$, where $O_e={\rm Ad}(G)e$ is the orbit of $e$. 
Thus the affine space $e+\mathfrak{z}_f$ is transversal to $O_e$ at $e$. 
This affine space is called the {\bf Kostant slice}.  

(iii)(b) Consider the $\Bbb C^\times$-action on $\g$ given by 
$$
t\circ x=t^{\frac{1}{2}{\rm ad}(h)-1}x.
$$
Show that this action preserves the decomposition of (ii), and 
the linear coordinates $b_i^*$ on $\mathfrak{z}_f$ have homogeneity degrees $m_i+1$ under this action. 

(iv) Let $(S\g^*)^\g=\Bbb C[p_1,...,p_r]$, $\deg p_i=d_i$, 
and let $\widetilde p_i(y):=p_i(e+y)$, $y\in \mathfrak{z}_f$. Show that 
$\widetilde p_i$ are polynomials of $b_j^*$ homogeneous under the $\Bbb C^\times$-action of (iii) of degrees $d_i$. Deduce from this and the identity $\sum_i (d_i-1)=\sum_i m_i$ proved in (i),(ii) that 
$$
d_i-1=m_i
$$ 
and thus $\widetilde p_i=b_i^*$ (under appropriate choice of basis).
Conclude that the differentials $dp_i$ are linearly independent at $e\in \g$.  

(v) Work out (i)-(iv) explicitly for $\g=\mathfrak{sl}_n$.  

(vi) Prove Proposition \ref{redu}. {\bf Hint:} View 
$\mathcal O(\mathcal N)$ as an algebra over $\mathcal R:=S\n_+\otimes S\n_-$. 
Use the arguments of Subsection \ref{kothe} to show that it is a free $\mathcal R$-module of rank $|W|$. Show that 
the specialization of $\mathcal O(\mathcal N)$ at a generic point $z\in \n_+^*\times \n_-^*$ is a semisimple algebra of dimension $|W|$ (use (iv)).
Now take $f\in \mathcal O(\mathcal N)$ such that $f^k=0$ for some $k$, and deduce that the specialization of $f$ at $z$ is zero. Conclude that $f=0$.  
\end{exercise} 

\begin{proposition}\label{irrre} (i) The orbit $O_e:={\rm Ad}(G)e$ is open and dense in $\mathcal N$.  

(ii) All regular nilpotent elements in $\g$ are conjugate to $e$. 

(iii) $\mathcal N$ is an irreducible affine variety. Thus $(S\g)_0$ is an integral domain.
\end{proposition} 

\begin{proof} (i) This follows from Corollary \ref{den} and the fact that every nilpotent element in $\g$ can be conjugated into $\n_+$. 

(ii) The orbit $O_x$ of every regular nilpotent element $x$ has the same dimension as $O_e$, so the statement follows from (i). Indeed, since $O_e$ is open and dense, $\mathcal N\setminus O_e$ 
has smaller dimension than $\mathcal N$, hence can't contain $O_x$. 

(iii) follows from (i) and Proposition \ref{redu}, since $O_e$ is smooth and connected (being an orbit of a connected group), hence irreducible.  
\end{proof} 

\begin{corollary} $U_\chi$ is an integral domain for all $\chi$.
\end{corollary} 

\begin{proof} This follows from Proposition \ref{irrre}(iii) since ${\rm gr}(U_\chi)=(S\g)_0$. 
\end{proof} 

\begin{exercise}\label{nilor} Let $e$ be a nilpotent element in a semisimple 
complex Lie algebra $\g$, and $\g^e$ be the centralizer of $e$. 
Let $(,)$ be the Killing form of $\g$. 

(i) Show that $(e,\g^e)=0$ (prove that for any $x\in \g^e$, the operator ${\rm ad}_e{\rm ad}_x$ is nilpotent). 

(ii) Show that there exists $h\in \g$ such that $[h,e]=2e$
(use that ${\rm Im}({\rm ad}_e)={\rm \g^e}^\perp$ to deduce that 
$e\in {\rm Im}({\rm ad}_e)$). 
 
(iii) Show that in (ii), $h$ can be chosen semisimple (consider the Jordan decomposition $h=s+n$). From now on we choose $h$ in such a way. 

(iv) Show that $\Bbb Ch\oplus\g^e$ is a Lie subalgebra of $\g$. 

(v) Assume that $\g^e$ is nilpotent. Show that there is a basis
of $\g$ in which the operator ${\rm ad}_x$ is upper triangular 
for all $x\in \Bbb Ch\oplus\g^e$ (use Lie's theorem). Deduce that $(h,x)=0$ for all $x\in \g^e$. 

(vi) Show that if $\g^e$ is nilpotent then 
there are $h,f\in \g$ such that $[h,e]=2e$, $[e,f]=h$ and $[h,f]=-2f$. 
In other words, there is a homomorphism of Lie algebras 
$\phi: \mathfrak{sl}_2\to \g$ such that $\phi(E)=e$, $\phi(H)=h$, $\phi(F)=f$. 
Show that $h$ is semisimple and $f$ is nilpotent. 

(vii) (Jacobson-Morozov theorem, part I) Show that the conclusion of (vi) 
holds for any $e$ (without assuming that $\g^e$ is nilpotent). ({\bf Hint:} use induction in $\dim \g$. If $\g^e$ is not nilpotent, use Jordan decomposition to find a nonzero semisimple element $x\in \g^e$ and consider the Lie algebra $\g^x$. Show that $\g':=[\g^x,\g^x]$ is semisimple and 
$e\in \g'$). 

(viii) Show that for given $e,h$, the homomorphism $\phi$ in (vi,vii) is unique (i.e., $f$ is uniquely determined by $e,h$). 

(ix) (Jacobson-Morozov theorem, part II) Show that for a fixed $e$, 
$\exp(\g^e)$ (the Lie subgroup corresponding to $\g^e$) is a closed Lie subgroup of the adjoint group $G_{\rm ad}$ corresponding to $\g$, and the element $h$ (hence also $f$) can be chosen uniquely up to conjugation by $\exp(\g_e)$. ({\bf Hint}: 
Let $h'$ be another choice of $h$, and consider the element $h'-h\in \g^e$.)

(x) Explain why the Jacobson-Morozov theorem extends to reductive Lie algebras (where by a nilpotent element we mean one that is nilpotent in any finite-dimensional representation). Give an elementary proof of this theorem for $\g=\mathfrak{gl}_n$ using only linear algebra. 

(xi) Show that there are finitely many conjugacy classes of nilpotent elements in $\g$, i.e., the nilpotent cone $\mathcal{N}$ has finitely many $G_{\rm ad}$-orbits. ({\bf Hint:} Consider the variety $X$ of homomorphisms 
$\phi: \mathfrak{sl}_2\to \g$ and show that it is a disjoint union of finitely many closed $G_{\rm ad}$-orbits. To this end, show that the tangent space to $X$ at each $x\in X$ coincides with the tangent space of the orbit $Gx$ at the same point, using that ${\rm Ext}^1_{\mathfrak{sl}_2}(\Bbb C,\g)=0$). 
\end{exercise} 

\section{\bf Maps of finite type, Duflo-Joseph theorem} 

\subsection{Maps of finite type} 

Let $M,N$ be $\g$-modules. Let $\Hom_{\rm fin}(M,N)$ be the 
space of linear maps from $M$ to $N$ which generate a finite-dimensional $\g$-module under the adjoint action $a\circ T:=[a,T]$. The elements of $\Hom_{\rm fin}(M,N)$ are called {\bf linear maps of finite type}. For example, a module homomorphism is a map of finite type, as it generates a trivial $1$-dimensional $\g$-module. 

\begin{exercise} Show that any map of finite type has the form \linebreak
$(f\otimes 1)\circ \Phi$, where $f\in V^*$ for some finite-dimensional $\g$-module $V$ and $\Phi: M\to V\otimes N$ is a module homomorphism. 
\end{exercise}   
  
Note that $\Hom_{\rm fin}(M,N)$ is a $\g$-bimodule with bimodule structure given by 
$$
(a,b)\circ T:=aT+Tb, 
$$
$a,b\in \g$. Moreover, it is clear that 
if $M$ has infinitesimal character $\chi$ and $N$ has infinitesimal character $\theta$ then $\Hom_{\rm fin}(M,N)$ has infinitesimal character 
$(\theta,\chi)$.     
  
\begin{proposition} If $M,N\in \mathcal O$ then $\Hom_{\rm fin}(M,N)$  is an admissible $\g$-bimodule.  
\end{proposition} 

\begin{proof} We must show that for every simple finite-dimensional $\g$-module $V$, the space 
$$
\Hom_\g(V,\Hom_{\rm fin}(M,N))=\Hom_\g(V,\Hom_{\Bbb C}(M,N))
$$
is finite-dimensional. Let $\mu(M,N,V)$ be its dimension (a nonnegative integer or 
infinity). Since the functor $(M,N)\mapsto \Hom_{\Bbb C}(M,N)$ is exact in both arguments, for any short exact sequence 
$$
0\to M_1\to M_2\to M_3\to 0
$$ 
we have 
$$
\mu(M_2,N,V)\le\mu(M_1,N,V)+\mu(M_3,N,V),
$$
$$
\mu(N,M_2,V)\le\mu(N,M_1,V)+\mu(N,M_3,V).
$$
Thus, since $M,N$ have finite length, it suffices to establish the result for $M,N$ simple. Then $M$ is a quotient of $M_\lambda$ and $N$ a submodule of $M_\mu^\vee$ for some $\lambda,\mu$, so 
$\Hom_{\Bbb C}(M,N)\subset \Hom_{\Bbb C}(M_\lambda,M_\mu^\vee)$. 
But by Exercise \ref{intertw}, for any finite-dimensional $\g$-module $V$, 
$$
\Hom_\g(V,\Hom_{\Bbb C}(M_\lambda,M_\mu^\vee))\cong \Hom_\g(V\otimes M_\lambda,M_\mu^\vee)\cong 
$$
$$
\Hom_\g(M_\lambda,V^*\otimes M_\mu^\vee)\cong V^*[\lambda-\mu].
$$  
This implies the statement. 
\end{proof} 

\begin{proposition}\label{tensv} For $M,N\in \mathcal O$ and a finite-dimensional $\g$-module $V$ we have 
$$
\Hom_{\rm fin}(M,V\otimes N)=V\otimes \Hom_{\rm fin}(M,N).
$$
\end{proposition} 

\begin{exercise} Prove Proposition \ref{tensv}. 
\end{exercise} 

\begin{proposition}\label{dimhom1} Let $V$ be a finite-dimensional 
$\g$-module. Then for any $\lambda\in \h^*$, we have 
$$
\dim \Hom_\g(M_\lambda,V\otimes M_\lambda)=\dim V[0].
$$ 
Thus the multiplicity of $V$ in $\Hom_{\rm fin}(M_\lambda,M_\lambda)$ equals $\dim V[0]$. 
\end{proposition} 

\begin{proof} By Exercise \ref{dimhom}, the statement holds if $M_\lambda$ is irreducible, i.e., generically. Thus 
$\dim \Hom_\g(M_\lambda,V\otimes M_\lambda)\ge \dim V[0]$, and it remains to prove the opposite inequality. Let $M_{\mu}$ 
be the simple Verma submodule of $M_\lambda$. Given $\Phi: M_\lambda\to V\otimes M_\lambda$, we claim that the restriction of $\Phi$ to $M_\mu$ must land in $V\otimes M_\mu$. Indeed, otherwise we will have a nonzero (hence injective) homomorphism $M_\mu\to V\otimes (M_{\lambda}/M_{\mu})$, which is impossible by growth considerations. 

But by Exercise \ref{dimhom}, the statement holds if $\lambda$ is replaced by $\mu$. So if it does not hold for $\lambda$ then there is a nonzero $\Phi$ which kills $M_\mu$. Thus $\Phi$ defines a nonzero homomorphism $M_\lambda/M_\mu\to M_\lambda\otimes V$, which is impossible since $M_\lambda\otimes V$ is a free, hence torsion free 
$U(\n_-)$-module, while every homogeneous vector in $M_\lambda/M_\mu$ is torsion (as this module does not contain free $U(\n_-)$-submodules by growth considerations). This establishes the proposition.
\end{proof} 

\begin{remark} Note that Proposition \ref{dimhom1} does not extend 
to maps $M_{\lambda+\nu}\to V\otimes M_{\lambda}$ where $\nu\in P$ is nonzero. Namely, if $M_{\lambda}$ is irreducible then we have 
$\dim \Hom_\g(M_{\lambda+\nu},V\otimes M_{\lambda})=\dim V[\nu]$, so in general 
$\dim \Hom_\g(M_{\lambda+\nu},V\otimes M_{\lambda})\ge \dim V[\nu]$, and the inequality can, in fact, be strict. The simplest example is $\g=\mathfrak{sl}_2$, $V=\Bbb C$, 
$\lambda=0$, $\nu=-2$, in which case the left hand side is $1$ and the right hand side is $0$. 

Also the expectation value map 
$$
\langle,\rangle: \Hom_\g(M_\lambda,V\otimes M_\lambda)\to V[0]
$$
need not be an isomorphism, even though its source and target have the same dimension. The simplest example is $\g=\mathfrak{sl}_2$, $\lambda=0$, and $V$ is the adjoint representation. 
We have 
$$
\dim\Hom_\g(M_0,V\otimes M_0)=\dim\Hom_\g(M_0,V\otimes M_{-2})=1,
$$ 
so the only (up to scaling) nonzero homomorphism $\Phi: M_0\to V\otimes M_{0}$ in fact lands in $V\otimes M_{-2}\subset V\otimes M_0$. Thus 
$\langle \Phi\rangle =0$.  
\end{remark} 

\subsection{The Duflo-Joseph theorem} 

\begin{proposition}\label{injee} The action homomorphism 
$$
\phi: U_{\chi_{\lambda+\rho}}\to \Hom_{\rm fin}(M_{\lambda},M_{\lambda})
$$ 
is injective. 
\end{proposition} 

\begin{proof} Let $M_\mu\subset M_\lambda$ be a 
simple Verma submodule with highest weight vector $v$. 
Let $B_{\mu,\beta}: U(\n_+)[\beta]\otimes U(\n_-)[-\beta]\to \Bbb C$ be the pairing defined by the equality 
$$
abv=B_{\mu,\beta}(a,b)v.
$$
As $M_\mu$ is simple, this pairing is nondegenerate. 

Consider the multiplication map
$$
\xi: U(\n_-)\otimes U(\n_+)\to U_{\chi_{\lambda+\rho}}.
$$ 
We claim that the map $\phi\circ \xi$ is injective, hence so are $\xi$ and $\phi|_{{\rm Im}\xi}$. 
Indeed, let $x\in U(\n_-)\otimes U(\n_+)$ be a nonzero element. We can uniquely write $x=\sum_{\alpha\in Q_+}x_\alpha$, where 
$x_\alpha\in U(\n_-)\otimes U(\n_+)[\alpha]$. 
Let $\beta\in Q_+$ be a minimal element such that 
$x_\beta=\sum_i b_i\otimes a_i\ne 0$, where $\lbrace a_i\rbrace$ is a basis 
of $U(\n_+)[\beta]$. Let $\lbrace a_i^*\rbrace$ be the dual basis of $U(\n_-)[-\beta]$ with respect to $B_{\mu,\beta}$. Then 
$$(\phi\circ \xi)(x)a_j^*v=b_jv.$$ Since $b_j$ are not all zero, there exists $j$ such that 
$b_jv\ne 0$. It follows that $(\phi\circ \xi)(x)\ne 0$, as claimed. 

Thus, denoting the PBW filtration by $F_n$, we have 
$$
\dim F_n(U_{\chi_{\lambda+\rho}}/{\rm Ker}\phi)\ge \dim  F_n(U(\n_-)\otimes U(\n_+))\ge Cn^{\dim \g-r}
$$
for some $C>0$. On the other hand, assume that ${\rm Ker}\phi\ne 0$ and consider the nonzero ideal 
$$
{\rm gr}({\rm Ker}\phi)\subset (S\g)_0=\mathcal O(\mathcal N).
$$ 
This ideal contains a principal ideal $\mathcal O(\mathcal N)f$, 
where $f\in \mathcal O(\mathcal N)$ is a nonzero homogeneous 
element. Since $\mathcal O(\mathcal N)$ is a domain 
(Proposition \ref{irrre}(iii)), this ideal is a free 
$\mathcal O(\mathcal N)$-module generated by $f$. 
$$
\dim F_n(U_{\chi_{\lambda+\rho}}/{\rm Ker}\phi)=
\dim {\rm gr}_{\le n} (\O(\mathcal N)/{\rm gr}({\rm Ker}\phi))
$$
$$
\le\dim {\rm gr}_{\le n} (\O(\mathcal N)/\O(\mathcal N)f)\le C'n^{\dim \g-r-1}.
$$
for some $C'>0$. 
So we get that $Cn^{\dim \g-r}\le C'n^{\dim \g-r-1}$. 
This is a contradiction, so  
${\rm Ker}\phi=0$ and thus
$\phi$ is injective.
\end{proof} 

\begin{corollary}\label{iso} (The Duflo-Joseph theorem) $\phi$ is an isomorphism.
\end{corollary}

\begin{proof} Consider the restriction $\phi_V$ of $\phi$ to the $V^*$-isotypic component. Thus 
$$
\phi_V: {\rm Hom}_\g(V^*,(U_{\chi_{\lambda+\rho}})_{\rm ad})\to \Hom_\g(M_{\lambda},V\otimes M_\lambda).
$$
By Kostant's theorem, the source of this map has dimension $\dim V[0]$, while by  
Proposition \ref{dimhom1}, so does the target. Since by Proposition \ref{injee} $\phi_V$ is injective, it follows that $\phi_V$ is an isomorphism for all $V$, hence so is $\phi$. 
\end{proof} 

\begin{corollary}\label{natmap} If $V$ is a finite-dimensional $\g$-module then the natural map 
$V\otimes U_{\chi_{\lambda+\rho}}\to \Hom_{\rm fin}(M_\lambda,V\otimes M_\lambda)$ is 
an isomorphism.
\end{corollary} 

\begin{proof} This follows from Proposition \ref{tensv} and Corollary \ref{iso}. 
\end{proof} 

\subsection{Infinitesimal characters of Harish-Chandra bimodules} 

\begin{corollary}\label{cechar} Let $V$ be a finite-dimensional $\g$-module and $\lambda\in \h^*$. 

(i) The left infinitesimal characters occurring in $V\otimes U_{\chi_{\lambda}}$ 
are $\chi_{\lambda+\nu}$ where $\nu$ runs over weights of $V$. 

(ii) If $M$ is a $\g$-module with infinitesimal character $\chi_\lambda$ then 
the infinitesimal characters occurring in $V\otimes M$ 
are among $\chi_{\lambda+\nu}$ where $\nu$ runs over weights of $V$. 

(iii) If $M$ is a nonzero Harish-Chandra $\g$-bimodule with infinitesimal character $(\chi_\lambda,\chi_\mu)$
then there is $w\in W$ such that $w\lambda-\mu\in P$.  
\end{corollary} 

\begin{proof} (i) This follows from Corollary \ref{natmap}.  

(ii) follows from (i) and the isomorphism 
$$
V\otimes M\cong (V\otimes U_{\chi_\lambda})\otimes_{U_{\chi_\lambda}}M.
$$

(iii) This follows from (i) since by Corollary \ref{quot} any irreducible Harish-Chandra bimodule is a quotient of $V\otimes U_{\chi_\mu}$ for some $\mu,V$. 
\end{proof} 

Let $HC_{\theta,\chi}(\g)$ be the category of Harish-Chandra $\g$-bimodules 
with generalized infinitesimal character  $(\theta,\chi)$.

\begin{corollary} The category of Harish-Chandra $\g$-bimodules $HC(\g)$ has a decomposition according to generalized infinitesimal characters: 
$$
HC(\g)=\oplus_{\theta,\chi} HC_{\theta,\chi}(\g),
$$
where $(\theta,\chi)$ runs over the distinct pairs $(\chi_{\lambda+\gamma},\chi_\lambda)$ with $\gamma\in P_+$ and $\lambda\in\h^*$. In particular, 
if $(\theta,\chi)$ cannot be written as $(\chi_{\lambda+\gamma},\chi_{\lambda})$, $\lambda\in \h^*$, $\gamma\in P_+$, then 
$HC_{\theta,\chi}(\g)=0$. 
\end{corollary}

\begin{proof} This follows from Exercise \ref{finquo1} and Corollary \ref{cechar}.
\end{proof} 

\section{\bf Principal series representations} 

\subsection{Residual finiteness of $U(\g)$} 

\begin{proposition}\label{resfi} The homomorphism $\phi: U(\g)\to \prod_{\lambda\in P_+}\End(L_\lambda)$ is injective. 
\end{proposition} 

\begin{proof} Let $x\in {\rm Ker}\phi$, and $G$ be the simply connected group with Lie algebra $\g$. Then by the Peter-Weyl theorem, 
$x$ acts by zero on $\mathcal O(G):=\oplus_{\lambda\in P_+}L_\lambda\otimes L_\lambda^*$ (where $x$ acts only on the first component). This means that 
the right-invariant differential operator on $G$ defined by $x$ is zero, i.e., $x=0$. 
\end{proof} 

\begin{exercise} Give another proof of Proposition \ref{resfi} which does not use the Peter-Weyl theorem. Take $x\in {\rm Ker}\phi$. 

(i) Show by interpolation that $x$ acts by zero in every Verma module $M_\lambda$. 

(ii) Show that if $x\in U(\g)$ acts by zero in $M_\lambda$ for all $\lambda$ then $x=0$. 
\end{exercise} 

Note that Proposition \ref{resfi} implies that any $z\in U(\g)$ which acts by a scalar in all $L_\lambda$ 
belongs to $Z(\g)$. Indeed, in this case for any $x\in U(\g)$, $[x,z]$ acts by zero in $L_\lambda$, hence $[x,z]=0$. 

\subsection{Principal series} Let $\lambda,\mu\in \h^*$, $\lambda-\mu\in P$. Define the {\bf principal series} bimodule
$$
\bold M(\lambda,\mu):=\Hom_{\rm fin}(M_{\lambda-\rho},M_{\mu-\rho}^\vee)\in HC_{\chi_{\mu},\chi_{\lambda}}(\g).
$$ 
Then we have 
\begin{equation}\label{princer}
\bold M(\lambda,\mu)=\oplus_{V\in {\rm irr}(\g)}V\otimes V^*[\lambda-\mu].
\end{equation} 
The bimodule $\bold M(\lambda,\mu)$ represents a certain functor that has a nice independent description. 

\begin{proposition} Let $X\in HC(\g)$. Then 
$$
\Hom_{\g-{\rm bimod}}(X,\bold M(\lambda,\mu))\cong \Hom_{(\mathfrak{b}_-,\mathfrak{b}_+)-{\rm bimod}}(X\otimes \Bbb C_{\lambda-\rho},\Bbb C_{\mu-\rho}).
$$
where the $(\mathfrak b_-,\mathfrak b_+)$-bimodule structure on $\Bbb C_{\mu-\rho}$ is defined by the character $(\mu-\rho,0)$ and on $\Bbb C_{\lambda-\rho}$ by the character $(0,\lambda-\rho)$. 
\end{proposition} 

\begin{proof} We have 
$$
\Hom_{\g-{\rm bimod}}(X,\bold M(\lambda,\mu))=\Hom_{\g-{\rm bimod}}(X\otimes M_{\lambda-\rho},M_{\mu-\rho}^\vee),
$$
where the right copy of $\g$ acts trivially on $M_{\mu-\rho}^\vee$ and the left copy 
of $\g$ acts trivially on $M_{\lambda-\rho}$. Frobenius reciprocity then yields 
$$
\Hom_{\g-{\rm bimod}}(X,\bold M(\lambda,\mu))=\Hom_{(\g,\b_+)-{\rm bimod}}(X\otimes \Bbb C_{\lambda-\rho},M_{\mu-\rho}^\vee).
$$
Since $X\otimes \Bbb C_{\lambda-\rho}$ is diagonalizable under the adjoint action of $\h$, 
on the right hand side we may replace $M_{\mu-\rho}^\vee$ with 
its completion $\widehat M_{\mu-\rho}^\vee$ (the Cartesian product of all weight spaces). 
Then applying Frobenius reciprocity again, we get the desired statement.  
\end{proof} 

Let us give an explicit realization of $\bold M(\lambda,\mu)$. 
By \eqref{princer}, $\bold M(\lambda,\mu)$ is spanned by elements 
$\Phi_{v,\ell}: M_{\lambda-\rho}\to M_{\mu-\rho}^\vee$, $v\in V,\ell\in V^*[\lambda-\mu]$, where 
$$
\Phi_{v,\ell}u:=(v\otimes 1,\Phi_\ell u), 
$$
and $\Phi_\ell: M_{\lambda-\rho}\to V^*\otimes M_{\mu-\rho}^\vee$ is the homomorphism 
for which $\langle \Phi_\ell\rangle=\ell$, for finite-dimensional 
$\g$-modules $V$. Moreover these elements easily express in terms of such elements for simple $V$.  Thus for any $V$ and $y\in V\otimes V^*[\lambda-\mu]$ 
we can define the linear map $\Phi_V(y): M_{\lambda-\rho}\to M_{\mu-\rho}^\vee$ 
which depends linearly on $y$ with $\Phi_V(v\otimes \ell)=\Phi_{v,\ell}$, and every element of $\bold M(\lambda,\mu)$ is of this form.  

\begin{proposition}\label{rightac} The right action of $\g$ on $\bold M(\lambda,\mu)$ 
is given by the formula 
$$
\Phi_V(v\otimes \ell)\cdot b=\Phi_{\g\otimes V}([b\otimes v]\bigotimes [(\lambda-\rho)\otimes \ell+\sum_{\alpha\in R_+}f_\alpha^*\otimes f_\alpha \ell]).
$$
\end{proposition} 

\begin{proof} Consider the homomorphism
$$
\Psi_\ell:=\sum_i b_i^*\otimes \Phi_\ell b_i: M_{\lambda-\rho}\to \g^*\otimes V^*\otimes M_{\mu-\rho}^\vee,
$$ where 
$\lbrace b_i\rbrace$ is a basis of $\g$ and $\lbrace b_i^*\rbrace$ the dual basis of $\g^*$. 
We have 
$$
\langle \Psi_\ell\rangle=\sum b_i^*\otimes \langle \Phi_\ell b_i\rangle\in \g^*\otimes V^*,
$$
where the expectation value map $\langle,\rangle$ is defined in Exercise \ref{intertw}. 
But
$$
\langle \Phi_\ell h\rangle=(\lambda-\rho,h)\ell,\ \langle \Phi_\ell e_\alpha\rangle=0,\ 
\langle \Phi_\ell f_\alpha\rangle=f_\alpha \ell 
$$
for $\alpha\in R_+$. 
Thus we get 
$$
\langle\Psi_\ell\rangle=(\lambda-\rho)\otimes \ell+\sum_{\alpha\in R_+}f_\alpha^*\otimes f_\alpha \ell,
$$
hence
$$
\Psi_\ell=\Phi_{(\lambda-\rho)\otimes \ell+\sum_{\alpha\in R_+}f_\alpha^*\otimes f_\alpha \ell}.
$$
This implies the statement since 
$$
(\Phi_V(v\otimes \ell)\cdot b)u=(v\otimes 1,\Phi_\ell bu)=(b\otimes v\otimes 1,\Psi_\ell u),\ u\in M_{\lambda-\rho}.
$$
\end{proof} 

This leads to a geometric construction  of the principal series.
Namely, let $G$ be the simply connected group with Lie algebra $\g$, $B=B_+$ be the Borel subgroup of $G$ whose Lie algebra is $\mathfrak{b}_+$ and $H=B/[B,B]$ the corresponding torus. Fix $\lambda,\mu\in \h^*$ with $\lambda-\mu\in P$. 
Define a real-analytic character 
$$
\psi_{\lambda,\mu}: H\to \Bbb C^\times
$$
by 
$$
\psi_{\lambda,\mu}(x):=\lambda(x)\mu(x^*)^{-1},
$$
where $x^*$ is the image of $x$ under the compact antiholomorphic 
involution $\sigma : H\to H$ (i.e., such that $H^\sigma=H_c$, the compact real form of $H$). For example, for $G=SL_2$, $\lambda,\mu$ are complex numbers with $\lambda-\mu$ an integer and $x^*=\overline x^{-1}$, so 
$$
\psi_{\lambda,\mu}(x)=x^\lambda \overline x^{\mu}=x^{\lambda-\mu}|x|^{2\mu}.
$$
Define $C^\infty_{\lambda,\mu}(G/B)$ to be the space of smooth functions on $G$ satisfying 
$$
F(gb)=F(g)\psi_{\lambda,\mu}(b). 
$$
This is naturally an admissible representation of 
$G$: we have $G/B=G_c/H_c$, so the multiplicity space of $V$ 
in $C^\infty_{\lambda,\mu}(G/B)$ is $V^*[\lambda-\mu]$; 
namely, $C^\infty_{\lambda,\mu}(G/B)^{\rm fin}=
C^\infty_{\lambda-\mu}(G_c/H_c)^{\rm fin}$, the space of  
$G_c$-finite functions on $G_c$ (under left translations) 
such that 
$$
F(gx)=F(g)\lambda(x)\mu(x)^{-1}
$$ 
for $x\in H_c$. 

\begin{proposition}\label{prinser1}  
We have an isomorphism 
$$
\xi: \bold M(\lambda,\mu)\to C^\infty_{\lambda-\rho,\mu-\rho}(G/B)^{\rm fin}
$$ 
as Harish-Chandra bimodules. 
Namely, $\xi(\Phi_{v,\ell})$ is the matrix 
coefficient $\psi_{v,\ell}(g):=(v,g\ell)$, $g\in G_c$. 
\end{proposition} 

\begin{exercise} Prove Proposition \ref{prinser1}. {\bf Hint:} Use Proposition \ref{rightac} to show that $\xi$ is a well defined isomorphism of ${\g}_{\rm ad}$-modules, and after applying $\xi$ the right action of $\g$ looks like 
$$
(\psi\cdot b)(g)=(\lambda-\rho)({\rm Ad}(g)b)\psi(g)+
\sum_{\alpha\in R_+}f_\alpha^*({\rm Ad}(g)b)(R(f_\alpha)\psi)(g),
$$
where $R(f_\alpha)$ is the left-invariant vector field equal to $f_\alpha$ at $1$. 
Then show that the right action of $\g$ on $C^\infty_{\lambda-\rho,\mu-\rho}(G/B)$ is given by the same formula. 
\end{exercise} 

\subsection{The functor $H_\lambda$}\label{funHla} 

Define the functor $H_\lambda: \mathcal O_\theta\to HC_{\theta,\chi_\lambda}$ given by 
$$
H_\lambda(X):=\Hom_{\rm fin}(M_{\lambda-\rho},X).
$$
Note that $H_\lambda(M_{\mu-\rho}^\vee)=\bold M(\lambda,\mu)$. 

\begin{proposition} \label{Hlaexact}
The functor $H_\lambda$ is exact when $\lambda$ is dominant. 
\end{proposition} 

\begin{proof} If $V$ is a finite-dimensional $\g$-module then 
$$
\Hom_\g(V,H_\lambda(X))=\Hom_\g(V\otimes M_{\lambda-\rho},X),
$$
which is exact as $V\otimes M_{\lambda-\rho}$ is projective. 
\end{proof} 

\section{\bf BGG reciprocity and BGG Theorem} 

\subsection{A vanishing lemma for Ext groups} 

\begin{lemma}\label{exti} Let $X\in \mathcal O$ be a free $U(\n_-)$-module. 
Then for any $\mu\in \h^*$ we have 
$$
{\rm Ext}^i_{\mathcal O}(X,M_\mu^\vee)=0,\ i>0. 
$$
\end{lemma} 

\begin{proof} Fix a projective resolution $P_\bullet$ of $X$ in $\mathcal O$ 
and consider the complex $\Hom_\g(P_\bullet,M_\mu^\vee)$ 
which computes the desired Ext groups. Since $P_i$ have a weight decomposition, 
$$
\Hom_\g(P_\bullet,M_\mu^\vee)=\Hom_\g(P_\bullet,\widehat M_\mu^\vee),
$$
where $\widehat M_\mu^\vee:=\prod_{\beta\in \h^*} M_\mu^\vee[\beta]$ is the completion of $M_\mu^\vee$. 
We have 
$$
\widehat M_\mu^\vee={\rm Coind}_{\mathfrak b_-}^\g(\Bbb C_\mu):=\Hom_{\b_-}(U(\g),\Bbb C_\mu)\cong \Hom_{\Bbb C}(U(\n_+),\Bbb C_\mu). 
$$ 
Thus, Frobenius reciprocity yields 
$$ 
\Hom_\g(P_\bullet,\widehat M_\mu^\vee)=\Hom_{\mathfrak b_-}(P_\bullet,\Bbb C_\mu). 
$$
By Corollary \ref{enoughpro}(ii), $P_i$ are free $U(\n_-)$-modules, so the exact sequence of $U(\n_-)$-modules
$$
...\to P_1\to P_0\to X\to 0
$$
is split. Thus the complex 
$\Hom_{\mathfrak b_-}(P_\bullet,\Bbb C_\mu)$ is exact in positive degrees, which implies the statement. 
\end{proof} 

\subsection{Standard filtrations} 

A {\bf standard (or Verma) filtration} on $X\in \mathcal O$ 
is a filtration for which successive quotients are Verma modules. 
$X$ is called {\bf standardly filtered} if it admits a standard filtration. 
It is clear that every standardly filtered object $X$ is necessarily a free 
$U(\n_-)$-module.  

\begin{corollary}\label{exti1} If $X$ is standardly filtered then 
$\Ext^i_{\mathcal O}(X,M_\mu^\vee)=0$ for all $\mu\in \h^*$ and $i>0$. 
\end{corollary} 

\begin{proof} This follows from Lemma \ref{exti}. 
\end{proof}

The converse also holds. In fact, we have 

\begin{theorem}\label{exti2} $X$ is standardly filtered if and only if 
$$
{\rm Ext}^1_{\mathcal O}(X,M_\lambda^\vee)=0
$$ 
for all $\lambda\in \h^*$. 
\end{theorem} 

\begin{proof}  Let $E$ be a finite-dimensional vector space, and suppose we have a short exact sequence in $\mathcal O$: 
$$
0\to K\to E\otimes M_\lambda\to Z\to 0
$$  
with $K[\lambda]=0$. 

\begin{lemma}\label{K=0} If $\Ext^1_\O(Z,M_\mu^\vee)=0$ for all $\mu\in \h^*$
then $K=0$ and $Z\cong E\otimes M_\lambda$.  
\end{lemma} 

\begin{proof}  
The long exact sequence of cohomology yields 
$$
...\to\Hom(E\otimes M_\lambda,M_\mu^\vee)\to \Hom(K,M_\mu^\vee)\to \Ext_\O^1(Z,M_\mu^\vee)=0.
$$
For $\lambda\ne \mu$, we have $\Hom(M_\lambda,M_\mu^\vee)=0$, so it follows that 
$\Hom(K,M_\mu^\vee)=0$. But we also have $\Hom(K,M_\lambda^\vee)=0$, as 
$K[\lambda]=0$, while every nonzero submodule of $M_\lambda^\vee$ contains $L_\lambda$. It follows that $K=0$. 
\end{proof} 

Now let us prove the theorem. We only need to prove the ``if" direction. We argue by induction in the length of $X$ (with the base case $X=0$ being trivial). Let 
$\lambda$ be a maximal weight in $P(X)$ and $E:=X[\lambda]$. Let $Z$ 
be the submodule of $X$ generated by $E$; it is a quotient of $E\otimes M_\lambda$ by a submodule $K$ with $K[\lambda]=0$. We have a short exact sequence 
$$
0\to Z\to X\to Y\to 0.
$$
Thus from the long exact sequence of cohomology we get an exact sequence 
$$
...\to \Hom(Z,M_\mu^\vee)\to \Ext^1_{\mathcal O}(Y,M_\mu^\vee)\to \Ext^1_{\mathcal O}(X,M_\mu^\vee)=0. 
$$
It follows that for $\mu\ne \lambda$ we have $\Ext^1_{\mathcal O}(Y,M_\mu^\vee)=0$, as in this case $\Hom(Z,M_\mu^\vee)=0$ (since $Z$ is a quotient of $E\otimes M_\lambda$). On the other hand, if $\mu=\lambda$ then 
by the argument in the proof of Lemma \ref{exti} we have
$$
\Ext^1_{\mathcal O}(Y,M_\lambda^\vee)=\Ext^1_{\mathcal C}(Y,\Bbb C_\lambda),
$$
where $\mathcal C$ is the category of $\h$-semisimple $\mathfrak \b_-$-modules. 
But $\Ext^1_{\mathcal C}(Y,\Bbb C_\lambda)=0$, as all weights of $Y$ are not $>\lambda$ and hence any short exact sequence of $\b_-$-modules 
$$
0\to \Bbb C_\lambda\to \widetilde Y\to Y\to 0
$$
canonically splits. By the induction assumption, it follows that 
$Y$ is standardly filtered, so by Corollary \ref{exti1}, 
${\rm Ext}^i(Y,M_\mu^\vee)=0$ for all $i\ge 1$, in particular for $i=1,2$. Thus the long exact sequence of Ext groups gives 
$$
\Ext^1(Z,M_\mu^\vee)=\Ext^{1}(X,M_\mu^\vee)=0,
$$ 
hence $Z=E\otimes M_\lambda$ by Lemma \ref{K=0}. This completes the induction step. 
\end{proof} 

\begin{corollary}\label{exti3} (i) Every $X\in \mathcal O$ which is a free 
$U(\n_-)$-module is standardly filtered. In particular, 
for any $\lambda\in \h^*$ and finite-dimensional $\g$-module $V$, the module $V\otimes M_\lambda$ 
is standardly filtered. 

(ii) Any projective object $P\in \O$ is standardly filtered. 
\end{corollary} 

\begin{proof} (i) Follows from Theorem \ref{exti2} and Lemma \ref{exti}. 

(ii) Immediate from Theorem \ref{exti2}.  
\end{proof} 

\subsection{BGG reciprocity} 
Denote by $d_{\lambda\mu}$ the multiplicity of $L_\mu$ in the Jordan-H\"older series 
of $M_\lambda$. Since characters of $L_\mu$ are linearly independent, these numbers are determined from the formula 
$$
\sum_\mu d_{\lambda\mu}{\rm ch}(L_\mu)={\rm ch}(M_\lambda)=\frac{e^\lambda}{\prod_{\alpha\in R_+}(1-e^{-\alpha})}. 
$$  
Thus the knowledge of $d_{\lambda\mu}$ is equivalent to the knowledge 
of the characters ${\rm ch}(L_\lambda)$. 

Since by Corollary \ref{exti3}(ii) the projective covers $P_\lambda$ of $L_\lambda$ are standardly filtered, we may also define the multiplicities $d_{\lambda\mu}^*$
of $M_\mu$ in $P_\lambda$. These are independent of the choice of the standard filtration and are determined by the formula 
$$
{\rm ch}(P_\lambda)=\sum_\mu d_{\lambda\mu}^*{\rm ch}(M_\mu)=
\sum_\mu d_{\lambda\mu}^*\frac{e^\mu}{\prod_{\alpha\in R_+}(1-e^{-\alpha})}. 
$$

\begin{theorem}\label{bggr} (BGG reciprocity) We have 
$d_{\lambda\mu}^*=d_{\mu\lambda}$. 
\end{theorem}

\begin{proof} We compute 
$\dim \Hom(P_\lambda,M_\mu^\vee)$ in two ways. 
First using the standard filtration of $P_\lambda$ and Lemma \ref{exti}, we have 
$\dim \Hom(P_\lambda,M_\mu^\vee)=d_{\lambda\mu}^*$. On the other hand, 
using that the multiplicity of $L_\lambda$ in $M_\mu^\vee$ is $d_{\mu\lambda}$, we 
get $\dim \Hom(P_\lambda,M_\mu^\vee)=d_{\mu\lambda}$. 
\end{proof} 

Let $c_{\lambda\mu}=\dim \Hom(P_\lambda,P_\mu)$ be the entries 
of the Cartan matrix $C$ of $\mathcal O$. They are equal to the multiplicities 
of $L_\lambda$ in $P_\mu$. 

\begin{corollary} We have 
$$
c_{\lambda\mu}=\sum_\nu d_{\nu\lambda}d_{\nu\mu}.
$$
In other words, $C=D^TD$ where $D=(d_{\lambda\mu})$.
\end{corollary} 

Note that since $D$ is upper triangular with respect to the partial order $\le$ with ones on the diagonal, it can be uniquely recovered from $C$ by Gauss decomposition. Thus the knowledge of $D$ is equivalent to the knowledge of $C$. 

\begin{example} Consider the structure of the category $\mathcal O_\chi$ for 
$\g=\mathfrak{sl}_2$. The only interesting case is $\chi=\chi_{\lambda+1}$ for $\lambda\in \Bbb Z_{\ge 0}$. Then the simple objects are $X=L_\lambda$ (finite-dimensional) and $Y=M_{-\lambda-2}$. 
By Proposition \ref{proje}, the projective cover $P_X$ is just the Verma module $M_\lambda$, which has composition series $[X,Y]$, starting from the head $X$. To determine 
$P_Y$, consider the tensor product $P:=M_{-1}\otimes L_{\lambda+1}$. This is projective with character 
$$
{\rm ch}(P)={\rm ch}(M_\lambda)+{\rm ch}(M_{\lambda-2})+...+{\rm ch}(M_{-\lambda-2}).
$$
Thus denoting by $\Pi_\lambda$ the projection functor to the generalized infinitesimal character $\chi_{\lambda+1}$,  we get that 
$$
{\rm ch}(\Pi_\lambda(P))={\rm ch}(M_\lambda)+{\rm ch}(M_{-\lambda-2}).
$$
Note that $M_{-\lambda-2}$ is not projective since $\Ext^1_\O(M_{-\lambda-2},L_\lambda)\ne 0$ (there is a nontrivial extension $M_\lambda^\vee$). Thus 
$\Pi_\lambda(P)$ is indecomposable (otherwise one of the summands in the decomposition would have to be $M_{-\lambda-2}$), i.e., $\Pi_\lambda(P)=P_Y$. 
Since it maps to $Y$ and receives an injection from $M_\lambda$, 
its composition series is $[Y,X,Y]$. This is the {\bf big projective object} 
of $\mathcal O_\chi$. We thus get for $\mathcal O_\chi$: 
$$
D=\begin{pmatrix} 1& 1\\ 0& 1\end{pmatrix},\ C= \begin{pmatrix} 1& 1\\ 1& 2\end{pmatrix}.
$$
We can now compute the (basic) algebra $A$ whose module category is equivalent to $\O_\chi$. This is the algebra $A=\End(P_X\oplus P_Y)$, and it has 
dimension $\sum_{i,j}c_{ij}=5$. The basis is formed by $1_X,1_Y$ and morphisms 
$a: P_X\to P_Y$, $b: P_Y\to P_X$ and $ab: P_Y\to P_Y$. 
Moreover, we have $ba=0$. Thus the algebra 
$A$ is the path algebra of the quiver with two vertices $x,y$ 
with edges $a: x\to y$ and $b: y\to x$ with the only relation $ba=0$.   
\end{example} 

\subsection{The duality functor} 

Let $\tau: \g\to \g$ be the {\bf Cartan involution} given by $\tau(e_i)=f_i$, $\tau(f_i)=e_i$, $\tau(h_i)=-h_i$. For $X\in \mathcal O$ let $X^\tau$ be the module 
$X$ twisted by $\tau$, and $X^\vee=(X^\tau)_{\rm fin}^*$, the $\h$-finite part of $(X^\tau)^*$. The following proposition is easy: 

\begin{proposition} (i) $X^\vee\in \mathcal O$ and has the same character and composition series as $X$. 

(ii) $(M_\lambda)^\vee=M_\lambda^\vee$, $L_\lambda^\vee=L_\lambda$. 

(iii) the assignment $X\mapsto X^\vee$ is an involutive equivalence of categories 
$\mathcal O\to \mathcal O^{\rm op}$ which preserves the decomposition into $\O_\chi(S)$. 
\end{proposition}

\begin{corollary} $\mathcal O$ has enough injectives, namely the injective hull of 
$L_\lambda$ is $P_\lambda^\vee$. 
\end{corollary}   

\subsection{The Jantzen filtration} 

It turns out that every Verma module $M_\lambda$ carries a canonical finite filtration by submodules called the {\bf Jantzen filtration}, which plays an important role in studying category $\mathcal O$. In fact, this filtration is defined much more generally, as follows. 

Let $k$ be a field and $V,W$ be free $k[[t]]$-modules  of the same rank $d<\infty$, and 
let $B\in \Hom(V,W)$ be such that $\det B:=\wedge^d B$ is nonzero. Let $V_0:= V/tV$. Define $V_m\subset V_0$ to be the space of all $v_0\in V_0$ such that there exists 
a lift $v\in V$ of $v_0$ for which $Bv\in t^mW$. 
It is clear that $V_0\supset V_1\supset V_2\supset...$ with $V_1={\rm Ker}B(0)$, and $V_m=0$ for some $m$. Thus we get a finite descending filtration $\lbrace V_j\rbrace$ of $V_0$ called the {\bf Jantzen filtration} attached to $B$.  

\begin{exercise}\label{JSF1} (i) Show that there exist unique nonnegative 
integers $n_1\le...\le n_d$ such that for some bases 
$e_1,...,e_d$ of $V$ and $f_1,...,f_d$ of $W$ over $k[[t]]$ 
one has $Be_i=t^{n_i}f_i$, and that ${\rm Coker}B\cong \oplus_{i=1}^d(k[t]/t^{n_i})$ as 
a $k[[t]]$-module. Deduce that the order of vanishing of $\det B$ 
at $t=0$ equals $\dim_k {\rm Coker}B=\sum_{i=1}^d n_i$.

(ii) Suppose $\dim V_j=d_j$ (so $d_0=d$). Show that for all $j\in \Bbb Z_{\ge 0}$, 
$n_i=j$ if and only if $d-d_j<i\le d-d_{j+1}$, and deduce 
the {\bf Jantzen sum formula}: the order of vanishing of $\det B$ 
at $t=0$ equals $\sum_{j\ge 1}d_j$. 

(iii) Suppose that $V,W$ are modules over some $k[[t]]$-algebra $A$ with $A_0:=A/tA$ (for example, $A=A_0[[t]]$ and $V,W$ are $A_0$-modules), and $B$ is an $A$-module homomorphism. Show that the Jantzen filtration of $V_0$ attached to $B$ is a filtration by $A_0$-submodules. 
\end{exercise} 

The Jantzen filtration on $M_\lambda$ is now defined 
using the homomorphism $B: M_{\lambda(t)}\to M_{\lambda(t)}^\vee$ 
over $A:=U(\g)[[t]]$ corresponding to the Shapovalov form, where $\lambda(t):=\lambda+t\rho$. 
Namely, we define it separately on each weight subspace. 
For example, $(M_\lambda)_1=J_\lambda$ 
is the maximal proper submodule of $M_\lambda$. 

\begin{exercise}\label{JSF} (Jantzen sum formula for $M_\lambda$) 
Use the Jantzen sum formula of Exercise \ref{JSF1} and the formula 
for the determinant of the Shapovalov form (Exercise \ref{Shapova}) to show that 
$$
\sum_{j\ge 1}{\rm ch}((M_\lambda)_j)=\sum_{\alpha\in R_+: (\lambda+\rho,\alpha^\vee)\in \Bbb Z_{\ge 1}}
{\rm ch}(M_{\lambda-(\lambda+\rho,\alpha^\vee)\alpha}). 
$$
\end{exercise}  

\subsection{The BGG theorem}

The following is the converse to Theorem \ref{vermath}. 

\begin{theorem}\label{BGGth} (Bernstein -- I. Gelfand -- S. Gelfand) 
If $L_{\mu-\rho}$ occurs in the composition series of $M_{\lambda-\rho}$
(i.e., $d_{\lambda-\rho,\mu-\rho}\ne 0$) then $\mu\preceq\lambda$.
\end{theorem} 

\begin{proof} It is clear that $\lambda-\mu\in Q_+$. 
The proof is by induction in the integer $n:=(\lambda-\mu,\rho^\vee)$. 
If $n=0$, the statement is obvious, so we only need 
to justify the induction step for $n>0$. Then $L_{\mu-\rho}$ occurs 
in $J_{\lambda-\rho}=(M_{\lambda-\rho})_1$, 
the degree $1$ part of the Jantzen filtration of $M_{\lambda-\rho}$. 
Thus by the Jantzen sum formula (Exercise \ref{JSF}), 
$L_{\mu-\rho}$ must occur in $M_{\lambda-\rho-(\lambda,\alpha^\vee)\alpha}=M_{s_\alpha\lambda-\rho}$ 
for some $\alpha\in R_+$ such that $(\lambda,\alpha^\vee)\in \Bbb Z_{\ge 1}$. 
By the induction assumption, we then have $\mu\preceq s_\alpha\lambda$. 
But $s_\alpha\lambda\prec\lambda$, so we get $\mu\prec\lambda$. 
\end{proof} 

\begin{corollary} The following conditions on $\mu\le \lambda$ are equivalent.

(i) $\mu\preceq \lambda$

(ii) $L_{\mu-\rho}$ occurs in $M_{\lambda-\rho}$.

(iii) $\dim \Hom(M_{\mu-\rho},M_{\lambda-\rho})\ne 0$.
\end{corollary}

\section{\bf Multiplicities in category $\mathcal O$}

The multiplicities $d_{\lambda\mu}$ are complicated in general, and 
the (eventually successful) attempt to understand them was one of the main 
developments that led to creation of geometric representation theory. 
These multiplicities are given by the {\bf Kazhdan-Lusztig conjecture} (1979) proved by Beilinson-Bernstein and independently by Brylinski-Kashiwara in 1981. 
By now several proofs of this conjecture are known, but they are complicated 
and beyond the scope of this course. However, let us give the statement of this result. To simplify the exposition, we do so for $\mathcal O_{\chi_\lambda}$ 
when $\lambda\in P_+$; it turns out that this case captures all the complexity of the situation, and the general case is similar.

\subsection{The Hecke algebra} 
Even to formulate the Kazhdan-Lusztig conjecture, we need to introduce an object which seemingly has nothing to do with our problem - the {\bf Hecke algebra} of $W$. Namely, recall that $W$ is defined by generators $s_i, i=1,...,r$ subject to
the braid relations 
$$
s_is_j...=s_js_i...,\ i\ne j,
$$
where the length of both words is $m_{ij}$ such that $a_{ij}a_{ji}=4\cos^2 \frac{\pi}{m_{ij}}$(for $a_{ji}a_{ij}=0,1,2,3$, $m_{ij}=2,3,4,6$), and also the relations $s_i^2=1$.
The same relations of course define the group algebra $\Bbb Z W$, in which the last relation can be written as the quadratic relation $(s_i+1)(s_i-1)=0$. The {\bf Hecke algebra} $H_q(W)$ of $W$ is defined over $\Bbb Z[q^{\frac{1}{2}},q^{-\frac{1}{2}}]$ by the generators $T_i$ satisfying the same braid relations 
$$
T_iT_j...=T_jT_i...,\ i\ne j,
$$
and the deformed quadratic relations
$$
(T_i+1)(T_i-q)=0.
$$ 

For every $w\in W$ we can define the element $T_w=T_{i_1}...T_{i_m}$ 
for every reduced decomposition $w=s_{i_1}....s_{i_m}$. This is independent of the reduced decomposition since any two of them can be related by using only the braid relations. Moreover, it is easy to see that the elements $T_w$ span $H_q(W)$, since 
any non-reduced product of $T_i$ can be expressed via shorter products by using the braid and quadratic relations for $T_i$. Moreover, we have 

\begin{proposition} $T_w,w\in W$ are linearly independent, so they form a basis 
of $H_q(W)$. Thus $H_q(W)$ is a free $\Bbb Z[q^{\frac{1}{2}},q^{-\frac{1}{2}}]$-module of rank $|W|$. 
\end{proposition} 

\begin{proof} 
Let $V$ be the free $\Bbb Z[q^{\frac{1}{2}},q^{-\frac{1}{2}}]$-module with basis 
$X_w,w\in W$. Define a left action of the free algebra with generators $T_i$ on $V$ by 
$$
T_iX_w=X_{s_iw}
$$
if $\ell(s_iw)=\ell(w)+1$ and 
$$
T_iX_w=(q-1)X_w+qX_{s_iw}
$$
if $\ell(s_iw)=\ell(w)-1$. 
We claim that this action factors through $H_q(W)$.  
To show this, define a right action of the same free algebra on $V$ by 
$$
X_wT_i=X_{ws_i}
$$
if $\ell(ws_i)=\ell(w)+1$ and 
$$
X_wT_i=(q-1)X_w+qX_{ws_i}
$$
if $\ell(ws_i)=\ell(w)-1$. It is easy to check by a direct computation 
that these two actions commute: 
\begin{equation}\label{commu}
(T_iX_w)T_j=T_i(X_wT_j).
\end{equation} 
Also the elements $X_1T_w$ clearly span $V$. 
Thus to prove the relations of $H_q(W)$ for the left action, it suffices 
to check them on $X_1$, which is straightforward. 

Since $T_wX_1=X_w$ are linearly independent, it follows that $T_w$ are linearly independent, as claimed.  
\end{proof} 

\begin{exercise} Check identity \eqref{commu}. 
\end{exercise} 

The quadratic relation for $T_i$ implies that it is invertible in the Hecke algebra, with inverse 
$$
T_i^{-1} = q^{-1}(T_i + 1 -q).
$$ 
These inverses satisfy the relation $(T_i^{-1} + 1)(T_i^{-1}-q^{-1}) = 0$ (obtained by multiplying the quadratic relation for $T_i$ by $-T_i^{-2}q^{-1}$), and also the braid relations. It follows that the Hecke algebra has an involutive automorphism $D$ that sends $q^{\frac{1}{2}}$ to $q^{-\frac{1}{2}}$ and each $T_i$ to $T_i^{-1}$. More generally one has $D ( T_w ) = T_{w^{-1}}^{- 1}$. 

\subsection{The Bruhat order} 
Recall that the partial {\bf Bruhat order} on $W$ is defined as follows: $y\le w$ if a reduced decomposition of $y$ can be obtained from a reduced decomposition of $w$ by crossing out some $s_i$; thus 
$y\le w$ implies that $\ell(y)\le \ell(w)$, and if the equality holds then $y=w$. Moreover, 
if $\ell(w)=\ell(y)+1$ then $y<w$ iff $y=y_1y_2$ and $w=y_1s_iy_2$ for some $i$, where 
$\ell(y)=\ell(y_1)+\ell(y_2)$. In this case we say that $w$ {\bf covers} $y$, and 
$y\le w$ iff there exists a sequence $y=x_0<x_1<...<x_m=w$ such that $x_{j+1}$ covers $x_j$ for all $j$ (here $m=\ell(w)-\ell(y)$).  

\begin{exercise} Show that if $y\le w$ then for any dominant $\lambda\in P$, $w\lambda\preceq y\lambda$, and 
the converse holds if $\lambda$ is regular (i.e., $W_\lambda=1$). 
\end{exercise} 

\begin{example} For type $A_1$ the Bruhat order is the covering relation $1<s$. For 
type $A_2$ the covering relations are 
$$
1<s_1,s_2<
s_1s_2,s_2s_1<s_1s_2s_1=s_2s_1s_2.
$$ 
\end{example} 

\subsection{Kazhdan-Lusztig polynomials} 
\begin{theorem} There exist unique polynomials $P_{y,w}\in \Bbb Z[q]$ 
such that 

(a) $P_{y,w}=0$ unless $y\le w$, and $P_{w,w}=1$; 

(b) If $y<w$ then $P_{y,w}$ has degree at most $\frac{\ell(w)-\ell(y)-1}{2}$; 

(c) The elements 
$$
C_w:=q^{-\frac{\ell(w)}{2}}\sum_y P_{y,w}(q)T_y\in H_q(W)
$$
satisfy $D(C_w)=C_w$. 
\end{theorem} 

\begin{proof} 
Let $y=s_{i_1}...s_{i_l}$ be a reduced decomposition of $y$. Then we have 
$$
T_{y^{-1}}^{-1}=\prod_{j=1}^l T_{i_j}^{-1}=q^{-\ell(y)}\prod_{j=1}^l (T_{i_j} + 1 -q).
$$
Thus there exist unique polynomials $R_{x,y}\in \Bbb Z[q]$ such that 
$$
D(T_y)=T_{y^{-1}}^{-1}=\sum_x q^{-\ell(x)}R_{x,y}(q^{-1})T_x,
$$
with $R_{x,y}=0$ unless  $x=y$ (in which case $R_{x,y}(q)=1$) or $\ell(x)<\ell(y)$. It is easy to check that $R_{x,y}$ can be computed using the following recursive rules: for a simple reflection $s$, 
$$
R_{x,y}=R_{sx,sy},\ sx<x,sy<y;
$$
$$
R_{x,y}=(q-1)R_{x,sy}+qR_{sx,sy}, sx>x,sy<y.
$$
(we have $R_{x,1}=\delta_{x,1}$ and for $y\ne 1$ there is always $i$ such that $s_iy<y$). 
This implies by induction in $\ell(y)$ that $R_{x,y}=0$ unless $x\le y$. Indeed, 
if $x':=sx<x,y':=sy<y$ then $R_{x,y}=R_{x',y'}$, so if this is nonzero then by the induction assumption $x'\le y'$, hence $sx'\le sy'$, i.e., $x\le y$.  
On the other hand, if $sx>x$, $sy<y$ and $R_{x,y}\ne 0$ then either $R_{x,sy}\ne 0$ or $R_{sx,sy}\ne 0$, hence either 
$x\le sy$ or $sx\le sy$. But 
each one of the inequalities $x\le sy$, $sx\le sy$
implies $x\le y$. 

We also see by induction that $\deg R_{x,y}\le \ell(y)-\ell(x)$. 

Now it is easy to compute that the condition that $D(C_w)=C_w$ is equivalent to the recursion 
$$ 
q^{\frac {\ell (w)-\ell (x)}{2}}P_{x,w}(q^{-1})-
q^{\frac{\ell (x)-\ell (w)}{2}}
P_{x,w}(q)=
$$
$$
\sum _{x<y}
(-1)^{\ell (x)+\ell (y)}
q^{\frac{-\ell(x)+2\ell (y)-\ell (w)}{2}}R_{x,y}(q^{-1})P_{y,w}(q).
 $$
 We can now see that this recursion has a unique solution $P_{x,w}$ with required properties, as
the two terms on the left are supposed to be polynomials in $q^{\frac{1}{2}}$ and $q^{-\frac{1}{2}}$ without constant terms. 
\end{proof} 

The elements $C_w$ form a basis of the Hecke algebra called the {\bf Kazhdan-Lusztig} basis, and the polynomials 
$P_{y,w}$ are called the {\bf Kazhdan-Lusztig polynomials}. 

\subsection{Kazhdan-Lusztig conjecture} 
The Kazhdan-Lusztig conjecture (now a theorem) is: 

\begin{theorem} (i) $P_{y,w}$ has non-negative coefficients. 
 
(ii) The multiplicity $[M_{y\bullet \lambda}:L_{w\bullet \lambda}]$ 
equals $P_{y,w}(1)$. 
\end{theorem} 

The polynomials $P_{y,w}$ have the property that if $y\le w$ then $P_{y,w}(0)=1$, so if in addition $\ell(w)-\ell(y)\le 2$ then $P_{y,w}(q)=1$ (indeed, it has to be a polynomial of degree $0$). Also if $w=w_0$ then
$P_{y,w}=1$ for all $y$. 

\begin{example} For type $A_2$ ($\g=\mathfrak{sl}_3$) we have the following decompositions in the Grothendieck group of $\mathcal O_{\chi_\lambda}$ (where we 
abbreviate $s_{i_1}...s_{i_k}\cdot \lambda$ as $i_1...i_k$: 
\begin{gather*}
M_{121}=L_{121}\\
M_{12}=L_{12}+L_{121}\\
M_{21}=L_{21}+L_{121}\\
M_1=L_1+L_{12}+L_{21}+L_{121}\\
M_2=L_2+L_{12}+L_{21}+L_{121}\\
M_\emptyset=L_\emptyset+L_1+L_2+L_{12}+L_{21}+L_{121}.
\end{gather*}
\end{example} 

\begin{exercise} Compute the Cartan matrix of the category $\mathcal O_{\chi_\lambda}$ for $\g=\mathfrak{sl}_3$ for regular weights $\lambda$. 
\end{exercise} 

\section{\bf Projective functors - I} 

\subsection{Projective functors and projective $\theta$-functors.} 

Let $\Rep(\g)_f$ be the category of $\g$-modules in which the center $Z(\g)$ of $U(\g)$ acts through its finite-dimensional quotient. We have 
$$
\Rep(\g)_f=\oplus_{\theta\in \h^*/W}\Rep(\g)_\theta,
$$ 
where 
$\Rep(\g)_\theta$ is the category of modules with generalized infinitesimal character $\theta$. 
Recall that for a finite-dimensional $\g$-module $V$ we have an exact functor
$F_V: \Rep(\g)\to \Rep(\g)$ given by $X\mapsto V\otimes X$ (e.g., $F_{\Bbb C}={\rm Id}$), and that if 
$M$ has infinitesimal character $\chi_\lambda$ then 
$$
F_V(M)=(V\otimes U_{\chi_\lambda})\otimes_{U_{\chi_\lambda}}M.
$$
Recall also that the infinitesimal characters occurring in the left $\g$-module $V\otimes U_{\chi_\lambda}$ are $\chi_{\lambda+\beta}$ for $\beta\in P(V)$ (Corollary \ref{cechar}); thus 
the infinitesimal characters occurring in $F_V(M)$ belong to the same set. 
It follows that 
$$
F_V(\Rep(\g)_{\chi_\lambda})\subset \oplus_{\beta\in P(V)}\Rep(\g)_{\chi_{\lambda+\beta}},
$$ 
hence $F_V$ maps $\Rep(\g)_f$ to itself. Finally note that 
$F_{V^*}$ is both right and left adjoint to $F_V$. 

\begin{definition} A {\bf projective functor} is an endofunctor of $\Rep(\g)_f$ 
which is isomorphic to a direct summand in $F_V$ for some $V$. 
\end{definition} 

\begin{example} For $\theta\in \h^*/W$ let $\Pi_\theta: \Rep(\g)_f\to \Rep(\g)_\theta$ be the projection. Then ${\rm Id}=\oplus_{\theta\in \h^*/W}\Pi_\theta$, hence $\Pi_\theta$ is a projective functor. 
\end{example}

It is easy to see that projective functors form a category which is closed under taking compositions, direct summands and finite direct sums, and every projective functor admits a left and right adjoint which are also projective functors (we'll see that they are isomorphic). It is also clear that every projective functor $F$ has a decomposition 
$$
F=\oplus_{\theta,\chi\in \h^*/W}\Pi_\chi\circ F\circ \Pi_\theta.
$$
Finally, projective functors obviously map category $\mathcal O$ to itself and by Proposition \ref{proje1}(i) send projectives of this category to projectives. 

For an infinitesimal character $\theta: Z(\g)\to \Bbb C$ let $\Rep(\g)_\theta^n\subset \Rep(\g)_
\theta$ be the subcategory of modules annihilated by $(\Ker \theta)^n$. In other words, 
$\Rep(\g)_\theta^n$ is the category of left modules over the algebra 
$$
U_\theta^{(n)}:=U(\g)/(\Ker \theta)^nU(\g).
$$ 
Every $M\in\Rep(\g)_\theta$ is the nested union of submodules $M_n\subset M$
of elements killed by $({\rm Ker}\theta)^n$, and $M_n\in \Rep(\g)_\theta^n$. Note that $U_\theta^{(1)}=U_\theta$ and 
$\Rep(\g)_\theta^1$ is the category of modules with infinitesimal character $\theta$. 

For a projective functor $F$ denote by $F(\theta)$ the restriction of $F$ to $\Rep(\g)_\theta^1$. 

\begin{definition} A {\bf projective $\theta$-functor} is a direct summand in 
$F_V(\theta)$.
\end{definition} 

For example, if $F$ is a projective functor then $F(\theta)$ is a projective $\theta$-functor. 

\begin{theorem}\label{projiso} Let $F_1,F_2$ be projective $\theta$-functors for $\theta=\chi_\lambda$. Let 
$$
i_\lambda: \Hom(F_1,F_2)\to \Hom(F_1(M_{\lambda-\rho}),F_2(M_{\lambda-\rho})).
$$
Then $i_\lambda$ is an isomorphism. 
\end{theorem} 

\begin{proof} It suffices to assume $F_j=F_{V_j}(\theta)$, $j=1,2$. Let $V=V_1^*\otimes V_2$. Then $\Hom(F_1,F_2)=\Hom({\rm Id}(\theta),F_V(\theta))$ and 
$$
\Hom(F_1(M_{\lambda-\rho}),F_2(M_{\lambda-\rho}))=\Hom_\g(M_{\lambda-\rho},V\otimes M_{\lambda-\rho}).
$$ 
Thus it suffices to show that the natural map 
$$
i_\lambda: \Hom({\rm Id}(\theta), F_V(\theta))\to \Hom_\g(M_{\lambda-\rho},V\otimes M_{\lambda-\rho})
$$
is an isomorphism. 

Recall that for associative unital algebras $A,B$, a right exact functor 
$F: A-{\rm mod}\to B-{\rm mod}$ has the form 
$F(X)=F(A)\otimes_A X$, where $F(A)$ is the corresponding 
$(B,A)$-bimodule. Thus if $F_1,F_2$ are two such functors then 
$\Hom(F_1,F_2)\cong \Hom_{(B,A)-{\rm bimod}}(F_1(A),F_2(A))$. 
Applying this to $A=U_\theta$ and $B=U(\g)$, we get 
$$
\Hom({\rm Id}(\theta),F_V(\theta))=\Hom_{(U(\g),U_\theta)-{\rm bimod}}(U_\theta,V\otimes U_\theta)=(V\otimes U_\theta)^{\g_{\rm ad}}.
$$
Moreover, upon this identification the map $i_\lambda$ becomes the natural map 
$$
i_\lambda: (V\otimes U_{\chi_\lambda})^{\g_{\rm ad}}\to \Hom(M_{\lambda-\rho},V\otimes M_{\lambda-\rho})^{\g_{\rm ad}}.
$$
But this map is an isomorphism by the Duflo-Joseph theorem, as it is obtained by restricting the Duflo-Joseph isomorphism
$$
U_{\chi_\lambda}\cong \Hom_{\rm fin}(M_{\lambda-\rho},M_{\lambda-\rho})
$$ 
to the multiplicity space of $V^*$. 
\end{proof} 

\subsection{Lifting projective $\theta$-functors.} 

\begin{proposition}\label{lif} (i) If $F_1,F_2$ are projective functors then 
every morphism $\phi: F_1(\theta)\to F_2(\theta)$ lifts to a morphism 
$\widehat\phi: F_1|_{\Rep(\g)_\theta}\to F_2|_{\Rep(\g)_\theta}$. 

(ii) If $F_1=F_2$ and $\phi^2=\phi$ then we can choose $\widehat \phi$ so that $\widehat\phi^2=\widehat\phi$.

(iii) If $\phi$ is an isomorphism then so is $\widehat \phi$. 
\end{proposition} 

\begin{proof} (i) It suffices to show that 
there exist morphisms 
$$
\phi_n: F_1|_{\Rep(\g)_\theta^n}\to F_2|_{\Rep(\g)_\theta^n}
$$ 
such that $\phi_n$ restricts to $\phi_{n-1}$ and $\phi_1=\phi$; then $\widehat \phi$ is the projective limit of $\phi_n$. As before, we may assume without loss of generality that $F_1={\rm Id}$ and $F_2=F_V$. As explained in the proof of Theorem \ref{projiso}, we have 
$$
\Hom(F_1|_{\Rep(\g)_\theta^n},F_2|_{\Rep(\g)_\theta^n})=
$$
$$
\Hom_{(U(\g),U_\theta^{(n)})-{\rm bimod}}(U_\theta^{(n)},V\otimes U_\theta^{(n)})=
(V\otimes U_\theta^{(n)})^{\g_{\rm ad}}. 
$$
This implies the statement, as the map $U_\theta^{(n)}\to U_\theta^{(n-1)}$ is onto and 
$V\otimes U_\theta^{(n)}$ is a semisimple $\g_{\rm ad}$-module. 

(ii) Let $F$ be a direct summand in $F_V$. Let $p: F_V\to F_V$ be the projection to $F$. Let $A:=\End(F(\theta))=p{\rm End}(F_V(\theta))p$ and $\phi\in A$. 
Let $F^n(\theta)$ be the restriction of $F$ to $\Rep(\g)_\theta^n$, so 
that $F^1(\theta)=F(\theta)$. We have 
$$
A_n:={\rm End}(F^n(\theta))=p{\rm End}(F_V^n(\theta))p=p({\rm End V}\otimes U_\theta^{(n)})^{\g_{\rm ad}} p.
$$
So we have a chain of surjective homomorphisms 
$$
...\to A_n\to A_{n-1}\to ...\to A_1=A 
$$
and our job is to show that $\phi$ admits a chain of lifts 
$$
...\mapsto \phi_n\mapsto \phi_{n-1}\mapsto...\mapsto \phi_1=\phi
$$
such that $\phi_n\in A_n$ and $\phi_n^2=\phi_n$. 

To this end, note that the kernel $I$ of the surjection $A_n\to A_{n-1}$ satisfies 
$I^2=0$, so $I$ is a left and right module over $A_n/I=A_{n-1}$. So we can construct the desired chain of lifts by induction in $n$ as follows. 
Pick any lift $e_*$ of $e_0:=\phi_{n-1}$. 
Then $e_*-e_*^2=a\in I$, and $e_0a=ae_0$.
We look for an idempotent $e$ in the form $e=e_*+b$, $b\in I$.
The equation $e^2=e$ is then equivalent to
$$
e_0b+be_0-b=a.
$$
Set $b=(2e_0-1)a$. Then
$$
e_0b+be_0-b=2e_0a+(1-2e_0)a=a,
$$
as desired. Now we can set $\phi_n=e$.  

(iii) If $\phi: F_1(\theta)\to F_2(\theta)$ is an isomorphism then it has the inverse $\psi: F_2(\theta)\to F_1(\theta)$ such that
$\phi\circ \psi=1$, $\psi\circ\phi=1$. Let $\widehat \phi=(\phi_n)$ be a lift of $\phi$. Our job is to show that $\phi_n$ are isomorphisms for all $n$, which yields (iii). 
We prove it by induction in $n$. 

The base is trivial, so we just need to do the induction step from $n-1$ to $n$. By the induction assumption, $\phi_{n-1}$ is invertible with $\phi_{n-1}^{-1}=\psi_{n-1}$. 
Let $\psi_n$ be a lift of $\psi_{n-1}$ and consider the composition $\psi_n\circ \phi_n$ in the corresponding algebra $A_n$. Let $I$ be the kernel of the map $A_n\to A_{n-1}$. Then $\psi_n\circ \phi_n=1+a$ where $a\in I$. Since $I^2=0$, setting $\psi_n':=(1-a)\circ \psi_n$, we get $\psi_n'\circ \phi_n=1$. Similarly we can construct $\psi_n''$ such that $\phi_n\circ \psi_n''=1$. Thus  $\psi_n'=\psi_n''$ is the inverse of $\phi_n$. 
This completes the induction step. 
\end{proof} 

\begin{corollary}\label{projfuu} (i) Let $F_1,F_2$ be projective functors. Then:
any isomorphism 
$F_1(M_{\lambda-\rho})\cong F_2(M_{\lambda-\rho})$ lifts to an isomorphism 
$$
F_1|_{\Rep(\g)_{\chi_\lambda}}\to F_2|_{\Rep(\g)_{\chi_\lambda}};
$$ 

(ii) Let $F$ be a projective functor. Then any decomposition $F(M_{\lambda-\rho})=\oplus_i M_i$ can be lifted 
to a decomposition $F=\oplus_i F_i$ where $F_i$ are projective functors 
and $F_i(M_{\lambda-\rho})=M_i$; 

(iii) Every projective $\theta$-functor is of the form $F(\theta)$ for a projective functor $F$. 
\end{corollary} 

\begin{proof} 
(i) follows from Proposition \ref{lif}(i),(iii) and Theorem \ref{projiso}. 

(ii) follows from Proposition \ref{lif}(ii). 

To prove (iii), let $H$ be a projective $\theta$-functor, so 
$H\oplus H'=F_V(\theta)$. Thus $H(M_{\lambda-\rho})\oplus H'(M_{\lambda-\rho})=F_V(M_{\lambda-\rho})$. 
So by (ii) there is a projective functor $F$ with 
$ F(\theta)(M_{\lambda-\rho}) \cong F(M_{\lambda-\rho})\cong H(M_{\lambda-\rho})$. Thus $H\cong F(\theta)$ by Theorem \ref{projiso}.  
\end{proof} 

\subsection{Decomposition of projective functors} 

\begin{proposition}\label{projfuu2} (i) Each projective functor $F$ is a direct sum of indecomposable projective functors. Moreover, for $F\circ \Pi_\theta$ this sum is finite.  

(ii) If $F=F\circ \Pi_{\chi_\lambda}$ for dominant $\lambda$ is an indecomposable projective functor then $F(M_{\lambda-\rho})=P_{\mu-\rho}$ for some $\mu\in \h^*$. 
\end{proposition} 

\begin{proof}  
(i) We have $F=\oplus_{\theta\in \h^*/W} F\circ \Pi_\theta$, so it suffices to show 
the statement for $F\circ \Pi_\theta$. Let $\theta=\chi_\lambda$, and consider 
$F\circ \Pi_\theta(M_{\lambda-\rho})\in \mathcal O$. Let us write this object 
as a finite direct sum of indecomposables, $\oplus_{i=1}^N M_i$. Then by Corollary \ref{projfuu}(ii) we get a decomposition 
$F\circ \Pi_\theta=\oplus_{i=1}^N F_i$, and all $F_i$ are indecomposable.  
 
(ii) Since $F$ is indecomposable and $M_{\lambda-\rho}$ is projective, $F(M_{\lambda-\rho})$ is indecomposable and projective, so the statement follows. 
\end{proof}

\section{\bf Projective functors - II}

\subsection{The Grothendieck group of $\O$}

The Grothendieck group $K(\O)$ of $\O$ is freely spanned by the classes of simple modules $[L_{\lambda-\rho}]$ or, more conveniently, by the classes of Verma modules $[M_{\lambda-\rho}]$, which we'll denote $\delta_\lambda$; so it is a basis of $K(\O)$. 
Put an inner product on $K(\O)$ by declaring this basis to be orthonormal. 
Note that if $P$ is projective then 
$$
([P],[M])=\dim\Hom(P,M).
$$ 
Indeed, in this case $\dim\Hom(P,M)$ is a linear function of $[M]$, and  
for $M=M_\mu$ by the BGG reciprocity we have: 
$$
\dim \Hom(P_\lambda,M_\mu)=d_{\mu\lambda}=d_{\lambda\mu}^*=(\sum_\nu d_{\lambda\nu}^*\delta_{\nu+\rho},\delta_{\mu+\rho})=([P_\lambda],[M_\mu]).  
$$

Since every projective functor $F$ is exact, it defines an endomorphism $[F]$ of $K(\O)$. For example, 
$$
[F_V]\delta_\lambda=\sum_{\beta}m_V(\beta)\delta_{\lambda+\beta},
$$
where $m_V(\beta)$ is the weight multiplicity of $\beta$ in $V$. 
Clearly $[F_1\oplus F_2]=[F_1]+[F_2]$ and   
$[F_1\circ F_2]=[F_1][F_2]$. 

\begin{theorem} (i) If $F_1,F_2$ are projective functors with $[F_1]=[F_2]$ 
then $F_1\cong F_2$.

(ii) If $(F,F^\vee)$ are an adjoint pair of projective functors then 
$[F]$ is adjoint to $[F^\vee]$ under the inner product on $K(\O)$.   

(iii) For a projective functor $F$, its left and right adjoint are isomorphic. 
\end{theorem} 

\begin{proof} (i) By Corollary \ref{projfuu}, to prove (i), it suffices to show that 
$$
F_1(M_{\lambda-\rho})\cong F_2(M_{\lambda-\rho})
$$ 
for all dominant $\lambda$. These objects are projective, so it is enough to check that they have the same character (or define the same element of $K(\O)$). This implies the claim. 

(ii) We need to show that $([F]x,y)=(x,[F^\vee]y)$. It suffices to take 
$x=[P]$ for projective $P$ and $y=[M]$. Then, since $F(P)$ is projective, we have 
$$
([F][P],[M])=([F(P)],[M])=\dim\Hom(F(P),M)=
$$
$$
\dim\Hom(P,F^\vee(M))=
([P],[F^\vee(M)])=
([P],[F^\vee][M]),
$$
as claimed. 

(iii) follows from (i),(ii). 
\end{proof} 

\subsection{$W$-invariance}

We have an action of the Weyl group $W$ on $K(\O)$ by $w\delta_\lambda:=\delta_{w\lambda}$. 

\begin{theorem}\label{commuW} If $F$ is a projective functor then $[F]$ commutes with $W$
on $K(\O)$. 
\end{theorem} 

\begin{proof} We may assume that $F=\Pi_\chi\circ F\circ \Pi_\theta$ 
for $\chi,\theta\in \h^*/W$ and $F$ is indecomposable. Let $\lambda$ be a dominant weight such that $\theta=\chi_\lambda$. Define 
$$
S=\lbrace\mu\in \lambda+P: \chi_\mu=\chi\rbrace.
$$ 
Let us say that $\lambda$ {\bf dominates} $\chi$ if for every $\mu\in S$ we have $\lambda-\mu\in P_+$. 

\begin{lemma}\label{auxx} When $\lambda$ dominates $\chi$ then 

(i) Theorem \ref{commuW} holds;  

(ii) For each $\mu\in S$ there exists an indecomposable projective functor $F_\mu$ sending $M_{\lambda-\rho}$ to $P_{\mu-\rho}$.
\end{lemma}  

\begin{proof} 
(i) For a finite-dimensional $\g$-module $V$, let $G_V:=\Pi_\chi\circ F_V\circ \Pi_\theta$. Since the character of $V$ is $W$-invariant, 
$[F_V]$ commutes with $W$, hence so does $[G_V]$. Thus 
it suffices to show that $[F]$ is an integer linear combination of $[G_V]$ for various $V$. 

By Proposition \ref{projfuu2}(ii), $F(M_{\lambda-\rho})=P_{\mu-\rho}$, where $\mu\in S$. 
Let $\beta:=\lambda-\mu$. By our assumption, $\beta\in P_+$. 
Define $n(\beta):=(\beta,2\rho^\vee)$, a non-negative integer. 
We will prove the required statement by induction in $n(\beta)$. 

The base of induction is $n(\beta)=0$, hence $\beta=0$ and $\mu=\lambda$. 
So $F(M_{\lambda-\rho})=P_{\lambda-\rho}=M_{\lambda-\rho}$. 
This implies that $F=\Pi_\theta$, so $[F]$ clearly commutes with $W$. 

So it remains to justify the induction step. Let $L:=L_\beta^*$, a finite-dimensional $\g$-module. 
Consider the decomposition of the functor $G_L$ 
into indecomposables (which we have shown to exist in Proposition \ref{projfuu2}(ii)): 
$G_L=\oplus_j F_{\nu_j}$, where $\nu_j\in S$ and 
$F_{\nu_j}(M_{\lambda-\rho})=P_{\nu_j-\rho}$ (this direct sum may contain repetitions). So $G_L(M_{\lambda-\rho})=\oplus_j P_{\nu_j-\rho}$. Thus 
$$
[G_L]\delta_\lambda=\sum_{j,\gamma}d_{\nu_j,\gamma}^*\delta_\gamma=
\sum_{j,\gamma}d_{\gamma,\nu_j}\delta_\gamma=\sum_j \delta_{\nu_j}+\sum_{j,\gamma>\nu_j}d_{\gamma,\nu_j}\delta_{\gamma}. 
$$
On the other hand, 
$$
[G_L]\delta_\lambda=[G_L(M_{\lambda-\rho})]=[\Pi_\chi(L\otimes M_{\lambda-\rho})]=
[\Pi_\chi]\sum_\eta m_{L}(\eta)\delta_{\lambda+\eta}=
$$
$$
[\Pi_\chi]\sum_\eta m_{L_\beta}(\eta)\delta_{\lambda-\eta}=
\sum_{\eta: \chi_{\lambda-\eta}=\chi} m_{L_\beta}(\eta)\delta_{\lambda-\eta}=\sum_{\nu: \chi_{\nu}=\chi} m_{L_\beta}(\beta+\mu-\nu)\delta_{\nu}=
$$
$$
\delta_\mu+\sum_{\nu>\mu: \chi_{\nu}=\chi} m_{L_\beta}(\beta+\mu-\nu)\delta_{\nu}. 
$$
These two formulas for $[G_L]\delta_\lambda$ jointly imply that $
\nu_j\ge \mu$ for all $j$, and only one of them equals $\mu$, i.e.,  
\begin{equation}\label{Fmu}
G_L=F_\mu\oplus\bigoplus_{\nu\in S,\nu>\mu}c_{\nu\mu}F_\nu
\end{equation} 
for some constants $c_{\nu\mu}\in \Bbb Z_{\ge 0}$. But if $\nu>\mu$ then $n(\lambda-\nu)<n(\lambda-\mu)$, so by the induction assumption $[F_\nu]$ for all $\nu>\mu$ in this sum are linear combinations of $[G_V]$ for various $V$. Thus so is $F_\mu$. But $F(M_{\lambda-\rho})=F_\mu(M_{\lambda-\rho})$, so $F\cong F_\mu$ and the induction step follows.

(ii) The functor $F_\mu$ from \eqref{Fmu} has the desired property. 
\end{proof} 

Now we are ready to prove the theorem in the general case. So $\lambda$ no longer needs to dominate $\chi$. However, for sufficiently large integer $N$, the weight
$\lambda+N\rho$ dominates both $\chi$ and $\theta$. Let $\theta_N:=\chi_{\lambda+N\rho}$. We have shown in Lemma \ref{auxx}(ii) that there exists 
an indecomposable projective functor $G=\Pi_\theta\circ G\circ \Pi_{\theta_N}$ such that $G(M_{\lambda+(N-1)\rho})=P_{\lambda-\rho}=M_{\lambda-\rho}$. Moreover, by Lemma \ref{auxx}(i), $W$ commutes with both $[G]$ and $[F\circ G]=[F][G]$. Thus for $w\in W$, 
$$
w[F]\delta_\lambda=w[F][G]\delta_{\lambda+N\rho}=[F][G]w\delta_{\lambda+N\rho}=[F]w[G]\delta_{\lambda+N\rho}=
[F]w\delta_\lambda=[F]\delta_{w\lambda}. 
$$
So for $u\in W$, 
$$
u[F]\delta_{w\lambda}=uw[F]\delta_\lambda=[F]uw\delta_\lambda=
[F]u\delta_{w\lambda},
$$
i.e., 
$$
u[F]\delta_\mu=[F]u\delta_\mu
$$
for all $\mu\in \h^*$, as claimed. 
\end{proof} 

\begin{lemma}\label{stab1} Let $\lambda\in \h^*$ be dominant and $\phi,\psi\in \lambda+P$,  $\psi\preceq\phi$. Then $(\lambda-\phi)^2\le (\lambda-\psi)^2$, and if $(\lambda-\phi)^2=(\lambda-\psi)^2$ then $\psi\in W_\lambda\phi$. 
\end{lemma} 

\begin{proof} Consider the subgroup $W_{\lambda+Q}\subset W$. By Proposition \ref{stab}, it is the Weyl group of a root system $R'\subset R$. Let us first prove the result when $\psi<_\alpha \phi$, $\alpha\in R$, i.e., 
$\psi=s_\alpha\phi$, $\psi\ne \phi$. 
Then $\alpha\in R'$ and thus by Corollary \ref{domchar} 
$$
(\lambda,\alpha^\vee)=a\in \Bbb Z_{\ge 0},\ (\phi,\alpha^\vee)=-(\psi,\alpha^\vee)=b\in \Bbb Z_{\ge 0}.
$$ 
We have $\lambda=\frac{1}{2}a\alpha+\lambda'$, 
$\phi=\frac{1}{2}b\alpha+\phi'$, $\psi=-\frac{1}{2}b\alpha+\phi'$. where $\lambda',\phi'$ are orthogonal to $\alpha$. Thus 
$$
(\lambda-\psi)^2-(\lambda-\phi)^2=((\tfrac{a+b}{2})^2-(\tfrac{a-b}{2})^2)\alpha^2=ab\alpha^2.
$$
So this is $\ge 0$, and if it is zero then either $b=0$, in which case $\phi=\psi$ and there is nothing to prove, or $a=0$, so $s_\alpha\lambda=\lambda$ and $s_\alpha\in W_\lambda$, as claimed.  

Now let us consider the general case. By assumption, there is a chain 
$$
\psi=\psi_m<_{\alpha^m} \psi_{m-1}...<_{\alpha^1}\psi_0=\phi,
$$
where $\alpha^1,...,\alpha^m$ are positive roots of $R$. 
Thus, as we've shown,  
$$
(\lambda-\psi_{i})^2\ge (\lambda-\psi_{i-1})^2
$$ 
for all $i\ge 1$, so $(\lambda-\phi)^2\le (\lambda-\psi)^2$. Moreover, if $(\lambda-\phi)^2=(\lambda-\psi)^2$ then $(\lambda-\psi_{i-1})^2=(\lambda-\psi_{i})^2$ for all $i\ge 1$ so $\psi_{i-1}\in W_\lambda\psi_i$, hence $\psi\in W_\lambda\phi$. 
\end{proof} 

\begin{remark} The last statement of Lemma \ref{stab1} fails 
if the partial order $\preceq$ is replaced with $\le$. For example, take $R=A_3$ and 
$\psi=(0,3,1,2)$, $\phi=(1,2,3,0)$, as in Remark \ref{falseingen} (so $\psi<\phi$ but $\psi\nprec\phi$), and let 
$\lambda:=(1,1,0,0)$. Then $(\lambda-\phi)^2=(\lambda-\psi)^2=10$, but 
$W_\lambda=\langle (12),(34)\rangle$, so $\psi\notin W_\lambda\phi$. 
\end{remark} 

\subsection{Classification of indecomposable projective functors}

Denote by $\Xi_0$ the set of pairs $(\lambda,\mu)$ of weights in $\h^*$ 
such that $\lambda-\mu\in P$, and let $\Xi:=\Xi_0/W$. So in general an element $\xi\in \Xi$ can be represented by more than one pair. Let us say that the pair $(\mu,\lambda)$ representing $\xi$ is {\bf proper} if 
$\lambda$ is dominant and $\mu$ is a minimal element of $W_\lambda\mu$ with respect to the partial order $\preceq$ (where $W_\lambda$ is the stabilizer of $\lambda$ in $W$). It is clear that any $\xi$ has a proper representative. This representative is not unique in general, but 
for every dominant $\lambda$ in the $W$-orbit of the second coordinate of $\xi$, there is a unique 
$\mu$ such that $(\mu,\lambda)$ is a proper representation of $\xi$ (indeed, $W_\lambda\mu$ has a unique minimal element).

\begin{theorem}\label{classpro} For any $\xi\in \Xi$ there exists an indecomposable projective functor $F_\xi$ such that $F_\xi(M_{\nu-\rho})=0$ if $\chi_\nu\ne \chi_\lambda$ and 
$F_\xi(M_{\lambda-\rho})=P_{\mu-\rho}$ for any proper representation $(\mu,\lambda)$ of $\xi$. The assignment 
$\xi\mapsto F_\xi$ is a bijection between $\Xi$ and the set of isomorphism classes of indecomposable projective functors. 
\end{theorem} 

\begin{proof} For a projective functor $F$ let 
$$
a_F(\mu,\lambda):=(\delta_\mu,[F]\delta_\lambda)
$$
be the matrix coefficients of $[F]$. If $\lambda$ is dominant then 
$F(M_{\lambda-\rho})$ is projective, so $a_F(\mu,\lambda)\ge 0$ for all $\mu\in \h^*$. Since by Theorem \ref{commuW} $[F]$ commutes with $W$, this holds for all $\lambda\in \h^*$. 

Let $S(F):=\lbrace (\mu,\lambda)\in \h^*\times \h^*: a_F(\mu,\lambda)>0\rbrace$. Since $a_F(\mu,\lambda)\ge 0$, if $F=\oplus_i F_i$ then $S(F)=\cup_i S(F_i)$. Also it is clear that $S(F_V)\subset \Xi_0$. 
It follows that $S(F)\subset \Xi_0$ for any $F$, so for $(\mu,\lambda)\in S(F)$ 
we have $\lambda-\mu\in P$. 

Let $S_*(F)$ be the set of elements 
of $S(F)$ for which $(\lambda-\mu)^2$ has maximal value (it is clear that $(\lambda-\mu)^2$ is bounded on $S(F)$, so $S_*(F)$ is nonempty if $F\ne 0$). Since by Theorem \ref{commuW} $[F]$ commutes with $W$, both $S(F)$ and $S_*(F)$ are $W$-invariant. 

We claim that if $F$ is indecomposable, then $S_*(F)$ is a single $W$-orbit. 
More specifically, recall that $F=F\circ \Pi_{\chi_\lambda}$ for some dominant $\lambda$ and $F(M_{\lambda-\rho})=P_{\mu-\rho}$ 
for some $\mu$. 

\begin{lemma}\label{prope} In this case $S_*(F)=\xi:=W(\mu,\lambda)$
and $(\mu,\lambda)$ is a proper representation of $\xi$.  
\end{lemma}

\begin{proof} It suffices to check that if $(\phi,\lambda)\in S_*(F)$ then 
$\phi\in W_\lambda\mu$ and $\mu\preceq \phi$. So let $(\phi,\lambda)\in S_*(F)$. Since $F$ is indecomposable, 
$\chi_\mu=\chi_\phi$, so there exists $w\in W$ such that $\mu=w\phi$. Moreover, by Theorem \ref{BGGth},  
$$
[P_{\mu-\rho}]=\sum_{\mu\preceq\eta}d_{\mu\eta}^*\delta_\eta,
$$
we get that $\mu\preceq\phi$. Thus we may apply Lemma \ref{stab1} with $\psi=\mu$. It follows that $(\lambda-\phi)^2\le (\lambda-\mu)^2$. 
But by the definition of $S_*(F)$, we have $(\lambda-\phi)^2\ge (\lambda-\mu)^2$. Thus $(\lambda-\phi)^2=(\lambda-\mu)^2$. Then Lemma \ref{stab1} implies that $\phi\in W_\lambda\mu$, as claimed.    
\end{proof} 

Thus to every indecomposable projective functor $F$ we have assigned 
$\xi=S_*(F)/W\in \Xi$. If $(\mu,\lambda)$ is a proper representation of $\xi$ then it follows that $F(M_{\lambda-\rho})=P_{\mu-\rho}$, so $F$ is completely determined by $\xi$ by Corollary \ref{projfuu}. It remains to show that any $\xi\in \Xi$ is obtained in this way. To this end, let $\xi=W(\mu,\lambda)$ (a proper representation), and let $V$ be a finite-dimensional $\g$-module with extremal weight $\mu-\lambda$. 
Then $(\mu-\lambda)^2\ge \beta^2$ for any weight $\beta$ of $V$, so 
$(\mu,\lambda)\in S_*(F_V)$. This implies that $(\mu,\lambda)\in S_*(F)$ for some indecomposable direct summand $F$ of $F_V$. Since $S_*(F)/W$ consists of one element, this $F$ must correspond to the element $\xi$.  
\end{proof} 

\section{\bf Applications of projective functors - I}

\subsection{Translation functors} \label{trafu}

Let $\theta,\chi\in \h^*/W$ and $V$ be a finite-dimensional irreducible $\g$-module. Write 
$F_{\chi,V,\theta}$ for the projective functor $\Pi_\chi\circ F_V\circ \Pi_\theta$, and let us view it as a functor $\Rep(\g)_\theta\to \Rep(\g)_\chi$. 
 
Pick dominant weights $\lambda,\mu\in \h^*$ such that $\theta=\chi_\lambda,\chi=\chi_\mu$,
and $\lambda-\mu\in P$ (this can be done if $F_{\chi,V,\theta}\ne 0$, which we will assume). 

\begin{theorem}\label{equi} If $W_\lambda=W_\mu$ and $V$ has extremal weight $\mu-\lambda$ then 
$F_{\chi,V,\theta}: \Rep(\g)_\theta\to \Rep(\g)_\chi$
is an equivalence of categories. A quasi-inverse 
equivalence is given by the functor $F_{\theta,V^*,\chi}$. 
\end{theorem} 

\begin{proof} It suffices to show that 
$$
F_{\chi,V,\theta}(M_{\lambda-\rho})=M_{\mu-\rho},\ F_{\theta,V^*,\chi}(M_{\mu-\rho})=M_{\lambda-\rho}.
$$ 
Indeed, then 
$$
F_{\theta,V^*,\chi}\circ F_{\chi,V,\theta}(M_{\lambda-\rho})=M_{\lambda-\rho},\ F_{\chi,V,\theta}\circ F_{\theta,V^*,\chi}(M_{\mu-\rho})=M_{\mu-\rho},  
$$ 
so
$$
F_{\theta,V^*,\chi}\circ F_{\chi,V,\theta}\cong {\rm Id}_{\Rep(\g)_\theta},\  
F_{\chi,V,\theta}\circ F_{\theta,V^*,\chi}\cong {\rm Id}_{\Rep(\g)_\chi},
$$ 
i.e., $F_{\chi,V,\theta}, F_{\theta,V^*,\chi}$ are mutually quasi-inverse equivalences. 

We only prove the first statement, the second one being similar.  
We have 
$$
F_{\chi,V,\theta}(M_{\lambda-\rho})=\Pi_\chi(V\otimes M_{\lambda-\rho}).
$$
By Corollary \ref{exti3}(i), $V\otimes M_{\lambda-\rho}$ has a standard filtration 
whose composition factors are $M_{\lambda+\beta-\rho}$ 
where $\beta$ is a weight of $V$. 
The only ones among them that survive the application of $\Pi_\chi$ are those for which $\chi_{\lambda+\beta}=\chi_\mu$, i.e., $\lambda+\beta=w\mu$ for some $w\in W$. So $w\mu\preceq\mu$ (as $\mu$ is dominant). Thus, applying Lemma \ref{stab1} with 
$\phi=\mu,\psi=w\mu$, we get 
$$
(\lambda-\mu)^2\le (\lambda-w\mu)^2=\beta^2.
$$ 
On the other hand, since 
$\mu-\lambda$ is an extremal weight of $V$, we have 
$(\lambda-\mu)^2\ge \beta^2$. It follows that 
$(\lambda-\mu)^2=\beta^2=(\lambda-w\mu)^2$. Thus by Lemma \ref{stab1} we may choose 
$w\in W_\lambda$. But since $W_\lambda\subset W_\mu$, it follows 
that $w\mu=\mu$, so $\beta=\mu-\lambda$. Since the weight multiplicity 
of an extremal weight is $1$, it follows that $F_{\chi,V,\theta}(M_{\lambda-\rho})=M_{\mu-\rho}$, as claimed. 
\end{proof} 

Theorem \ref{equi} shows that for dominant $\lambda$ the category 
$\Rep(\g)_{\chi_\lambda}$ depends (up to equivalence)  only on the coset $\lambda+P$ and the subgroup $W_\lambda\subset W$. In view of Theorem \ref{equi}, the functors $F_{\chi,V,\theta}$ are called {\bf translation functors} (as they translate between different infinitesimal characters). 

\begin{remark} Suppose we only have $W_\lambda\subset W_\mu$ instead of $W_\lambda=W_\mu$ (with all the other assumptions being the same). 
Then the proof of Theorem \ref{equi} still shows that 
$F_{\chi,V,\theta}(M_{\lambda-\rho})=M_{\mu-\rho}$. 
Thus $[F_{\chi,V,\theta}]\delta_\lambda=\delta_\mu$, and since by Theorem \ref{commuW} $[F_{\chi,V,\theta}]$ is $W$-invariant, it follows that $[F_{\chi,V,\theta}]\delta_\nu=\delta_\mu$ for all $\nu\in W_\mu\lambda$. 

On the other hand, we no longer have 
$F_{\theta,V^*,\chi}(M_{\mu-\rho})=M_{\lambda-\rho}$,
in general. Namely, the proof of Theorem \ref{equi}
shows that $F_{\theta,V^*,\chi}(M_{\mu-\rho})$ 
has a filtration whose successive quotients are 
$M_{\nu-\rho}$, $\nu\in W_\mu \lambda$, each occurring with multiplicity $1$
(so the length of this filtration is $|W_\mu/W_\lambda|$). Thus
$$
[F_{\theta,V^*,\chi}]\delta_\mu=\sum_{\nu\in W_\mu\lambda}\delta_\nu. 
$$
It follows that 
$$
[F_{\chi,V,\theta}][F_{\theta,V^*,\chi}]\delta_\mu=|W_\mu/W_\lambda|\delta_\mu,
$$
hence $F_{\chi,V,\theta}\circ F_{\theta,V^*,\chi}(M_{\mu-\rho})=|W_\mu/W_\lambda|M_{\mu-\rho}$ (as the left hand side is projective). 
Thus $F_{\chi,V,\theta}\circ F_{\theta,V^*,\chi}\cong |W_\mu/W_\lambda|{\rm Id}$. 
\end{remark} 

\begin{remark} Let $\C\subset \Rep(\g)$ be a full subcategory invariant under all $F_V$ and $\Pi_\theta$, and $\C_\theta:=\Pi_\theta\C=\C\cap \Rep(\g)_\theta$. Then Theorem \ref{equi} implies that if $W_\lambda=W_\mu$ then 
the functors $F_{\chi,V,\theta},F_{\theta,V^*,\chi}$ are mutually 
quasi-inverse equivalences between $\C_\theta$ and $\C_\chi$. Interesting examples of this include: 

1. $\C=\mathcal O$. In this case we obtain that for dominant $\lambda$ the category 
$\O_{\chi_\lambda}$ up to equivalence depends only on $\lambda+P$ and the stabilizer $W_\lambda$. In particular, for regular dominant integral $\lambda$ all these categories are equivalent. 

2. $\C$ is the category of $\g$-modules which are locally finite and semisimple with respect to a reductive Lie subalgebra $\mathfrak k\subset \mathfrak \g$. If $\mathfrak{k}$ is the fixed subalgebra of an involution of $\g$, this category contains the category of $(\g_{\Bbb R},K)$-modules for any connected compact group $K$ such that ${\rm Lie}K=\mathfrak k$. Namely, it is just the subcategory of modules that integrate to $K$. 
\end{remark} 

\subsection{Two-sided ideals in $U_\theta$ and submodules of Verma modules}

Let $\theta=\chi_\lambda$ for dominant $\lambda$. Let $\Omega_\theta$ denote the lattice 
of two-sided ideals in $U_\theta$ (i.e., the set of two-sided ideals equipped with the operations of sum and intersection). Likewise, let $\Omega(\lambda)$ be the lattice of submodules of $M_{\lambda-\rho}$. We have a map $\nu: \Omega_\theta\to \Omega(\lambda)$ given by 
$\nu(J)=JM_{\lambda-\rho}$. It is clear that $\nu$ preserves inclusion and arbitrary sums. 

\begin{theorem}\label{lattice} (i) $I\subset J$ iff $\nu(I)\subset \nu(J)$. In particular, $\nu$ is injective.

(ii) The image of $\nu$ is the set of submodules of $M_{\lambda-\rho}$ 
which are quotients of direct sums of $P_{\mu-\rho}$ where $\chi_\mu=\chi_\lambda$, $\mu\preceq \lambda$ and $\mu\preceq W_\lambda\mu$. 

(iii) If $\lambda$ is regular (i.e., $W_\lambda=1$) then $\nu$ is an isomorphism of lattices. 
\end{theorem}

\begin{proof} 
(i) Let $F$ be a projective $\theta$-functor, and $\phi: F\to {\rm Id}_\theta$ 
a morphism of functors $\Rep(\g)_\theta^1\to \Rep(\g)$. Let 
$M(\phi,F)$ be the image of the map $\phi_{M_{\lambda-\rho}}: F(M_{\lambda-\rho})\to M_{\lambda-\rho}$ and $J(\phi,F)$ 
the image of $\phi_{U_\theta}: F(U_\theta)\to U_\theta$. 
Note that $\phi_{U_\theta}$ is a morphism of $(U(\g),U_\theta)$-bimodules, 
so $J(\phi,F)$ is a subbimodule of $U_\theta$, i.e., a 2-sided ideal. Let 
$a: U_\theta\to M_{\lambda-\rho}$ be the surjection given by $a(u)=uv_{\lambda-\rho}$. 
Then by functoriality of $\phi$
$$
a\circ \phi_{U_\theta}=\phi_{M_{\lambda-\rho}}\circ F(a).
$$ 
Hence
$$
\nu(J(\phi,F))=J(\phi,F)M_{\lambda-\rho}=J(\phi,F)v_{\lambda-\rho}=a(J(\phi,F))=
$$
$$
{\rm Im}(a\circ \phi_{U_\theta})=
{\rm Im}(\phi_{M_{\lambda-\rho}}\circ F(a))={\rm Im}(\phi_{M_{\lambda-\rho}})=M(\phi,F).
$$

Let us show that any 2-sided ideal $J$ in $U_\theta$ is of the form $J(\phi,F)$ for some $F$, 
$\phi$. Since $U_\theta$ is Noetherian, $J$ is generated by some finite-dimensional subspace $V\subset J$ which can be chosen $\g_{\rm ad}$-invariant. Then by Frobenius reciprocity the $\g_{\rm ad}$-morphism $\iota: V\to U_\theta$ can be lifted to a morphism of $(U(\g),U_\theta)$-bimodules $\widehat\phi: V\otimes U_\theta=F_V(U_\theta)\to U_\theta$, i.e., to a functorial morphism 
$\phi: F_V(\theta)\to {\rm Id}_\theta$. It is clear that then $J=J(\phi,F)$.  
 
We are now ready to prove (i), i.e., that $M(\phi,F)\subset M(\phi',F')$ implies $J(\phi,F)\subset J(\phi',F')$. Since $F(M_{\lambda-\rho})$, 
$F'(M_{\lambda-\rho})$ are projective, the inclusion $M(\phi,F)\hookrightarrow M(\phi',F')$ lifts to 
a map $\widetilde \alpha: F(M_{\lambda-\rho})\to 
F'(M_{\lambda-\rho})$, i.e., $\phi_{M_{\lambda-\rho}}'\circ \widetilde \alpha=\phi_{M_{\lambda-\rho}}$. But by Theorem \ref{projiso}, morphisms of projective $\theta$-functors are the same as morphisms of the images of $M_{\lambda-\rho}$ under these functors. Thus there is 
$\alpha: F\to F'$ which maps to $\widetilde \alpha$ and such that $\phi'\circ \alpha=\phi$. Hence 
$$
J(\phi,F)={\rm Im}(\phi_{U_\theta})\subset {\rm Im}(\phi_{U_\theta}')=J(\phi',F'), 
$$
and (i) follows. 

(ii) The proof of (i) implies that the image of $\nu$ consists exactly of the submodules $M(\phi,F)$.
Such a submodule is the image of $F(M_{\lambda-\rho})$ under a morphism. But $F$ is a projective $\theta$-functor, so by Corollary \ref{projfuu}(iii), it is of the form $\widetilde F(\theta)$, where 
$\widetilde F$ is a projective functor. Also by Theorem \ref{classpro}, $\widetilde F$ is a direct sum of $F_\xi$, 
so $F(M_{\lambda-\rho})$ is a direct sum of $P_{\mu-\rho}$, where
$(\mu,\lambda)$ is a proper representation of $\xi$. Thus $\mu\preceq \lambda$ 
and $\mu\preceq W_\lambda\mu$. Conversely, if for such $\mu$ we have a homomorphism 
$\gamma: P_{\mu-\rho}=F_\xi(M_{\lambda-
\rho})\to M_{\lambda-\rho}$ then $\gamma=\phi_{M_{\lambda-\rho}}$ where 
$\phi: F_\xi(\theta)\to {\rm Id}_\theta$. So ${\rm Im}(
\gamma)=\nu(J(\phi,F_\xi(\theta)))$. Since $\nu$ preserves sums, (ii) follows. 

(iii) Every submodule of $M_{\lambda-\rho}$ is a quotient of a direct sum of $P_{\mu-\rho}$ with $\chi_\mu=\chi_\lambda,\mu\le \lambda$. Hence by Corollary \ref{domchar} $\mu\preceq \lambda$, as $\lambda$ is dominant. (This also follows from Theorem \ref{BGGth}). 
So if $W_\lambda=1$ then by (ii) $\nu$ is surjective, hence bijective by (i). Since $I\cap J$ is the largest of all ideals contained both in $I$ and in $J$ and similarly for submodules, $\nu$ also preserves intersections by (i). Thus $\nu$ is an isomorphism of lattices.
\end{proof} 

\begin{corollary} Let $\theta=\chi_\lambda$ where $\lambda$ is dominant. 
If $M_{\lambda-\rho}$ is irreducible then $U_\theta$ is a simple algebra. Conversely, if $U_\theta$ is simple then $M_{\mu-\rho}$ is irreducible for all $\mu$ with $\chi_\mu=\theta$. 
\end{corollary} 

\begin{proof} The direct implication follows from Theorem \ref{lattice}. 
For the reverse implication, suppose for some distinct $\mu_1,\mu_2\in W\lambda$, we have 
$M_{\mu_1-\rho}\hookrightarrow M_{\mu_2-\rho}$ and $M_{\mu_1-\rho}$ is simple. Then in view of the Duflo-Joseph theorem we have 
an inclusion 
$$
J:=\Hom_{\rm fin}(M_{\mu_2-\rho},M_{\mu_1-\rho})\hookrightarrow \Hom_{\rm fin}(M_{\mu_2-\rho},M_{\mu_2-\rho})=U_\theta,
$$
and $J$ is a proper 2-sided ideal (as it does not contain $1$) which is not zero 
(as $M_{\mu_1-\rho}\cong M_{\mu_1-\rho}^\vee$ and hence 
for a finite-dimensional $\g$-module $V$, 
$\Hom(M_{\mu_2-\rho},V\otimes M_{\mu_1-\rho})\cong V[\mu_2-\mu_1]$).
\end{proof} 

Using the determinant formula for the Shapovalov form, this gives an explicit description
of the locus of $\theta\in \h^*/W$ where $U_\theta$ is simple. 

\section{\bf Applications of projective functors - II} 

\subsection{Duflo's theorem on primitive ideals in $U_\theta$} 

Recall that a {\bf prime ideal} in a commutative ring $R$ is a proper ideal $I$ such that if $xy\in I$ then $x\in I$ or $y\in I$. This definition is not good for noncommutative rings: for example, 
the zero ideal in the matrix algebra ${\rm Mat}_n(\Bbb C)$, $n\ge 2$, would not be prime, even though this algebra is simple; so  ${\rm Mat}_n(\Bbb C)$ would have no prime ideals at all. However, the definition can be reformulated so that it works well for noncommutative rings. 

\begin {definition} A proper 2-sided ideal $I$ in a (possibly non-commutative) ring $R$ is {\bf prime} if whenever the product $XY$ of two 2-sided ideals $X,Y\subset R$ 
is contained in $I$, either $X$ or $Y$ must be contained in $I$. 
\end{definition} 

Note that for commutative rings this coincides with the usual definition. Indeed, if $I$ is prime in the noncommutative sense and if $xy\in I$ then $(x)(y)\subset I$, so $(x)\subset I$ or $(y)\subset I$, i.e. $x$ or $y$ is in $I$. Conversely, if $I$ is prime in the commutative sense and 
$X,Y$ are not contained in $I$ then there exist $x\in X,y\in Y$ not in $I$, so $xy\notin I$, i.e., $XY$ is not contained in $I$. But in the noncommutative case the two definitions differ, e.g. $0$ is clearly a prime ideal (in the noncommutative sense) in any simple algebra, e.g. in the matrix algebra  ${\rm Mat}_n(\Bbb C)$.  

A ring $R$ is called {\bf prime} if $0$ is a prime ideal in $R$. For example, if $R$ 
is an integral domain then it is prime, and the converse holds if $R$ is commutative. 
On the other hand, there are many noncommutative prime rings which are not domains, e.g. simple rings, such as the matrix algebras ${\rm Mat}_n(\Bbb C),n\ge 2$. Also it is clear that an ideal $I\subset R$ is prime iff the ring $R/I$ is prime (thus every maximal ideal is prime, so prime ideals always exist). If moreover $R/I$ is a domain, one says that $I$ is {\bf completely prime}.  

Another important notion is that of a {\bf primitive ideal}. 

\begin{definition} An ideal $I\subset R$ is {\bf primitive} if it is the annihilator of a simple $R$-module $M$.  
\end{definition} 

It is easy to see that every primitive ideal $I$ is prime: if $X,Y$ are 2-sided ideals in $R$ and $XY\subset I$ then $XYM=0$, so if $Y$ is not contained in $I$ then $YM\ne 0$. Thus $YM=M$ (as $M$ is simple), hence $XM=XYM=0$, so $X\subset I$. Also for a commutative ring a primitive ideal is the same thing as a maximal ideal. Indeed, if $I$ is maximal then $R/I$ is a field, so a simple $R$-module, and $I$ is the annihilator of $R/I$. Conversely, if $I$ is primitive and is the annihilator of a simple module $M$ then $M=R/J$ is a field and $I=J$, so $I$ is maximal. 

\begin{exercise} Show that every maximal ideal in a unital ring is primitive, and give a counterexample to the converse. 
\end{exercise} 

We see that in general a prime ideal need not be primitive, e.g. the zero ideal in $\Bbb C[x]$. Nevertheless, for $U_\theta$ we have the following remarkable theorem due to M. Duflo:  

\begin{theorem} Every prime ideal $J\subset U_\theta$ is primitive and moreover is the annihilator of a simple highest weight module $L_{\mu-\rho}$, where $\chi_\mu=\theta$.
\end{theorem} 

\begin{proof} The module $M:=M_{\lambda-\rho}/\nu(J)$ has finite length, so let us endow it with a filtration by submodules $F_k=F_kM$ with simple successive quotients $L_1,...,L_n$ ($L_k=F_k/F_{k-1}$). Let $I_k\subset U_\theta$ be the annihilators of $L_k$. Since $JM=0$, 
we have $J\subset I_k$ for all $k$. Also $I_kF_k\subset F_{k-1}$, so 
$I_1...I_nM=0$, hence $I_1...I_nM_{\lambda-\rho}\subset JM_{\lambda-\rho}$. 
By Theorem \ref{lattice}(i), this implies that $I_1....I_n\subset J$. 
Since $J$ is prime, this means that there exists $m$ such that $I_m\subset J$. 
Then $J=I_m$, i.e. $J$ is the annihilator of $L_m$. But $L_m=L_{\mu-\rho}$ for some $\mu$ such that 
$\chi_\mu=\chi_\lambda=\theta$. 
\end{proof} 

Note that the choice of $\mu$ is not unique, for example, for $J=0$ and generic $\theta$, any of the $|W|$ possible choices of $\mu$ is good. In fact, the proof of Duflo's theorem shows that for every dominant $\lambda$ such that $\theta=\chi_\lambda$, we can choose $\mu\in W\lambda$ such that $\mu\preceq \lambda$. 

\subsection{Classification of simple Harish-Chandra bimodules} \label{classsim}

Denote by $HC_\theta^n$ the category of Harish-Chandra bimodules over $\g$ 
annihilated on the right by the ideal $({\rm Ker}\theta)^n$. These categories form a nested sequence; denote the corresponding nested union by $HC_\theta$. Recall that we have a direct sum decomposition 
$HC=\oplus_{\theta\in \h^*/W} HC_\theta$. This implies that every simple Harish-Chandra bimodule belongs 
to $HC_\theta^1$ for some infinitesimal character $\theta$. 

Recall also that for a finite-dimensional $\g$-module $V$, in $HC_\theta^1$ we have the object 
$V\otimes U_\theta$. Moreover, this object is projective: for $Y\in HC_\theta^1$ we have 
$$
\Hom(V\otimes U_\theta,Y)=\Hom_{\g-{\rm bimod}}(V\otimes U(\g),Y)=\Hom_{\g_{\rm ad}}(V,Y),
$$
which is an exact functor since $Y$ is a locally finite (hence semisimple) $\g_{\rm ad}$-module. 
Finally, since $Y$ is a finitely generated bimodule locally finite under $\g_{\rm ad}$, there exists a finite-dimensional $\g_{\rm ad}$-submodule $V\subset Y$ that generates $Y$ as a bimodule. Then the homomorphism 
$$\widehat i: V\otimes U(\g)\to Y$$ corresponding to $i: V\hookrightarrow Y$ is surjective and factors through the module $V\otimes U_\theta$. Thus $Y$
is a quotient of $V\otimes U_\theta$. Thus we have 

\begin{lemma}\label{enopro} The abelian category $HC_\theta^1$ has enough projectives. 
\end{lemma}

We also note that this category has finite-dimensional Hom spaces. Indeed, 
if $Y_1,Y_2\in HC_\theta^1$ then $Y_1$ is a quotient of $V\otimes U_\theta$ for some $V$, so $\Hom(Y_1,Y_2)\subset \Hom(V\otimes U_\theta,Y_2)=\Hom_{\g_{\rm ad}}(V,Y_2)$, which is finite-dimensional. Finally, note that this category is Noetherian: any nested sequence of subobjects 
of an object stabilizes. 

It thus follows from the Krull-Schmidt theorem that in $HC_\theta^1$, every object of $HC_\theta^1$ is uniquely a finite direct sum of indecomposables, and from Proposition \ref{genera} the indecomposable projectives and the simples of $HC_\theta^1$ are labeled by the same index set. It remains to describe this labeling set. 

\begin{theorem}\label{classi} The simples (and indecomposable projectives) in $HC_\theta^1$ are labeled by 
the set $\Xi_\theta$ of $\xi\in \Xi$ admitting a proper representation $(\mu,\lambda)$ with $\chi_\lambda=\theta$, via $\xi\in \Xi_\theta\mapsto \bold L_\xi,\bold P_\xi$. Namely, if $\xi=(\mu,\lambda)$ is such a proper representation then 
$\bold P_\xi$ is the unique indecomposable projective in $HC_\theta^1$ such that $\bold P_\xi\otimes_{U(\g)}M_{\lambda-\rho}=P_{\mu-\rho}$. 
\end{theorem} 

\begin{proof} Every indecomposable projective is a direct summand of $V\otimes U_\theta$. But 
$(V\otimes U_\theta)\otimes_{U(\g)}Y=F_V(\theta)(Y)$. Thus from the classification of projective functors (Theorem \ref{classpro}) it follows that the indecomposable summands of $V\otimes U_\theta$ are $\bold P_\xi$ such that $\bold P_\xi\otimes_{U(\g)}\text{--}\cong F_\xi(\theta)$. 
\end{proof} 

\begin{corollary} Objects in $HC_\theta^1$, hence in $HC_\theta$ and $HC$, have finite length.
\end{corollary} 

\begin{proof} Recall that $HC_\theta^1=\oplus_\chi HC_{\chi,\theta}^1$, the decomposition according to left generalized infinitesimal characters. By  Theorem \ref{classi}, each subcategory $HC_{\chi,\theta}^1$ 
has finitely many simple objects. Thus the statement for $HC_\theta^1$ follows from Proposition \ref{genera}. 
Now, every object of $HC_\theta^n$ has a filtration of length $n$ by powers of ${\rm Ker}\theta$, and the successive quotients lie in $HC_\theta^1$, which implies the corollary. 
\end{proof}

\subsection{Equivalence between category $\mathcal O$ and category of Harish-Chandra bimodules}

Let $\theta=\chi_\lambda$ where $\lambda$ is dominant. Let $\mathcal O_{\lambda+P}$ be the full subcategory of $\O$ consisting of modules with weights in $\lambda+P$.  Define the functor 
$$
T_\lambda: HC_\theta^1\to \O_{\lambda+P}
$$
given by $T_\lambda(Y)=Y\otimes_{U(\g)}M_{\lambda-\rho}$. Also 
let $\O(\lambda)$ be the full subcategory of $\O_{\lambda+P}$ of modules $M$ which admit a presentation 
$$
Q_1\to Q_0\to M\to 0,
$$
where $Q_0,Q_1$ are direct sums of $P_{\mu-\rho}$ with $\mu\in \lambda+P$ and $\mu\preceq W_\lambda\mu$. 
 
Note that the functor $T_\lambda$ is left adjoint to the functor $H_\lambda$ defined in Subsection \ref{funHla}: 
$H_\lambda(X)={\rm Hom}_{\rm fin}(M_{\lambda-\rho},X)$. Indeed, 
$$
\Hom(T_\lambda(Y),X)=\Hom(Y\otimes_{U(\g)}M_{\lambda-\rho},X)=
$$
$$
\Hom(Y,\Hom(M_{\lambda-\rho},X))=\Hom(Y,\Hom_{\rm fin}(M_{\lambda-\rho},X))=\Hom(Y,H_\lambda(X)). 
$$

\begin{theorem}\label{equivaa} (J. Bernstein -- S. Gelfand) (i) If $\lambda$ is a regular weight then the functor  $T_\lambda$ is an equivalence of categories, with quasi-inverse $H_\lambda$. 

(ii) In general, $T_\lambda$ is fully faithful and defines an equivalence 
$$
HC_\theta^1\cong \O(\lambda),
$$ 
with quasi-inverse $H_\lambda$. 
\end{theorem}  

\begin{remark} Note that if $\lambda$ is not regular then the subcategory $\mathcal O(\lambda)\subset \O$ need not be closed under taking subquotients (even though it is abelian by Theorem \ref{equivaa}). 
Also the functor $T_\lambda$ (and thus the inclusion $\O(\lambda)\hookrightarrow \O$)
need not be (left) exact. So if $f: X\to Y$ is a morphism in $\O(\lambda)$ then 
its kernels in $\O(\lambda)$ and in $\O$ may differ, and in particular the latter may not belong 
to $\O(\lambda)$. See Example \ref{nonex}. 
\end{remark} 

\begin{proof} (i) is a special case of (ii), so let us prove (ii). To this end, we'll use the following general 
fact. 

\begin{proposition}\label{abcat} 
Let $\A,\B$ be abelian categories such that $\A$ has enough projectives and 
$T: \A\to \B$ a right exact functor which maps projectives to projectives.
Suppose that $T$ is fully faithful on projectives, i.e., for any projectives $P_0,P_1\in \A$, the natural map 
$\Hom(P_1,P_0)\to \Hom(T(P_1),T(P_0))$ is an isomorphism. Then $T$ is fully faithful, and defines an equivalence of $\A$ onto the subcategory of objects $Y\in \B$ which admit a presentation 
$$
T(P_1)\to T(P_0)\to Y\to 0
$$
for some projectives $P_0,P_1\in \A$.   
\end{proposition} 

\begin{proof} We first show that $T$ is faithful. Let $X,X'\in \A$ and $a: X\to X'$. Pick presentations 
$$
P_1\to P_0\to X\to 0,\ P_1'\to P_0'\to X'\to 0.
$$
We have maps $p_0: P_0\to X$, $p_0': P_0'\to X'$, $p_1: P_1\to P_0$, $p_1': P_1'\to P_0'$.
There exist morphisms $a_0: P_0\to P_0'$, $a_1: P_1\to P_1'$ such that $(a_1,a_0,a)$ 
is a morphism of presentations. 

Suppose $T(a)=0$. Then $T(p_0')T(a_0)=0$. Thus $Y:={\rm Im}T(a_0)\subset 
{\rm Ker}T(p_0')={\rm Im}T(p_1')$. 
Thus the map $T(a_0): T(P_0)\to Y$ lifts to $b: T(P_0)\to T(P_1')$ 
such that $T(a_0)=T(p_1')b$. Since $T$ is full on projectives, we have $b=T(c)$ for some $c: P_0\to P_1'$, 
so $T(a_0)=T(p_1')T(c)=T(p_1'c)$. Since $T$ is faithful on projectives, this implies that $a_0=p_1'c$. 
Thus ${\rm Im}a_0\subset {\rm Im}p_1'={\rm Ker}p_0'$. It follows that $p_0'a_0=0$, hence $ap_0=0$. But $p_0$ is an epimorphism, hence $a=0$, as claimed. 

Now let us show that $T$ is full. Let $X,X'\in \A$ and $b: T(X)\to T(X')$.
The functor $T$ maps the above presentations of $X,X'$ into presentations 
of $T(X),T(X')$ (as it is right exact and maps projectives to projectives):
$$
T(P_1)\to T(P_0)\to T(X)\to 0,\ T(P_1')\to T(P_0')\to T(X')\to 0,
$$
and we can find $b_0: T(P_0)\to T(P_0'), b_1: T(P_1)\to T(P_1')$ such that 
$(b_1,b_0,b)$ is a morphism of presentations. Since $T$ is fully faithful on projectives, 
there exist $a_0,a_1$ such that $T(a_0)=b_0,T(a_1)=b_1$ and $a_0p_1=p_1'a_1$. 
Thus $a_0$ maps ${\rm Im}p_1={\rm Ker}p_0$ into ${\rm Im}p_1'={\rm Ker}p_0'$. This implies that $a_0$ descends to $a: X\to X'$, and $T(a)T(p_0)=T(p_0')b_0$. Hence $(T(a)-b)T(p_0)=0$, so since $T(p_0)$ is an epimorphism we get $T(a)=b$, as claimed. 

If $Y\in {\rm Im}(T)$ then $Y=T(X)$ where $X$ has presentation 
$$
P_1\to P_0\to X\to 0.
$$ 
Thus $Y$ has presentation 
$$
T(P_1)\to T(P_0)\to Y\to 0.
$$ 
Conversely, if $Y$ has such a presentation as a cokernel of a morphism $f: T(P_1)\to T(P_0)$ then $f=T(g)$ where $g: P_1\to P_0$, and $Y=T({\rm Coker}(g))$, which proves the last claim of the proposition. 
\end{proof} 

Now we are ready to prove Theorem \ref{equivaa}. By Lemma \ref{enopro}, $HC_\theta^1$ has enough projectives. Also the functor $T_\lambda$ is right exact, as it is given by tensor product. Further, if $P$ is projective then $\Hom(T_\lambda(P),Y)=\Hom(P,H_\lambda(Y))$
is exact in $Y$ since $H_\lambda$ is exact by 
Proposition \ref{Hlaexact} and $P$ is projective. Thus $T_\lambda(P)$ is projective. Finally, the fact that $T_\lambda$ is fully faithful on 
projectives was one of the main results about projective functors (Theorem \ref{projiso}). 
So Proposition \ref{abcat} applies to $\A=HC_\theta^1$, $\B=\O_{\lambda+P}$, $T=T_\lambda$. Moreover, the image of $T_\lambda$ is precisely the category $\O(\lambda)$ by the classification of projective functors (Theorem \ref{classpro}).

For an equivalence of categories, a right adjoint is a quasi-inverse. Thus $H_\lambda$ 
is quasi-inverse of $T_\lambda$, as claimed. The theorem is proved. 
\end{proof} 

\begin{corollary}\label{realiz} Every Harish-Chandra bimodule $M$ with right infinitesimal character $\theta$ is realizable as $\Bbb V^{\rm fin}$ where $\Bbb V$ is a (not necessarily unitary) admissible representation of the complex simply connected group $G$ corresponding to $\g$ 
on a Fr\'echet space. 
\end{corollary} 
  
\begin{proof} Let us prove the statement if $\theta=\chi_\lambda$ where $\lambda$ 
is a regular dominant weight (the general proof is similar). 

We have seen in Subsection \ref{funHla} that $H_\lambda(M_{\mu-\rho}^\vee)$ is the principal series module 
$\bold M(\lambda,\mu)={\rm Hom}_{\rm fin}(M_{\lambda-\rho},M_{\mu-\rho}^\vee)$. Thus by Theorem \ref{equivaa} $\bold M(\lambda,\mu)$ is injective in $HC_\theta^1$ if $\mu$ is dominant (since $M_{\mu-\rho}$ is projective, hence $M_{\mu-\rho}^\vee$ is injective). Moreover, since every indecomposable projective in $\O_{\lambda+P}$ is a direct summand of $V\otimes M_{\mu-\rho}$ 
for some dominant $\mu$ and finite-dimensional $\g$-module $V$, it follows that every indecomposable injective is a direct summand in $V\otimes M_{\mu-\rho}^\vee$ for some $V$ and dominant $\mu$. 
Hence by Theorem \ref{equivaa}, every indecomposable injective in $HC_\theta^1$ is a direct summand in $V\otimes \bold M(\lambda,\mu)$ for some $V$ and dominant $\mu$. Thus any $Y\in HC^1_\theta$ is contained in a direct sum of objects $V\otimes \bold M(\lambda,\mu)$ for finite-dimensional $V$ and dominant $\mu$. Since principal series modules $\bold M(\lambda,\mu)$ are realizable in a Fr\'echet space 
by Proposition \ref{prinser1}, we are done by Corollary \ref{bije}.
\end{proof}  
 
\begin{exercise} (i) Generalize the proof of Corollary \ref{realiz} to non-regular dominant weights $\lambda$. 

(ii) Generalize Corollary \ref{realiz} to any Harish-Chandra bimodule with {\it generalized} infinitesimal character $\theta$, and then to any Harish-Chandra bimodule. 

{\bf Hint.} Recall that $C^\infty_{\lambda,\mu}(G/B)$ is the space of smooth functions $F$ on $G$
which satisfy the differential equations 
$$
(R_b-\lambda(b))F=(R_{\overline b}-\mu(\overline b))F=0
$$ 
for $b\in \b$ and $\overline b\in \overline \b$ (here $R_b$ is the vector field corresponding to the right translation by $b$). Now for $N\ge 1$ consider the space $C^\infty_{\lambda,\mu,N}(G/B)$ of smooth functions $F$ on $G$ satisfying the differential equations 
$$
(R_b-\lambda(b))^NF=(R_{\overline b}-\mu(\overline b))^NF=0. 
$$ 
(Note that $C^\infty_{\lambda,\mu,1}(G/B)=C^\infty_{\lambda,\mu}(G/B)$.) 
Show that  $C^\infty_{\lambda,\mu,N}(G/B)$ are admissible representations of $G$ on Fr\'echet spaces. Then mimic the proof of Corollary \ref{realiz} using these instead of $C^\infty_{\lambda,\mu}(G/B)$. 
\end{exercise} 

\section{\bf Representations of $SL_2(\Bbb C)$}

\subsection{Harish-Chandra bimodules for $\mathfrak{sl}_2(\Bbb C)$}

Let us now work out the simplest example, $\g=\mathfrak{sl}_2(\Bbb C)$. 
In this case $\h^*=\Bbb C$, $P=\Bbb Z$, $\chi_\lambda=\lambda^2$. 
So by Theorem \ref{classi}, irreducible Harish-Chandra bimodules $\bold L_\xi$ are parametrized 
by pairs $\xi=(\mu,\lambda)$ of complex numbers such that $\lambda-\mu$ is an integer, 
modulo the map $(\mu,\lambda)\mapsto (-\mu,-\lambda)$,  and we may (and will) assume that $(\mu,\lambda)$ 
is a proper representation of $\xi$, i.e., $\lambda\notin \Bbb Z_{<0}$ and if $\lambda=0$ then $\mu\in \Bbb Z_{\le 0}$. Let us describe these bimodules 
in terms of the principal series bimodules $\bold M(\lambda,\mu)$. 

\begin{proposition} (i) The principal series bimodule $\bold M(\lambda,\mu)$ is irreducible and isomorphic 
to $\bold M(-\lambda,-\mu)$ unless $\lambda,\mu$ are nonzero integers of the same sign.  
Otherwise such bimodules are pairwise non-isomorphic. 

(ii) If $\lambda,\mu$ are both nonzero integers of the same sign then $\bold M(\lambda,\mu)$ is indecomposable and has a finite-dimensional constituent $L_{|\lambda|-1}^*\otimes L_{|\mu|-1}$, which is a submodule if $\lambda>0$ and quotient if $\lambda<0$. The other composition factor is $\bold M(\lambda,-\mu)\cong \bold M(-\lambda,\mu)$, which is irreducible. 

(iii) If $\xi=(\mu,\lambda)$ is a proper representation with $\lambda\notin \Bbb Z_{\ge 1}$ then $\bold L_\xi=\bold M(\lambda,\mu)$. If $\xi=(\mu,\lambda)$ where $\lambda\in \Bbb Z_{\ge 1}$ then $\bold L_\xi=L_{\lambda-1}^*\otimes L_{\mu-1}$ if $\mu\ge 1$ and $\bold L_\xi=\bold M(\lambda,\mu)$ if $\mu\le 0$. 
\end{proposition} 

\begin{proof} (i),(ii) Consider first the case when $\lambda$ and $\mu$ are both non-integers. Then 
the weights $\pm\lambda$ are dominant and $M_{\pm \mu-1}^\vee$ are simple, so by Theorem \ref{equivaa} $\bold M(\lambda,\mu)$ is also simple and isomorphic to $\bold M(-\lambda,-\mu)$. 

Now suppose $\lambda,\mu$ are integers. Recall that $\bold M(\lambda,\mu)$ 
decomposes over the diagonal copy of $\g$ as 
\begin{equation}\label{sameK}
\bold M(\lambda,\mu)=\oplus_{j\ge 0}L_{|\lambda-\mu|+2j}. 
\end{equation} 
If $\lambda=0$ and $\mu\ge 0$, then by Theorem \ref{classi} all composition factors of $\bold M(0,\pm\mu)$ are isomorphic. These bimodules have finite length and are multiplicity-free under the diagonal $\g$ by \eqref{sameK}, so they are simple. 
On the other hand, if $\mu>0$, we have two bimodules $\bold M(0,-\mu),\bold M(0,\mu)$ and a natural map 
$$
a: \bold M(0,\mu)=\Hom_{\rm fin}(M_{-1},M_{\mu-1}^\vee)\to \bold M(0,-\mu)=\Hom_{\rm fin}(M_{-1},M_{-\mu-1}^\vee)
$$ 
induced by the surjection $M_{\mu-1}^\vee\to M_{-\mu-1}^\vee$. 
The kernel of this map is 
${\rm Ker}a=\Hom_{\rm fin}(M_{-1},L_{\mu-1})=0$, which implies 
that $a$ is an isomorphism (as the $K$-type of the bimodules $\bold M(0,\mu), \bold M(0,-\mu)$ is the same by \eqref{sameK}). So we have the simple bimodule $\bold M(0,\mu)=\bold M(0,-\mu)$. If $\mu=0, \lambda\ne 0$, the situation is similar, as $\lambda$ and $\mu$ play a symmetric role. 
 
It remains to consider the situation when $\lambda,\mu\in \Bbb Z\setminus 0$. So let $n,m$ be positive integers. By Theorem \ref{equivaa}, the bimodule $\bold M(n,-m)$ is simple, as it corresponds 
to the simple module $M_{-m-1}^\vee$. Similarly, $\bold M(-n,m)$ is simple. 
Now, we have homomorphisms
$$
a: \bold M(n,m)\to \bold M(n,-m), b: \bold M(n,-m)\to \bold M(-n,-m).
$$
Since  $\bold M(n,-m)$ is simple and $a\ne 0$, it is surjective, so in view of \eqref{sameK} we have a short exact sequence 
$$
0\to L_{n-1}^*\otimes L_{m-1}\to \bold M(n,m)\to \bold M(n,-m)\to 0.
$$
Similarly, since $b\ne 0$, it is injective, so 
in view of \eqref{sameK} we have a short exact sequence 
$$
0\to \bold M(n,-m)\to \bold M(-n,-m)\to L_{n-1}^*\otimes L_{m-1}\to 0.
$$
Moreover, these sequences are not split 
by Theorem \ref{equivaa}. This proves (i),(ii). 

(iii) follows immediately from (i),(ii). 
The proposition is proved. 
\end{proof} 

\begin{example} \label{nonex} 
One may also describe explicitly the projectives $\bold P_\xi$. 
As an example let us do so for $\xi=(-1,0)$. 
Consider the tensor product $L_1\otimes U_0$, which is a projective object. 
We have $(L_1\otimes U_0)\otimes_{U_0}M_{-1}=L_1\otimes M_{-1}=P_{-2}$, the big projective object with composition series $[M_{-2},\Bbb C,M_{-2}]$. 
Thus $L_1\otimes U_0=\bold P_\xi$.
Over the diagonal copy of $\g$ we have 
$$
\bold P_\xi=L_1\otimes U_0=L_1\otimes (L_0\oplus L_2\oplus...)=2L_1\oplus 2L_3\oplus...
$$
Thus we have a short exact sequence 
\begin{equation}\label{shor}
0\to \bold L_\xi\to \bold P_\xi\to \bold L_\xi\to 0,
\end{equation}
where $\bold L_\xi=\bold M(0,-1)=\bold M(0,1)$, which is not split. 

This shows that the functor $T_\lambda=T_0$ is not exact in this case. Indeed, 
$T_0(\bold L_\xi)=M_0^\vee$ ($M_0^\vee\in \mathcal O(0)$ with presentation $P_{-2}\to P_{-2}\to M_0^\vee\to 0$ and $H_0(M_0^\vee)=\bold M(0,1)$), so the image of \eqref{shor} 
under $T_0$ is the sequence 
$$
0\to M_0^\vee\to P_{-2}\to M_0^\vee\to 0,
$$
which is not exact in the leftmost nontrivial term (the cohomology is $\Bbb C$). 
This sequence is, however, exact in the category $\mathcal O(0)$, which 
has just two indecomposable objects $M_0^\vee$ and $P_{-2}$ (so $\O(0)$
is not closed under taking subquotients and the inclusion $\mathcal O(0)\hookrightarrow \O$
is not exact). 
\end{example} 
 
\subsection{Representations of $SL_2(\Bbb C)$.}

Let us now consider representations of $G=SL_2(\Bbb C)$. We have $\g=\mathfrak{sl}_2(\Bbb C)$, $K=SU(2)$. We have already classified the irreducible Harish-Chandra (bi)modules and shown that the only ones are finite-dimensional modules and principal series modules. Moreover, we realized the principal series module $\bold M(\lambda,\mu)$ as the space of $K$-finite vectors in the space of smooth functions  
$F: G\to \Bbb C$ such that 
$$
F(gb)=F(g)t(b)^{\lambda-\mu}|t(b)|^{2\mu-2},\ b\in B,
$$
where $B\subset G$ is the subgroup of upper triangular matrices. Thus, similarly to the real case, setting $\mu-\lambda=m\in \Bbb Z$, we may represent $\bold M(\lambda,\mu)$ as the space of polynomial tensor fields on $\Bbb C\Bbb P^1=S^2$ 
of the form 
$$
\omega=\phi(u)(du)^{\frac{m}{2}}|du|^{1-\mu},
$$
and we have an admissible realization of $\bold M(\lambda,\mu)$ by the vector space $C^\infty_{\lambda-1,\mu-1}(G/B)$ 
of smooth tensor fields of the same form. The (right) action of the group $G$ on this space is 
given by 
$$
\left(\phi \circ \begin{pmatrix} a & b\\ c & d \end{pmatrix}\right)(u)=\phi\left(\frac{au+b}{cu+d}\right)(cu+d)^{-m}|cu+d|^{2\mu-2}.
$$
 We may also upgrade this realization to a Hilbert space realization by completing it with respect to the inner product 
 $$
 \norm{\omega}^2=\int_{S^2}|\phi(u)|^2(1+|u|^2)^{m-2{\rm Re}\mu+2}dA,
 $$
where $dA$ is the rotation-invariant probability measure on $S^2$. However, this inner product is not $G$-invariant, in general; 
it is only $G$-invariant if ${\rm Re}\mu=\frac{m}{2}$, i.e., $\mu=\frac{m}{2}+s$, $s\in i\Bbb R$. This shows that 
the $(\g,K)$-modules $\bold M(-\frac{m}{2}+s,\frac{m}{2}+s)$ are unitary and irreducible for any imaginary $s$, with the Hilbert space completion being $L^2_{-\frac{m}{2}+s-1,\frac{m}{2}+s-1}(G/B)$ -- the {\bf unitary principal series}.  

Also the trivial representation is obviously unitary. 
Are there any other unitary irreducible representations? Clearly, they cannot be finite-dimensional. However, the answer is yes. To find them, let us first 
determine which $\bold M(\lambda,\mu)$ are Hermitian. 
It is easy to show that this happens whenever $\lambda^2=\overline \mu^2$, i.e., $\lambda=\pm \overline\mu$. If $\lambda=-\overline \mu$, we get $2{\rm Re}\mu=m$, so $\mu=\frac{m}{2}+s$, $\lambda=-\frac{m}{2}+s$, $s\in i\Bbb R$, exactly as above. On the other hand, if $\lambda=\overline \mu$ then we get 
$\mu-\overline\mu=m$, which implies that $m=0$, i.e., $\lambda=\mu\in \Bbb R$. In this case by Theorem \ref{classi} the module $\bold M(\mu,\mu)$ is irreducible if and only if $\mu\notin \Bbb Z\setminus\{0\}$. Thus we see that 
for $0<|\mu|<1$, this module is unitary, as we have a continuous family of simple Hermitian modules $X(c):=\bold M(\sqrt{c},\sqrt{c})$ for $c\in (-\infty,1)$, and these modules are in the unitary principal series for $c\le 0$. This family of unitary modules for $c>0$ ($0<|\mu|<1$) is called the {\bf complementary series}; it is analogous 
to the complementary series in the real case. 

It remains to consider the intervals $m<|\mu|<m+1$ for $m\in \Bbb Z_{\ge 1}$. If $\bold M(\mu,\mu)$ is unitary 
for at least one point in such interval, then it is so for the whole interval (since nondegenerate Hermitian forms on the multiplicity spaces of $SU(2)$-modules cannot change signature in a continuous family), and taking the limit $\mu\to m+1$, we see that $L_m^*\otimes L_m$, which is a composition factor of $\bold M(m+1,m+1)$, would have to be unitary, which it is not. This shows that we have no unitary modules in these intervals. Thus we obtain the following result. 

\begin{theorem} (Gelfand-Naimark)
The irreducible unitary representations of $SL_2(\Bbb C)$ are Hilbert space completions 
of the following unitary Harish-Chandra modules: 

$\bullet$ Unitary principal series $\bold M(-\frac{m}{2}+s,\frac{m}{2}+s)$, $m\in \Bbb Z$, $s\in i\Bbb R$; 

$\bullet$ Complementary series $\bold M(s,s)$, $0<|s|<1$;

$\bullet$ The trivial representation $\Bbb C$. 

Here $\bold M(-\frac{m}{2}+s,\frac{m}{2}+s)\cong \bold M(\frac{m}{2}-s,-\frac{m}{2}-s)$, $\bold M(s,s)=\bold M(-s,-s)$ and there are no other isomorphisms. 
\end{theorem} 

\begin{exercise} Compute the map $M\mapsto M^\vee$ from Exercise \ref{dualmod} 
on the set of irreducible Harish-Chandra modules for $SL_2(\Bbb R)$ and $SL_2(\Bbb C)$. 
\end{exercise} 

\begin{exercise}\label{comserC} The following exercise is the complex analog of Exercise \ref{comserR}. 

(i) Show that for $-1<s<0$ the formula 
$$
(f,g)_s:=\int_{\Bbb C^2}f(y)\overline{g(z)}|y-z|^{-2s-2}dyd\overline ydzd\overline z
$$
defines a positive definite inner product on the space $C_0(\Bbb C)$ 
of continuous functions $f: \Bbb C\to \Bbb C$ 
with compact support ({\it Hint}: pass to Fourier transforms). 

(ii) Let $\mathcal H_s$ be the completion of $C_0(\Bbb C)$ under $(,)_s$. Show that it identifies with the space of tempered distributions $f$ whose Fourier transforms are functions satisfying
$$
\norm{f}_s^2=c_s\int_{\Bbb C}|\widehat f(\xi)|^2|\xi|^{2s}d^2\xi<\infty,
$$
where $c_s>0$ and $d^2\xi$ is Lebesgue measure.  

(iii) Let us view $\mathcal H_s$ as the space of 
distributional tensor fields\linebreak[4] $f(y)|dy|^{1-s}$, where $f$ is as in (ii).  
Show that the complementary series unitary representation $\widehat{\bold M}(s,s)$ of 
$SL_2(\Bbb C)$ may be realized in $\mathcal H_s$ 
with the natural action of each $g\in G$ extended by continuity from continuous compactly supported tensor fields whose images under $g$ are compactly supported.
\end{exercise} 

\section{\bf Geometry of complex semisimple Lie groups} 

\subsection{The Borel-Weil theorem} Let $G$ be a simply connected semisimple complex Lie group with Lie algebra $\g$ and a Borel subgroup $B$ generated by a maximal torus $T\subset G$ 
and the 1-parameter subgroups $\exp(te_i)$, $i\in \Pi$. Given an integral weight $\lambda\in P$, we can define the corresponding algebraic (in particular, holomorphic) line bundle $\mathcal L_\lambda$ on the flag variety $G/B$. Namely, the total space $T(\mathcal L_\lambda)$ of $\mathcal L_\lambda$ is $(G\times \Bbb C)/B$, where $B$ acts by 
$$
(g,z)b=(gb,\lambda(b)^{-1}z),
$$
and the line bundle $\mathcal L_\lambda$ is defined by the projection $\pi: T(\mathcal L_\lambda)\to G/B$ to the first component. So this bundle is $G$-equivariant, i.e., $G$ acts on $T(\mathcal L_\lambda)$ by left multiplication preserving the projection map $\pi$. We also see that smooth sections of $\mathcal L_\lambda$ are smooth functions 
$F: G\to \Bbb C$ such that 
$$
(g,F(g))b=(gb,F(gb)),\ b\in B,
$$
which yields 
$$
F(gb)=\lambda(b)^{-1}F(g),\ b\in B.
$$

It follows that the space of smooth sections 
$\Gamma_{C^\infty}(G/B,\mathcal L_\lambda)$
coincides with the admissible $G$-module $C^\infty_{-\lambda,0}(G/B)$,
realizing the principal series module $\bold M(-\lambda+\rho,\rho)=\Hom_{\rm fin}(M_{-\lambda},M_0^\vee)$. 

\begin{remark} Recall that $H^2(G/B,\Bbb Z)=P$.
It is easy to check that the first Chern class $c_1(\mathcal L_\lambda)$ equals $\lambda$. This motivates the minus sign in the definition of $\mathcal L_\lambda$.
\end{remark}  
 
\begin{example} Let $G=SL_2(\Bbb C)$, so that $B$ is the subgroup of upper triangular matrices with determinant $1$ and  $G/B=\Bbb C\Bbb P^1$. 
Then sections of $\mathcal L_m$ are functions 
$F: G\to \Bbb C$ such that $F(gb)=t(b)^{-m}F(g)$, where 
$t(b)=b_{11}$. Thus $\mathcal L_m\cong \mathcal O(-m)$.  
\end{example} 

Let us now consider holomorphic sections of $\mathcal L_\lambda$. 
The space $V_\lambda$ of such sections is a proper subrepresentation
of $C^\infty_{-\lambda,0}(G/B)$, namely the subspace where the left copy of $\g$ 
(acting by antiholomorphic vector fields) acts trivially. Thus $V_\lambda^{\rm fin}=
\Hom_{\rm fin}(M_{-\lambda},\Bbb C)\subset \Hom_{\rm fin}(M_{-\lambda},M_0^\vee)$, and 
$V_\lambda=V_\lambda^{\rm fin}$ since $V_\lambda^{\rm fin}$ is finite-dimensional. It follows that $V_\lambda^{\rm fin}=0$ unless
$\lambda\in -P_+$, and in the latter case $V_\lambda=L_{-\lambda}^*=L_\lambda^-=
L_{w_0\lambda}$, the 
finite-dimensional representation of $G$ with lowest weight $\lambda$. Thus we obtain

\begin{theorem}\label{bw} (Borel-Weil) Let $\lambda\in P$. If $\lambda\in P_+$ then we have an isomorphism of $G$-modules 
$$
\Gamma(G/B,\mathcal L_{-\lambda})\cong L_\lambda^*.
$$
If $\lambda\notin P_+$ then $\Gamma(G/B,\mathcal L_{-\lambda})=0$. 
\end{theorem} 

\begin{example} Let $G=SL_2(\Bbb C)$. Then Theorem \ref{bw} says that 
$$
\Gamma(\Bbb C\Bbb P^1,\O(m))\cong L_m=\Bbb C^{m+1}
$$ 
as representations of $G$. 
\end{example} 

More generally, suppose $\lambda\in P$ and $(\lambda,\alpha_i^\vee)=0$, $i\in S$ for a subset $S\subset \Pi$ of the set of simple roots. Then we have a parabolic subgroup $P_S\subset G$ 
generated by $B$ and also $\exp(tf_i)$ for $i\in S$, and $\lambda$ extends to a $1$-dimensional representation of $P_S$. Thus we can define the line bundle $\mathcal L_{\lambda,S}$ 
on the partial flag variety $G/P_S$ in the same way as $\mathcal L_\lambda$, and we have 
$\mathcal L_\lambda=p_S^*\mathcal L_{\lambda,S}$, where $p_S: G/B\to G/P_S$ 
is the natural projection. 

Note that any holomorphic section of 
$\mathcal L_\lambda$ is just a function when restricted to a fiber $F\cong P_S/B$ 
of the fibration $p_S$ (a compact complex manifold), so by the maximum principle it must be constant. It follows that 
$\Gamma(G/B,\mathcal L_\lambda)=\Gamma(G/P_S,\mathcal L_{\lambda,S})$. Thus we get 

\begin{corollary}\label{bw1} Let $\lambda\in P$ with $(\lambda,\alpha_i^\vee)=0$, $i\in S$. Then 
$$
\Gamma(G/P_S,\mathcal L_{-\lambda,S})\cong L_\lambda^*.
$$
if $\lambda\in P_+$, 
otherwise $\Gamma(G/P_S,\mathcal L_{-\lambda,S})=0$. 
\end{corollary} 

\begin{example} Let $G=SL_n(\Bbb C)=SL(V)$, $V=\Bbb C^n$, and $P_S\subset G$ be the subgroup of matrices $b$ such that $b_{r1}=0$ for $r>1$ (this corresponds to $S=\lbrace 2,...,n-1\rbrace$). Then $G/P_S=\Bbb C\Bbb P^{n-1}=\Bbb PV$. The condition $(\lambda,\alpha_i^\vee)=0$, $i\in S$ means that $\lambda=m\omega_1$, and in this case $\mathcal L_{m,S}=\O(-m)$. So Corollary \ref{bw1} 
says that 
$$
\Gamma(\Bbb PV,\O(m))=L_{m\omega_{n-1}}=S^mV^*
$$
for $m\ge 0$, and zero for $m<0$. This is also clear from elementary considerations, as 
by definition $\Gamma(\Bbb PV,\O(m))$ is the space of homogeneous polynomials on $V$ 
of degree $m$.  
\end{example} 

In fact, for $\lambda\in P_+$ we can construct an isomorphism $L_\lambda^*\cong \Gamma(G/B,\mathcal L_{-\lambda})$ explicitly as follows. Let $v_\lambda$ be a highest weight vector of $L_{\lambda}$,
$\ell\in L_{\lambda}^*$, and $F_\ell(g):=(\ell,gv_\lambda)$. Then 
$$
F_\ell(gb)=\lambda(b)F_\ell(g),\ b\in B. 
$$
Thus the assignment $\ell\to F_\ell$ defines a linear map 
$L_\lambda^*\to \Gamma(G/B,\mathcal L_{-\lambda})$ 
which is easily seen to be an isomorphism. 

This shows that the bundle $\mathcal L_{-\lambda}$ is {\bf globally generated}, i.e., 
for every $x\in G/B$ there exists $s\in \Gamma(G/B,\mathcal L_{-\lambda})$ such that $s(x)\ne 0$. 
In other words, we have a regular map $i_\lambda: G/B\to \Bbb PL_\lambda$ defined as follows. 
For $x\in G/B$, choose a basis vector $u$ of the fiber of $\mathcal L_{-\lambda}$ at $x$ and define
$i_\lambda(x)\in L_\lambda$ by the equality
$$
s(x)=i_\lambda(x)(s)u
$$
 for $s\in \Gamma(G/B,\mathcal L_{-\lambda})\cong L_\lambda^*$. Then $i_\lambda(x)$ is well defined (does not depend on the choice of $u$) up to scaling and is nonzero, so gives rise to a well defined element of the projective space $\Bbb PL_\lambda$. Another definition of 
this map is 
$$
i_\lambda(x)=x(\Bbb Cv_\lambda).
$$
This shows that $i_\lambda$ is an embedding when $\lambda$ is regular, i.e., in this case 
the line bundle $\mathcal L_{-\lambda}$ is {\bf very ample}. On the other hand, if 
$\lambda$ is not necessarily regular and $S$ is the set of $j$ such that $(\lambda,\alpha_j^\vee)=0$ 
then $i_\lambda: G/P_S\to \Bbb PL_\lambda$ is an embedding, so the bundle 
$\mathcal L_{-\lambda,S}$ over the partial flag variety $G/P_S$ is very ample. 

\begin{example} Let $G=SL_n(\Bbb C)$ and $\lambda=\omega_k$. 
Then $S=[1,n-1]\setminus k$, so $P_S\subset G$ is the subgroup
of matrices with $g_{ij}=0$, $i>k,j\le k$ and $G/P_S$ is the Grassmannian ${\rm Gr}(k,n)$ 
of $k$-dimensional subspaces in $\Bbb C^n$. In this case $L_\lambda=\wedge^k\Bbb C^n$, so 
$i_\lambda$ is the Pl\"ucker embedding ${\rm Gr}(k,n)\hookrightarrow \Bbb P(\wedge^k \Bbb C^n)$. 
\end{example} 

\subsection{The Springer resolution} Recall that a {\bf resolution of singularities} of an irreducible algebraic variety $X$ is a morphism $p: Y\to X$ from a smooth variety $Y$ that is proper (for example, projective\footnote{Recall that a morphism $f: X\to Y$ is said to be {\bf projective} if $f=\pi\circ \widetilde f$ where $\widetilde f: X\to Z\times Y$ is a closed embedding for some projective variety $Z$ and $\pi: Z\times Y\to Y$ is the projection to the second component.}) and birational. Hironaka proved in 1960s that any variety over a field of characteristic zero has a resolution of singularities. However, it is not unique and this theorem does not provide a nice explicit construction of a resolution. 

A basic example of a singular variety arising in Lie theory is the nilpotent cone $\mathcal{N}$ of a semisimple Lie algebra $\g$. This variety turns out to admit a very explicit equivariant resolution called the {\bf Springer resolution}, which plays an important role in representation theory. 

To define the Springer resolution, consider the cotangent bundle $T^*\mathcal F$ of the flag variety $\mathcal F$ of $G$. Recall that $\mathcal F$ is the variety of Borel subalgebras $\b\subset \g$. For $\b\in \mathcal F$, we have an isomorphism $\g/\b\cong T_\b\mathcal F$ defined by the action of $G$. Thus $T^*\mathcal F$ can be viewed as the set of pairs $(\b,x)$, where $x\in (\g/\b)^*$. Note that $(\g/\b)^*\cong \b^\perp$ under the Killing form, and $\b^\perp=[\b,\b]$ is the maximal nilpotent subalgebra of $\b$. Thus $T^*\mathcal F$ is the variety of pairs $(\b,x)$ where $\b\in \mathcal F$ is a Borel subalgebra of $\g$ and $x\in \b$ is a nilpotent element. 

Now we can define the {\bf Springer map} $p: T^*\mathcal F\to \mathcal N$ given by $p(\b,x)=x$. Note that this map is $G$-equivariant, so its fibers over conjugate elements of $\N$ are isomorphic. 
 
\begin{theorem} The Springer map $p$ is birational and projective, so it is a resolution of singularities. 
\end{theorem} 

\begin{proof} To show that $p$ is birational, it suffices to prove that if $e\in \mathcal{N}$ is regular, the Borel subalgebra $\b$ containing $e$ is unique. To this end, note that $\dim T^*\mathcal F=2\dim \mathcal F=\dim \N$ and the map $p$ is surjective (as any nilpotent element is contained in a Borel subalgebra). Thus $p$ is generically finite, i.e., $p^{-1}(e)$ is a finite set, and our job is to show that it consists of one element. 

We may fix a decomposition $\g=\n_+\oplus \h\oplus \n_-$ and 
assume that $e=\sum_{i=1}^re_i$. Then we have $[\rho^\vee,e]=e$, so 
the group $\lbrace t^{\rho^\vee},t\ne 0\rbrace\cong \Bbb C^\times$ acts on $p^{-1}(e)$ 
(as any Borel subalgebra containing $e$ also contains $te$). Since $p^{-1}(e)$ is finite, this action must be trivial. Thus $\rho^\vee$ normalizes every $\b\in p^{-1}(e)$, hence is contained in every such $\b$. But $\rho^\vee$ is regular, so is contained in a unique Cartan subalgebra, namely $\h$. Since every semisimple element in a Borel subalgebra $\b\subset \g$ is contained in a Cartan subalgebra sitting inside $\b$, it follows that $\h\subset \b$ for all $\b\in p^{-1}(e)$. Thus $[\omega_i^\vee,e]=e_i\in \b$ for all $i$. It follows that $\b=\b_+:=\h\oplus \n_+$, i.e., $|p^{-1}(e)|=1$, as claimed. 

Now let us show that $p$ is projective. Let $\widetilde p: T^*\mathcal F\to \mathcal F\times \mathcal N$ be the map defined by $\widetilde p(\b,x)=(\b,x)$. This is clearly a closed embedding (the image is defined by the equation $x\in \b$). But $p=\pi\circ \widetilde p$ where $\pi: \mathcal F\times \mathcal N\to \mathcal N$ is the projection to the second component. 
Thus $p$ is projective, as claimed. 
\end{proof} 

\begin{remark} The preimage $p^{-1}(e)$ for $e\in \N$ is called the {\bf Springer fiber}. If $e$ is not regular, $p^{-1}(e)$ has positive dimension. It is a projective variety, which is in general singular, reducible and has complicated structure, but it plays an important role in representation theory.
\end{remark} 

\begin{example} Let $\g=\mathfrak{sl}_2$. Then $\N$ is the usual quadratic cone 
$yz+x^2=0$ in $\Bbb C^3$, and $T^*\mathcal F=T^*\Bbb CP^1$ is the blow-up 
of the vertex in this cone. 
\end{example} 

\subsection{The symplectic structure on coadjoint orbits} 

Recall that a smooth real manifold, complex manifold or algebraic variety $X$ is {\bf symplectic} if it is equipped with a nondegenerate closed 2-form $\omega$. It is clear that in this case $X$ has even dimension. 

\begin{theorem}\label{KK} (Kirillov-Kostant) Let $G$ be a connected real or complex Lie group or complex algebraic group. Then every $G$-orbit in $\g^*$ has a natural symplectic structure. 
\end{theorem} 

\begin{proof} Let $O$ be a $G$-orbit in $\g^*$ and $f\in O$. Then $T_fO=\g/\g_f$ where $\g_f$ is the set of $x\in \g$ such that $f([x,y])=0$ for all $y\in \g$. Define a skew-symmetric bilinear form $\omega_f: \g\times \g\to \Bbb C$ 
given by $\omega_f(y,z)=f([y,z])$. It is clear that ${\rm Ker}\omega_f=\g_f$, so 
$\omega_f$ defines a nondegenerate form on $\g/\g_f=T_fO$. This defines a nondegenerate $G$-invariant differential $2$-form $\omega$ on $O$. 

It remains to show that $\omega$ is closed. Let $L_x$ be the vector field on $O$ 
defined by the action of $x\in \g$; thus $L_{[x,y]}=[L_x,L_y]$. It suffices to show that 
for any $x,y,z\in \g$ we have $d\omega(L_x,L_y,L_z)=0$.  
By Cartan's differentiation formula we have 
$$
d\omega(L_x,L_y,L_z)={\rm Alt}(L_x\omega(L_y,L_z)-\omega([L_x,L_y],L_z)),
$$
where Alt denotes the sum over cyclic permutations of $x,y,z$. Since $\omega$ is $G$-invariant, this yields 
$$
d\omega(L_x,L_y,L_z)(f)={\rm Alt}(\omega(L_y,L_{[x,z]}))(f)=f({\rm Alt}([y,[x,z]])),
$$
which vanishes by the Jacobi identity. 
\end{proof} 

\begin{corollary}\label{cooo1} The singular locus of the nilpotent cone $\N$ has codimension $\ge 2$. 
\end{corollary}

\begin{proof} This follows since $\mathcal N$ has finitely many orbits (Exercise \ref{nilor}) and by Theorem \ref{KK} they all have even dimension. 
\end{proof} 

\begin{corollary}\label{cooo2} $\N$ is normal (i.e., the algebra $\O(\N)$ is integrally closed in its quotient field). 
\end{corollary} 

\begin{proof} This follows from Corollary \ref{cooo1} since $\N$ is a complete intersection and any complete intersection whose singular locus has codimension $\ge 2$ is necessarily normal (\cite{H}, Chapter II, Prop. 8.23).  
\end{proof} 

\subsection{The algebra of functions on $T^*\mathcal F$}

We will first recall some facts about normal algebraic varieties. 

\begin{proposition}\label{normm} Let $Y$ be an irreducible normal algebraic variety. Then 

(i) (\cite{Eis}, Proposition 11.5) The singular locus of $Y$ has codimension $\ge 2$. 

(ii) (\cite{Eis}, Proposition 11.4) If $U\subset Y$ is an open subset and $Y\setminus U$ has codimension $\ge 2$ then any regular function $f$ on $U$ extends to a regular function on $Y$. In particular, any regular function on the smooth locus of $Y$ extends to a regular function on $Y$. 

(iii) Zariski main theorem (\cite{H}, Corollary III.11.4). If $X$ is irreducible and $p: X\to Y$ is a proper birational morphism then 
fibers of $p$ are connected. 
\end{proposition}

\begin{proposition}\label{leee} Let $Y$ be an irreducible normal affine algebraic variety and 
$p: X\to Y$ be a resolution of singularities. Then the homomorphism
$p^*: \O(Y)\to \O(X)$ is an isomorphism. 
\end{proposition}

\begin{proof} It is clear that $p^*$ is injective, so we only need to show it is surjective. 
Let $f\in \O(X)$. Since every fiber of $p$ is proper, and also connected due to normality of $Y$ by Proposition \ref{normm}(iii), $f$ is constant along this fiber. So $f=h\circ p$ for $h: Y\to \Bbb C$ a rational function. It remains to show that $h$ is regular. We know that $h$ is regular on the smooth locus of $Y$ (as it is defined at all points of $Y$). Thus the result follows from the normality of $Y$ and Proposition \ref{normm}(i),(ii). 
\end{proof} 

\begin{theorem}\label{isograd} Let $p: T^*\mathcal F\to \N$ be the Springer resolution. Then the map $p^*: \O(\N)\to \O(T^*\mathcal F)$
is an isomorphism  of graded algebras. 
\end{theorem} 

\begin{proof} This follows from Proposition \ref{leee} and the normality of $\N$ (Corollary \ref{cooo2}). 
\end{proof} 

\section{\bf D-modules - I} 

We would now like to formulate the Beilinson-Bernstein localization theorems.
We first review generalities about differential operators and $D$-modules. 

\subsection{Differential operators} 

Let $\bold k$ be an algebraically closed field of characteristic zero. Let $X$ be a smooth affine algebraic variety over $\bold k$. 
Let $\O(X)$ be the algebra of regular functions on $X$. Following Grothendieck, we define inductively  
the notion of a {\it differential operator of order (at most) $N$ on $X$}. Namely, a differential operator of order $-1$ is zero, 
and a $k$-linear operator $L: \O(X)\to \O(X)$ is a differential operator of order $N\ge 0$ if 
for all $f\in \O(X)$, the operator $[L,f]$ is a differential operator of order $N-1$.

Let $D_N(X)$ denote the space of differential operators of order $N$. We have 
$$
0=D_{-1}(X)\subset \O(X)=D_0(X)\subset D_1(X)\subset...\subset D_N(X)\subset...
$$
and $D_i(X)D_j(X)\subset D_{i+j}(X)$, which implies that the nested union $D(X):=\cup_{i\ge 0}D_i(X)$ 
is a filtered algebra.

\begin{definition} $D(X)$ is called {\bf the algebra of differential operators} on $X$. 
\end{definition} 

\begin{exercise}\label{ex1} Prove the following statements. 

1. $[D_i(X),D_j(X)]\subset D_{i+j-1}(X)$ for $i,j\ge 0$. In particular, $[,]$ makes
$D_1(X)$ a Lie algebra naturally isomorphic to ${\rm Vect}(X)\ltimes \O(X)$, where ${\rm Vect}(X)$ is the Lie algebra of vector fields on $X$.  

2. Suppose $x_1,...,x_n\in \O(X)$ are regular functions such that $dx_1,...,dx_n$ form a basis in each cotangent space to $X$ (for every $p\in X$ this can always be achieved by replacing $X$ by an affine neighborhood of $p$). 
Let $\partial_1,...,\partial_n$ be the corresponding vector fields. For $\bold m=(m_1,...,m_n)\in \Bbb Z_{\ge 0}^n$, let $|\bold m|:=\sum_{i=1}^n m_i$ and $\partial^{\bold m}:=\partial_1^{m_1}...\partial_n^{m_n}$. 
Then $D_N(X)$ is a free finite rank $\O(X)$-module (under left multiplication) with basis $\lbrace\partial^{\bold m}\rbrace$ with $|\bold m|\le N$, 
and $D(X)$ is a free $\O(X)$-module with basis $\lbrace\partial^{\bold m}\rbrace$ for all $\bold m$. 

3. One has $\gr D(X)=\oplus_{i\ge 0}\Gamma(X,S^iTX)=\O(T^*X)$. In particular, $D(X)$ is left and right Noetherian. 

4. $D(X)$ is generated by $\O(X)$ and elements $L_v$, $v\in {\rm Vect}(X)$ (depending linearly on $v$), with defining relations
\begin{equation}\label{rela} 
[f,g]=0,\ [L_v,f]=v(f),\ L_{fv}=fL_v,\ [L_v,L_w]=L_{[v,w]}, 
\end{equation}
where $f,g\in \O(X), v,w\in {\rm Vect}(X)$.

5. If $U\subset X$ is an affine open set then the multiplication map 
$\O(U)\otimes_{\O(X)}D(X)\to D(U)$ is a filtered isomorphism. 
\end{exercise}

\subsection{$D$-modules}

\begin{definition} A
{\bf left} (respectively, {\bf right}) {\bf $D$-module} on $X$ is a left (respectively, right) $D(X)$-module. 
\end{definition} 

\begin{example} 1. $\O(X)$ is an obvious example of a left $D$-module on $X$. 
Also, $\Omega(X)$ (the space of top differential forms on $X$) 
is naturally a right $D$-module on $X$, via $\rho(L)=L^*$ (the adjoint differential operator to $L$
with respect to the ``integration pairing'' between functions and top forms).
More precisely, $f^*=f$ for $f\in \O(X)$, and $L_v^*$ is the action of the vector field $-v$ 
on top forms (by Lie derivative). Finally, $D(X)$ is both a left and a right $D$-module on $X$.  

2. Suppose $\bold k=\Bbb C$, and $f$ is a holomorphic function defined on some open set in $X$ (in the usual topology). 
Then $M(f):=D(X)f$ is a left $D$-module. 
We have a natural surjection $D(X)\to M(f)$
whose kernel is the left ideal generated by the linear differential equations satisfied by $f$. 
E.g. 
$M(1)=\O(X)=D(X)/D(X){\rm Vect}(X)$, $M(x^s)=D(\Bbb C)/D(\Bbb C)(x\partial-s)$ if $s\notin \Bbb Z_{\ge 0}$, $M(e^x)=D(\Bbb C)/D(\Bbb C)(\partial -1)$. 
Similarly, if $\xi$ is a distribution (e.g., a measure) then $\xi\cdot D(X)$ 
is a right $D$-module. For instance, $\delta\cdot D(\Bbb C)= D(\Bbb C)/x D(\Bbb C)$, where $\delta$ is the delta-measure on the line.    
\end{example} 

\begin{exercise}
Show that if $X$ is connected then $\O(X)$ is a simple $D(X)$-module. Deduce that for any nonzero regular function $f$ on $X$, $M(f)=\O(X)$.
\end{exercise}

\subsection{$D$-modules on non-affine varieties} 

Now assume that $X$ is a smooth variety which is not necessarily affine. Recall 
that a {\bf quasicoherent sheaf} on $X$ is a sheaf $M$ of $\O_X$-modules (in Zariski topology) such that 
for any affine open sets $U\subset V\subset X$ the restriction map induces an isomorphism of $\O(U)$-modules $\O(U)\otimes_{\O(V)}M(V)\cong M(U)$. Exercise \ref{ex1}(5) implies that there exists a canonical quasicoherent sheaf of algebras $D_X$ on $X$ such that $\Gamma(U,D_X)=D(U)$ for any affine open set $U\subset X$. This sheaf is called the {\bf sheaf of differential operators on $X$}.

\begin{definition} A {\bf left} (respectively, {\bf right}) {\bf $D$-module} on $X$ is a quasicoherent 
sheaf of left (respectively, right) $D_X$-modules. The categories of left (respectively, right)
$D$-modules on $X$ (with obviously defined morphisms) are denoted by ${\mathcal{M}}_l(X)$ and ${\mathcal{M}}_r(X)$.    
\end{definition}

It is clear that these are abelian categories. 
We will mostly use the category ${\mathcal{M}}_l(X)$ and denote it shortly by $\M(X)$. 

Note that if $X$ is affine, this definition is equivalent to the previous one (by taking global sections). 

As before, the basic examples are $\O_X$ (a left $D$-module), $\Omega_X$ (a right $D$-module), $D_X$ (both a left and a right $D$-module). 

We see that the notion of a $D$-module on $X$ is local. For this reason, many questions about $D$-modules are local and reduce to the case of affine varieties. 

\subsection{Connections} The definition of a $D_X$-module can be reformulated in terms of connections on an $\O_X$-module. Namely, in differential geometry we have a theory of connections on vector bundles. An algebraic vector bundle on $X$ is the same thing as a coherent, locally free $\O_X$-module. It turns out that the usual definition of a connection, when written algebraically, 
makes sense for any $\O_X$-module (i.e., quasicoherent sheaf), not necessarily coherent or locally free. 

Namely, let $X$ be a smooth variety and $\Omega^i_X$ be the $\O_X$-module of differential $i$-forms on $X$.  

\begin{definition} A {\bf connection} on an $\O_X$-module $M$ 
is a $\bold k$-linear morphism of sheaves $\nabla: M\to M\otimes_{\O_X}\Omega^1_X$ 
such that 
$$
\nabla(fm)=f\nabla(m)+m\otimes df
$$ 
for local sections $f$ of $\O_X$ and $m$ of $M$.
\end{definition} 

Thus for each $v\in {\rm Vect}(X)$ we have the operator of covariant derivative $\nabla_v: M\to M$ given on local sections by $\nabla_v(m):=\nabla(m)(v)$. 

\begin{exercise} Let $X$ be an affine variety. Show that the operator $m\mapsto ([\nabla_v,\nabla_w]-\nabla_{[v,w]})m$ is $\O(X)$-linear in $v,w,m$.  
\end{exercise} 

Given a connection $\nabla$ on $M$, define the $\O_X$-linear map 
$$
\nabla^2: M\to M\otimes_{\O_X}\Omega^2_X
$$ 
given on local sections by 
$$
\nabla^2(m)(v,w):=([\nabla_v,\nabla_w]-\nabla_{[v,w]})m.
$$
This map is called the {\bf curvature} of $\nabla$. We say that $\nabla$ is {\bf flat} 
if its curvature vanishes: $\nabla^2=0$. 

\begin{proposition}
A left $D_X$-module is the same thing as an $\O_X$-module with a flat connection. 
\end{proposition} 

\begin{proof} Given an $\O_X$-module $M$ with a flat connection $\nabla$, we extend the $\O_X$-action to a $D_X$-action by $\rho(L_v)=\nabla_v$. 
The first three relations of \eqref{rela} then hold for any connection, while the last relation holds due to flatness of $\nabla$. 
Conversely, the same formula can be used to define a flat connection $\nabla$ on any $D_X$-module $M$. 
\end{proof} 

\begin{exercise} Show that if a left $D$-module $M$ on $X$ is $\O$-coherent (i.e. a coherent sheaf on $X$) then it is locally free, i.e., is a vector bundle with a 
flat connection, and vice versa.   
\end{exercise} 

\subsection{Direct and inverse images} 

Let $\pi: X\to Y$ be a morphism of smooth affine varieties. This morphism gives rise to a homomorphism 
$\pi^*: \O(Y)\to \O(X)$, making $\O(X)$ an $\O(Y)$-module, and a morphism of vector bundles $\pi_*: TX\to \pi^*TY$. 
This induces a map on global sections $\pi_*: {\rm Vect}(X)\to \O(X)\otimes_{\O(Y)}{\rm Vect}(Y)$. 

Define
$$
D_{X\to Y}=\O(X)\otimes_{\O(Y)}D(Y).
$$
This is clearly a right $D(Y)$-module. Let us show that it also has a commuting left $D(X)$-action. 
The left action of $\O(X)$ is obvious, so it remains to construct a flat connection. Given a vector field 
$v$ on $X$, let 
\begin{equation}\label{e2}
\nabla_v(f\otimes L)=v(f)\otimes L+f\pi_*(v)L,\ f\in \O(X),\ L\in D(Y), 
\end{equation}
where we view $\pi_*(v)$ as an element of $D_{X\to Y}$. This is well defined since for $a\in \O(Y)$ one has 
$[\pi_*(v),a]=v(a)\otimes 1$. 

\begin{exercise} Show that this defines a flat connection on $D_{X\to Y}$. 
\end{exercise} 

Now we define the {\bf inverse image functor} $\pi^\bullet: {\mathcal{M}}_l(Y)\to {\mathcal{M}}_l(X)$ by 
$$
\pi^\bullet(N)=D_{X\to Y}\otimes_{D(Y)}N
$$ 
and the {\it direct image functor}
$\pi_\bullet: {\mathcal{M}}_r(X)\to {\mathcal{M}}_r(Y)$ by 
$$
\pi_\bullet(M)=M\otimes_{D(X)}D_{X\to Y}.
$$ 
Note that at the level of quasicoherent sheaves, $\pi^\bullet$ is the usual inverse image. 

These functors are right exact and compatible with compositions. Also by definition, $D_{X\to Y}=\pi^\bullet(D(Y))$. 

Note that $\pi^\bullet(N)=\O(X)\otimes_{\O(Y)}N$ as an $\O(X)$-module (i.e., the usual pullback of $\O$-modules), 
with the connection defined by the formula similar to \eqref{e2}: 
$$
\nabla_v(f\otimes m)=v(f)\otimes m+f\nabla_{\pi_*(v)}(m),\ f\in \O(X),\ m\in M.
$$
This means that the definition of $\pi^\bullet$ is local both on $X$ and on $Y$. 
On the contrary, the definition of $\pi_\bullet$ is local only on $Y$ but not on $X$. 
For example, if $Y$ is a point and $\dim X=d$ then 
$\pi_\bullet\Omega_X=H^d(X,\bold k)$, the algebraic de Rham cohomology of $X$ of degree $d$. 

Thus we can use the same definition locally to define $\pi^\bullet$ for any morphism of smooth varieties, and $\pi_\bullet$ for an affine morphism (i.e. such that $\pi^{-1}(U)$ is affine for any affine open set $U\subset Y$), for example, a closed embedding. These functors are right exact, so one can also consider the corresponding derived functors 
$\pi_*:=L\pi_\bullet$ and $L\pi^\bullet$. In fact, it is convenient to define 
$\pi^!:=L\pi^\bullet[d]$ where $d:=\dim X-\dim Y$ (assuming that $X,Y$ have pure dimension).   

On the other hand, due to the non-local nature of direct image with respect to $X$ the correct functor $\pi_*$ for a non-affine morphism is not the derived functor of anything and can be defined only in the derived category. 

\section{\bf The Beilinson-Bernstein Localization Theorem} 

\subsection{The Beilinson-Bernstein localization theorem for the zero infinitesimal character} 

Let $\g$ be a complex semisimple Lie algebra and $U_0$ be the maximal quotient of $U(\g)$ 
corresponding to the infinitesimal character $\chi_\rho=\chi_{-\rho}$ of the trivial representation of $\g$. Recall that 
${\rm gr}(U_0)=\O(\N)$. Let $G$ be the corresponding simply connected complex group and $\mathcal F$ 
the flag variety of $G$; thus $\mathcal F\cong G/B$ for a Borel subgroup $B\subset G$. Let $D(\mathcal F)$
be the algebra of global differential operators on $\mathcal F$; it is clear that ${\rm gr}D(\mathcal F)\subset \O(T^*\mathcal F)$. Also, we have a natural filtration-preserving action map 
$a: U(\g)\to D(\mathcal F)$, induced by the Lie algebra homomorphism 
$\g\to {\rm Vect}(\mathcal F)$. 

\begin{theorem}\label{bebe} (Beilinson-Bernstein, \cite{BB}) (i) The homomorphism \linebreak $a: U(\g)\to D(\mathcal F)$ factors through a homomorphism $a_0: U_0\to D(\mathcal F)$.

(ii) One has ${\rm gr}(a_0)=p^*$ where $p$ is the Springer map $T^*\mathcal F\to \N$. 

(iii) ${\rm gr}D(\mathcal F)=\O(T^*\mathcal F)$ and $a_0$ is an isomorphism. 
\end{theorem} 

\begin{proof} (i) Let $z\in Z(\g)$ be an element acting by zero in the trivial representation of $\g$. 
Our job is to show that for any rational function $f\in \Bbb C(\mathcal F)$ we have 
$a(z)f=0$. Writing $\mathcal F$ as $G/B$, we may view $f$ as a rational function on $G$ 
such that $f(gb)=f(g)$, $b\in B$. The function $a(z)f$ on $G$ is the result of action on $f$ of the right-invariant differential operator $L_z$ corresponding to $z$: $a(z)f=L_zf$. Since $z$ is central, this operator is also left-invariant: $L_z=R_z$. Since $z$ acts by zero on the trivial representation, using the Harish-Chandra isomorphism, we may write $z$ as $\sum_i c_ib_i$, where $b_i\in \b:={\rm Lie}(B)$ and $c_i\in U(\g)$. 
Thus $R_z=\sum_i R_{c_i}R_{b_i}$. But $R_{b_i}f=0$ since $f$ is invariant under right translations by $B$. Thus $R_zf=0$ and we are done.

(ii) It suffices to check the statement in degrees $0$ and $1$, where it is straightforward. 

(iii) The statement follows from (i), (ii) and the fact that $p^*$ is an isomorphism (Theorem \ref{isograd}).   
\end{proof} 

The isomorphism $a_0$ gives rise to two functors: the functor 
of global sections $\Gamma: \M(\mathcal F)\to D(\mathcal F)-{\rm mod}\cong U_0-{\rm mod}$ 
and the functor of localization ${\rm Loc}: U_0-{\rm mod}\cong D(\mathcal F)-{\rm mod}\to 
\M(\mathcal F)$ given by ${\rm Loc}(M)(U):=D(U)\otimes_{D(\mathcal F)}M$ for an affine open set 
$U\subset \mathcal F$. Note that by definition the functor ${\rm Loc}$ is left adjoint to $\Gamma$. 

The following theorem is a starting point for the geometric representation theory of semisimple Lie algebras (in particular, for the original proof of the Kazhdan-Lusztig conjecture). 

\begin{theorem}\label{locth} (Beilinson-Bernstein localization theorem, \cite{BB}) The functors $\Gamma$ and ${\rm Loc}$ 
are mutually inverse equivalences. Thus the category $U_0-{\rm mod}$ is canonically equivalent 
to the category of $D$-modules on the flag variety $\mathcal F$. 
\end{theorem}

We will not give a proof of this theorem here. 

Theorem \ref{locth} motivates the following definition. 

\begin{definition} A smooth algebraic variety $X$ is said to be \linebreak {\bf D-affine} 
if the global sections functor $\Gamma: \M(X)\to D(X)-{\rm mod}$ 
is an equivalence (hence ${\rm Loc}$ is its inverse). 
\end{definition} 

It is clear that any affine variety is $D$-affine. Also we have 

\begin{corollary} Partial flag varieties of semisimple algebraic groups are $D$-affine. 
\end{corollary} 

\subsection{Twisted differential operators and $D$-modules}  

We would now like to generalize the localization theorem to nonzero infinitesimal characters. 
To do so, we have to replace usual differential operators and $D$-modules by twisted ones. 

Let $T$ be an algebraic torus with character lattice $P:=\Hom(T,\Bbb C^\times)$ and $\widetilde X$ be a principal $T$-bundle over a smooth algebraic variety $X$ (with $T$ acting on the right). In this case, given $\lambda\in P$, we can define the line bundle $\mathcal L_\lambda$ on $X$ whose total space is $\widetilde X\times_T\Bbb C_\lambda$, where $\Bbb C_\lambda$ is the 1-dimensional representation of $T$ corresponding to $\lambda$, and we can consider the sheaf $D_{\lambda,X}$ 
of differential operators acting on local sections of $\mathcal L_\lambda$ (rather than functions). 

Moreover, unlike the bundle $\mathcal L_\lambda$, the sheaf $D_{\lambda,X}$ makes sense 
not just for $\lambda\in P$ but more generally for $\lambda\in P\otimes_{\Bbb Z}\Bbb C$. Namely, 
assuming for now that $\lambda\in P$, we may think of rational sections of $\mathcal L_\lambda$ 
as rational functions $F$ on $\widetilde X$ such that $F(yt)=\lambda(t)^{-1}F(y)$ for $y\in \widetilde X$. 
A differential operator $D$ on $\widetilde X$ may be applied to such a function, and if $\xi\in \mathfrak{t}:={\rm Lie}(T)$ then the first order differential operator $R_\xi+\lambda(\xi)$ acts by zero: $(R_\xi+\lambda(\xi))F=0$. 
Thus given an affine open set $U\subset X$ with preimage $\widetilde U\subset \widetilde X$,  
the space 
$$
D_\lambda(U):=(D(\widetilde U)/D(\widetilde U)(R_\xi+\lambda(\xi),\xi\in \mathfrak{t}))^T
$$ 
is naturally an associative algebra  which acts on rational sections of $\mathcal L_\lambda$ (check it!). Moreover, 
it is easy to check that $D_\lambda(U)=D_{\lambda,X}(U)$. Now it remains to note that 
the definition of $D_\lambda(U)$ does not use the integrality of $\lambda$, thus makes sense 
for all $\lambda\in P\otimes_{\Bbb Z}\Bbb C$. 

Thus for any $\lambda\in P\otimes_{\Bbb Z}\Bbb C$ we obtain a quasicoherent sheaf of algebras $D_{\lambda,X}$ on $X$ which is called the sheaf of {\bf $\lambda$-twisted differential operators}. If $\lambda=0$, this sheaf coincides with the sheaf $D_X$ of usual differential operators, and in general it has very similar properties, for example ${\rm gr}(D_{\lambda,X}(U))=\O(T^*U)$ for any affine open set $U\subset X$. A quasicoherent sheaf on $X$ with the structure of a (left or right) $D_{\lambda,X}$-module is called a (left or right) {\bf $\lambda$-twisted $D$-module} on $X$. For example, if $\lambda\in P$ then $\mathcal L_\lambda$ is a left $D_{\lambda,X}$-module. The category of such modules is denoted by $\M^\lambda(X)$ (of course, it depends on the principal bundle $\widetilde X$ but we do not indicate it in the notation). Note that for $\beta\in P$ 
we have an equivalence $\M^\lambda(X)\cong \M^{\lambda+\beta}(X)$ defined by tensoring with $\mathcal L_\beta$. 

\begin{example} Let $\mathcal L$ be a line bundle on $X$ 
and $c\in \bold k$. Let $\widetilde X$ be the subset of nonzero vectors 
in the total space of $\mathcal L$. We have a natural action of $T:=\bold k^\times$ 
on $\widetilde X$ by dilations, and $c$ defines a character of ${\rm Lie}(T)$. 
Thus we can define the sheaf $D_{c,L,X}$ of twisted differential 
operators on $X$, and if $c\in \Bbb Z$ then $D_{c,L,X}=D_X(L^{\otimes c})$ 
is the sheaf of differential operators on $L^{\otimes c}$. For example, 
if $\Omega_X$ is the canonical bundle of $X$ then $D_{1,\Omega,X}=D_X(\Omega)$
is naturally isomorphic to the sheaf of usual differential operators with opposite multiplication, $D_X^{\rm op}$. 
\end{example} 

Thus tensoring with $\Omega$ defines a canonical equivalence 
$$
\M_l(X)\cong \M_r(X)
$$ 
(i.e., the sheaf $D_X$ is Morita equivalent, although not in general isomorphic, to $D_X^{\rm op}$). We may therefore not distinguish between these categories any more, identifying them by this equivalence, and can use left or right $D$-modules depending 
on what is more convenient. 

\subsection{The localization theorem for non-zero infinitesimal characters} 
We are now ready to generalize the localization theorem to non-zero infinitesimal characters. 
Let $U_\lambda$ be the maximal quotient of $U(\g)$ corresponding to  
the infinitesimal character $\chi_{\lambda-\rho}$. Recall that ${\rm gr}(U_\lambda)=\O(\N)$. 

Let $\widetilde{\mathcal F}:=G/[B,B]$. We have a right action of $T:=B/[B,B]$ 
on this variety by $y\mapsto yt$, defining the structure of a principal $T$-bundle 
$\widetilde{\mathcal F}\to \mathcal F$. Thus for every $\lambda\in P\otimes_{\Bbb Z}\Bbb C=\h^*$ 
we have a sheaf of $\lambda$-twisted differential operators $D_{\lambda,\mathcal F}=D_\lambda$ on $\mathcal F$. 
For example, if $\lambda\in P$ then $D_\lambda$ is the sheaf of differential operators
acting on sections of the line bundle $\mathcal L_\lambda$ appearing in the Borel-Weil theorem (Theorem \ref{bw}).
Let $D_\lambda(\mathcal F)$
be the algebra of global $\lambda$-twisted differential operators on $\mathcal F$; it is clear that ${\rm gr}D_\lambda(\mathcal F)\subset \O(T^*\mathcal F)$. Also, we have a natural filtration-preserving action map $a: U(\g)\to D_\lambda(\mathcal F)$.

\begin{theorem} (Beilinson-Bernstein) (i) The map 
$$
a: U(\g)\to D_\lambda(\mathcal F)
$$ 
factors through a map $a_\lambda: U_\lambda\to D_\lambda(\mathcal F)$.

(ii) One has ${\rm gr}(a_\lambda)=p^*$ where $p$ is the Springer map $T^*\mathcal F\to \N$. 

(iii) ${\rm gr}D_\lambda(\mathcal F)=\O(T^*\mathcal F)$ and $a_\lambda$ is an isomorphism. 
\end{theorem} 

\begin{proof} The proof is completely parallel to the proof of Theorem \ref{bebe}.  
\end{proof} 

As in the untwisted case, the isomorphism $a_\lambda$ gives rise to two functors: the functor 
of global sections 
$$
\Gamma: \M^\lambda(\mathcal F)\to D_\lambda(\mathcal F)-{\rm mod}\cong U_\lambda-{\rm mod}
$$ 
and the functor of localization 
$$
{\rm Loc}: U_\lambda-{\rm mod}\cong D_\lambda(\mathcal F)-{\rm mod}\to 
\M_\lambda(\mathcal F)
$$ 
given by ${\rm Loc}(M)(U):=D_\lambda(U)\otimes_{D_\lambda(\mathcal F)}M$ for an affine open set 
$U\subset \mathcal F$. Moreover, as before, ${\rm Loc}$ is left adjoint to $\Gamma$. 

Let us say that $\lambda\in \h^*$ is {\bf antidominant} if $-\lambda$ is dominant (cf. Subsection \ref{domwei}).  

\begin{theorem}\label{bebe2}(Beilinson-Bernstein localization theorem) If $\lambda-\rho$ is regular and antidominant then the functors $\Gamma$ and ${\rm Loc}$ are mutually inverse equivalences. Thus the category $U_\lambda-{\rm mod}$ is canonically equivalent to the category of $D_\lambda$-modules on the flag variety $\mathcal F$. 
\end{theorem}

\begin{remark} 1. As explained above, for $\beta\in P$ we have an equivalence $\M^\lambda(\mathcal F)\cong \M^{\lambda+\beta}(\mathcal F)$ defined by tensoring with $\mathcal L_\beta$. On the other side of the Beilinson-Bernstein equivalence this corresponds to translation functors defined in Subsection \ref{trafu}. 

2. The first statement of Theorem \ref{bebe2} fails if $\lambda-\rho$ is not assumed regular and antidominant. Indeed, if 
$\lambda$ is integral but not antidominant then by the Borel-Weil theorem (Theorem \ref{bw}) $\Gamma(\mathcal F,\mathcal L_\lambda)=0$, so the functor $\Gamma$ is not faithful. The second statement of Theorem \ref{bebe2} also fails if $\lambda\in P$ and $\lambda-\rho$ is not regular. 

For example, for $\g=\mathfrak{sl}_2$ and $\lambda\in \Bbb Z$, the localization theorem holds 
for $\lambda\le 0$. For $\lambda\ge 2$ the first statement fails but we still have an equivalence 
$\M^\lambda(\mathcal F)\cong U_\lambda-{\rm mod}$ (as $U_\lambda\cong U_{-\lambda+2}$), albeit not given by $\Gamma$. But for $\lambda=1$ there is no such equivalence at all; in fact, one can show that 
the category $U_\lambda-{\rm mod}$, unlike $\M^\lambda(\mathcal F)$, has infinite cohomological dimension. 
\end{remark} 
 
\section{\bf D-modules - II} 

We would now like to explain how the Beilinson-Bernstein localization theorem can be used to  classify various kinds of irreducible representations of $\g$. For this we will need to build up a bit more background on $D$-modules. 

\subsection{Support of a quasicoherent sheaf} 
Let $M$ be a quasicoherent sheaf on a variety $X$, and $Z\subset X$ a closed subvariety. 
We will say that $M$ is {\bf supported} on $Z$ if for any affine open set $U\subset X$, 
regular function $f\in \O(U)$ vanishing on $Z$, and $v\in M(U)$, there exists $N\in \Bbb Z_{\ge 0}$ 
such that $f^Nv=0$.  The {\bf support} ${\rm Supp}(M)$ is then defined as the intersection of all closed subvarieties $Z\subset X$ such that $M$ is supported on $Z$. So $M$ is supported on $Z$ iff the support of $M$ is contained in $Z$. 
 
In particular, we can talk about support of a (left or right, possibly twisted) $D$-module on a smooth variety $X$.  
The category of $D$-modules on $X$ supported on $Z$ will be denoted by $\M_Z(X)$. 

\begin{example} It is easy to see that $\Bbb C[x,x^{-1}]$ is a left $D$-module on $\Bbb A^1$, and $\Bbb C[x]$ is its submodule. These modules have full support $\Bbb A^1$. On the other hand, consider the quotient $\delta_0:=\Bbb C[x,x^{-1}]/\Bbb C[x]$.\footnote{In analysis $\delta_0$ arises as the $D$-module generated by the $\delta$-function at zero, which motivates the notation.} It is clear that $\delta_0$ has a basis $v_i=x^{-i}$, $i\ge 1$, with $xv_i=v_{i-1}$, $xv_1=0$, $\partial v_i=-iv_{i+1}$. Thus the support of $\delta_0$ is $\lbrace 0\rbrace$. 
\end{example} 

\subsection{Restriction to an open subset}

Recall that if $\A$ is an abelian category and $\B\subset \A$ a Serre subcategory (i.e., a full subcategory closed under taking subquotients and extensions) then one can form the quotient 
category $\A/\B$ with the same objects as $\A$, but with $\Hom_{\A/\B}(X,Y)$ being the direct limit 
of $\Hom_\A(X',Y/Y')$ over $X'\subset X$ and $Y'\subset Y$ such that $X/X',Y'\in \B$. One can show that $\A/\B$ is an abelian category. The natural functor 
$F: \A\to \A/\B$ is then called the {\bf Serre quotient functor}. This functor is essentially surjective, its kernel is $\B$, and it maps simple objects to simple objects or zero. Thus $F$ defines a bijection between simple objects of $\A$ not contained in $\B$ and simple objects of $\A/\B$. 

For example, 
if $X$ is a variety, $Z\subset X$ a closed subvariety, ${\rm Qcoh}(X)$ the category of quasicoherent sheaves on $X$ and ${\rm Qcoh}_Z(X)$ the full subcategory of sheaves supported on $Z$ 
then ${\rm Qcoh}(X)/{\rm Qcoh}_Z(X)\cong {\rm Qcoh}(X\setminus Z)$. 
The corresponding Serre quotient functor is the restriction $M\mapsto M|_{X\setminus Z}$. 

Now assume that $X$ is smooth. Let $j: X\setminus Z\hookrightarrow X$ be the open embedding. Then we have a {\bf restriction functor} on $D$-modules 
$$
j^!: \M(X)\to \M(X\setminus Z)
$$ 
which is the usual restriction functor at the level of sheaves; it is also called the {\bf inverse image} or {\bf pull-back} functor, since it is a special case of the inverse image functor defined above. Thus $j^!(M)=0$ if and only if $M$ is supported on $Z$ and the functor $j^!$ is a Serre quotient functor which induces an equivalence $\M(X)/\M_Z(X)\cong \M(X\setminus Z)$.

The functor $j^!$ has a right adjoint {\bf direct image (or push-forward) functor} 
$$
j_*: \M(X\setminus Z)\to \M(X),
$$ 
which is just the sheaf-theoretic direct image (=push-forward). Namely, for an affine open $U\subset X$,  $j_*M(U):=M(U\setminus Z)$ regarded as a module over $D(U)\subset D(U\setminus Z)$. While the functor $j^!$ is exact, the functor 
$j_*$ is only left exact, in general (as so is the push-forward functor for sheaves). 
In particular, $j_*$ is {\bf not} the (right exact) direct image defined above since the morphism $j$ is not affine, in general; rather it is the zeroth cohomology of the full direct image functor defined 
on the derived category of $D$-modules, which we will not discuss here. They do agree, however, when $j$ is affine (e.g., when $Z$ 
is a hypersurface).

In particular, $j^!$ defines a bijection between isomorphism classes of simple $D_X$-modules which are not supported on $Z$ and simple $D_{X\setminus Z}$-modules, given by $M\mapsto j^!M$. 

The inverse map is defined as follows. Given $L\in \M(X\setminus Z)$, consider the $D$-module $j_*L$. Since $j_*$ is right adjoint to $j^!$, the module $j_*L$ does not contain nonzero submodules supported on $Z$. Now define $j_{!*}L$ to be the intersection of all submodules $N$ of $j_*L$ such that 
$j_*L/N$ is supported on $Z$. This gives rise to a functor $j_{!*}: \M(X\setminus Z)\to \M(X)$ 
(not left or right exact in general). Then if $L$ is irreducible, so is $j_{!*}L$, and $j^!j_{!*}L\cong L$, while for $M\in \M(X)$ irreducible and not supported on $Z$ we have $j_{!*}j^!M\cong M$. The functor $j_{!*}$ is called the {\bf Goresky-MacPherson extension} or {\bf minimal (or intermediate) extension} functor. 

\begin{proposition}\label{irrsup} The support of an irreducible $D$-module is irreducible. 
\end{proposition} 

\begin{proof} Let $M$ be a $D_X$-module with support $Z$. Assume that 
$Z$ is reducible: $Z=Z_1\cup Z_2$ where $Z_1$ is an irreducible component of $Z$ and $Z_2$ the union of all the other components. Let $Y=Z_1\cap Z_2$, a proper subset in $Z_1$ and $Z_2$. 
Let $Z^\circ=Z\setminus Y$, $Z_i^\circ=Z_i\setminus Y$ and $X^\circ=X\setminus Y$. Then $Z^\circ=Z_1^\circ\cup Z_2^\circ$ is disconnected: $Z_1^\circ,Z_2^\circ$ are closed nonempty subsets of $Z^\circ$ and $Z_1^\circ\cap Z_2^\circ=\emptyset$. Let $M_1,M_2$ be the sums of all subsheaves of 
$M|_{X^\circ}$ which are killed by localization away from $Z_1^\circ$, respectively $Z_2^\circ$. It is easy to show that $M_i$ are nonzero submodules of $M|_{X^\circ}$ and $M|_{X^\circ}\cong M_1\oplus M_2$. Thus $M|_{X^\circ}$ is reducible and hence so is $M$. 
\end{proof} 

\subsection{Kashiwara's theorem} 

Let $X$ be a smooth variety and $Z\subset X$ a smooth closed subvariety with closed embedding $i: Z\hookrightarrow X$. For $M\in \M(X)$ define $M_Z$ to be the sheaf on $X$ whose sections on an affine open set $U\subset X$ are the vectors in $M(U)$ annihilated by regular functions on $U$ vanishing on $Z$. Thus the $\O(U)$-action on $M_Z(U)$ factors through $\O(Z\cap U)$. Also it is easy to see that $M_Z(U)$ depends only on $Z\cap U$, i.e., it gives rise to a quasicoherent sheaf 
$i^\dagger M$ on $Z$ with sections 
$$
i^\dagger M(V):=M_Z(U)
$$
for affine open $U\subset X$ such that $V=Z\cap U$. 
Moreover, if $v$ is a vector field on $U$ tangent to $V$ then 
$v$ preserves the ideal of $V$, hence acts naturally on $i^\dagger M(V)$. Furthermore, 
the action of $v$ on this space depends only on the vector field on $V$ induced by $v$. Thus 
$i^\dagger M(V)$ carries an action of the Lie algebra ${\rm Vect}(V)$. Together with the action of $\O(V)$, this defines an action of $D(V)$ on $i^\dagger M(V)$. We conclude that $i^\dagger M$ is naturally a $D_Z$-module. Thus we have defined a left exact functor 
$$
i^\dagger: \M(X)\to \M(Z).
$$ 
It is called the {\bf shifted inverse image} functor. This terminology is motivated by the following exercise.

\begin{exercise} Show that $i^\dagger=L^d i^\bullet$ and $i^\bullet=R^di^\dagger$, where $L^d,R^d$ are the $d$-th left, respectively right derived functors and $d=\dim X-\dim Z$. 
\end{exercise} 

\begin{theorem}\label{kashth} (Kashiwara) The functor $i^\dagger$ is an equivalence of categories $\M_Z(X)\to \M(Z)$. 
\end{theorem}  

The proof is not difficult, but we will skip it (see \cite{HTT}). 

The inverse of the functor $i^\dagger$ is called the {\bf direct image} functor and denoted $i_*: \M(Z)\to \M_Z(X)$, as it is a special case of the direct image functor defined above for affine morphisms.  
If we view $i_*$ as a functor $\M(Z)\to \M(X)$ then it has both left and right adjoint, which are $i^!$ and $i^\dagger$, respectively. 

Let us give a prototypical example. 

\begin{example} Let $X=\Bbb A^1$, $Z=\lbrace  0\rbrace$. Then $\M(Z)={\rm Vect}$
and $i_*(V)=V\otimes \delta_0$. So in this case Kashiwara's theorem reduces to the claim 
that $\Ext^1(\delta_0,\delta_0)=0$. 
\end{example} 
 
\begin{remark} We note that the above formalism and results extend in a straightforward manner to the case of twisted $D$-modules.   
\end{remark}  
 
\subsection{Equivariant $D$-modules}

Let $X$ be an algebraic variety with an action of an affine algebraic group $G$. Let us review the notion of a $G$-{\bf equivariant quasicoherent sheaf} on $X$. Roughly speaking, this is a quasicoherent sheaf $\mathcal{E}$ on $X$ equipped with a system of isomorphisms $\phi_g: g(\mathcal E)\cong \mathcal E, g\in G$ such that $\phi_{gh}=\phi_g\circ g(\phi_h)$ and $\phi_g$ depends on $g$ algebraically. To give a formal definition, note that the group structure gives us a multiplication map $m:G\times G\twoheadrightarrow G,$ and the action of $G$ gives us a map $\rho:G\times X\twoheadrightarrow X$. We have a commutative diagram
$$
\xymatrix{
& G\times G\times X \ar[dl]^{m\times\operatorname{id}}_{} \ar[dr]_{\operatorname{id}\times\rho} & \\
G\times X \ar[dr]^{\rho} && G\times X \ar[dl]_{\rho}\\
& X &.
}
$$

\begin{definition}
A $G$-{\bf equivariant quasicoherent sheaf} on $X$ is a quasicoherent sheaf $\mathcal{E}$ on $X$ equipped with an isomorphism 
$$
\phi: \rho^*\mathcal{E}\cong \O_G\boxtimes\mathcal{E}
$$ 
making the following diagram commutative:

\centerline{
\xymatrixcolsep{5pc}
\xymatrix{
(\operatorname{id}\times\rho)^*\rho^*\mathcal{E}\ar@{=}[d]\ar[r]^-{(\operatorname{id}\times\rho)^*\phi}&(\operatorname{id}\times\rho)^*(\O_G\boxtimes\mathcal{E})\ar[r]&\O_G\boxtimes\rho^*\mathcal{E}\ar[d]^-{\O_G\boxtimes\phi}\\
(m\times\operatorname{id})^*\rho^*\mathcal{E}\ar[r]^-{(m\times\operatorname{id})^*\phi}&(m\times\operatorname{id})^*(\O_G\boxtimes\mathcal{E})\ar[r]&\O_G\boxtimes \O_G\boxtimes\mathcal{E}\\
}}
\end{definition}

Thus $\phi$ comprises all the isomorphisms $\phi_g$, which therefore satisfy the equality $\phi_{gh}=\phi_g\circ g(\phi_h)$ and depend on $g$ algebraically.  

We now wish to define the notion of a $G$-{\bf equivariant $D_X$-module}. To this end, recall
that for any $D_X$-module $\mathcal E$, the quasicoherent sheaf $\rho^*\mathcal E$ carries a natural structure of a $D_{G\times X}$-module (the $D$-module inverse image).  We now make the following definition. 

\begin{definition}
A {\bf weakly $G$-equivariant $D$-module} on $X$ is a $D_X$-module $\mathcal{E}$ with a $G$-equivariant quasicoherent sheaf structure, where $\phi$ is $D_X$-linear.
\end{definition} 

Note that if $\mathcal E$ is a weakly equivariant $D_X$-module then we have two (in general, different) actions of $\mathfrak{g}=\operatorname{Lie}(G)$ on $\mathcal{E}.$ First of all, the $G$-action on $X$ gives us maps $\mathfrak{g}\to\operatorname{Vect}(X)\to D(X),$ and so the $D$-module structure on $\mathcal{E}$ gives us  a $\mathfrak{g}$-action $x\mapsto b_0(x)$ on $\mathcal{E}$. Note that this action does not depend on the choice of the weakly equivariant structure $\phi.$

On the other hand, we have a $\mathfrak{g}$-action on $\O_G\boxtimes\mathcal{E}$ coming from the $G$-action on $G\times X$ given by $g\cdot(h,x)=(gh,x).$ Translating this along $\phi$, we get a $\mathfrak{g}$-action on $\rho^*\mathcal{E}$. Restricting to $1\times X$, this gives us another $\mathfrak{g}$-action $x\mapsto b_\phi(x)$ on $\mathcal{E}.$

\begin{definition}
A (strongly) $G$-{\bf equivariant $D_X$-module} is a weakly $G$-equivariant $D_X$-module where these two $\mathfrak{g}$-actions agree: $b_\phi=b_0$ (or, equivalently, where $\phi$ is $D_{G\times X}$-linear.)
\end{definition} 

In general, since $[b_0(x),L]=[b_\phi(x),L]$ for $L\in D_X$, 
the operator $\rho_\phi(x):=b_\phi(x)-b_0(x)$ is a $D$-module endomorphism of $\mathcal E$. 
Moreover, it is easy to see that $\rho_\phi$ is a Lie algebra homomorphism 
$\g\to \End(\mathcal E)$. In particular, if $\mathcal E$ is irreducible then 
by Dixmier's lemma (Lemma \ref{Dixlemm}), $\End(\mathcal E)=\Bbb C$, so $\rho_\phi$ is just a character of $\g$. 
Thus if $\g=[\g,\g]$ is perfect (for example, semisimple) then every weakly $G$-equivariant irreducible $D_X$-module is actually (strongly) $G$-equivariant.   

\begin{remark}
A given $D_X$-module may have many weakly $G$-equivariant structures, but if $G$ is connected, then it can only have one $G$-equivariant structure. This is because the $\mathfrak{g}$-action on $\mathcal{E}$ is determined by the map $\mathfrak{g}\to D(X)$ and this action can be integrated to a $G$-equivariant structure in a unique way (recall that we always work over a field of characteristic $0$.)

Furthermore, any $D_X$-linear map of $G$-equivariant $D_X$-modules is automatically compatible with the $G$-action. This is because such a map is necessarily $\mathfrak{g}$-linear, which implies that it is in fact $G$-linear. These two facts combined show that the category of $G$-equivariant $D_X$-modules is a full subcategory of the category of $D_X$-modules. Stated another way, $G$-equivariance of a $D_X$-module is a property, not a structure.
\end{remark}

\begin{example}\label{previex}
Consider the case where $X$ is a point. Then $D_X\cong \Bbb C$ and so a $D_X$-module is just a vector space. A weakly $G$-equivariant $D_X$-module is then simply a locally algebraic representation of $G$. This representation gives a $G$-equivariant structure if and only if $\mathfrak{g}$ acts by $0$, i.e.,  the connected component of the identity $G_0\subset G$ acts trivially. Thus a $G$-equivariant $D_X$-module is just a representation of the component group $G/G_0$. Conversely, any locally algebraic representation $V$ of $G$ gives rise to a weakly $G$-equivariant $D$-module on $X$ which is equivariant  iff $G_0$ acts trivially on $V$, so that $V$ is a representation of $G/G_0$. 
\end{example} 

\begin{example}\label{homspace} Let $X=G/H$, where $G$ is an algebraic group and $H$ a closed subgroup of $G$. 
Then we claim that a $G$-equivariant $D_X$-module is the same thing as an $H$-equivariant 
$D$-module on a point, i.e., a representation of the component group $H/H_0$. Indeed, given an $H/H_0$-module $V$, we can define a $G$-equivariant vector bundle 
$$
(G\times V)/H\to X=G/H,
$$ 
where 
$H$ acts on $G\times V$ via $(g,v)h=(gh,h^{-1}v)$. Note that this can be written as 
$\frac{(G/H_0)\times V}{H/H_0}$ (as $H_0$ acts on $V$ trivially). This shows that 
this vector bundle has a natural flat connection, i.e. is a  $D_X$-module $L(X,V)$, which is clearly 
$G$-equivariant. The assignment $V\mapsto L(X,V)$ 
is the desired equivalence. In the case $H=G$, this reduces to Example \ref{previex}. 
\end{example} 

\begin{exercise} (i) Define the algebraic group $L:=G\times_{G/G_0}H/H_0$ of pairs 
$(g,h)$, $g\in G$, $h\in H/H_0$ which map to the same element of $G/G_0$; thus we have a short exact sequence 
$$
1\to G_0\to L\to H/H_0\to 1.
$$
Show that the category of weakly $G$-equivariant $D$-modules on $G/H$ 
is naturally equivalent to the category of representations of $L$, such that the 
subcategory of strongly $G$-equivariant $D$-modules 
is identified with the subcategory of representations of $L$ pulled back from the second factor $H/H_0$ (i.e., those with trivial action of $G_0$), and the subcategory of modules of the form $\O(G/H)\otimes V$ 
where $V$ is a $G$-module is identified with the category of representations of $L$ pulled back from the first factor $G$.

(ii) Let $\Delta: H\to L$ be the map defined by $\Delta(h)=(h,h)$. 
Show that the forgetful functor from weakly $G$-equivariant $D$-modules on $G/H$ 
to $G$-equivariant quasicoherent sheaves on $G/H$ corresponds to the pullback functor $\Delta^*$. 
\end{exercise} 

\begin{exercise}\label{nateq} Let $X$ be a smooth variety with an action of an affine algebraic group $G$ and $H\subset G$ be a closed subgroup. Show that the category of $H$-equivariant $D$-modules on $X$ is naturally equivalent to the category of $G$-equivariant $D$-modules on $X\times G/H$ with diagonal action of $G$ (note that when $X$ is a point, this reduces to 
Example \ref{homspace}). 
\end{exercise}

\begin{exercise} Let $X$ be a principal $G$-bundle over a smooth variety $Y$. 
Show that the category of $G$-equivariant $D_X$-modules is naturally equivalent to 
the category of $D_Y$-modules. Namely, given a $G$-equivariant $D_X$-module $M$, 
for an affine open set $U\subset Y$ let $\widetilde U$ be the preimage of $U$ in $X$ 
and let $\overline M(U):=M(\widetilde U)^G$. Then $\overline M$ is a $D_Y$-module, and the assignment $M\mapsto \overline M$ is a desired equivalence.  
\end{exercise} 

The notion of a weakly equivariant $D$-module often arises in the following setting. Let $T$ be an algebraic torus and let $\widetilde{X}$ be a principal $T$-bundle over $X$. 

\begin{definition} 
A {\bf monodromic $D$-module over $X$} (with respect to the bundle $\widetilde{X}\twoheadrightarrow X$) is a weakly $T$-equivariant $D_{\widetilde{X}}$-module.
\end{definition} 

\begin{example} A monodromic $D$-module over $X$ with $\rho_\phi=\lambda\in {\rm Lie}(T)^*$ 
is the same thing as a $\lambda$-twisted $D$-module on $X$, i.e., a $D_{\lambda,X}$-module. 
\end{example} 
 
\begin{proposition}\label{daffine}
Assume that $X$ is a $D$-affine variety and that $K$ is an affine algebraic group acting on $X$. Let $D(X)$ be the ring of global sections of $D_X.$ Then the category of $K$-equivariant $D_X$-modules is equivalent to the category of $D(X)$-modules $M$ endowed with a locally finite $K$-action whose differential coincides with the action of ${\rm Lie}(K)$ on $M$ coming from the map ${\rm Lie}(K)\to D(X).$
\end{proposition}

\begin{exercise} Prove Proposition \ref{daffine}. 
\end{exercise} 

In particular, by the Beilinson-Bernstein localization theorem, Proposition \ref{daffine} applies to $X=\mathcal F\cong G/B$ and $K$ a closed subgroup of $G$, and moreover it extends to the case of $\lambda$-twisted differential operators on $\mathcal F$ for $\lambda\in \h^*$ with $\lambda-\rho$ regular and antidominant. Thus we get 

\begin{corollary}\label{equica} If $\lambda-\rho$ is regular and antidominant then the functors $\Gamma,{\rm Loc}$ restrict to mutually inverse equivalences between the category of $(\g,K)$-modules with infinitesimal character $\chi_{\lambda-\rho}$ and the category of $K$-equivariant $D_\lambda$-modules on $\mathcal F$. 
\end{corollary}  

\section{\bf Applications of D-modules to representation theory} 
 
\subsection{Classification of irreducible equivariant $D$-modules for actions with finitely many orbits}

\begin{theorem}\label{orb} Let $X$ be a smooth variety and $K$ a connected algebraic group acting on $X$ with finitely many orbits. Then there are finitely many irreducible $K$-equivariant $D$-modules on $X$. Namely, they are parametrized by pairs $(O,V)$ where $O$ is an orbit of $K$ on $X$ and $V$ is an irreducible representation of the component group $H/H_0$ of the stabilizer $H:=K_x$ for $x\in O$, $(O,V)\mapsto M(O,V)$.  
\end{theorem}  

\begin{proof} Let $M$ be an irreducible $K$-equivariant $D$-module on $X$. Then by Proposition \ref{irrsup}, the support $Z$ of $M$ is irreducible. Thus $Z=\overline O$ for a single orbit $O$ of $K$. Let $Z_0=\overline O\setminus O$, and $U=X\setminus Z_0$. Then $U$ is a $K$-stable open subset of $X$ and $O$ is closed in $U$. Also $M|_U$ is a simple $D_U$-module supported on $O$. Let $i: O\hookrightarrow U$ be the closed embedding. By Kashiwara's theorem (Theorem \ref{kashth}) $i^\dagger M$ is a simple $K$-equivariant $D$-module on $O$. Thus by Example \ref{homspace} $i^\dagger M=L(O,V)$ for some irreducible representation $V$ of the component group of the stabilizer $K_x$, $x\in O$. Also it is clear that $L(O,V)$ gives rise to a simple $K$-equivariant $D$-module on $X$, namely, $M(O,V):=j_{!*}i_*L(O,V)$, where $j: U\hookrightarrow X$ is the open embedding. This proves the theorem. 
\end{proof} 
 
\begin{remark} Theorem \ref{orb} can be extended in a straightforward way to weakly equivariant $D$-modules. In this case, recall that the weakly equivariant structure on an irreducible $D$-module $M$ defines a character 
$\rho:  \mathfrak{k}\to \Bbb C$, where $\mathfrak{k}={\rm Lie}K$. Theorem \ref{orb} 
then holds with the only change: rather than being a representation of $H/H_0$, 
$V$ now needs to be a representation of $H$ in which ${\rm Lie}(H)$ acts by the character $\rho$. 
The proof is analogous to the case $\rho=0$.

In particular, this applies to the case of twisted $D$-modules. In this case we have a principal $T$-bundle $p: \widetilde X\to X$ and a character $\lambda\in \mathfrak{t}^*$, $\mathfrak{t}={\rm Lie}(T)$. 
Suppose $K$ acts on $X$ preserving this bundle; i.e., it acts on $\widetilde X$ and commutes with $T$. 
So we have a $K\times T$-action on $\widetilde X$ and a $K$-equivariant $\lambda$-twisted $D$-module on $X$ is just a weakly $K\times T$-equivariant $D$-module on $\widetilde X$ with $\rho({\rm k},t):=\lambda(t)$. Now, for every $K$-orbit $O$ on $X$, we have the stabilizer $K_x$, $x\in O$, and a homomorphism $\xi_x: K_x\to T$ defined by the condition that $(g,\xi_x(g))$ acts trivially on $p^{-1}(x)$ for $g\in K_x$. This defines a character 
$\lambda_x=\lambda\circ d\xi_x$ of ${\rm Lie}(K_x)$, 
and the simple $K$-equivariant $D_\lambda$-modules on $X$ are 
$M(O,V)$ where $V$ is an irreducible representation of $K_x$ with ${\rm Lie}(K_x)$ acting by 
the character $\lambda_x$. 
\end{remark} 

\subsection{Classification of irreducible Harish-Chandra modules}  

Let $G_{\Bbb R}$ be a connected real semisimple algebraic group, $K_{\Bbb R}\subset G_{\Bbb R}$ a maximal compact subgroup, $G,K\subset G$ their complexifications. By Corollary \ref{equica}, if $\lambda-\rho$ is regular and antidominant then the Beilinson-Bernstein equivalence restricts to an equivalence between the category of $(\g,K)$-modules with infinitesimal character $\chi_{\lambda-\rho}$ and the category of $K$-equivariant $D_\lambda$-modules on $\mathcal F=G/B$.  

\begin{proposition}\label{finma} The group $K$ acts on $\mathcal F$ with finitely many orbits. 
\end{proposition}  

We will not give a proof of this proposition. For the proof and description of the set of orbits, see \cite{RS}. 

Proposition \ref{finma} along with Theorem \ref{orb} allows us to classify irreducible $(\g,K)$-modules (i.e., Harish-Chandra modules) for a regular infinitesimal character (the general case can be handled similarly). 

Namely, let $H\subset B\subset G$ be a maximal torus and Borel subgroup of $G$; so $H\cong B/[B,B]$. Note that $K\times H$ acts on $\widetilde{\mathcal F}=G/[B,B]$. So for a $K$-orbit $O$ on $\mathcal F=G/B$ and $x\in O$, we have a homomorphism $\xi_x: K_x\to H$ such that $(g,\xi_x(g))$ acts trivially on the fiber over $x$ in $\widetilde{\mathcal F}$ for $g\in K_x$. 

Let $\chi$ be a regular infinitesimal character for $\g$ and $\lambda$ be a weight with $\lambda-\rho$ regular and antidominant and $\chi=\chi_{\lambda-\rho}$ (take $\lambda-\rho$ minimal in the corresponding $W$-orbit under $\le$). 

\begin{theorem}\label{irrehc} Irreducible $(\g,K)$-modules with (pure) infinitesimal character $\chi$ are $\pi(O,V)$ 
where $O$ is a $K$-orbit on $\mathcal F$ and $V$ an irreducible representation of $K_x$, $x\in O$ such that ${\rm Lie}(K_x)$ acts via the character $\lambda_x$. 
Namely, $\pi(O,V)$ corresponds to $M(O,V)$ under the Beilinson-Bernstein equivalence.  
\end{theorem} 

\begin{example} Let $G_{\Bbb R}=SL_2(\Bbb R)$. Let $\lambda\in \Bbb C$, $\lambda\notin \Bbb Z_{>0}$ and set $\chi=\chi_{\lambda-1}$ (so $\chi\ne \chi_0$). 
In this case $\mathcal F=\Bbb C\Bbb P^1$
is the Riemann sphere, and $K=\Bbb C^\times$ acts by ${\rm k}\circ z:={\rm k}^2z$. 
Thus we have three orbits: $0,\infty$, and $\Bbb C^\times$. 
For the orbit $\Bbb C^\times$ we have $K_x=\Bbb Z/2$, so we have two irreducible 
representations $V=\Bbb C_\pm$, which generically correspond to principal series representations $\pi(\Bbb C^\times,V_{\pm})=P_\pm(1-\lambda)$ (see Section \ref{sl2R}). The other two orbits have a connected stabilizer, and $\lambda_x=\pm\lambda$. 
Thus for such orbits representations exist only for $\lambda\in \Bbb Z_{\le 0}$. 
It is easy to see that these are exactly the discrete series representations $M_{\lambda-2}^+, M_{-\lambda+2}^-$. Also for such points one of the principal series representations is reducible 
($P_+(1-\lambda)$ for even $\lambda$ and $P_-(1-\lambda)$ for odd $\lambda$) and $\pi(\Bbb C^\times,V_+)$, respectively $\pi(\Bbb C^\times,V_-)$ is actually the finite-dimensional representation $L_{-\lambda}$. Namely, in the Grothendieck group we have 
$[P_\pm(1-\lambda)]=[M_{\lambda-2}^+]+[M_{-\lambda+2}^-]+[L_{-\lambda}]$, while 
$P_\mp(1-\lambda)$ is simple. Thus we have four irreducible representations in this case. 

Note that this agrees with our classification of irreducible representations of $SL_2(\Bbb R)$ for regular infinitesimal characters discussed in Section \ref{sl2R}. 
\end{example} 

\begin{example} Let $G$ be a simply connected complex semisimple group regarded as a real group. Then its maximal compact subgroup is $G_c$, so its complexification is $G$, and $G_{\Bbb C}=G\times G$, so that the inclusion $(G_c)_{\Bbb C}=G\hookrightarrow G_{\Bbb C}=G\times G$
is the diagonal embedding. The flag variety is $\mathcal F\times \mathcal F=G/B\times G/B$. 
Thus Harish-Chandra bimodules with infinitesimal character $(\chi_{\mu-\rho},\chi_{\lambda-\rho})$ 
for antidominant $\lambda,\mu$ are $\pi(O,V)$ where 
$O$ runs over orbits of $G$ on $G/B\times G/B$ and $V$ 
over appropriate representations of isotropy groups. Note that 
orbits of $G$ on $G/B\times G/B$ are in a natural bijection with orbits of $B$ on $G/B$, which are the Schubert cells 
$C_w$ labeled by $w\in W$. One can check that the condition for existence of $V$ on the orbit $C_w$ 
is that $\lambda-w\mu$ is integral, and then $V$ is unique (as the isotropy groups are connected in this case). Thus we find that the irreducible Harish-Chandra bimodules with such infinitesimal character are labeled by elements $w$ such that $\lambda-w\mu\in P$, which agrees with the classification 
we obtained in Subsection \ref{classsim}.  
\end{example} 

\begin{exercise} Classify irreducible Harish-Chandra modules for $SL_3(\Bbb R)$ with a regular infinitesimal character. 

{\bf Hint.} Classify orbits of $SO_3(\Bbb C)$ on $SL_3(\Bbb C)/B$. This is equivalent to classification of flags in a 3-dimensional complex inner product space $E$ under the action of $SO(E)$. Then classify possible representations $V$ of the isotropy group for each orbit. 
\end{exercise} 

\begin{remark} The $K$-orbits on $G/B$ can be classified in explicit combinatorial terms. 
Together with Theorem \ref{irrehc}, this leads to an alternative proof, using the localization theorem, of the {\bf Langlands classification} 
of irreducible Harish-Chandra modules (obtained by Langlands in 1973 by a different method, 8 years before the localization theorem was proved, \cite{La}). This classification requires a serious separate discussion which is beyond the scope of these notes. 
\end{remark} 

\subsection{Applications to category $\mathcal O$} 

Let us now see how this approach allows us to study category $\O$ for a semisimple Lie algebra $\g$. 

Consider the category $\C$ of weakly $B\times B$-equivariant finitely generated $D$-modules on $G$ which are equivariant under $[B,B]\times [B,B]$ (it is easy to see that such modules have finite length). Thus for $M\in \C$, we have a homomorphism $\rho: \h\oplus \h\to \End(M)$, so $M=\oplus_{\mu,\lambda} M(\mu,\lambda)$ where $M(\mu,\lambda)$ is the generalized eigenspace 
for $\h\oplus \h$ with eigenvalue $(\mu,\lambda)\in \h^*\times \h^*$. Thus we have a decomposition 
$\C=\oplus_{\mu,\lambda}\C_{\mu,\lambda}$. 

Let $\C_{[\mu],\lambda},\C_{\mu,[\lambda]},\C_{[\mu],[\lambda]}$ 
be the full subcategories of $\C_{\mu,\lambda}$ consisting of objects on which the eigenvalues in square brackets are pure (without Jordan blocks). Thus we have 
$$
\C_{\mu,\lambda}\supset \C_{[\mu],\lambda},\C_{\mu,[\lambda]}\supset \C_{[\mu],[\lambda]}
$$ 
and all simple objects of $\C_{\mu,\lambda}$ are contained in $\C_{[\mu],[\lambda]}$. 
These objects are labeled by Bruhat cells $BwB\subset G$, $w\in W$ and representations $V$ 
of the isotropy group satisfying an appropriate condition. As before, the condition for $V$ to exist is 
that $\lambda-w\mu\in P$, thus $\C_{\mu,\lambda}=0$ unless $\lambda-w\mu\in P$ for some $w\in W$. 

We also see that $\C_{\lambda,\mu}\cong \C_{\lambda+\beta,\mu+\gamma}$ 
for $\beta,\gamma\in P$, and the same applies to its subcategories. 

Let us now try to describe these categories representation-theoretically. 
To this end, note that we may interpret $\C_{[\mu],[\lambda]}$ as the category 
of weakly $B$-equivariant $D_\lambda$-modules on $G/B$ with (pure) equivariance character $\mu$. 
So if $\lambda$ is antidominant, we get that $\C_{[\mu],[\lambda]}$ is 
equivalent to the full subcategory $\O_{[\mu],[\lambda]}$
 of the category $\O_{\chi_{\lambda-\rho}}$ of objects with pure infinitesimal character $\chi_{\lambda-\rho}$ 
and weights in $\mu+P$. Similarly, 
$\C_{\mu,\lambda},\C_{[\mu],\lambda},\C_{\mu,[\lambda]}$
are equivalent to 
$\O_{\mu,\lambda},\O_{[\mu],\lambda},\O_{\mu,[\lambda]}$, where the corresponding (infinitesimal) character is pure if square brackets are present and generalized if not. 

Now note that flipping left and right, we get equivalences 
$\C_{\lambda,\mu}\cong \C_{\mu,\lambda}$, 
$\C_{[\lambda],\mu}\cong \C_{\mu,[\lambda]}$,
$\C_{\lambda,[\mu]}\cong \C_{[\mu],\lambda}$,
$\C_{[\lambda],[\mu]}\cong \C_{[\mu],[\lambda]}$. 
If $\lambda,\mu$ are both antidominant, this yields equivalences of representation categories
$\O_{\lambda,\mu}\cong \O_{\mu,\lambda}$, 
$\O_{[\lambda],\mu}\cong \O_{\mu,[\lambda]}$,
$\O_{\lambda,[\mu]}\cong \O_{[\mu],\lambda}$,
$\O_{[\lambda],[\mu]}\cong \O_{[\mu],[\lambda]}$. 
While the first equivalence is easy to see 
representation theoretically using translation functors, the others are not. 
They are clear from geometry but somewhat mysterious from the viewpoint of representation theory
(although they can be understood using the Bernstein-Gelfand equivalence between category $\O$ and the category of Harish-Chandra bimodules, Theorem \ref{equivaa}). 

\begin{example} If $\lambda,\mu\in P$, these categories are independent of $\lambda,\mu$. 
Namely, let $\O_0$ be the category $\O$ for the trivial generalized infinitesimal character, and $\widetilde \O_0$ be its Serre closure (the category of modules admitting a finite filtration whose successive quotients are in $\O_0$; i.e. the action of $\h$ is not necessarily diagonalizable but is only assumed locally finite). We may also define the category $\O_0^*$ of modules in $\widetilde \O_0$ which have pure infinitesimal character, and $\overline \O_0\subset \O_0$ of modules with both pure infinitesimal character and diagonalizable action of $\h$. Then the above four categories are exactly $\widetilde \O_0,\O_0,\O_0^*,\overline \O_0$. 
In particular, we obtain an equivalence $\O_0\cong \O_0^*$ which is not obvious representation-theoretically. 
\end{example} 

Finally, we note that Exercise \ref{nateq} applied to $X=G/B$ and $H=B$
gives a transparent geometric proof of Theorem \ref{equivaa}.

\end{document}